\documentclass[a4paper]{article}
\usepackage[T1]{fontenc}
\usepackage{amsmath,amssymb,amsthm,mathrsfs}
\usepackage{cite}
\usepackage{mathtools}
\usepackage{enumitem}

\usepackage[pdftex,bookmarksopen=true,bookmarks=true,unicode,setpagesize]{hyperref}
\hypersetup{colorlinks=true,linkcolor=black,citecolor=black}

\usepackage{pgf,tikz}
\usetikzlibrary{arrows,patterns}

\let\equation=\gather
\let\endequation=\endgather

\allowdisplaybreaks[3]
\numberwithin{equation}{section}

\makeatletter
\renewcommand*{\@fnsymbol}[1]{\ensuremath{\ifcase#1\or 1\or 2\or
   3\else\@ctrerr\fi}}
\makeatother

\newtheorem{theorem}{Theorem}[section]
\newtheorem{corollary}[theorem]{Corollary}
\newtheorem{lemma}[theorem]{Lemma}
\newtheorem{proposition}[theorem]{Proposition}
\theoremstyle{definition}
\newtheorem{definition}[theorem]{Definition}
\newtheorem{example}[theorem]{Example}
\theoremstyle{remark}
\newtheorem{remark}[theorem]{Remark}

\newcounter{assum}
\newenvironment{assum}[1][]{\ifx\newenvironment#1\newenvironment\refstepcounter{assum}\fi\equation\tag{\ensuremath{\mathrm{A}\theassum#1}}}{\endequation}
\DeclareMathOperator*{\esssup}{esssup}
\DeclareMathOperator*{\essinf}{essinf}

\newcommand{\R}{\mathbb{R}}
\newcommand{\X}{{\mathbb{R}^d}}
\renewcommand{\S}{{S^{d-1}}}
\newcommand{\N}{\mathbb{N}}
\newcommand{\Z}{\mathbb{Z}}

\newcommand{\eps}{\varepsilon}
\newcommand{\la}{\lambda}
\newcommand{\La}{\Lambda}
\newcommand{\ka}{\varkappa}
\newcommand{\Tau}{\Upsilon}

\newcommand{\x}{\mathcal{X}}

\newcommand{\D}{\mathscr{D}}
\newcommand{\Mn}{{\mathcal{M}}}
\newcommand{\K}{\mathscr{K}}
\renewcommand{\L}{\mathfrak{L}}
\newcommand{\Tauout}{\mathscr{O}}
\newcommand{\Tauin}{\mathscr{C}}
\newcommand{\1}{1\!\!1}
\newcommand{\m}{{\mathfrak{m}}}
\newcommand{\A}{{\mathfrak{a}}}
\newcommand{\h}{{\mathfrak{h}}}
\newcommand{\T}{\mathfrak{t}}
\newcommand{\An}{\mathcal{A}}

\renewcommand{\Re}{\mathrm{Re}\,}
\renewcommand{\Im}{\mathrm{Im}\,}
\newcommand{\dist}{\mathrm{dist}\,}

\newcommand{\supp}{\mathrm{supp}\,}
\newcommand{\inter}{{\mathrm{int}}}

\newcommand{\Buc}{C_{ub}(\X)}
\newcommand{\BUC}{C_{ub}(\X\times\R_+)}
\newcommand{\M}{{\mathcal{M}_\theta}(\R)}
\newcommand{\Utheta}{{U_\theta}}
\newcommand{\Ltheta}{{L_\theta}}
\newcommand{\Xinf}{\mathcal{X}_{\infty}}
\newcommand{\tXinf}{\tilde{\mathcal{X}}_{\infty}}
\newcommand{\xt}{\mathcal{X}_T}

\newcommand{\Vxi}{\mathcal{V}_\xi}
\newcommand{\Wxi}{\mathcal{W}_\xi}
\newcommand{\Uxi}{\mathcal{U}_\xi}
\newcommand{\locun}{\xRightarrow{\,\mathrm{loc}\ }}

\title{Traveling waves and long-time behavior in a doubly nonlocal Fisher--KPP equation}

\author{Dmitri Finkelshtein\thanks{Department of Mathematics,
Swansea University, Singleton Park, Swansea SA2 8PP, U.K. ({\tt d.l.finkelshtein@swansea.ac.uk}).} \and Yuri Kondratiev\thanks{Fakult\"{a}t
f\"{u}r Mathematik, Universit\"{a}t Bielefeld, Postfach 110 131, 33501 Bielefeld,
Germany ({\tt kondrat@math.uni-bielefeld.de}).} \and Pasha Tkachov\thanks{Fakult\"{a}t
f\"{u}r Mathematik, Universit\"{a}t Bielefeld, Postfach 110 131, 33501 Bielefeld,
Germany ({\tt ptkachov@math.uni-bielefeld.de}).}}

\begin{document}

\maketitle

\begin{abstract}
We consider a Fisher--KPP-type equation, where both diffusion and nonlinear parts are nonlocal, with anisotropic probability kernels. Under minimal conditions on the coefficients, we prove existence, uniqueness, and uniform space-time boundedness of a positive solution. We investigate existence, uniqueness, and asymptotic behavior of monotone traveling waves for the equation. We also describe the existence and main properties of the front of propagation.



\end{abstract}

\tableofcontents

\section{Introduction}
We will deal with the following nonlinear nonlocal evolution equation
\begin{equation}\label{log}
\dfrac{\partial u}{\partial t}(x,t)=\ka^{+}\bigl(a^{+}*u\bigr)(x,t)-mu(x,t)-\ka^{-}u(x,t)\bigl(a^{-}*u\bigr)(x,t)
\end{equation}
with a bounded initial condition $u(x,0)=u_0(x)$, $x\in\X$, $d\geq1$. Constants $m, \ka^\pm$
are assumed to be positive, and $(a^\pm*u)(x,t)$ mean the convolutions (in~$x$)
between $u$ and nonnegative integrable probability kernels $a^\pm=a^\pm(x)\geq0$ on~$\X$; namely,
\[
(a^\pm*u)(x,t)=\int_\X a^\pm(x-y)u(y,t)dy, \qquad \int_\X a^\pm(x)\,dx=1.
\]

The equation \eqref{log} first appeared, probably, in \cite{Dur1988} (a `crabgrass model'), for~$\ka^+a^+=\ka^-a^-$, and later in \cite{BP1997} (a model of spatial ecology), for~different kernels. The meaning of $u(x,t)$ is the (approximate) value of the local density of a system in a point $x\in\X$ at a moment of time $t\geq0$. A short review for the history of derivation of \eqref{log} see in Subsection~\ref{subsec:history} below.
This equation was considered as a spatial (nonhomogeneous) version of the classical logistic (Verhulst) equation
\begin{equation}\label{Ver}
\frac{du}{dt}=(\ka^+-m)u(t)-\ka^-(u(t))^2,
\end{equation}
corresponding to $u(x,t)=u(t)$, $x\in\X$. Of course, in the original logistic model one needs to assume that $\ka^+>m$; then \eqref{Ver} has two stationary nonnegative solutions: unstable $u=0$ and stable $u=\frac{\ka^+-m}{\ka^-}$. For $\ka^+\leq m$, \eqref{Ver} has the unique stationary
stable solution $u=0$.

The equation \eqref{log} can be rewritten as follows:
\begin{equation}\label{2nl}
  \dfrac{\partial u}{\partial t}(x,t)=(L_{a^+}u)(x,t)+F(u,a^-*u)(x,t),
\end{equation}
where, for a bounded function $v$ on $\X$, the operator
 \begin{equation}\label{jump}
( L_{a^+}v)(x)=\ka^{+}\int_\X a^+(x-y)[v(y)-v(x)]\,dy
 \end{equation}
describes the so-called nonlocal diffusion, see e.g. \cite{AMRT2010} and references below, and $F$ is a mapping on bounded functions, given by
\begin{equation}\label{nonlinearoper}
  F(v_1,v_2)(x)=\ka^{-}v_1(x)\bigl(\theta-v_2(x)\bigr), \qquad \theta=\frac{\ka^+-m}{\ka^-}.
\end{equation}
For the known results about \eqref{2nl}, one can refer to \cite{FM2004,FKK2009,FKKozK2014}, in the general case; to \cite{YY2013,PS2005,WZ2006}, in the case $\theta>0$, i.e. $\ka^+>m$, see also details below; and to \cite{TW2011,TW2011a}, for $\ka^+=m$.

The equation \eqref{2nl} may be addressed to a doubly nonlocal Fisher--KPP equation.
Recall that the classical Fisher--KPP (Kolmogorov--Petrovski--Piskunov) equation in $\X$
goes back to \cite{Fis1937,KPP1937} and has the form
\begin{equation}\label{kpp}
\dfrac{\partial u}{\partial t}(x,t)=\Delta u(x,t)+f\bigl(v(x,t)\bigr),
\end{equation}
see the seminal paper \cite{AW1978}. Here $\Delta$ is the Laplace operator on $\X$, and $f$ is a nonlinear monostable function on $\R$: namely, let $\theta>0$, cf. \eqref{nonlinearoper}, then we assume that $f(0)=f(\theta)=0$, $f'(0)>0$, $f'(\theta)<0$; for example,
\begin{equation}\label{nonlinfunc}
f(s)= \ka^- s(\theta-s), \quad s\geq0.
\end{equation}

The Fisher--KPP equation may be informally obtained from \eqref{2nl} under two scaling procedures.
Namely, let
\[
\int_\X a^+(y) (|y_i|+|y_i|^2)\,dy<\infty, \qquad
\int_\X a^+(y) y_i\,dy=0, \quad 1\leq i\leq d,
\]
and consider the following scaling (for small $\eps>0$)
\begin{equation}\label{scaling}
a^+(x)\longmapsto a_\eps^+(x):=\eps^{-d} a^+(\eps^{-1}x), \ x\in\X, \qquad \ka^+\longmapsto \eps^{-2}.
\end{equation}
Then, we obtain
\begin{align*}
(L_{a_\eps^+}v)(x)&=\eps^{-2}\int_\X a^+(y)[v(\eps y+x)-v(x)]\,dy\notag\\
&=\frac{1}{2}\int_\X a^+(y)\bigl( v''(x) y, y\bigr)_\X\,dy+ o(\eps)=\alpha\Delta v(x)+ o(\eps),
\end{align*}
where $(\cdot,\cdot)_\X$ is the scalar product on $\X$, and
$\alpha=\frac{1}{2}\int_\X a^+(y) y_1^2\,dy$.

As a result, one gets the following nonlocal Fisher--KPP equation
\begin{equation}\label{l+nl}
\dfrac{\partial u}{\partial t}  =\alpha \Delta u+F(u,a^-*u).
\end{equation}
For the theory of such equations with different mappings $F$ `similar' to \eqref{nonlinearoper}, see e.g. \cite{Gou2000,BNPR2009,FZ2011,GVA2006,ABVV2010,AK2015,NPT2011,AC2012,HR2014,FH2015,dCPV+2009} and some details below.

On the other hand, the scaling
\begin{equation}\label{scaling2}
a^-\longmapsto a^-_\eps(x)=\eps^{-d}a(\eps^{-1}x)
\end{equation}
yields $F(v,a^-_\eps*v)\to f(v)$, $\eps\to0$,
where $f$ is given by \eqref{nonlinfunc}. Hence one gets from \eqref{2nl} another nonlocal Fisher--KPP equation
\begin{equation}\label{nl+l}
\dfrac{\partial u}{\partial t} =L_{a^+}u+f(u).
\end{equation}
For a general monostable $f$ as above, this equation was considered in e.g. \cite{Sch1980,Yag2009,BCGR2014,Cov2006,Gar2011,CD2005,CDM2008,AGT2012,CDM2008a,Cov2007,CD2007,Kao2012,LSW2010,SLW2011}, see also some details below.

Combination of the both considered scalings produces hence the classical Fisher--KPP equation \eqref{kpp}, with $f$  given by \eqref{nonlinfunc}, from the equation \eqref{2nl}. It should be stressed, however, that we do not state that solutions to \eqref{2nl} converge to solutions  to \eqref{l+nl}, \eqref{nl+l}, or \eqref{kpp}. Up to our knowledge, a convergence of solutions to \eqref{l+nl} or \eqref{nl+l} to solutions to \eqref{kpp} was not considered rigorously in the literature as well.

Of course, there are a~lot of generalisations for the equations \eqref{2nl}, \eqref{l+nl}, \eqref{nl+l}: the monostable-type function $f$ may depend on time and space variables (e.g. nonlocal reaction-diffusion equation in a periodic media), the mapping $F$ may include a convolution in time or just a time-delay, and many others. For some recent generalisations, see e.g. \cite{LSW2010,Sun2014,HMW2012,CDM2013,ABVV2010,Kao2012,
ZM2014,SZ2012,Pan2014,DW2015,SS2015,SLW2011,TZ2003,LZ2007,RS2012,XX2015}.

Let us formulate the main problems traditionally addressed to the equations above.
\begin{enumerate}[label=\textnormal{(P\arabic*})]
\item \label{P-exist} Existence and uniqueness of solutions in Banach spaces of functions on $\X$, e.g. in $L^\infty(\X)$, $\Buc$ (the space of uniformly continuous functions with $\sup$-norm), or $L^1(\X)$.
\item \label{P-uniform} Uniform in time bounds for the norms of the solutions in the Banach spaces.
\item \label{P-stable} Existence and stability of stationary solutions.
\item \label{P-waves} Existence, uniqueness and properties of the traveling waves: solutions of the special form $u(x,t)=\psi(x\cdot\xi-ct)$, where $\psi$ is a function on $\R$ called the profile of a wave, $\xi$ belongs to the unit sphere $\S$ in $\X$, $x\cdot \xi=(x,\xi)_\X$ is the scalar product on $\X$, and $c\in\R$ describes the speed of the wave. Depending on the class of functions $\psi$ the question may be referred to decaying waves, bounded waves {\em etc.}
\item \label{P-front} Existence and time-behavior of the front of propagation, i.e. a set $\Gamma_t=\X\setminus(\Tauin_t\cup\Tauout_t)$, such that for any $x_t\in\Tauin_t$, the values of $u(x_t,t)$ will converge (as $t\to\infty$) to the upper stationary solution ($\theta$ in the notations above), whereas, for any $y_t\in\Tauout_t$, the values of $u(y_t,t)$ will converge to the low stationary solution (i.e.~to~$0$).
\end{enumerate}

We present now an overview of our results concerning the problems \ref{P-exist}--\ref{P-front} for the equation \eqref{log}/\eqref{2nl}, and compare them with the existing results in the literature, including some information about `partially local' equations \eqref{l+nl} and \eqref{nl+l}.

\smallskip

\textbf{Problem \ref{P-exist}} \quad We will study \eqref{log} in the spaces $\Buc$ and $L^\infty(\X)$. To get an answer on the problem \ref{P-exist}, one does not need any further assumptions on parameters $m,\ka^\pm>0$ and probability kernels $0\leq a^\pm\in L^1(\X)$ (see Theorem~\ref{thm:exist_uniq_BUC} and Remark~\ref{rem:exist_uniq_Linf}). We use standard fixed point arguments, which take into account, however, the negative sign before $a^-$ in \eqref{log}. The solution hence may be constructed on a time-interval $[\tau,\tau+\Delta\tau]$, whereas the $\Delta\tau$ depends on the supremum of the solution at $\tau$. Since the values at the moment $\tau+\Delta\tau$ might be bigger, the next time-interval appears, in general, shorter. The mentioned usage of the negative sign allows us to show that, however, the series of the time-intervals diverges, and thus one can construct solution on an arbitrary big time-interval.

\smallskip

\textbf{Problem \ref{P-uniform}} \quad In spite of the possible growth of solution's space-supremum in time, we show (Theorem~\ref{thm:unibdd}) that the solution in $\Buc$ remains uniformly bounded in time on $[0,\infty)$ under very weak assumptions: one needs only that $a^-$ would be separated from zero in a neighbourhood of the origin and that $a^+$ would have a regular behavior at infinity, e.g. $a^+(x)\leq p(|x|)$, where $|\cdot|$ denotes the Euclidean norm in $\X$ and $p\in L^1(\R)$ monotonically decays at $\pm\infty$. This result is an analog of \cite[Theorem 1.2]{HR2014} for the equation \eqref{l+nl} (the latter used, however, the advantages of the powerful PDE technique for the linear part, that is absent in our case).

\smallskip

The rest of our results requires additional hypotheses. For the shortness, some of them are presented here in a more strict form than we really need (compare them with the real assumptions \eqref{as:chiplus_gr_m}--\eqref{expradialmoment} within the paper); and surely, a particular result requires a part of the assumptions only.

\begin{enumerate}[label=\textnormal{(H\arabic*})]
\item $0\leq a^\pm\in L^1(\X)\cap L^\infty(\X)$, and $\ka^+>m$, i.e. $\theta=\frac{\ka^+-m}{\ka^-}>0$.
\item \label{H-compare} the function
\[
J_\theta:=\ka^+ a^+-\theta \ka^-a^-=\ka^+ a^+-(\ka^+-m)a^-
\]
is almost everywhere (a.e., in the sequel) non-negative and it is separated from $0$ a.e. in a neighbourhood of the origin.
\item \label{H-expdecay} There exists $\la>0$, such that $\displaystyle\int_\X a^+(x) e^{\la |x|}\,dx<\infty$.
\end{enumerate}

Let us compare these hypotheses with existing in the literature. First, we are working in the multi-dimensional settings, cf. \cite{PS2005,FKK2009,FKKozK2014}. We show (Proposition~\ref{prop:monot_sol}) how the problem \ref{P-waves} may be reduced to a one-dimensional equation, whose kernels, however, will depend on a direction $\xi\in\S$. Regarding to this, it should be emphasised, that we do not assume that $a^+$ is symmetric; we deal with the so-called anisotropic settings, cf.~\cite{CDM2008,SLW2011,AGT2012} for the equation \eqref{nl+l}.

The hypothesis \ref{H-expdecay} is called the Mollison condition, see~\cite{Mol1972,Mol1972a}. In particular, it holds if $a^+$ exponentially decays as $|x|\to\infty$. The equation \eqref{nl+l}, under the Mollison condition \ref{H-expdecay} or its weaker form \ref{H-expdecayxi} (see below), was considered in \cite{AGT2012,CDM2008,CD2007,BCGR2014}. The corresponding results in \cite{YY2013,WZ2006} about our equation \eqref{log} required, however, symmetric and quickly decaying $a^+$; the latter meant that \ref{H-expdecay} must hold for all $\la>0$. Note that \cite{WZ2006} dealt with a system of equations for a multi-type epidemic model, which is reduced in the one-type case to \eqref{log} with $\ka^+a^+\equiv\ka^-a^-$. It is worth noting also that we do not need a continuity of $a^+$ as well.

The most restrictive, in some sense, hypothesis is \ref{H-compare}. It implies the comparison principle for the equation \eqref{log}, cf.~Theorem~\ref{thm:compar_pr}, Proposition~\ref{prop:comp_pr_BUC}. In~particular, the latter states that the solution will be inside the strip $0\leq u(x,t)\leq \theta$, for all $t>0$, provided that the initial condition $u(x,0)$ was inside this strip. On the other hand, we show that \ref{H-compare} is, in some sense, a necessary condition to have a comparison principle at all (Remark~\ref{rem:necineq}).

\smallskip

\textbf{Problem \ref{P-stable}} \quad In Subsection~\ref{subsec:comparison_pr}, we show also that $u\equiv\theta$ is a uniformly and asymptotically stable solution, whereas $u\equiv0$ is an unstable one. The absence of non-constant stationary solutions is shown in Proposition~\ref{uniqstationarysolutions}, see also Problem~\ref{P-front} below.

The maximum principle is considered in Subsection~\ref{subsec:max_principle}, cf.~Theorem~\ref{thm:strongmaxprinciple}. In~particular, we prove that the solution to \eqref{log} is strictly positive, even for a compactly supported initial condition $u_0(x):=u(x,0)$, and lies strictly less than $\theta$, for any $u_0\not\equiv\theta$ (Proposition~\ref{prop:u_gr_0}, Corollary~\ref{cor:lesstheta}).

It is worth noting that the luck of the comparison principle for the equation \eqref{l+nl} leads to a non-trivial behavior: if $\ka^-$ is big enough, then the upper stationary solution $u\equiv\theta$ may not be stable, moreover, a stationary inhomogeneous solution may exist, see \cite{GVA2006,ABVV2010}, and also \cite{BNPR2009,NPT2011}; for the further results about \eqref{l+nl} without the comparison principle technique, see e.g. \cite{AC2012,HR2014}. Note also that the hypothesis \ref{H-compare} is not preserved in course of the scaling \eqref{scaling}, which, recall, produces \eqref{l+nl} from \eqref{2nl}: indeed, the inequality $\eps^{-2-d}a^+(\eps^{-1}x)\geq \theta \ka^- a^-(x)$ can not hold for a.a.~$x\in\X$ and any $\eps>0$ simultaneously, as this would mean that $a^-$ is `localised' arbitrary close to the origin. This is a possible reason why the comparison principle which we show for \eqref{2nl} loses for \eqref{l+nl}. However, if, additionally, $a^-$ depends on $\eps$ as in the second scaling \eqref{scaling2}, the inequality in \ref{H-compare} becomes possible, provided that $\eps^{-2}>\theta\ka^-$, thus the comparison principle for the classical Fisher--KPP equation \eqref{kpp} may be informally obtained from our results.

\smallskip

\textbf{Problem \ref{P-waves}} \quad We study monotonically non-increasing traveling waves only (i.e. the profile $\psi$ is a non-increasing function on $\R$). To ensure the existence of a traveling wave solution to \eqref{log} in a direction $\xi\in\S$ the Mollison condition \ref{H-expdecay} can be relaxed as follows:
\begin{enumerate}[start=3,label=\textnormal{(H{\arabic*}${}_\xi$)}]
\item \label{H-expdecayxi} There exists $\la>0$, such that $\A_\xi(\la):=\displaystyle\int_\X a^+(x) e^{\la \, x\cdot\xi}\,dx<\infty$.
\end{enumerate}

Namely, we prove that there exists a minimal traveling wave speed $c_*(\xi)\in\R$, such that, for any $c\geq c_*(\xi)$, there exists a traveling wave in the direction $\xi$ with the speed $c$; and, for any $c<c_*(\xi)$, such a traveling wave does not exist (Theorem~\ref{thm:trwexists}). We use here an abstract result from \cite{Yag2009} and apply it to \eqref{log} similarly to how it was done in \cite{Yag2009} for \eqref{nl+l}. This allow us to prove the existence of such finite $c_*(\xi)$ without an assumption about a quick decaying of $a^+$ in the direction $\xi$; i.e. that we do not need that \ref{H-expdecayxi} holds, for all $\la>0$, in contrast to \cite{YY2013,WZ2006}. It is worth noting that the hypothesis \ref{H-compare} evidently holds under the assumptions from \cite{WZ2006}, where $\ka^+=\ka^-$, $a^+=a^-$, as well as it holds under the assumptions from \cite{YY2013}, where one of the considered cases may be rewritten in the form $\frac{\partial}{\partial t}u=J_\theta*u-mu+\ka^-(\theta-u)(a^-*u)$, which is equivalent to \eqref{log}.

A specific feature of the equation \eqref{log} is that any monotonic traveling wave with a non-zero speed $c\geq c_*(\xi)$ has a smooth profile $\psi_c\in C^\infty(\R)$, whereas, for the traveling wave with the zero speed (which does exist, if only $c_*(\xi)\leq0$), one~can only prove that its profile $\psi_0\in C(\R)$ (Proposition~\ref{prop:reg_trw}, Corollary~\ref{cor:infsmoothprofile}), in~contrast to the equation \eqref{nl+l}, cf.~\cite{CDM2008}, where a weaker smoothness was shown. This allow us to consider the equation for traveling waves point-wise:
\begin{equation}\label{trweqn}
  c\psi'(s)+\ka^{+}(\check{a}^{+}*\psi)(s)-m\psi(s)-\ka^{-}\psi(s) (\check{a}^{-}*\psi)(s)=0, \quad s\in\R,
\end{equation}
where the kernels $\check{a}^\pm$ are obtained by the integration of $a^\pm$ over the orthogonal complement $\{\xi\}^\bot$, see \eqref{apm1dim} below.  Moreover, in Proposition~\ref{prop:psidecaysstrictly}, we show that $\psi$ is a strictly decaying function.

We study properties of the solutions to \eqref{trweqn} using a bilateral-type Laplace transform: $(\L\psi)(z)=\int_\R \psi(s) e^{z s}\,ds$, $\Re z>0$. To~do this, we prove that any solution \eqref{trweqn} has a positive abscissa $\la_0(\psi)$ of this Laplace transform, i.e. that $(\L\psi)(\la)<\infty$, for some $\la>0$ (Proposition~\ref{prop:trw_exp_est}). Moreover, in Theorem~\ref{thm:speedandprofile}, we prove, in particular, that $\la_0(\psi)$ is finite and bounded by $\la_0(\check{a}^+)$; note that the latter abscissa will be infinite in the case of quickly decaying kernel $a^+$, i.e. whe \ref{H-expdecayxi} holds, for any $\la>0$. We also find in Theorem~\ref{thm:speedandprofile} the explicit formula for $c_*(\xi)$:
\begin{equation*}
  c_*(\xi)=\inf_{\la>0}\frac{\ka^+\A_\xi(\la)-m}{\la},
\end{equation*}
where $\A_\xi$ is defined in \ref{H-expdecayxi}; and we show that the dependence of the abscissa $\la_0(\psi_c)$ for a traveling wave profile $\psi_c$ corresponding to a speed $c$ is strictly decreasing in $c\geq c_*(\xi)$. Note that this expression for the minimal traveling wave speed coincides with the known one for the equation~\eqref{nl+l}, see e.g. \cite{CDM2008}.

Thus, for `exponentially decaying' $a^+$ (i.e. if there exists a finite supremum of $\la$'s for which \ref{H-expdecayxi} does hold), it is possible the situation in which the abscissa $\la_*=\la_0(\psi_{c_*(\xi)})$ of the traveling wave with the minimal possible speed coincides with $\la_0(\check{a}^+)$. This case is traditionally more difficult for an analysis of profiles' properties, cf. e.g. \cite[Theorem 3, Remark 8]{AGT2012}. We consider this special case in details and describe it in terms of the function $a^+$ and the parameters $m,\ka^\pm$, cf.~Definition~\ref{def:VWxi}, Theorem~\ref{thm:speedandprofile}.

The variety of possible situations demonstrates the following natural example, cf.~Example~\ref{ex:spefunc}. Let
\begin{equation}\label{ex-intro}
  a^+(x)=\frac{\alpha e^{-\mu|x|}}{1+|x|^q}, \quad q\geq0,\ \mu>0,
\end{equation}
where $\alpha>0$ is a normalising constant. Then, for any $\xi\in\S$, the abscissa $\la_0(\check{a}^+)=\mu$ is finite. We show that the strict inequality $\la_*<\mu$ always hold, for $q\in[0,2]$. Next, there exist critical values $\mu_*>0$ and $m_*\in (0,\ka^+)$, such that, for $q>2$, one has $\la_*<\mu$ if $\mu>\mu_*$ or if $\mu\in(0,\mu_*]$ and $m\in(m_*,\ka^+)$. Respectively, for $q>2$, $\mu\in(0,\mu_*]$, and $m\in(0,m_*]$, we show the equality $\la_*=\mu$, see~Theorem~\ref{thm:speedandprofile}.

To study the uniqueness of traveling waves, we find also the exact asymptotic at $\infty$ of the profiles of traveling waves with non-zero speeds. Namely, we show in Proposition~\ref{prop:charfuncpropert}, that,
for a profile $\psi$ corresponding to the speed $c\neq0$,
\begin{equation}\label{eq:trw_asympt-intro}
\psi(t)\sim D e^{-\la_0(\psi) t},\quad c >c_*(\xi), \qquad
\psi(t)\sim D t\,e^{-\la_0(\psi) t}, \quad c=c_*(\xi),
\end{equation}
as $t\to\infty$. Here $D>0$ is a constant which may be chosen equal to $1$ by a shift of $\psi$ (see Remark~\ref{rem:shifting}). To get \eqref{eq:trw_asympt-intro}, one needs an additional assumption in the critical case for the speed $c=c_*(\xi)$; for example, in terms of the function \eqref{ex-intro}, this assumption does not hold for the case $q\in(2,3]$, $\mu\in(0,\mu_*]$, $m=m_*$ only (Remark~\ref{very-critical-case}).

The asymptotic \eqref{eq:trw_asympt-intro} was known for the equation \eqref{nl+l}, cf. e.g. \cite{AGT2012,CC2004,CDM2008}. In the two latter references, there was used a version of the Ikehara theorem which belongs to Delange \cite{Del1954}. However, we have met here with the following problem.

Both the classical Ikehara theorem (see e.g. \cite{Wid1941}) and the Ikeraha--Delange theorem \cite{Del1954} (see also \cite{EE1985}) dealt with functions growing at infinity to $\infty$. In~\cite{CC2004,CDM2008}, the corresponding results were postulated for functions (decreasing or increasing) which tend to $0$ (on  $\infty$ or $-\infty$, respectively). We~did not find any arguments why we could apply or how one could modify the proofs of Ikehara-type theorems for such functions without proper additional assumptions. The~natural assumption under which it can be realized is that the decreasing function $\psi(s)$ (a~traveling wave in our context) must become an increasing one, being multiplied on an exponent $e^{\nu s}$, for a big enough $\nu>0$.

Under such an assumption the Ikehara-type theorems might hold true, however, one needs more to cover the aforementioned case $\la_*=\mu$. In this case, the Laplace transform of $\check{a}^+$ is not analytic at its abscissa, that was a requirement for the mentioned theorems. Therefore, we used an another modification of the Ikehara theorem, the so-called Ikehara--Ingham theorem \cite{Ten1995}. Under the assumption that a constant $\nu$ as above exists, we prove in Proposition~\ref{prop:tauber} a version of the Ikehara--Ingham theorem for such decreasing functions. Next, using the ideas from \cite{ZLW2012}, we show that, for any solution to \eqref{trweqn} with $c\neq0$, such a $\nu$ does exist.

Note also that the technique from \cite{AGT2012} did not require the usage of Ikehara-type theorem, however, even for the local nonlinearity like in \eqref{nl+l} it did not work in the critical case above.

The asymptotic \eqref{eq:trw_asympt-intro} allows us to prove the uniqueness of the profiles for a traveling wave with a non-zero speed (Theorem~\ref{thm:tr_w_uniq}). We follow there the technique proposed in \cite{CC2004}.

\smallskip

\textbf{Problem \ref{P-front}} \quad The question of the front of propagation for the equations \eqref{l+nl}, \eqref{nl+l} was studied less intensively. For our equation \eqref{2nl}, one can refer to \cite{PS2005} ($\ka^+a^+=\ka^- a^-$, see also below) and \cite{WZ2006} ($\ka^+a^+=\ka^- a^-$, $d=1$, quickly decaying kernels). Note that the generalization in \cite{ZM2014} does not cover the equation \eqref{2nl}. One of the traditional way for the study of the front of propagation for integro-differential equations is the usage of abstract Weinberger's results from \cite{Wei1982} (which are going back to \cite{AW1978}, for the Fisher--KPP equation \eqref{kpp}). The information we obtained for the traveling waves allow us to describe in more details the behavior of $u(tx,t)$ `out of the front'; here $u$ is the solution to \eqref{2nl}. Namely, in Theorem~\ref{thm:outoffront}, we prove that, for a proper compact convex set $\Tau_1$, the function $u(tx,t)$ decays exponentially in time, uniformly in $x\in\X\setminus\Tauout$, for any open $\Tauout\supset\Tau_1$, provided that the initial condition decays in space quicker than any exponent (in particular, we do not require a compactly supported initial condition).

To describe the behavior of $u(tx,t)$, for $x\in\Tau_1$, we start with an adaption of the results from \cite{Wei1982} to our case. However, that abstract technique required that the initial condition should be separated from $0$ on a set which can not be described explicitly (the existence of such a set was shown only, cf.~Lemma~\ref{lem:conv_to_theta-1} and Proposition~\ref{prop:conv_to_theta_cont_time} below). To avoid this restriction, we find, in Proposition~\ref{prop:subsolution}, an explicit sub-solution to \eqref{2nl}, and, moreover, we prove, in Proposition~\ref{useBrandle}, that this sub-solution indeed becomes a minorant for the solution, after a finite time. This arguments allow us to show that $u(tx,t)$ converges to $\theta$ uniformly in $x\in\Tauin$, for any compact $\Tauin\subset\Tau_1$ (Theorem~\ref{thm:convtotheta}, Corollary~\ref{cor_essinf}). In notations~of Problem~\ref{P-front}, it means informally that $\Gamma_t\approx t\,\partial \Tau_1$.

As a consequence, we prove that, under additional technical assumptions, there are not other non-negative time-stationary solutions to \eqref{2nl} except constant solutions $0$ and $\theta$ (Proposition~\ref{uniqstationarysolutions}).

The Mollison condition \ref{H-expdecayxi} is crucial: we show in Theorem~\ref{thm:infinitespeed} and Corollary~\ref{cor:nowaves} that the absence of a $\la$ and a $\xi\in\S$ which ensure \ref{H-expdecayxi} leads to an infinite speed of propagation (i.e. the compact set $\Tau_1$ above may be chosen arbitrary big) and hence to the absence of traveling waves at all. The corresponding result for \eqref{nl+l} was received in \cite{Gar2011} and it is goes back to \cite{Mol1972,Mol1972a} mentioned above. The results of \cite{PS2005} cover Theorems~\ref{thm:outoffront}, \ref{thm:convtotheta}, and \ref{thm:infinitespeed}, for the equation \eqref{2nl} with $\ka^+a^+=\ka^- a^-$; however, a lot of details of the proofs (which used completely another technique) were omitted.

To summarize, the structure of the paper is the following. In~Section~\ref{sec:basic}, we study Problems~\ref{P-exist} and \ref{P-uniform}; Section~\ref{sec:comparison} is devoted to comparison and maximum principles, and, partially, to Problem~\ref{P-stable}. Traveling waves, Problem~\ref{P-waves}, are considered in Section~\ref{sec:tr-waves}. The long-time behavior, i.e. Problem~\ref{P-front}, and the rest of Problem~\ref{P-stable} are the topics of Section~\ref{sec:long-time}. In Subsection~\ref{subsec:history}, we present some historical comments about the derivation of the equation~\eqref{2nl}; and, finally, in Subsection~\ref{subsec:remarks}, we discuss some remarks and open problems.

\section{Existence, uniqueness, and boundedness} \label{sec:basic}
Let $u=u(x,t)$ describe the local density of a system at the point $x\in\X$, $d\geq 1$, at the moment of time $t\in I$, where $I$ is either a finite interval $[0,T]$, for some $T>0$, or the whole $\R_{+}:=[0,\infty)$. The time evolution
of $u$ is given by the following initial value problem
\begin{equation}
\begin{cases}
\begin{aligned}
\dfrac{\partial u}{\partial t}(x,t)=\ &\ka^{+}(a^{+}*u)(x,t)-mu(x,t)\\&-\ka^{-}u(x,t)(a^{-}*u)(x,t),
\qquad x\in\X,\ t\in I\setminus \{0\},
\end{aligned}\\[3mm]
u(x,0)=u_0(x),\qquad \qquad \qquad \qquad \qquad \quad\,\; x\in\X,
\end{cases}\label{eq:basic}
\end{equation}
which we will study in a class of bounded in $x$ nonnegative functions.

Here $m>0$, $\ka^\pm>0$ are constants, and functions  $0\leq a^\pm \in L^{1}(\X)$ are probability densities:
\begin{equation}\label{normed}
  \int_{\X }a^+(y)dy=\int_{\X }a^-(y)dy=1.
\end{equation}
Here and below, for a function $u=u(y,t)$, which is (essentially) bounded in $y\in\X$, and
a function (a kernel) $a\in L^1(\X)$, we denote
\begin{equation}\label{def:conv}
(a*u)(x,t):=\int_{\X }a(x-y)u(y,t)dy.
\end{equation}

We assume that $u_0$ is a bounded function on $\X$. For technical
reasons, we will consider two Banach spaces of bounded real-valued functions
on $\X$: the space $\Buc$ of  bounded uniformly continuous functions
on $\X$ with $\sup$-norm and the space $L^{\infty}(\X)$
of  essentially bounded (with respect to the Lebesgue measure) functions on $\X$ with $\esssup$-norm. Let also $C_b(\X)$ and $C_0(\X)$ denote the spaces of continuous functions on $\X$ which are bounded and have compact supports, correspondingly.

Let $E$ be either $\Buc$ or $L^{\infty}(\X)$. Consider the equation \eqref{eq:basic}
in $E$; in particular, $u$ must be continuously differentiable in $t$, for $t>0$, in the sense of
the norm in $E$. Moreover, we consider $u$ as an element from the space
$C_{b}(I\rightarrow E)$ of continuous bounded functions on $I$ (including
$0$)
with values in $E$ and with the following norm
\begin{equation*}
\Vert u\Vert_{C_{b}(I\rightarrow E)}=\sup\limits _{t\in I}\Vert u(\cdot,t)\Vert_{E}.
\end{equation*}
Such a solution is said to be a classical solution to \eqref{eq:basic}; in particular, $u$ will continuously (in the sense of the norm in $E$) depend on the initial condition $u_0$.

We will also use the space $C_{b}(I\rightarrow E)$ with $I=[T_1,T_2]$, $T_1>0$. For simplicity of notations, we denote
\[
\x_{T_1,T_2}:=C_{b}\bigl([T_1,T_2]\rightarrow \Buc\bigr), \qquad T_2>T_1\geq0,
\]
and the corresponding norm will be denoted by $\|\cdot\|_{T_1,T_2}$. We set also $\x_T:=\x_{0,T}$, $\|\cdot\|_T:=\|\cdot\|_{0,T}$, and
\[
\Xinf:=C_{b}\bigl(\R_{+}\rightarrow \Buc\bigr)
\]
with the corresponding norm $\lVert\cdot\rVert_{\infty}$. The upper index `+' will denote the cone of nonnegative functions in the corresponding space, namely,
\[
\x_\sharp^+:=\{u\in\x_\sharp\mid u\geq0\},
\]
where $\sharp$ is one of the sub-indexes above.
Finally, the corresponding sets of functions with values in $L^\infty(\X)$ will be denoted by the tilde above, e.g.
\begin{align*}
\tilde{\x}_T &:=C_b([0,T]\to L^\infty(\X)), \\
 \tilde{\x}_T^+&:=\bigl\{u\in\tilde{\x}_T \mid u(\cdot,t)\geq0, \ t\in[0,T], \ \mathrm{a.a.} x\in\X\bigr\}.
\end{align*}

We will also omit the sub-index for the norm $\|\cdot\|_E$ in $E$, if it is  clear whether we are working with $\sup$- or $\esssup$-norm.

We start with a simple lemma.
\begin{lemma}\label{le:simple}
  Let $a\in L^1(\X)$, $f\in L^\infty(\X)$. Then $a*f\in\Buc$. Moreover, if $v\in C_b(I\to E)$, $I\subset\R_+$, then $a*v\in C_b(I\to \Buc)$.
\end{lemma}
\begin{proof}
  The convolution is a bounded function, as
  \begin{equation}\label{convbdd}
    |(a*f)(x)|\leq \|f\|_E \,\|a\|_{L^1(\X)}, \qquad a\in L^1(\X), f\in E.
  \end{equation}
  Next, let $a_n\in C_0(\X)$, $n\in\N$, be such that $\|a-a_n\|_{L^1(\X)}\to0$, $n\to\infty$. For any $n\geq1$, the proof of that $a_n*f\in\Buc$ is straightforward. Next, by \eqref{convbdd}, $\|a*f-a_n*f\|\to0$, $n\to\infty$. Hence $a*u$ is a uniform limit of uniformly continuous functions that fulfilles the proof of the first statement. The second statement is followed from the first one and the inequality \eqref{convbdd}.
\end{proof}

The following theorem yields existence and uniqueness of a solution
to \eqref{eq:basic} on a finite time-intervals $[0,T]$.

\begin{theorem}\label{thm:exist_uniq_BUC}
Let $u_0\in \Buc$ and $u_0(x)\ge0$, $x\in\X $.
Then, for any $T>0$, there exists a unique nonnegative solution $u$ to the equation
\eqref{eq:basic} in $\Buc$, such that $u\in\mathcal{{X}}_{T}$.
\end{theorem}
\begin{proof}
Let $T>0$ be arbitrary. Take any $0\leq v\in\xt$. For any $\tau\in[0,T)$, consider
the following linear equation in the space $\Buc$ on the interval $[\tau,T]$:
\begin{equation}
\begin{cases}
\begin{aligned}
\dfrac{\partial u}{\partial t}(x,t)=-mu(x,t)&-\ka^{-}u(x,t)(a^{-}*v)(x,t)\\&+\ka^{+}(a^{+}*v)(x,t), \qquad t\in(\tau,T],
\end{aligned}\\[2mm]
u(x,\tau)=u_{\tau}(x),
\end{cases}
\label{eq:exist_uniq_BUC:basic_lin}
\end{equation}
where $0\leq u_s\in \Buc$, $s>0$, are some functions, and $u_0$ is the same as in \eqref{eq:basic}.
By Lemma~\ref{le:simple}, in the right hand side (r.h.s. in the sequel) of \eqref{eq:exist_uniq_BUC:basic_lin}, there is a time-dependent linear bounded operator (acting in $u$) in the space $\Buc$ whose coefficients are continuous on $[\tau,T]$.
Therefore, there exists a unique solution
to \eqref{eq:exist_uniq_BUC:basic_lin} in $\Buc$ on $[\tau,T]$, given by $u=\Phi_\tau v$ with
\begin{equation}
(\Phi_\tau v)(x,t):=(Bv)(x,\tau,t)u_\tau(x)+\int_\tau^t(Bv)(x,s,t)\ka^{+}(a^{+}*v)(x,s)\,ds,
\label{eq:exist_uniq_BUC:Phi_v}
\end{equation}
for $x\in\X$, $t\in[\tau,T]$, where we set
\begin{equation}
(Bv)(x,s,t):=\exp\biggl(-\int _{s}^t\bigl(m+\ka^{-}(a^{-}*v)(x,p)\bigr)\,dp\biggr),\label{eq:exist_uniq_BUC:B}
\end{equation}
for $x\in\X$, $t,s\in[\tau,T]$. Note that, in particular, $(\Phi_\tau v)(\cdot,t), (Bv)(\cdot,s,t)\in\Buc$.
Clearly, $(\Phi_\tau v)(x,t)\geq0$ and, for any $\Tau\in(\tau,T]$,
\begin{equation}\label{eq:est}
  \|\Phi_\tau v(\cdot,t)\|\leq \|u_\tau\|+\ka^+ (\Tau-\tau) \|v\|_{\tau,\Tau}, \qquad t\in[\tau,\Tau],
\end{equation}
where we used \eqref{convbdd}. Therefore, $\Phi_\tau$ maps $\x_{\tau,\Tau}^+$ into itself, $\Tau\in(\tau,T]$.

Let now $0\leq\tau < \Tau\leq T$, and take any $v,w\in\x_{\tau,\Tau}^+$.
By \eqref{eq:exist_uniq_BUC:Phi_v}, one has, for any $x\in\X$, $t\in[\tau,\Tau]$,
\begin{equation}\label{estbyJ}
\bigl|(\Phi_\tau v)(x,t)-(\Phi_\tau w)(x,t)\bigr|\leq J_1+J_2,
\end{equation}
where
\begin{align*}
J_1&:=\bigl|(Bv)(x,\tau,t)-(Bw)(x,\tau,t)\bigr|u_\tau(x),\\
J_2&:=\ka^{+}\int_\tau^t\bigl|(Bv)(x,s,t)(a^{+}*v)(x,s)
-(Bw)(x,s,t)(a^{+}*w)(x,s)\bigr|\,ds.
\end{align*}
Since $|e^{-a}-e^{-b}|\leq |a-b|$, for any constants $a,b\geq0$, one has, by \eqref{eq:exist_uniq_BUC:B}, \eqref{convbdd},
\begin{equation}
J_1\leq \ka^- (\Tau-\tau)\lVert u_\tau\rVert \lVert v-w\rVert_{\tau,\Tau}.\label{eq:exist_uniq_BUC:Phi_est_i}
\end{equation}
Next, for any constants $a,b,p,q\geq0$,
\begin{equation*}
  \bigl|pe^{-a}-qe^{-b}\bigr|\leq e^{-a}|p-q|+q\max\bigl\{e^{-a},e^{-b}\bigr\}|a-b|,
\end{equation*}
therefore, by \eqref{eq:exist_uniq_BUC:B}, \eqref{convbdd},
\begin{align}
J_2&\leq  \ka^{+}\int_\tau^t(Bv)(x,s,t)\bigl(a^{+}*|v-w|\bigr)(x,s)\,ds\notag\\
&\quad+\ka^{+}\int_\tau^t\max\bigl\{(Bv)(x,s,t),(Bw)(x,s,t)\bigr\}(a^{+}*w)(x,s)\notag\\
&\qquad \times\ka^{-}\int_s^t\bigl(a^-*|v-w|\bigr)(x,r)\,dr\,ds\notag\\
&\leq \ka^{+}(\Tau-\tau) \lVert v-w\rVert_{\tau,\Tau}+\ka^{+}\ka^{-}\lVert w\rVert_{\tau,\Tau}\lVert v-w\rVert_{\tau,\Tau}\int_\tau^t e^{-m(t-s)}(t-s)\,ds\notag\\
&\leq \ka^+ \Bigl(1+\frac{\ka^-}{me}\lVert w\rVert_{\tau,\Tau}\Bigr) (\Tau-\tau) \lVert v-w\rVert_{\tau,\Tau},\label{est2in1}
\end{align}
as $re^{-r}\leq e^{-1}$, $r\geq0$.

For any $T_2>T_1\geq0$, we define
\begin{equation*}
  \x_{T_1,T_2}^+(r):=\bigl\{ v\in \x_{T_1,T_2}^+ \bigm| \|v\|_{T_1,T_2} \leq r\bigr\}, \quad r>0.
\end{equation*}

Take any $\mu\geq\lVert u_\tau\rVert$. By \eqref{eq:est}--\eqref{est2in1}, one has, for any $v,w\in \x_{\tau,\Tau}^+(r)$, $r>0$,
\begin{align*}
\bigl|(\Phi_\tau v)(x,t)-(\Phi_\tau w)(x,t)\bigr| &\leq
\Bigl(\ka^-\mu +\ka^+ +\frac{\ka^+ \ka^-}{me}r\Bigr) (\Tau -\tau)\lVert v-w\rVert_{\tau,\Tau},\\
\bigl|(\Phi_\tau v)(x,t)\bigr| &\leq \mu+\ka^+ r (\Tau -\tau).
\end{align*}
Therefore, $\Phi_\tau$ will be a contraction mapping on the set $\x_{\tau,\Tau}^+(r)$ if only
\begin{equation}\label{need1}
 \Bigl(\ka^-\mu +\ka^+ +\frac{\ka^+ \ka^-}{me}r\Bigr) (\Tau -\tau) <1 \quad \text{and} \quad \mu+\ka^+ r (\Tau -\tau)\leq r.
\end{equation}
Take any $\alpha\in(0,1)$ and set
\begin{equation}\label{settings}
\begin{gathered}
C:=\ka^{-}+\dfrac{\ka^{+}\ka^{-}}{me},  \qquad r:=\mu+\frac{\alpha\ka^+}{C}, \\ \Tau :=\tau+\frac{\alpha}{Cr}=\tau+\frac{\alpha}{C\mu+\alpha\ka^+}.
\end{gathered}
\end{equation}
Then, the second inequality in \eqref{need1} evidently holds (and it is just an equality), and the first one may be rewritten as follow
\begin{equation*}
 \Bigl(C\mu +\ka^+ +\frac{\ka^+ \ka^-}{me}\frac{\alpha\ka^+}{C}\Bigr)\frac{\alpha}{Cr} <1,
\end{equation*}
or, equivalently,
\begin{equation}\label{need11}
 \alpha C\mu  +\alpha^2\frac{\ka^+ \ka^-}{me}\frac{\ka^+}{C}<C\mu.
\end{equation}
To fulfill \eqref{need11}, one should choose $\alpha\in(0,1)$ such that
\begin{equation}\label{need111}
  \frac{\alpha^2}{1-\alpha}<\frac{C^2\mu m e}{(\ka^+)^2\ka^-}.
\end{equation}
Since function $f(\alpha)=\frac{\alpha^2}{1-\alpha}$ is strictly increasing on $[0,1)$ and $f(0)=0$, one can always choose $\alpha\in(0,1)$ that satisfies \eqref{need111}.

As a result, choosing $\mu=\mu(\tau)> \lVert u_\tau\rVert$ (to include the case $u_\tau\equiv0$) and $\alpha$ that satisfies \eqref{need111}, one gets that $\Phi_\tau$ will be a contraction on the set $\x_{\tau,\Tau}^+(r)$ with $\Tau$ and $r$ given by \eqref{settings}; the latter set naturally formes a complete metric space. Therefore, there exists a unique $u\in \x_{\tau,\Tau}^+(r)$ such that $\Phi_\tau u=u$. This $u$ will be a solution to \eqref{eq:basic} on $[\tau,\Tau]$.

To fulfill the proof of the statement, one can do the following. Set $\tau:=0$, choose any $\mu_1>\|u_0\|$ and fix an $\alpha$ that satisfies \eqref{need111} with $\mu=\mu_1$. One gets a solution $u$ to \eqref{eq:basic} on $[0,\Tau_1]$ with
$\Tau_1=\frac{\alpha}{C\mu_1+\alpha \ka^+}$, $\lVert u\rVert_{\Tau_1}\leq \mu_1+\frac{\alpha\ka^+}{C}$.

Iterating this scheme, take sequentially, for each $n\in\N$, $\tau:=\Tau_n$, $u_{\Tau_n}(x):=u(x,\Tau_n)$, $x\in\X$,
\[
\mu_{n+1}:=\mu_n+\frac{\alpha\ka^+}{C}\geq \|u_{\Tau_n}\|.
\]
Since $\mu_{n+1}>\mu_n$, the same $\alpha$ as before will satisfy \eqref{need111} with $\mu=\mu_{n+1}$ as well. Then, one gets a solution $u$ to
\eqref{eq:basic} on $[\Tau_n,\Tau_{n+1}]$ with initial condition $u_{\Tau_n}$, where
\[
\Tau_{n+1}:=\Tau_n+\frac{\alpha}{C\mu_{n+1}+\alpha \ka^+},
\]
and
\[
\|u\|_{\Tau_n,\Tau_{n+1}}\leq\mu_{n+1}+\frac{\alpha\ka^+}{C}=\mu_{n+2}.
\]
As a result, we will have a solution $u$ to \eqref{eq:basic} on intervals $[0,\Tau_1]$, $[\Tau_1,\Tau_2]$, \ldots, $[\Tau_n,\Tau_{n+1}]$, $n\in\N$, where $\mu_{n+1}=\mu_1+n\frac{\alpha\ka^+}{C}$, and, thus,
\begin{equation}\label{Tn}
  \Tau_{n+1}:=\Tau_n+\frac{\alpha}{C\mu_1+(n+1)\alpha \ka^+}.
\end{equation}
By Lemma~\ref{le:simple}, the r.h.s. of \eqref{eq:basic}, will be continuous on each of constructed time-intervals, therefore, one has that $u$ is continuously differentiable on $(0,\Tau_{n+1}]$ and solves \eqref{eq:basic} there. By \eqref{Tn},
\[
\Tau_{n+1}:=\sum_{j=1}^{n+1}\frac{\alpha}{C\mu_1+j\alpha \ka^+}\to\infty , \quad n\to\infty,
\]
therefore, one has a solution to \eqref{eq:basic} on any $[0,T]$, $T>0$.

To prove uniqueness, suppose that $v\in\xt$ is a solution to \eqref{eq:basic} on $[0,T]$, with $v(x,0)\equiv u_0(x)$, $x\in\X$. Choose $\mu_1>\|v\|_T\geq\|u_0\|$. Since $\{\mu_n\}_{n\in\N}$ above is an increasing sequence, $v$ will belong to each of sets  $\x_{\Tau_n,\Tau_{n+1}}^+(\mu_{n+1})$, $n\geq0$, $\Tau_0:=0$, considered above. Then, being solution to \eqref{eq:basic} on each $[\Tau_n,\Tau_{n+1}]$, $v$~will be a fixed point for $\Phi_{\Tau_n}$. By the uniqueness of such a point, $v$ coincides with $u$ on each $[\Tau_n,\Tau_{n+1}]$ and, thus, on the whole $[0,T]$.
\end{proof}

\begin{remark}
\label{rem:exist_uniq_Linf}
The statement of Theorem \ref{thm:exist_uniq_BUC} holds true for solutions in $L^\infty(\X)$ with $u\in\tilde{X}_T$: the proof will be mainly identical. See also \cite[Theorem 4.1]{FKKozK2014}.
\end{remark}

Consider the following quantity
\begin{equation}\label{theta_def}
  \theta:=\frac{\ka^+-m}{\ka^-}\in\R.
\end{equation}

Theorem \ref{thm:exist_uniq_BUC} has a simple corollary:
\begin{corollary}\label{cor:startwithconst}
  Let $t_0\geq0$ be such that the solution $u$ to \eqref{eq:basic} is a constant in space at the moment of time $t_0$, namely, $u(x,t_0)\equiv u(t_0)\geq0$, $x\in\X$. Then
  this solution will be a constant in space for all further moments of time, more precisely,
\begin{equation}\label{homogensol}
  u(x,t)= u(t)=\frac{u(t_0)}{u(t_0) g_\theta(t)+\exp(-\theta\ka^- t)}\geq0, \qquad x\in\X, \ t\geq t_0,
\end{equation}
where
\[
g_\theta(t)=\begin{cases}
  \dfrac{1-\exp(-\theta\ka^-t)}{\theta}, &\theta\neq0,\\
  \ka^-t, & \theta=0,
\end{cases} \qquad
t\geq t_0.
\]
In particular, $u(t)\to\max\{0,\theta\}$, $t\to\infty$.
\end{corollary}
\begin{proof}
First of all, we note that in the proof of Theorem~\ref{thm:exist_uniq_BUC} we proved that the problem \eqref{eq:basic} has a unique solution. Next, straightforward calculations show that \eqref{homogensol} solves \eqref{eq:basic} for $\tau=t_0$, that implies the first statement. The last statement is also straightforward then.
\end{proof}

\begin{remark}
  Note that \eqref{homogensol} solves the classical logistic equation, cf.~\eqref{Ver}:
  \begin{equation}\label{eq:homogen}
  \frac{d}{dt} u(t)=\ka^- u(t) (\theta - u(t)), \quad t>t_0,\quad u(t_0)\geq0.
  \end{equation}
\end{remark}

By Lemma~\ref{le:simple}, the mapping $A^+ v=\ka^+ a^+*v$ defines a linear operator on $\Buc$, which is evidently bounded: by~\eqref{convbdd} and $A^+ 1=\ka^+$, one has $\|A^+\|=\ka^+$. Then a solution $u$ to \eqref{eq:basic}  satisfies the following equation
\[
u(x,t)=e^{-tm}e^{tA^+}u_0(x)-\int_0^t e^{-(t-s)m}e^{(t-s)A^+}\ka^-u(x,s)(a^-*u)(x,s)\,ds.
\]
Therefore, $u(x,t)\geq0$ implies $u(x,t)\leq e^{-tm}e^{tA^+}u_0(x)$, $x\in\X$, $t\geq0$; and hence, by Theorem~\ref{thm:exist_uniq_BUC}, $0\leq u_0\in \Buc$ yields
\begin{equation}\label{expbdd}
    \|u(\cdot,t)\|\leq e^{(\ka^+-m)t}\|u_0\|, \qquad t\geq0.
  \end{equation}

In particular, for $m>\ka^+$, the solution $u(x,t)$ to \eqref{eq:basic} exponentially quickly in $t$ tends to $0$, uniformly in $x\in\X$.

We proceed now to show that, in fact, the solution to \eqref{eq:basic} is uniformly bounded in time on the whole $\R_+$, provided that the kernel $a^-$ does not degenerate in a neighborhood of the origin and $a^+$ has an integrable decay at~$\infty$.

\begin{definition}\label{def:loc_conv}
Let $\1_A$ denote the indicator function of a measurable set $A\subset\X$. Recall that a sequence $f_{n}\in L_{\mathrm{loc}}^{\infty}(\X)$ is said to be locally uniformly convergent to an $f\in L_{\mathrm{loc}}^{\infty}(\X)$, if
$\1_\La f_{n} \to \1_\La f$ in $L^{\infty}(\X)$, $n\to\infty$, for any compact $\La\subset\X$. We denote this convergence by $f_{n}\locun f$. We will use the same notation to say that, for some $T>0$ and $v_n, v\in L_{\mathrm{loc}}^{\infty}(\X\times [0,T])$, one has $\1_\La v_{n} \to \1_\La v$ in $L^{\infty}(\X\times [0,T])$, for any compact $\La\subset\X$.
\end{definition}

We start with a simple statement useful for the sequel.
\begin{lemma}\label{convluc}
Let $a\in L^1(\X)$, $\{f_n,f\}\subset L^\infty(\X)$, $\|f_n\|\leq C$, for some $C>0$, and $f_n\locun f$. Then $a*f_n \locun a*f$.
\end{lemma}
\begin{proof}
  Let $\{a_m\}\subset C_0(\X)$ be such that $\|a_m-a\|_{L^1(\X)}\to0$, $m\to\infty$, and denote $A_m:=\mathrm{supp}\, a_m$. Note that, there exists $D>0$, such that $\|a_m\|_{L^1(\X)}\leq D$, $m\in\N$. Next, for any compact $\La\subset\X$,
  \begin{align*}
  |\1_\La (x) (a_m*(f_n-f))(x)|&\leq \int_\X \1_{A_m}(y) \1_\La (x) |a_m(y)| |f_n(x-y)-f(x-y)|\,dy\\
  &\leq \|a_m\|_{L^1(\X)} \|\1_{\La_m}(f_n-f)\|\to0, n\to\infty,
  \end{align*}
  for some compact $\La_m\subset\X$. Next,
  \begin{align*}
  \|\1_\La (a*(f_n-f))\|&\leq\|\1_\La (a_m*(f_n-f))\|+\|\1_\La ((a-a_m)*(f_n-f))\|
  \\&\leq D\|\1_{\La_m}(f_n-f)\|+(C+\|f\|)\|a-a_m\|_{L^1(\X)},
  \end{align*}
  and the second term may be arbitrary small by a choice of $m$.
\end{proof}

The next theorem is an adaptation of \cite[Theorem 1.2]{HR2014}.

Below, $|\cdot|=|\cdot|_\X$ denotes the Euclidean norm in $\X$, $B_r(x)$ is a closed ball in $\X$ with the center at $x\in\X$ and the radius $r>0$; and $b_r$ is a volume of this ball. Consider also, for any $z\in\Z^d$, $q>0$, a hypercube in $\X$ with the center at $2qz\in\X$ and the side $2q$:
\[
H_q(z):=\{y\in\X\mid 2z_iq-q\leq y_i\leq 2z_iq+q, i=1,\ldots,d\}.
\]
\begin{theorem}\label{thm:unibdd}
Suppose that there exists $r_0>0$ such that
\begin{equation}\label{infpos}
  \alpha:=\inf\limits_{|x|\leq r_0} a^-(x)>0.
\end{equation}
Suppose also that, for some $q\in\bigl(0,\frac{r_0}{2\sqrt{d}}\bigr]$,
\begin{equation}\label{seriesconv}
  a^+_q:=\sum_{z\in\Z^d}\sup\limits_{x\in H_q(z)} a^+(x) <\infty
\end{equation}
(e.g. let, for some $\eps>0$, $A>0$, one have $a^+(x)\leq \frac{A}{1+|x|^{d+\eps}}$, for a.a. $x\in\X$).
Then, the solution $u\geq0$ to \eqref{eq:basic}, with $0\leq u_0\in \Buc$, belongs to $\Xinf$.
\end{theorem}
\begin{proof}
If $m\geq\ka^+$ then the statement is trivially followed from \eqref{expbdd}. Suppose that $m<\ka^+$ and rewrite \eqref{eq:basic} in the form
\begin{equation}\label{usefulform}
\frac{\partial}{\partial t} u(x,t)=(L_{a^+} u)(x,t)+\ka^- u(x,t)\bigl(\theta - (a^-*u)(x,t)\bigr),
\end{equation}
where $\theta=\dfrac{\ka^+-m}{\ka^-}>0$ and the operator $L_{a^+}$ acts in $x$ and is given by~\eqref{jump}.

It is easily seen that $H_q(z)\subset B_{q\sqrt{d}}(2qz)$, $z\in\Z^d$, $q>0$. Take any $q\leq\frac{r_0}{2\sqrt{d}}$ such that \eqref{seriesconv} holds, and set
$r=q\sqrt{d}\leq\frac{r_0}{2}$.
Define
\begin{equation}\label{defofv}
  v( x,t) :=(\1_{B_{r}(0)}*u)(x,t)=\int_{B_{r}(x)}u(y,t)\,dy.
\end{equation}
By Lemma~\ref{le:simple}, Theorem~\ref{thm:exist_uniq_BUC}, \eqref{expbdd}, $0\leq v\in\xt$, $T>0$, and
\begin{equation*}
  \|v(\cdot,t)\| \leq b_{r} e^{(\ka^+-m)t}\|u_0\|, \quad t\geq0.
\end{equation*}
Note that, by \eqref{jump},
\[
L_{a^+}v=\ka^+a^+*\1_{B_{r}(0)}*u-\ka^+\1_{B_{r}(0)}*u=\1_{B_{r}(0)}*(L_{a^+}u).
\]
Therefore,
\begin{align}
\frac{\partial }{\partial t}v(x,t)-(L_{a^{+}}v)(x,t)
=&\Bigl(\1_{B_{r}(0)}*\frac{\partial }{\partial t} u\Bigr)(x,t)-\bigl(\1_{B_{r}(0)}*(L_{a^+}u)\bigr)(x,t)\notag\\
=&\,\ka ^{-}\int_{B_{r}(x)} u\left( y,t\right) \left( \theta -( a^{-}\ast u) \left(
y,t\right) \right) dy.\label{eq:strange}
\end{align}
By \eqref{defofv}, one has $\|v(\cdot ,0)\| \leq b_{r}\|u_0\|$. Set
\begin{equation}\label{Mchoice}
M>\max \Bigl\{ b_{r}\left\Vert u_0\right\Vert ,\frac{\theta}{\alpha}\Bigr\}.
\end{equation}

First, we will prove that
\begin{equation}\label{estforv}
  \|v(\cdot,t)\|\leq M, \qquad t\geq0.
\end{equation}
On the contrary, suppose that there exists $t'>0$ such that $\|v(\cdot,t')\| >M$. By~\eqref{defofv} and Lemma~\ref{le:simple}, $\|v(\cdot,t)\|$ is continuous in $t$. Next, since $\|v(\cdot,0)\|<M$, there exists $t_0>0$ such that $\|v(\cdot,t_0)\| =M$ and $\|v(\cdot,t)\|<M$, for all $t\in [0,t_0)$.

Consider the sequence $\{x_n\}\subset\X$ such that $v(x_n,t_0)\to M$, $n\to\infty$.
Define the following functions:
\[
u_n(x,t):=u(x+x_n,t), \quad v_n(x,t):=v(x+x_n,t)=(\1_{B_{r}(0)}*u_n)(x,t),
\]
for $x\in\X, t\geq0$. Take any $T>0$. Evidently, $u\in C_{ub}(\X\times[0,T])$, then, for any $\eps>0$, there exists $\delta>0$ such that, for any $x,y\in\X$, $t,s\in[0,T]$, with $|x-y|_{\X}+|t-s|<\delta$, one has $|u_n(x,t)-u_n(y,s)|=|u(x+x_n,t)-u(y+x_n,s)|<\eps$. And, by \eqref{expbdd},
\begin{equation} \label{uniest}
\|u_n(\cdot,t)\| \leq \|u(\cdot,t)\| \leq e^{(\ka^+-m)T}\|u_0\|, \quad n\in\N, t\in [0,T].
\end{equation}
Hence $\{u_n\}$ is a uniformly bounded and uniformly equicontinuous sequence of functions on $\X\times[0,T]$. Thus, by a version of the Arzel\`{a}--Ascoli Theorem, see e.g. \cite[Appendix C.8]{Eva2010}, there exists a subsequence $\{u_{n_k}\}$ and a continuous function $u_\infty$ on $\X\times[0,T]$ such that $u_{n_k}\locun u_\infty$. Moreover, one can easily show that $u_\infty\in C_{ub}(\X\times[0,T])$.
By~\eqref{uniest} and Lemma~\ref{convluc}, $v_{n_k}\locun v_\infty =\1_{B_{r}(0)}*u_\infty$, moreover, $v_\infty\in C_{ub}(\X\times[0,T])$.

It is easily seen that both parts of \eqref{eq:strange} belong to $\xt$. Hence one can integrate \eqref{eq:strange} on $[0,t]\subset[0,T]$, namely,
\begin{align}
v(x,t)=v(x,0)&+\int_0^t (L_{a^{+}}v)(x,s)\,ds\notag\\&
+\ka ^{-}\int_0^t \int_{B_{r}(x)} u\left( y,s\right) \left( \theta -( a^{-}\ast u) \left(
y,s\right) \right) dy\,ds.\label{eq:strange2}
\end{align}
Substitute $x+x_{n_k}$ instead of $x$ into \eqref{eq:strange2} and use twice the  integration by substitution in the second integral, then one gets the same equality \eqref{eq:strange2}, but for $v_{n_k}$, $u_{n_k}$ instead of $v,u$, respectively. Next, by~Lemma~\ref{convluc} and the dominated convergence arguments, one can pass to the limit in $k$ in the obtained equality. As a result, one get \eqref{eq:strange2} for $v_\infty$, $u_\infty$ instead of $v$ and $u$, respectively.
Next, since $C_{ub}(\X\times[0,T])\subset \xt$, the integrands with respect to $s$ in the left hand side (l.h.s. in the sequel) of the modified equation \eqref{eq:strange2} (with $u_\infty, v_\infty\in \xt$) will belong to $\xt$ as well. As~a~result, $v_\infty$ will be differentiable in $t$ in the sense of the norm in $\Buc$. Finally, after differentiation, one get \eqref{eq:strange} back, but for $v_\infty, u_\infty$, namely,
\begin{multline}
\frac{\partial }{\partial t}v_\infty(x,t)-(L_{a^{+}}v_\infty)(x,t)
\\=\ka ^{-}\int_{B_{r}(x)} u_\infty\left( y,t\right) \left( \theta -( a^{-}\ast u_\infty) \left(
y,t\right) \right) dy.\label{eq:strange3}
\end{multline}

Going back to the definition of $x_n$, one can see that
\begin{equation}\label{touchM}
v_\infty(0,t_0)=\lim_{k\to\infty}v_{n_k}(0,t_0)=\lim_{k\to\infty}v(x_{n_k},t_0)=M,
\end{equation}
whereas, for any $x\in\X$, $t\in[0,t_0)$, $v_\infty(x,t)=\lim\limits_{k\to\infty}v(x+x_{n_k},t)\leq M$. Therefore, $\frac{\partial }{\partial t}v_\infty(0,t_0)\geq0$ and, by~\eqref{jump}, $(L_{a^+} v_\infty)(0,t_0)\leq0$. Then, by \eqref{eq:strange3},
\begin{equation}\label{sasd}
  \int_{B_{r}(0)} u_\infty\left( y,t_0\right) \left( \theta -( a^{-}\ast u_\infty) \left(
y,t_0\right) \right) dy\geq0.
\end{equation}
Next, the function $u_\infty(\cdot,t_0)$, by the construction above, is nonnegative. It can not be identically equal to $0$ on $B_{r}(0)$, since otherwise, by \eqref{defofv}, $v_\infty(0,t_0)=0$ that contradicts \eqref{touchM}. Hence by~\eqref{sasd}, the function $\theta - (a^-*u_\infty)(\cdot,t_0)$ cannot be strictly negative on $B_{r}(0)$. Thus, there exists $y_0\in B_{r}(0)$ such that $\theta \geq (a^-*u_\infty)(y_0,t_0)$. Since $2r\leq r_0$, one has that $\inf\limits_{x\in B_{2r}(0)}a^-(x)\geq \alpha$, cf.~\eqref{infpos}. Therefore, one can continue:
\begin{align*}
\theta &\geq (a^-*u_\infty)(y_0,t_0)\geq \int_{B_{2r}(0)} a^-(y) u_\infty (y_0-y,t_0)\,dy\\&\geq \alpha \int_{B_{2r}(0)}u_\infty(y_0-y,t_0)\,dy
= \alpha \int_{B_{2r}(y_0)}u_\infty(y,t_0)\,dy\\
&\geq\alpha \int_{B_{r}(0)}u_\infty(y,t_0)\,dy=\alpha v_\infty(0,t_0)=\alpha M,
\end{align*}
that contradicts \eqref{Mchoice}. Therefore, our assumption was wrong, and \eqref{estforv} holds.

We proceed now to show that $\|u(\cdot,t)\|$ is uniformly bounded in time. By \eqref{defofv}, \eqref{estforv}, \eqref{seriesconv}, one has, for $r=q\sqrt{d}$,
\begin{align}
  (a^+*u)(x,t)&=\sum_{z\in\Z^d}\int_{H_q(z)} a^+(y)u(x-y,t)\,dy\notag\\
  &\leq \sum_{z\in\Z^d}\sup\limits_{y\in H_q(z)}a^+(y)\int_{B_r(2qz)} u(x-y,t)\,dy\notag\\
  &=\sum_{z\in\Z^d}\sup\limits_{y\in H_q(z)}a^+(y)\int_{B_r(x-2qz)} u(y,t)\,dy
  \leq M a^+_q.\label{convest}
\end{align}
Therefore, by \eqref{eq:basic}, \eqref{convest}, using the same arguments as for the proof of \eqref{expbdd} one gets that
\begin{align}
  0\leq u(x,t) &\leq e^{-mt}u_0(x)+\int_0^t e^{-(t-s)m}\ka^+ Ma^+_q\,ds\notag\\ &=
e^{-mt}u_0(x)+\dfrac{\ka^+ Ma^+_q}{m}(1-e^{-mt})\notag \\&\leq \max\Bigl\{\frac{\ka^+ Ma^+_q}{m},\|u_0\|\Bigr\}, \qquad x\in\X, \ t\geq0,
  \label{globalbound}
\end{align}
that fulfills the proof.
\end{proof}
\begin{remark}
  It should be stressed that we essentially used the uniform continuity of the solution to prove Theorem~\ref{thm:unibdd}.
\end{remark}

Under conditions of Theorem~\ref{thm:unibdd}, the solution $u$ will be uniformly continuous on $\X\times\R_+$, namely, the following simple proposition holds true.
\begin{proposition}\label{prop:u_in_BUC}
Let $u$ be a solution to \eqref{eq:basic} with $u_0\in \Buc$, and suppose that there exists $C>0$, such that
\[
\left|u(x,t)\right|\le C,\quad x\in\X ,\ t\ge0.
\]
Then $u\in \BUC$. Moreover, $\|u(\cdot,t)\|\in C_{ub}(\R_+)$.
\end{proposition}
\begin{proof}
Being solution to \eqref{eq:basic}, $u$ satisfies the integral equation
\[
u(x,t)=u_0(x)+\int _0^t\bigl(\ka^{+}(a^{+}*u)(x,s)-\ka^{-}u(x,s)(a^{-}*u)(x,s)-mu(x,s)\bigr)\,ds.
\]
Hence for any $x,y\in\X$, $0\leq\tau< t$, one has
\begin{align*}
\left|u(x,t)-u(y,\tau)\right|&\leq\int_\tau^t(2\ka^{+}C+2\ka^{-}C^{2}+2mC)ds\\
&=2(\ka^{+}+\ka^{-}C+m)C(t-\tau),
\end{align*}
that fulfills the proof of the first statement. Then, the second one follows from the inequality $\bigl\lvert\|u(\cdot,t)\|-\|u(\cdot,\tau)\|\bigr\rvert\leq \|u(\cdot,t)-u(\cdot,\tau)\|$.
\end{proof}

\section{Around the comparison principle}
\label{sec:comparison}

The comparison principle is one of the basic tools for the study of elliptic and parabolic PDE. It is widely use for the nonlocal diffusion equation \eqref{nl+l} (see e.g. \cite{CDM2008}), however, it does not hold, in general, for \eqref{l+nl} (see e.g. \cite{AC2012,HR2014} and the references therein). We will find the sufficient conditions  (see \eqref{as:chiplus_gr_m} and \eqref{as:aplus_gr_aminus} below), under which the comparison principle for the equation \eqref{eq:basic} does hold and which will be the basic conditions for all our further settings. Moreover, one can show a necessity of these conditions (Remark~\ref{rem:necineq}). Subsection~\ref{subsec:max_principle} is devoted to the maximum principle, which is a counterpart of the comparison one for parabolic ODE. In particular, Theorem~\ref{thm:strongmaxprinciple} states that graphs of two different solutions to \eqref{eq:basic} never touch. The last Subsection gives further technical tools which will be explored through the paper.

\subsection{Comparison principle}\label{subsec:comparison_pr}

Let $T>0$ be fixed. Define the sets $\xt^1$ and $\tilde{\mathcal{X}}_T^1$ of functions from $\xt$, respectively, $\tilde{\mathcal{X}}_T$, which are continuously differentiable on $(0,T]$ in the sense of the norm in $\Buc$, respectively, in $L^\infty(\X)$. Here and below we consider the left derivative at $t=T$ only. For any $u$ from $\xt^1$ one can define the following function
\begin{multline}\label{Foper}
  (\mathcal{F}u)(x,t):=\dfrac{\partial u}{\partial t}(x,t)-\ka^{+}(a^{+}*u)(x,t)\\+mu(x,t)+\ka^{-}u(x,t)(a^{-}*u)(x,t)
\end{multline}
for all $t\in(0,T]$ and all $x\in\X$.
Moreover, for any $u\in\tilde{\mathcal{X}}_T^1$, one can consider the function $\dfrac{\partial u}{\partial t}(\cdot,t)\in L^\infty(\X)$, for all $t\in(0,T]$. Then, one can also define \eqref{Foper}, which will considered a.e. in $x\in\X$ now.

\begin{theorem}\label{thm:compar_pr}
Let there exist $c>0$, such that
\begin{equation}\label{compare}
  \ka^+ a^+(x)\geq c\ka^- a^-(x), \quad \text{a.a. } x\in\X.
\end{equation}
Let $T\in(0,\infty)$ be fixed and functions $u_{1},u_{2}\in\xt^1$ be such that, for any $(x,t)\in\X \times(0,T]$,
\begin{gather}
(\mathcal{F}u_{1})(x,t)\leq(\mathcal{F}u_{2})(x,t),\label{eq:max_pr_BUC:ineq}\\
u_{1}(x,t)\geq0, \qquad 0\leq u_{2}(x,t)\leq c, \qquad u_{1}(x,0)\leq u_{2}(x,0).\label{eq:init_for_compar}
\end{gather}
Then $u_{1}(x,t)\leq u_{2}(x,t)$, for all $(x,t)\in\X \times[0,T]$. In particular, $u_{1}\leq c$.
\end{theorem}
\begin{proof}
Define the following function
\begin{equation}
f(x,t):=(\mathcal{F}u_{2})(x,t)-(\mathcal{F}u_{1})(x,t)\geq0, \quad x\in\X, t\in(0,T],\label{dif_pos}
\end{equation}
cf. \eqref{eq:max_pr_BUC:ineq}. We set
\begin{equation}
K= m+\ka^{-}\Vert u_{1}\Vert_{T},\label{eq:max_pr_BUC:K_as}
\end{equation}
and consider a linear mapping
\begin{multline}\label{mapF}
  F(t,w):=Kw-mw+\ka^{+}(a^{+}*w)\\-\ka^{-}w(a^{-}*u_{1})-\ka^{-}u_{2}(a^{-}*w)+e^{Kt}f,
\end{multline}
for $w\in\xt$.
By \eqref{eq:init_for_compar}, \eqref{dif_pos}, \eqref{eq:max_pr_BUC:K_as}, \eqref{compare}, \eqref{convbdd}, $w\geq0$ implies
\begin{align}\notag
F(t,w)&=Kw-mw+(\ka^{+}a^{+}-c\ka^-a^-)*w\\&\quad +\ka^{-}(c-u_{2})(a^{-}*w)
-\ka^{-}w(a^{-}*u_{1})+e^{Kt}f \notag\\
&\geq Kw-mw+(\ka^{+}a^{+}-c\ka^{-} a^{-})*w-\ka^{-}\Vert u_{1}\Vert_{T}w+e^{Kt}f\ge0.
\label{Fispos}
\end{align}
Define also the function
\begin{equation*}
v(x,t):=e^{Kt}(u_{2}(x,t)-u_{1}(x,t)), \quad x\in\X, t\in[0,T].
\end{equation*}
Clearly, $v\in\xt^1$, and it is straightforward to check that
\begin{equation}\label{tautology}
  F(t,v(s,t))=\frac{\partial}{\partial t} v(x,t),
\end{equation}
for all $x\in\X$, $t\in(0,T]$. Therefore, $v$ solves the following integral equation in $\Buc$:
\begin{equation}
\begin{cases}
\displaystyle v(x,t)=v(x,0)+\int_0^t F(s,v(x,s))ds, & \quad (x,t)\in\X {\times}(0,T],\\[3mm]
v(x,0)=u_{2}(x,0)-u_{1}(x,0), & \quad x\in\X,
\end{cases}\label{eq:max_pr_BUC:u2-u1_lin}
\end{equation}
where $v(x,0)\geq0$, by \eqref{eq:init_for_compar}.

Consider also another integral equation in $\Buc$:
\begin{align}
\tilde{v}(x,t)&=(\Psi \tilde{v})(x,t)\label{eq:max_pr_BUC:pos_sol} \\
\shortintertext{where}
(\Psi w)(x,t)&:=v(x,0)+\int_0^t\max\{F(s,w(x,s)),0\}\,ds,\qquad w\in\xt.\label{eq:max_pr_BUC:Psi}
\end{align}
It is easily seen that $w\in\xt^+$ yields $\Phi w\in\xt^+$.
Next, for any $\tilde{T}< T$ and for any $w_{1},w_{2}\in\mathcal{X}_{\tilde{T}}^+$, one gets from \eqref{mapF}, \eqref{eq:max_pr_BUC:Psi}, that
\begin{align}
\|\Psi w_{1}-\Psi w_{2}\|_{\tilde{T}}
&\leq\tilde{T}(K+m+\ka^{+}+\ka^{-}\Vert u_{1}\Vert_{T}+c\ka^{-})\|w_{2}-w_{1}\|_{\tilde{T}}\notag\\
&: =q_T \tilde{T}\|w_{2}-w_{1}\|_{\tilde{T}},\label{contractioncomp}
\end{align}
where we used the elementary inequality $\lvert \max\{a,0\}-\max\{b,0\}\rvert\leq |a-b|$, $a,b\in\R$.
Therefore, for $\tilde{T}<(q_T)^{-1}$, $\Psi$ is a contraction on $\mathcal{X}_{\tilde{T}}^+$. Thus, there exists a unique solution
to \eqref{eq:max_pr_BUC:pos_sol} on $[0,\tilde{T}]$. In the same
way, the solution can be extended on $[\tilde{T},2\tilde{T}], [2\tilde{T},3\tilde{T}]$, \ldots, and therefore, on the whole $[0,T]$.
By \eqref{eq:max_pr_BUC:pos_sol}, \eqref{eq:max_pr_BUC:Psi},
\begin{equation}\label{result}
\tilde{v}(x,t)\ge v(x,0)\ge0,
\end{equation}
hence, by \eqref{Fispos}, \eqref{eq:max_pr_BUC:Psi},
\begin{equation}\label{safas}
\tilde{v}(x,t)=v(x,0)+\int_0^t F(s,\tilde{v}(x,s))\,ds=:\Xi(\tilde{v})(x,t).
\end{equation}
Since $\tilde{v}\in\xt$, \eqref{safas} implies that $\tilde{v}$ is a solution to \eqref{eq:max_pr_BUC:u2-u1_lin} as well.
The same estimate as in \eqref{contractioncomp} shows that $\Xi$ is a contraction on $\mathcal{X}_{\tilde{T}}$, for small enough $\tilde{T}$. Thus $\tilde{v}=v$ on $\X \times[0,\tilde{T}]$, and one continue this consideration as before on the whole $[0,T]$.
Then, by \eqref{result}, $v(x,t)\geq0$ on $\X \times[0,T]$, that yields the statement.
\end{proof}

\begin{remark}\label{rem:max_pr_Linf}
The previous theorem holds true in $\tilde{\mathcal{X}}_T^1$. Here and below, for the $L^\infty$-case, one can assume that \eqref{eq:max_pr_BUC:ineq}, \eqref{eq:init_for_compar} hold almost everywhere in $x$ only.
\end{remark}

From the proof of Theorem~\ref{thm:compar_pr}, one can see that we used the fact that $u_1,u_2$ belong to $\xt^1$ to ensure that \eqref{tautology} implies \eqref{eq:max_pr_BUC:u2-u1_lin} only. For technical reasons we will need to extend the result of Theorem~\ref{thm:compar_pr} for a wider class of functions. Naturally, to get
\eqref{eq:max_pr_BUC:u2-u1_lin} from \eqref{tautology}, it is enough to assume absolute continuity of $v(x,t)$ in $t$, for a fixed $x$. Consider the corresponding statement.

For any $T\in(0,\infty]$, define the set $\D_T$ of all functions $u:\X\times\R_+\to\R$, such that, for all $t\in[0,T)$, $u(\cdot,t)\in \Buc$, and, for all $x\in\X$, the function $f(x,t)$ is absolutely continuous in $t$ on $[0,T)$. Then, for any $u\in\D_T$, one can define the function \eqref{Foper}, for all $x\in\X$ and a.a. $t\in[0,T)$.

\begin{proposition}\label{compprabscont}
The statement of Theorem~\ref{thm:compar_pr} remains true, if we assume that $u_1,u_2\in\D_T$ and, for any $x\in\X$, the inequality \eqref{eq:max_pr_BUC:ineq} holds for a.a.~$t\in(0,T)$~only.
\end{proposition}
\begin{proof}
  One can literally repeat the proof of Theorem~\ref{thm:compar_pr}; for any $x\in\X$, the function \eqref{dif_pos} and the mapping \eqref{mapF} will be defined for a.a. $t\in(0,T)$ now (and it will not be a mapping on $\xt$, of course). Similarly, \eqref{Fispos} and \eqref{tautology} hold, for all $x$ and a.a. $t$. However, for any $x\in\X$, one gets that \eqref{eq:max_pr_BUC:u2-u1_lin} holds still for all $t\in[0,T]$. Hence, the rest of the proof remains the same, stress that, in general, $F(t)v\notin\xt$, whereas $\Xi(v)\in\xt$, cf.~\eqref{safas}.
\end{proof}

The standard way to use Theorem~\ref{thm:compar_pr} is to take $u_1$ and $u_2$ which solve \eqref{eq:basic}, thus, $\mathcal{F}u_1=\mathcal{F}u_2=0$, and \eqref{eq:max_pr_BUC:ineq} holds. Then Theorem~\ref{thm:compar_pr} gives a comparison between these solutions provided that there exists a comparison between the initial conditions. However, to do this, one needs to know {\it \`{a} priori} that $u_2(x,t)\leq c$. For example, one can demand that $c$ is not smaller than the constant in the r.h.s. of \eqref{globalbound}. Another possibility is to compare the solution to \eqref{eq:basic} with the solution to its homogeneous version \eqref{eq:homogen} (with $t_0=0$).

Namely, let \eqref{compare} hold, $0<\upsilon\leq c$, and, cf.~\eqref{homogensol},
\begin{equation*}
\begin{split}
\psi(t,\upsilon)&:=\dfrac{\upsilon}{\upsilon g_\theta(t)+\exp(-\theta\ka^- t)}\geq0,\\ g_\theta(t)&:=\lim\limits_{y\to\theta}\dfrac{1-\exp(-y\ka^-t)}{y}\geq0.
\end{split}
\end{equation*}
It is easily seen that, for $\theta\leq0$, $\psi(t,\upsilon)$ decreases monotonically to $0$ on $t\in[0,\infty)$: exponentially fast, for $\theta<0$, and linearly fast, for $\theta=0$. In particular, $\psi(t,\upsilon)\leq \upsilon\leq c$, $t\geq0$. As a result,
\begin{itemize}
  \item if $\ka^+\leq m$ and $0\leq u_0\in \Buc$ be such that $\|u_0\|\leq c$, then $\|u(\cdot,t)\|\leq \psi\bigl(t,\|u_0\|\bigr)$. In particular, $u$ converges to $0$  uniformly in space as $t\to\infty$.
\end{itemize}

Next, for $\theta>0$, the function $\psi(t,\upsilon)$ increases monotonically to $\theta$ on $t\in[0,\infty)$, if~$\upsilon<\theta$; and it decreases monotonically to $\theta$, if $\upsilon>\theta$, and, clearly, $\psi(t)\equiv\theta$, if~$\upsilon=\theta$.~Therefore, if~\eqref{compare} holds with $c>\theta$ and $0 <\|u_0\|\leq c$ then $\psi(t,\|u_0\|)\leq \|u_0\|\leq c$, and therefore, $\|u(\cdot,t)\|\leq \psi\bigl(t,\|u_0\|\bigr)\to\theta$, $t\to\infty$. Set also $\inf\limits_{\X} u_0(x)=:\beta\geq0$, then one can apply the comparison principle to the functions $u_1=\psi(t,\beta)$ and $u_2=u$. (Note that $\psi(t,0)=0$.)
As a result,
\begin{itemize}
  \item if $\ka^+>m$ and $0\leq u_0\in \Buc$ be such that $0 <\|u_0\|\leq c$, then $\psi(t,\beta)\leq u(x,t)\leq \psi\bigl(t,\|u_0\|\bigr)$, $x\in\X$, $t\geq0$, where $\beta=\inf\limits_{\X} u_0(x)\geq0$. In particular, if $\beta>0$ then $u$ converges to $\theta$ exponentially fast as $t\to\infty$ and uniformly in space.
\end{itemize}

Consider the case in which \eqref{compare} holds with $c\geq\theta$ and $\|u_0\|\leq \theta$,  in more details. Then, one can set $u_2\equiv \theta$ (that is a solution to \eqref{eq:basic}), and $\|u(\cdot,t)\|\leq\theta=\psi(t,\theta)$. Of course, for this case it is enough to have \eqref{compare} with $c=\theta$ only. The latter constitutes the following basic assumptions for the most part of our further results:
\begin{assum}\label{as:chiplus_gr_m}
\ka^{+}>m,
\end{assum}\vspace{-4mm}
\begin{assum}\label{as:aplus_gr_aminus}
\ka^{+}a^{+}(x)\geq(\ka^{+}-m)a^{-}(x),\quad \text{a.a.}\ x\in\X.
\end{assum}

\begin{proposition} \label{prop:comp_pr_BUC}
Suppose that \eqref{as:chiplus_gr_m} and \eqref{as:aplus_gr_aminus} hold.
Let $0\leq u_0\in \Buc$ be an initial condition to \eqref{eq:basic} and $u\in\xt$ be the corresponding solutions on any $[0,T]$, $T>0$. Suppose that
$0\leq u_0(x)\leq \theta$, $x\in\X$.
Then $u\in\Xinf$, with $\|u\|_\infty\leq \theta$.

Let $v_0\in \Buc$ be another initial condition to \eqref{eq:basic} such that $u_0(x)\leq v_0(x)\leq\theta$, $x\in\X$; and $v\in\Xinf$ be the corresponding solution. Then
\begin{equation*}
u(x,t)\leq v(x,t), \quad x\in\X, t\geq0.
\end{equation*}

If, additionally, $\beta:=\inf\limits_{x\in\X} u_0(x)>0$, then
\begin{equation}\label{doubleest}
  \dfrac{\beta\theta}{\beta +(\theta-\beta)\exp(-\theta\ka^- t)}\leq u(x,t)\leq \theta, \quad x\in\X, t\geq0.
\end{equation}
In particular,
\begin{equation*}
\|u(\cdot,t)-\theta\|\leq \frac{\theta-\beta}{\beta} \exp(-\theta\ka^- t), \quad t\geq0.
\end{equation*}
\end{proposition}
\begin{proof}
The first two parts were proved above; note that $\theta\ka^-=\ka^+-m$. The last one is followed from the definition of the function $\psi$ above and the estimate for the difference between low and upper bounds in \eqref{doubleest}.
\end{proof}
\begin{remark}
\label{rem:comp_pr_Linf}
The same result may be formulated for $\tilde{\mathcal{X}}_T$ and $\tilde{\mathcal{X}}_\infty$. All inequalities will hold true almost everywhere only.
\end{remark}

We did not consider all possible relations between $c$, $\theta>0$, and $\|u_0\|$. In particular, the previous-type considerations do not cover the situation in which \eqref{compare} holds with $c<\theta$. In such a case, the solution to \eqref{eq:homogen} (with $t_0=0$) can not be considered as a function $u_2$ in Theorem~\ref{thm:compar_pr} since that solution tends to $\theta$ as $t\to\infty$, hence, \eqref{eq:init_for_compar} will not hold. This situation remains open.

Another case, which is not covered by the comparison method is the following: let $\theta>0$, i.e. \eqref{as:chiplus_gr_m} holds, and $\|u_0\|>c$. However, it may be analyzed using stability arguments provided that $c\geq\theta$, the latter evidently implies \eqref{as:aplus_gr_aminus}. Under assumptions \eqref{as:chiplus_gr_m}, \eqref{as:aplus_gr_aminus}, we set, cf.~\eqref{Jq} below,
\begin{equation}\label{diffofkernels}
  J_\theta(x):=\ka^+a^+(x)-(\ka^+-m)a^-(x)\geq0,\quad x\in\X.
\end{equation}
Next, denote the r.h.s. of \eqref{eq:basic} by $G(u)$. Recall, that $G(\theta)=0$, hence, $u^*\equiv\theta$ is a stationary solution to~\eqref{eq:basic}. Another stationary solution is $u_*\equiv0$. Consider the stability property of these solutions. To do this, find the linear operator $G'(u)$ on $\Buc$: for $v\in\Buc$,
\begin{align}\label{derofG}
  G'(u)v&=\frac{d}{ds} G(u+sv)\biggr|_{s=0}\notag\\&=\ka^+(a^+*v)-mv-\ka^-v(a^-*u)-\ka^-u(a^-*v).
\end{align}
Therefore, by \eqref{diffofkernels},
\[
G'(\theta)v=\ka^+(a^+*v)-mv-\ka^-\theta v-\ka^-\theta(a^-*v)=J_\theta*v-\ka^+ v.
\]
By \eqref{diffofkernels}, $\int_\X J_\theta(x)\,dx=m$, thus, the spectrum $\sigma(A)$ of the operator $Av:=J_\theta *v$ on $\Buc$ is a subset of $\{z\in\mathbb{C}\mid |z|\leq m\}$. Therefore,
\[
\sigma(G'(\theta))=\sigma(A-\ka^+)\subset\{z\in\mathbb{C}\mid |z+\ka^+|\leq m\}\subset \{z\in\mathbb{C}\mid \mathrm{Re}\, z<0\},
\]
by \eqref{as:chiplus_gr_m}. Hence, by e.g. \cite[Chapter VII]{DK1974}, $u^*\equiv\theta$ is uniformly and asymptotically stable solution, in the sense of Lyapunov, i.e., for any $\eps>0$ there exists $\delta>0$ such that, for any solution $u\in\Buc$ to \eqref{eq:basic} and for all $t_1\geq0$, the inequality $\|u(\cdot,t_1)-\theta\|<\delta$ implies that, for any $t\geq t_1$, $\|u(\cdot,t)-\theta\|<\eps$; and, for some $\delta_0>0$, the inequality $\|u(\cdot,t_1)-\theta\|<\delta_0$ yields $\lim\limits_{t\to\infty}\|u(\cdot,t)-\theta\|=0$.
In~particular, it works if $\theta<\|u_0\|\leq \theta+\delta_0$. Moreover, it is possible to show that $u^*\equiv\theta$ is a globally asymptotically (exponentially) stable solution to \eqref{eq:basic}, that means, in particular, that $\|u_0\|>\theta$ may be arbitrary; we expect to discuss this in a forthcoming paper.

Note also, that, by \eqref{derofG},
$G'(0)v=\ka^+(a^+*v)-mv$. Then, the operator $G'(0)$ has an eigenvalue $\ka^+-m>0$ whose corresponding eigenfunctions will be constants on $\X$. Therefore $\sigma(G'(0))$ has points in the right half-plane and since $G''(0)$ exists, one has, again by \cite[Chapter VII]{DK1974}, that $u_*\equiv0$ is unstable, i.e. there exists a solution $u$ such that $\inf\limits_\X |u(x,t)|\geq \eps$, for some $\eps>0$, for all $x\in\X$ and for all $t\geq t_0=t_0(\eps)$. It is worth noting that, $u_*\equiv\theta$ is a locally stable solution, see Subsection~\ref{subsec:loc_stab} below.

The natural question arises whether it is possible to characterize some properties of the solution to \eqref{eq:basic} without comparison between $a^+$ and $a^-$ like \eqref{compare}. For a particular answer, one can refer to \cite[Theorem~4.3]{FKKozK2014}, namely, there was proved that if \eqref{compare} holds on a set $\Omega\subset\X$ of positive Lebesgue measure, and if $\int_\Omega (\ka^+a^+(x)-c\ka^-a^-(x))\,dx<m$, then $\|u_0\|<\theta$ implies $\|u(\cdot,t)\|\leq\theta$, $t>0$.

\begin{remark}\label{rem:necineq}
The condition \eqref{compare} is the necessary one to have a comparison principle for
nonnegative (essentially) bounded by the constant $c$ solutions to \eqref{eq:basic}, provided that $c\geq\theta$. To show this, consider, for simplicity, the case $c=\theta$. Let the condition \eqref{as:aplus_gr_aminus} fails in a ball $B_{r}(y_0)$ only, ${r}>0$, $y_0\in\X$, i.e. $J_\theta(x)<0$, for a.a. $x\in B_{r}(y_0)$, where $J_\theta$ is given by \eqref{diffofkernels}. Take any $y\in B_{r}(y_0)$ with $\frac{{r}}{4}<|y-y_0|<\frac{3{r}}{4}$, then $y_0\notin B_{\frac{{r}}{4}}(y)$ whereas $B_{\frac{{r}}{4}}(y)\subset B_{r}(y_0)$.
Take $u_0\in\Buc$ such that $u_0(x)=\theta$, $x\in \X\setminus B_{\frac{{r}}{4}}(y)$, and $u_0(x)< \theta$, $x\in B_{\frac{{r}}{4}}(y)$. Since $\int_\X J_\theta(x)\,dx=m$, one has
\begin{align*}
\frac{\partial u}{\partial t}(y_0,0)&=-m\theta+\ka^+(a^+*u)(y_0,\theta)-\ka^-\theta (a^-*u)(y_0,0)\notag\\&=(J_\theta*u)(y_0,0)-m\theta=(J_\theta*(u_0-\theta))(y_0)
\\&
=\int_{B_{\frac{{r}}{4}}(y)} J_\theta(y_0-x)(u_0(x)-\theta)\,dx>0,\notag
\end{align*}
Therefore, $u(y_0,t)>u(y_0,0)=\theta$, for small enough $t>0$, and hence, the statement of Proposition~\ref{prop:comp_pr_BUC} does not hold in this case. The similar counterexample may be considered if \eqref{compare} fails, for $c>\theta$. Note that the case $c<\theta$ is again unclear.
\end{remark}

\subsection{Maximum principle}\label{subsec:max_principle}

The maximum principle is a `standard counterpart' of the comparison principle, see e.g. \cite{Cov2007}.

We will present sufficient conditions that solutions to \eqref{eq:basic} never reach at positive times the stationary values $\theta$ and $0$, provided that the corresponding initial conditions were not these constants. Moreover, we will prove the so-called strong maximum principle (Theorem~\ref{thm:strongmaxprinciple}), cf. e.g. \cite{CDM2008}.

Through the rest of the paper we will suppose that \eqref{as:chiplus_gr_m}, \eqref{as:aplus_gr_aminus} hold and $\theta>0$ is given by \eqref{theta_def}. Under these assumptions, for any $q\in(0,\theta]$, one can generalize the function \eqref{diffofkernels} as follows
\begin{equation}\label{Jq}
    \begin{split}
    J_q(x):&=\ka^+a^+(x)-q\ka^-a^-(x), \\
    &\geq  \ka^+a^+(x)-\theta\ka^-a^-(x)\geq 0, \qquad x\in\X.
    \end{split}
  \end{equation}
since $\theta\ka^-=\ka^+-m$ and \eqref{as:aplus_gr_aminus} holds.

\begin{definition}
For $\theta>0$, given by \eqref{theta_def}, consider the following sets
\begin{align}
\Utheta&:=\{f\in \Buc \mid 0\le f(x)\le\theta,\ x\in\X \},\label{eq:B_theta}\\
\Ltheta&:=\{f\in L^{\infty}(\X)\mid 0\le f(x)\le\theta,\ \text{for a.a.}\ x\in\X \}.\label{eq:L_theta}
\end{align}
\end{definition}

We introduce also the following assumption:
\begin{assum}\label{as:a+nodeg}
\text{there exists $\rho,\delta>0$ such that} \ a^{+}(x)\geq\rho, \text{ for a.a. } x\in B_\delta(0).
\end{assum}

Prove that then the solutions to \eqref{eq:basic} (or, equivalently, \eqref{usefulform}) are strictly positive; this is quite common feature of linear parabolic equations, however, in general, it may fail for nonlinear ones.

\begin{proposition}\label{prop:u_gr_0}
Let \eqref{as:chiplus_gr_m}, \eqref{as:aplus_gr_aminus}, \eqref{as:a+nodeg}
hold. Let $u_0\in \Utheta$, $u_0\not\equiv0$, $u_0\not\equiv\theta$, be~the~initial condition to \eqref{eq:basic}, and $u\in\Xinf$ be the corresponding solution. Then
\[
u(x,t)>\inf_{\substack{y\in\X\\ s>0}}u(y,s)\geq0, \qquad x\in\X, t>0.
\]
\end{proposition}
\begin{proof}
By Theorem~\ref{thm:exist_uniq_BUC} and Proposition~\ref{prop:comp_pr_BUC}, $0\leq u(x,t)\leq \theta$, $x\in\X$, $t\geq0$. Then, by \eqref{usefulform},
\begin{equation}\label{hintineq}
  \dfrac{\partial u}{\partial t}(x,t)-(L_{a^{+}}u)(x,t) \ge0.
\end{equation}
Prove that, under \eqref{hintineq}, $u$ cannot attain its infimum on $\X\times(0,\infty)$ without being a constant. Indeed, suppose that, for some $x_0\in\X$, $t_0>0$,
\begin{equation}\label{minleq}
u(x_0,t_0)\leq u(x,t), \quad x\in\X, t>0.
\end{equation}
Then, clearly,
\begin{equation}\label{minpoint}
  \dfrac{\partial u}{\partial t}(x_0,t_0)=0,
\end{equation}
and \eqref{hintineq} yields
$(L_{a^{+}}u)(x_0,t_0)\le0$. On the other hand, \eqref{minleq} and \eqref{jump} imply
$(L_{a^{+}}u)(x_0,t_0)\ge0$. Therefore,
\begin{equation}\label{backgrforfurther}
  \int_\X a^+(x_0-y)(u(y,t_0)-u(x_0,t_0)) \,dy=0.
\end{equation}
Then, by \eqref{as:a+nodeg}, for all $y\in B_\delta(x_0)$,
\begin{equation}
u(y,t_0)=u(x_0,t_0).\label{eq:getmax}
\end{equation}
By the same arguments, for an arbitrary $x_{1}\in\partial B_{\delta}(x_0)$,
we obtain \eqref{eq:getmax}, for all $y\in B_{\delta}(x_{1})$.
Hence, \eqref{eq:getmax} holds on $B_{2\delta}(x_0)$, and so on. As a result, \eqref{eq:getmax}~holds, for all $y\in\X $, thus $u(\cdot,t_0)$ is a constant.
Then, considering \eqref{eq:basic} at $(x_0,t_0)$, and taking into account \eqref{minpoint}, one gets $u(x_0,t_0) (\theta - u(x_0,t_0))=0$ with $u(x,t_0)=u(x_0,t_0)$, $x\in\X$; cf. \eqref{eq:homogen}. By \eqref{minleq}, $u(x_0,t_0)=\theta\geq\sup_{y\in\X,s>0}u(y,s)$ implies $u\equiv\theta$, that contradicts $u_0\not\equiv\theta$. Hence $u(x,t_0)=u(x_0,t_0)=0$, $x\in\X$. Then, by \eqref{homogensol}, $u(x,t)=0$, $x\in\X$, $t\geq t_0$.
And now one can consider the reverse time in \eqref{eq:basic} starting from $t=t_0$. Namely, we set $w(x,t):=u(x,t_0-t)$, $t\in[0,t_0]$, $x\in\X$. Then $w(x,0)=v(t_0)=0$, $x\in\X$, and
\begin{equation}
\dfrac{\partial w}{\partial t}(x,t)= mw(x,t) -\ka^{+}(a^{+}*w)(x,t)
 +\ka^{-}w(x,t)(a^{-}*w)(x,t).
\label{eq:inverse_basic}
\end{equation}
The equation \eqref{eq:inverse_basic} has a unique classical solution in $\Buc$ on $[0,t_0]$. Indeed, if $w_1,w_2\in\mathcal{X}_{t_0}$ both solve \eqref{eq:inverse_basic}, then the difference $w_2-w_1$ is a solution to the following linear equation
\begin{equation}
\begin{aligned}
\dfrac{\partial h}{\partial t}(x,t)&=mh(x,t) -\ka^{+}(a^{+}*h)(x,t)\\
 &\quad +\ka^{-}h(x,t)(a^{-}*w_2)(x,t)\ka^{-}w_{1}(x,t)(a^{-}*h)(x,t),
\end{aligned}
\label{eq:inverse_basic_lin}
\end{equation}
with $h(x,0)=0$, $x\in\X$. The r.h.s. of \eqref{eq:inverse_basic_lin}, for any $w_1,w_2\in\mathcal{X}_{t_0}$, is a bounded linear operator on $\Buc$, therefore, there exists a unique solution to \eqref{eq:inverse_basic_lin}, hence, $h\equiv0$. As a result, $w_1\equiv w_2$. Since $w\equiv0$ satisfies \eqref{eq:inverse_basic} with the initial condition above, one has $u(x,t_0-t)=0$, $t\in[0,t_0]$, $x\in\X$. Hence, $u(\cdot,t)\equiv0$, for all $t\geq0$, that contradicts $u_0\not\equiv0$. Thus, the initial assumption was wrong, and \eqref{minleq} can not hold.
\end{proof}

In contrast to the case of the infimum, the solution to \eqref{eq:basic} may attain its supremum but not the value $\theta$. One can prove this under a modified version of \eqref{as:a+nodeg}: suppose that, cf. \eqref{Jq},
\begin{assum}\label{as:aplus-aminus-is-pos}
\begin{gathered}
\text{there exists $\rho,\delta>0$, such that} \\
J_\theta(x)=\ka^{+}a^{+}(x)-(\ka^{+}-m)a^{-}(x)\geq\rho, \text{ for a.a. }  x\in B_\delta(0).
\end{gathered}
\end{assum}
As a matter of fact, under \eqref{as:aplus-aminus-is-pos}, a much stronger statement than unattainability of $\theta$ does hold.

\begin{theorem}\label{thm:strongmaxprinciple}
Let \eqref{as:chiplus_gr_m}, \eqref{as:aplus_gr_aminus}, \eqref{as:aplus-aminus-is-pos}
hold. Let $u_1,u_2\in\Xinf$ be two solutions to \eqref{eq:basic}, such that $u_1(x,t)\leq u_2(x,t)\leq\theta$, $x\in\X$, $t\geq0$. Then either $u_1(x,t)= u_2(x,t)$, $x\in\X$, $t\geq0$ or
$u_1(x,t)< u_2(x,t)$, $x\in\X$, $t>0$.
\end{theorem}
\begin{proof}
Let $u_1(x,t)\leq u_2(x,t)$, $x\in\X$, $t\geq0$, and suppose that there exist $t_0>0$, $x_0\in\X $,
such that $u_1(x_0,t_0)=u_2(x_0,t_0)$. Define $w:=u_2-u_1\in\Xinf$. Then $w(x,t)\geq0$ and $w(x_0,t_0)=0$, hence $\frac{\partial}{\partial t}w(x_0,t_0)=0$. Since both $u_1$ and $u_2$ solve \eqref{eq:basic}, one easily gets that $w$ satisfies the following linear equation
\begin{multline}\label{lineareqw}
  \frac{\partial}{\partial t}w(x,t)=(J_\theta* w)(x,t)+\ka^-(\theta-u_1(x,t))(a^-*w)(x,t)\\-w(x,t)\bigl(m+(a^-*u_2)(x,t)\bigr);
\end{multline}
or, at the point $(x_0,t_0)$, we will have
\begin{equation}\label{atx0t0}
  0=(J_\theta* w)(x_0,t_0)+\ka^-(\theta-u_1(x_0,t_0))(a^-*w)(x_0,t_0).
\end{equation}
Since the both summands in \eqref{atx0t0} are nonnegative, one has $(J_\theta* w)(x_0,t_0)=0$. Then, by \eqref{as:aplus-aminus-is-pos}, we have that $w(x,t_0)=0$, for all $x\in B_\delta (x_0)$. Using the same arguments as in the proof of Proposition~\ref{prop:u_gr_0}, one gets that $w(x,t_0)=0$, $x\in\X$. Then, by Corollary~\ref{cor:startwithconst}, $w(x,t)=0$, $x\in\X$, $t\geq t_0$. Finally, one can reverse the time in the linear equation \eqref{lineareqw} (cf.~the proof of Proposition~\ref{prop:u_gr_0}), and the uniqueness arguments imply that $w\equiv 0$, i.e. $u_1(x,t)= u_2(x,t)$, $x\in\X$, $t\geq0$. The statement is proved.
\end{proof}

By choosing $u_2\equiv\theta$ in Theorem~\ref{thm:strongmaxprinciple}, we immediately get the following
\begin{corollary}\label{cor:lesstheta}
Let \eqref{as:chiplus_gr_m}, \eqref{as:aplus_gr_aminus}, \eqref{as:aplus-aminus-is-pos}
hold. Let $u_0\in \Utheta$, $u_0\not\equiv\theta$, be~the~initial condition to \eqref{eq:basic}, and $u\in\Xinf$ be the corresponding solution. Then
$u(x,t)<\theta$, $x\in\X$, $t>0$.
\end{corollary}

\subsection{Further toolkits}\label{subsec:loc_stab}

We start with the proof that any solution to \eqref{eq:basic} is locally stable with respect to the locally uniform convergence of Definition~\ref{def:loc_conv}, provided that \eqref{compare} holds. This stability is very `weak', for example, $u_*\equiv 0$, being unstable solution (see Subsection~\ref{subsec:comparison_pr} above), will be still locally stable.

\begin{theorem}\label{thm:loc_stab_BUC}
Let \eqref{as:chiplus_gr_m}, \eqref{as:aplus_gr_aminus} hold.
Let $T>0$ be fixed. Consider a sequence of functions $u_{n}\in\xt$ which are solutions
to \ref{eq:basic} with uniformly bounded initial conditions: $u_{n}(\cdot,0)\in \Utheta$, $n\in\N$. Let $u\in\xt$ be a solution to \eqref{eq:basic} with initial condition
$u(\cdot,0)$ such that
$u_{n}(\cdot,0)\locun u(\cdot,0)$. Then
$u_{n}(\cdot,t)\locun u(\cdot,t)$, uniformly in $t\in[0,T]$.
\end{theorem}
\begin{proof}
It is easily seen that $u(\cdot,0)\in\Utheta$. By Proposition~\ref{prop:comp_pr_BUC}, $u_n(\cdot,t), u(\cdot,t)\in\Utheta$, $n\in\N$, for any $t\geq0$.
We define, for any $n\in\N$, the following functions on $\X$:
\[
\overline{u}_{n}(x,0):=\max\left\{ u_{n}(x,0),u(x,0)\right\},\qquad \underline{u}_{n}(x,0):=\min\left\{ u_{n}(x,0),u(x,0)\right\}.
\]
Then, clearly, $0\le\underline{u}_n(x,0)\le u(x,0)\le\overline{u}_n(x,0)\le\theta$, $x\in\X$, $n\in\N$. Hence the corresponding solutions $\overline{u}_{n}(x,t)$, $\underline{u}_{n}(x,t)$ to  \eqref{eq:basic} belongs to $\Utheta$ as well. By Theorem \ref{prop:comp_pr_BUC}, one has
\[
\underline{u}_{n}(x,t)\leq u(x,t)\leq\overline{u}_{n}(x,t), \quad x\in\X, t\in[0,T].
\]
In the same way, one gets $\underline{u}_{n}(x,t)\leq u_n(x,t)\leq\overline{u}_{n}(x,t)$ on
$\X\times[0,T]$. Therefore, it is enough to prove that $\overline{u}_{n}$ and $\underline{u}_{n}$ converge locally uniformly to $u$.

Prove that $\overline{u}_{n}\locun u$. For any $n\in\N$, the function $h_{n}(\cdot,t)=\overline{u}_{n}(\cdot,t)-u(\cdot,t)\in\Utheta$, $t\geq0$, satisfies the equation
$\frac{\partial}{\partial t} h_{n}= A_n h_n$ with
$h_{n,0}(x):=h_n(x,0)=\overline{u}_{n}(x,0)-u(x,0)\geq0$, $x\in\X$, where, for any $0\leq h\in\xt$,
\[
A_{n}h:=-mh+\ka^{+}(a^{+}*h)-\ka^{-}h(a^{-}*\overline{u}_{n})
-\ka^{-}u(a^{-}*h).
\]
For any $u_n$ and $u$, $A_n$ is a bounded linear operator on $\Buc$, therefore,
$h_{n}(x,t)=(e^{tA_{n}}h_{n,0})(x)$, $x\in\X$, $t\in[0,T]$. Since $u\geq0$, one has that, for any $0\leq h\in\xt$,
$(A_nh)(x,t)\leq (Ah)(x,t)$, $x\in\X$, $t\in[0,T]$, where a bounded linear operator $A$ is given on $\Buc$ by
\[
Ah:=\ka^{+}(a^{+}*h)-\ka^{-}u(a^{-}*h).
\]
Next, the series expansions for $e^{tA_n}$ and $e^{tA}$ converge in the topology of norms of operator on the space $\Buc$. Then, for any $n\in\N$, and for $x\in\X$, $t\in[0,T]$,
\begin{equation}
h_n(x,t)=(e^{tA_{n}}h_{n,0})(x)\leq(e^{TA}h_{n,0})(x)=\sum_{m=0}^{\infty}\dfrac{T^{m}}{m!}A^m h_{n,0}, \label{eq:loc_stab_BUC:suff_est}
\end{equation}
and, moreover, for any $\eps>0$ one can find $M=M(\eps)\in\N$, such that we get from \eqref{eq:loc_stab_BUC:suff_est} that
\begin{equation}
h_n(x,t)\leq\sum\limits _{m=0}^{M}\dfrac{T^{m}}{m!}A^m h_{n,0}(x)+\eps \theta, \quad x\in\X, t\in[0,T]. \label{eq:loc_stab_BUC:final_est}
\end{equation}
as $h_{n,0}\in\Utheta$, $n\in\N$.
Finally, the assumptions of the statement yield that $h_{n,0}\locun 0$. Then, by \eqref{eq:loc_stab_BUC:suff_est} and Lemma~\ref{convluc}, $h_n(x,t)\locun 0$ uniformly in $t\in[0,T]$. Hence, $\overline{u}_{n}\locun u$ uniformly on $[0,T]$. The convergence $\underline{u}_{n}\locun u$ may be proved by an analogy.
\end{proof}

\begin{remark}\label{rem:cfwith}
An analogous statement holds in the space $\tilde{\mathcal{X}}_T$, $T>0$.
\end{remark}

In the case of measurable bounded functions, cf. Remark~\ref{rem:cfwith}, we will need also a weaker form of the local stability above.
\begin{proposition}\label{prop:stab_Linf_ae}
Let \eqref{as:chiplus_gr_m}, \eqref{as:aplus_gr_aminus}
hold. Let $T>0$ be fixed. Consider a sequence of functions $u_{n}\in\tilde{\mathcal{X}}_T$ which are solutions
to \ref{eq:basic} with uniformly bounded initial conditions: $u_{n}(\cdot,0)\in \Ltheta$, $n\in\N$. Let $u\in\tilde{\mathcal{X}}_T$ be a solution to \ref{eq:basic} with initial condition $u(\cdot,0)$ such that
$u_{n}(x,0)\to u(x,0)$, for a.a. $x\in\X$. Then
$u_{n}(x,t)\to u(x,t)$, for a.a. $x\in\X$, uniformly in $t\in[0,T]$.
\end{proposition}
\begin{proof}
The proof will be fully analogous to that for Theorem \ref{thm:loc_stab_BUC} until the inequality \eqref{eq:loc_stab_BUC:final_est}, in which $M=M(\eps,x)$ now. The rest of the proof is the same, taking into account that an analogue of Lemma~\ref{convluc} with both convergences almost everywhere holds true by the dominated convergence theorem.
\end{proof}

In the sequel, it will be useful to consider the solution to \eqref{eq:basic} as a nonlinear transformation of the initial condition.

\begin{definition}\label{def:def:Q_t}
For a fixed $T>0$, define the mapping $Q_{T}$ on $L^{\infty}_{+}(\X):=\{f\in L^\infty(\X)\mid f\geq0 \ \mathrm{a.e.}\}$, as follows
\begin{equation}
(Q_{T}f)(x):=u(x,T),\quad x\in\X,\label{def:Q_T}
\end{equation}
where $u(x,t)$ is the solution to \eqref{eq:basic} with the initial condition $u(x,0)=f(x)$.
\end{definition}
Let us collect several properties of $Q_T$ needed below.
\begin{proposition}
\label{prop:Q_def}
Let \eqref{as:chiplus_gr_m}, \eqref{as:aplus_gr_aminus} hold. The mapping $Q=Q_{T}:L^{\infty}_{+}(\X)\to L^{\infty}_{+}(\X)$ satisfies the following properties
\begin{enumerate}[label=\textnormal{(Q\arabic*)}]
  \item $Q:\Ltheta\to \Ltheta$, $Q:\Utheta\to \Utheta$, \label{eq:QBtheta_subset_Btheta}
  \item let $T_y:L^{\infty}_{+}(\X)\to L^{\infty}_{+}(\X)$, $y\in\X$, be a translation operator, given by
      \begin{equation}\label{shiftoper}
      (T_y f)(x)=f(x-y), \quad x\in\X;
\end{equation}
then \label{prop:QTy=TyQ}
      \begin{equation}
        (QT_{y}f)(x)=(T_{y}Qf)(x), \quad x,y\in\X,\label{eq:QTy=TyQ}
      \end{equation}
  \item $Q0=0$, $Q\theta=\theta$, and $Q r>r$, for any constant $r\in(0,\theta)$, \label{prop:Ql_gr_l}
  \item if $f(x)\leq g(x)$, for a.a. $x\in\X$, then $(Qf)(x)\leq (Qg)(x)$, for a.a. $x\in\X$;\label{prop:Q_preserves_order}
  \item if $f_{n}\locun f$, then $(Qf_{n})(x)\to (Qf)(x)$, for a.a. $x\in\X$.\label{prop:Q_cont}
\end{enumerate}
\end{proposition}
\begin{proof}
The property \ref{eq:QBtheta_subset_Btheta} follows from Remark~\ref{rem:comp_pr_Linf} and Proposition~\ref{prop:comp_pr_BUC}.
To prove \ref{prop:QTy=TyQ} we note that, by \eqref{def:conv}, $T_y(a^\pm*u)=a^\pm*(T_y u)$, and then, by \eqref{eq:exist_uniq_BUC:B}, $B(T_y v)=T_y(Bv)$, therefore, by \eqref{eq:exist_uniq_BUC:Phi_v}, if $\tau=0$ and $u_\tau=T_y f$, then $\Phi_\tau T_y =T_y \Phi $, where $\Phi$ is given by \eqref{eq:exist_uniq_BUC:Phi_v} with $f$ in place of $u_\tau$ only. As a result, $\Phi_\tau^n T_y=T_y \Phi^n$ hence
\[
Q_\Tau(T_y f)=\lim_{n\to\infty}\Phi_\tau^n T_y f
=\lim_{n\to\infty}T_y \Phi^n f=T_y(Q_\Tau f);
\]
and one can continue the same considerations on the next time-interval.
The property~\ref{prop:Ql_gr_l} is a straightforward consequence of Corollary~\ref{cor:startwithconst}; indeed, \eqref{homogensol} implies, for $\alpha_T:=\exp(-\theta\ka^-T)\in(0,1)$,
\[
  Q_T r-r=\frac{\theta r}{r (1-\alpha_T)+\theta\alpha_T}-r
  =\frac{r(\theta-r)(1-\alpha_T)}{r (1-\alpha_T)+\theta\alpha_T}>0.
\]
The property \ref{prop:Q_preserves_order} holds also by
Remark~\ref{rem:comp_pr_Linf} and Proposition~\ref{prop:comp_pr_BUC}
The property \ref{prop:Q_cont} is a weaker version of Remark~\ref{rem:cfwith} and Proposition~\ref{prop:stab_Linf_ae}.
\end{proof}

Let $\S$ denotes a unit sphere in $\X$ centered at the origin:
\begin{equation}\label{unitsphere}
\S =\bigl\{x\in\X\bigm| |x|=1\bigr\};
\end{equation}
in particular, $S^{0}=\{-1,1\}$.

\begin{definition}\label{def:monotoneindirection}
A function $f\in L^\infty(\X)$ is said to be increasing (decreasing, constant) along the vector $\xi\in\S $ if, for a.a. $x\in\X $, the function $f(x+s\xi)=(T_{-s\xi}f)(x)$ is increasing (decreasing, constant) in $s\in\R$, respectively.
\end{definition}

\begin{proposition}\label{prop:monot_along_vector_sol}
Let \eqref{as:chiplus_gr_m}, \eqref{as:aplus_gr_aminus} hold. Let $u_0\in \Ltheta$ be the initial condition for the equation \eqref{eq:basic} which is
increasing (decreasing, constant) along a vector $\xi\in\S $; and $u(\cdot,t)\in\Ltheta$, $t\geq0$, be the corresponding solution (cf. Proposition~\ref{prop:comp_pr_BUC} and Remark~\ref{rem:comp_pr_Linf}). Then, for any $t>0$, $u(\cdot,t)$ is increasing (decreasing, constant, respectively) along the $\xi$.
\end{proposition}
\begin{proof}
Let $u_0$ be decreasing along a $\xi\in\S $. Take any $s_1\leq s_2$ and consider
two initial conditions to \eqref{eq:basic}: $u_0^i(x)=u_0(x+s_i\xi)=(T_{-s_i\xi}u_0)(x)$, $i=1,2$. Since $u_0$ is decreasing, $u_0^1(x)\geq u_0^2(x)$, $x\in\X$. Then, by Proposition~\ref{prop:Q_def},
\[
T_{-s_1\xi}Q_tu_0=Q_tT_{-s_1\xi}u_0=Q_t u_0^1\geq Q_tu_0^2=Q_tT_{-s_2\xi}u_0=
T_{-s_2\xi}Q_tu_0,
\]
that proves the statement. The cases of a decreasing $u_0$ can be considered in the same way. The constant function along a vector is decreasing and decreasing simultaneously.
\end{proof}

For the sequel, we need also to show that any solution to \eqref{eq:basic} is bounded from below by a solution to the corresponding equation with `truncated' kernels~$a^\pm$.
Namely, suppose that the conditions \eqref{as:chiplus_gr_m}, \eqref{as:aplus_gr_aminus} hold. Consider a family of Borel sets $\{\Delta_R\mid R>0\}$, such that $\Delta_R\nearrow\X$, $R\to\infty$. Define, for any $R>0$, the following kernels:
\begin{equation}\label{trkern}
  a_R^{\pm}(x)=\1_{\Delta_R}(x)a^{\pm}(x),\quad x\in\X,
\end{equation}
and the corresponding `truncated' equation, cf. \eqref{eq:basic},
\begin{equation}
\begin{cases}
\begin{aligned}
\dfrac{\partial w}{\partial t}(x,t)&= \ka^{+}(a_R^+*w)(x,t)-mw(x,t)\\
&\quad -\ka^{-}w(x,t)(a_R^-*w)(x,t), \qquad x\in\X,\  t>0,
\end{aligned}\\
w(x,0)=w_{0}(x), \qquad \qquad \qquad \qquad \qquad \quad \, x\in\X.
\end{cases}\label{eq:basic_R}
\end{equation}
We set
\begin{equation}\label{ARdef}
  A_R^\pm:=\int_{\Delta_R}a^\pm(x)\,dx \nearrow 1, \quad R\to\infty,
\end{equation}
by \eqref{normed}. Then the non-zero constant solution to \eqref{eq:basic_R} is equal to
\begin{equation}\label{defofthetaR}
\theta_R=\dfrac{\ka^+A_R^+-m}{\ka^-A_R^-}\to \theta, \quad R\to\infty,
\end{equation}
however, the convergence $\theta_R$ to $\theta$ is, in general, not monotonic. Clearly, by~\eqref{as:chiplus_gr_m}, $\theta_R>0$ if only
\begin{equation}\label{bigR}
  A_R^+>\frac{m}{\ka^+}\in(0,1).
\end{equation}

\begin{proposition}\label{lowestsuppCub}
Let \eqref{as:chiplus_gr_m}, \eqref{as:aplus_gr_aminus} hold, and $R>0$ be such that \eqref{bigR} holds, cf.~\eqref{ARdef}. Let $w_0\in\Buc$ be such that $0\leq w_0(x)\leq \theta_R,\ x\in\X$. Then there exists the unique solution $w\in\Xinf$ to \eqref{eq:basic_R}, such that
\begin{equation}\label{wlessthetaR}
0\leq w(x,t)\leq\theta_R, \quad x\in\X,\ t>0.
\end{equation}
Let $u_0\in\Utheta$ and $u\in\Xinf$ be the corresponding solution to \eqref{eq:basic}. If $w_0(x)\leq u_0(x), x\in \X$, then
\begin{equation}\label{ineqtrunc}
w(x,t)\leq u(x,t),\quad x\in\X, \ t>0.
\end{equation}
\end{proposition}
\begin{proof}
Denote $\Delta_R^c:=\X\setminus \Delta_R$. We have
\begin{align*}
  \theta-\theta_R&=\frac{\ka^+A^-_R-mA^-_R-\ka^+A^+_R+m}{\ka^-A^-_R}
  =\frac{\ka^+(1-A^+_R)-(\ka^+-m)(1-A^-_R)}{\ka^-A^-_R}\\
  &=\frac{1}{\ka^-A^-_R}\int_{\Delta_R^c}\bigl( \ka^+ a^+(x)-(\ka^+-m)a^-(x)\bigr)\,dx\geq0,
\end{align*}
by \eqref{as:aplus_gr_aminus}. Therefore,
\begin{equation}\label{thetaRlesstheta}
  0<\theta_R\leq\theta.
\end{equation}
Clearly, \eqref{as:aplus_gr_aminus} and \eqref{thetaRlesstheta} yield
\begin{equation}\label{as:aplus_geq_aminus_R}
\ka^+a_R^+(x)\geq \theta_R\ka^-a_R^-(x),\quad x\in\X.
\end{equation}
Thus one can apply Proposition~\ref{prop:comp_pr_BUC} to the equation~\eqref{eq:basic_R} using trivial equalities $a^\pm_R(x)=A_R^\pm \tilde{a}^\pm_R(x)$, where the kernels $\tilde{a}^\pm_R(x)=(A_R^\pm)^{-1}a^\pm_R(x)$ are normalized, cf. \eqref{normed}; and the inequality \eqref{as:aplus_geq_aminus_R} is the corresponding analog of \eqref{as:aplus_gr_aminus}, according to \eqref{defofthetaR}. This proves the existence and uniqueness of the solution to \eqref{eq:basic_R} and the bound \eqref{wlessthetaR}.

Next, for $\mathcal{F}$ given by \eqref{Foper}, one gets from \eqref{trkern} and \eqref{eq:basic_R}, that the solution $w$ to \eqref{eq:basic_R} satisfies the following equality
\begin{multline}\label{dopR}
  (\mathcal{F}w)(x,t)=-\ka^+\int_{\Delta_R^c} a^+(y) w(x-y,t)\,dy\\
+\ka^-w(x,t)\int_{\Delta_R^c} a^-(y)w(x-y,t)\,dy.
\end{multline}
By \eqref{wlessthetaR}, \eqref{thetaRlesstheta}, \eqref{as:aplus_gr_aminus}, one gets from \eqref{dopR} that
\begin{align*}
(\mathcal{F}w)(x,t)
&\leq-\ka^+\int_{\Delta_R^c} a^+(y) w(x-y,t)\,dy+\ka^-\theta\int_{\Delta_R^c} a^-(y)w(x-y,t)\,dy \\
&\leq0=(\mathcal{F}u)(x,t),
\end{align*}
 where $u$ is the solution to \eqref{eq:basic}. Therefore, we may apply Theorem~\ref{thm:compar_pr} to get the statement.
\end{proof}
\begin{remark}
The statements of Proposition~\ref{lowestsuppCub} remains true for the functions from $L^\infty(\X)$ (the inequalities will hold a.e.\! only).
\end{remark}

\section{Traveling waves}\label{sec:tr-waves}

Traveling waves were studied intensively for the original Fisher--KPP equation \eqref{kpp}, see e.g. \cite{AW1978,Bra1983,HN2001}; for locally nonlinear equation with nonlocal diffusion \eqref{nl+l}, see e.g. \cite{CDM2008,Yag2009,SLW2011}; and for nonlocal nonlinear equation with local diffusion \eqref{l+nl}, see e.g. \cite{BNPR2009,NPT2011,AC2012,HR2014}.

Through this section we will mainly work in $L^\infty$-setting, see Remarks~\ref{rem:exist_uniq_Linf}, \ref{rem:max_pr_Linf}, \ref{rem:comp_pr_Linf}, \ref{rem:cfwith} above. Recall that we will always assume that \eqref{as:chiplus_gr_m} and \eqref{as:aplus_gr_aminus} hold, and $\theta>0$ is given by \eqref{theta_def}.

Let us give a brief overview for the results of this Section. First, we will show (Proposition~\ref{prop:monot_sol}) that the study of a traveling wave solution to the equation \eqref{eq:basic} in a direction $\xi\in\S$ (cf.~Definition~\ref{def:trw} below) may be reduced to the study of the corresponding one-dimensional equation \eqref{eq:basic_one_dim}, whose kernels are given by \eqref{apm1dim}. The existence and properties of the traveling wave solutions will be considered under the so-called Mollison condition \eqref{aplusexpla}, cf. e.g.\cite{AGT2012,CDM2008,CD2007,BCGR2014,Mol1972,Mol1972a}. Namely, in Theorem~\ref{thm:trwexists} we will prove that, for any $\xi\in\S$, there exists $c_*(\xi)\in\R$, such that, for any $c\geq c_*(\xi)$, there exists a traveling wave with the speed $c$, and, for any $c<c_*(\xi)$,  such a traveling wave does not exist. Moreover, we will find an expression for $c_*(\xi)$, see \eqref{minspeed}. We will that the profile of a traveling wave with a non-zero speed is smooth, whereas the zero-speed traveling wave (provided it exists, i.e. if $c_*(\xi)\leq0$) has a continuous profile (Proposition~\ref{prop:reg_trw}, Corollary~\ref{cor:infsmoothprofile}).
In Theorem~\ref{thm:speedandprofile}, we will show a connection between traveling wave speeds and the corresponding profiles.
Next, using the Ikehara--Delange-type Tauberian theorem (Proposition~\ref{prop:tauber}), we will find the exact asymptotic of a decaying traveling wave profile at $+\infty$ (Proposition~\ref{asymptexact}). This will allow us to prove the uniqueness (up to shifts) of a traveling wave wave profile with a given speed $c\geq c_*(\xi)$ (Theorem~\ref{thm:tr_w_uniq}).

\subsection{Existence and properties of traveling waves}
\begin{definition}
  Let $\M$ denote the set of all decreasing and right-continuous functions $f:\R\to[0,\theta]$.
\end{definition}

\begin{remark}\label{rem:inclus}
  There is a natural embedding of $\M$ into $L^\infty(\R)$. According to this, for a function $f\in L^\infty(\R)$, the inclusion $f\in\M$ means that there exists $g\in\M$, such that $f=g$ a.s. on $\R$.
\end{remark}

\begin{definition}\label{def:trw}
Let  $\tilde{\x}_\infty^1:=\tilde{\x}_\infty\cap C^{1}((0,\infty)\to L^{\infty}(\X))$.
A function $u\in \tilde{\x}_\infty^1$ is said to be a traveling
wave solution to the equation \eqref{eq:basic} with
a speed $c\in\R$ and in a direction $\xi\in\S $ if and only if (iff, in the sequel) there
exists a function $\psi\in\M$,
such that
\begin{equation}\label{trwv}
  \begin{aligned}
  &\psi(-\infty)=\theta, \qquad \psi(+\infty)=0,\\
  &u(x,t)=\psi(x\cdot\xi-ct),\quad t\geq0, \ \mathrm{a.a.}\ x\in\X.
  \end{aligned}
\end{equation}
Here and below $\S $ is defined by \eqref{unitsphere} and $x\cdot y=(x,y)_\X$ is the scalar product in $\X$. The function $\psi$ is said to be the profile for the traveling wave, whereas $c$ is its speed.
\end{definition}

We will use some ideas and results from \cite{Yag2009}.

To study traveling wave solutions to \eqref{eq:basic}, it is natural to consider
the corresponding initial conditions of the form
 \begin{equation}\label{trwvincond}
u_0(x)=\psi(x\cdot\xi),
\end{equation}
for some $\xi\in\S $, $\psi\in\M$. Then the solutions will have a special form as well, namely, the following proposition holds.
\begin{proposition}\label{prop:monot_sol}
Let $\xi\in\S $, $\psi\in\M$, and an initial condition to \eqref{eq:basic} be given by
$u_0(x)=\psi(x\cdot\xi)$, a.a.\;$x\in\X$; let also $u\in\tilde{\x}_\infty^+$ be the corresponding solution. Then there exist a function $\phi:\R\times\R_+\to[0,\theta]$, such that $\phi(\cdot,t)\in\M$, for any $t\geq0$, and
\begin{equation}\label{repres}
  u(x,t)=\phi(x\cdot\xi,t),\quad t\geq0,\ \mathrm{a.a.}\ x\in\X.
\end{equation}

Moreover, there exist functions $\check{a}^\pm$ (depending on $\xi$) on $\R$ with
$0\leq \check{a}^\pm\in L^1(\R)$, $\int_\R \check{a}^\pm(s)\,ds=1$, such that $\phi$ is a solution to the following one-dimensional version of \eqref{eq:basic}:
\begin{equation}
\begin{cases}
\begin{aligned}
\dfrac{\partial \phi}{\partial t}(s,t)&=\ka^{+}(\check{a}^{+}*\phi)(s,t)-m\phi(s,t)\\&\quad
-\ka^{-}\phi(s,t)(\check{a}^{-}*\phi)(s,t), \qquad t>0, \ \mathrm{a.a.}\ s\in\R,
\end{aligned}\\
\phi(s,0)=\psi(s),\qquad \mathrm{a.a.}\ s\in\R.
\end{cases}\label{eq:basic_one_dim}
\end{equation}
\end{proposition}
\begin{proof}
Choose any $\eta\in\S $ which is orthogonal to the $\xi$. Then the initial condition $u_0$ is constant along $\eta$, indeed, for any $s\in\R$,
\[
u_0(x+s\eta)=\psi((x+s\eta)\cdot\xi)=\psi(x\cdot\xi)=u_0(x),\quad \mathrm{a.a.}\ x\in\X.
\]
Then, by Proposition~\ref{prop:monot_along_vector_sol}, for any fixed $t>0$, the solution $u(\cdot,t)$ is constant along $\eta$ as well. Next, for any $\tau\in\R$, there exists $x\in\X$ such that $x\cdot\xi=\tau$; and, clearly, if $y\cdot\xi=\tau$ then $y=x+s\eta$, for some $s\in\R$ and some $\eta$ as above. Therefore, if we just set, for a.a. $x\in\X$, $\phi(\tau,t):=u(x,t)$, $t\geq0$, this definition will be correct a.e.\! in $\tau\in\R$; and it will give \eqref{repres}. Next, for a.a. fixed $x\in\X$, $u_0(x+s\xi)=\psi(x\cdot\xi+s)$ is decreasing in $s$, therefore, $u_0$ is decreasing along the $\xi$, and by Proposition~\ref{prop:monot_along_vector_sol},
$u(\cdot,t)$, $t\geq0$, will be decreasing along the $\xi$ as well. The latter means that, for any $s_1\leq s_2$, we have, by \eqref{repres},
\[
\phi(x\cdot\xi+s_1,t)=u(x+s_1\xi,t)\geq u(x+s_2\xi,t)=\phi(x\cdot\xi+s_2,t),
\]
and one can choose in the previous any $x$ which is orthogonal to $\xi$ to prove that $\phi$ is decreasing in the first coordinate.

To prove the second statement, for $d\geq2$, choose any $\{\eta_{1},\ \eta_{2},\ ...,\ \eta_{d-1}\}\subset\S $ which form a complement of $\xi\in\S $ to an orthonormal basis in $\X $.
Then, for a.a.\,\,$x\in\X$, with $x=\sum_{j=1}^{d-1}\tau_j\eta_j+s\xi$, $\tau_1,\ldots,\tau_{d-1},s\in\R$, we have (using an analogous expansion of $y$ inside the integral below an taking into account that any linear transformation of orthonormal bases preserves volumes)
\begin{align}
&\quad (a^\pm*u)(x,t)=\int_\X a^\pm(y)u(x-y,t)dy\notag\\
&=\int_\X a^\pm\biggl(\sum_{j=1}^{d-1}\tau_j'\eta_j+s'\xi\biggr)\,
    u\biggl(\sum_{j=1}^{d-1}(\tau_j-\tau_j')\eta_j+(s-s')\xi,t\biggr)\,d\tau_{1}'\ldots d\tau_{d-1}'ds'\notag\\
&=\int_\R\Biggl(\int_{\R^{d-1}}a^\pm\biggl(\sum_{j=1}^{d-1}\tau_j'\eta_j+s'\xi\biggr)\,
d\tau_1'\ldots d\tau_{d-1}'\Biggr)u\bigl((s-s')\xi,t\bigr)\,ds',\label{reducingto1dim}
\end{align}
where we used again Proposition~\ref{prop:monot_along_vector_sol} to show that $u$ is constant along the vector $\eta=\sum_{j=1}^{d-1}(\tau_j-\tau_j')\eta_j$ which is orthogonal to the $\xi$.

Therefore, one can set
\begin{equation}\label{apm1dim}
\check{a}^\pm(s):=\begin{cases}
\displaystyle \int_{\R^{d-1}} a^\pm(\tau_1\eta_1+\ldots+\tau_{d-1}\eta_{d-1}+s\xi)\,d\tau_1\ldots d\tau_{d-1}, &d\geq2,\\[3mm]
a^\pm(s\xi), &d=1.
\end{cases}
\end{equation}
It is easily seen that $\check{a}^\pm=\check{a}^\pm_\xi$ does not depend on the choice of $\eta_1,\ldots,\eta_{d-1}$, which constitute a basis in the space $H_\xi:=\{x\in\X\mid x\cdot\xi=0\}=\{\xi\}^\bot$.
Note that, clearly,
\begin{equation}\label{cleareq}
\int_\R \check{a}^\pm(s)\,ds=\int_\X a^\pm(y)\,dy=1.
\end{equation}
Next, by \eqref{repres}, $u\bigl((s-s')\xi,t\bigr)=\phi(s-s',t)$, therefore, \eqref{reducingto1dim} may be rewritten as
\[
(a^\pm*u)(x,t)=\int_\R \check{a}^\pm(s')\phi(s-s',t\bigr)\,ds'=:(\check{a}^\pm*\phi)(s,t),
\]
where $s=x\cdot\xi$. The rest of the proof is obvious now.
\end{proof}
\begin{remark}\label{rem:multi-one}
%By Proposition~\ref{prop:monot_sol}, the solution $u$ has always the form \eqref{repres} provided that \eqref{trwvincond} holds, for some $\psi\in\M$ and some $\xi\in\S $. As a matter of fact, to prove the existence of traveling waves we will need to show that there exist functions $\psi\in\M$ such that
%\begin{equation}\label{tw1d}
%\phi(s,t)=\psi(s-ct).
%\end{equation}
%According to the second statement of Proposition~\ref{prop:monot_sol}, it means that the one-dimensional version \eqref{eq:basic_one_dim} of the initial equation \eqref{eq:basic} should have a traveling wave solution with the profile $\psi$ and the speed $c$ given by \eqref{tw1d}. If such solution is found and if the initial condition to \eqref{eq:basic} is given by~\eqref{trwvincond}, then \eqref{trwv} holds.
Let $\xi\in\S$ be fixed and $\check{a}^\pm$ be defined by \eqref{apm1dim}. Let $\phi$ be a traveling wave solution to the equation \eqref{eq:basic_one_dim} (in the sense of Definition~\ref{def:trw}, for $d=1$) in the direction $1\in S^0=\{-1,1\}$, with a profile $\psi\in\M$ and a speed $c\in\R$. Then the function $u$ given by
\begin{equation}\label{tw1d}
u(x,t)=\psi(x\cdot\xi-ct)=\psi(s-ct)=\phi(s,t),
\end{equation}
for $x\in\X$, $t\geq0$, $s=x\cdot\xi\in\R$,
is a traveling wave solution to \eqref{eq:basic} in the direction $\xi$, with the profile $\psi$ and the speed $c$.
\end{remark}

\begin{remark}\label{incrinsteadofdecr}
  One can realize all previous considerations for increasing traveling wave, increasing solution along a vector $\xi$ etc. Indeed, it is easily seen that the function $\tilde{u}(x,t)=u(-x,t)$ with the initial condition $\tilde u_0(x)=u_0(-x)$ is a solution to the equation \eqref{eq:basic} with $a^\pm$ replaced by $\tilde{a}^\pm(x)=a^\pm(-x)$; note that $(a^\pm*u)(-x,t)=(\tilde{a}^\pm*\tilde{u})(x,t)$.
\end{remark}

\begin{remark}\label{shiftoftrw}
  It is a straightforward application of \eqref{eq:QTy=TyQ}, that if $\psi\in\M$, $c\in\R$ gets \eqref{trwv} then, for any $s\in\R$, $\psi(\cdot+s)$ is a traveling wave to \eqref{eq:basic} with the same $c$.
\end{remark}

We will need also the following simple statement.
\begin{proposition}\label{prop:Qtilde}
Let \eqref{as:chiplus_gr_m}, \eqref{as:aplus_gr_aminus} hold and $\xi\in\S $ be fixed.
  Define, for an arbitrary $T>0$, the mapping $\tilde{Q}_{T}:L^{\infty}(\R)\to L^{\infty}(\R)$ as follows: $\tilde{Q}_T\psi(s)=\phi(s,T)$, $s\in\R$, where $\phi:\R\times\R_+\to[0,\theta]$ solves \eqref{eq:basic_one_dim} with $0\leq\psi\in L^{\infty}_{+}(\R)$. Then such a $\tilde{Q}_T$ is well-defined, satisfies all properties of Proposition~\ref{prop:Q_def} (with $d=1$), and, moreover, $\tilde{Q}_T(\M)\subset\M$.
\end{proposition}
\begin{proof}
  Consider one-dimensional equation \eqref{eq:basic_one_dim}, where $\check{a}^\pm$ are given by \eqref{apm1dim}. The latter equality together with \eqref{as:aplus_gr_aminus} imply that \begin{equation}\label{acheckpos}
\ka^+\check{a}^+(s)\geq (\ka^+-m)\check{a}^-(s), \quad \text{a.a.} \ s\in\R.
\end{equation}
Therefore, all previous results (e.g. Theorem~\ref{thm:exist_uniq_BUC}) hold true for the solution to \eqref{eq:basic_one_dim} as well. In particular, all statements of Proposition~\ref{prop:Q_def} hold true, for $Q=\tilde{Q}_T$, $d=1$. Moreover, by the proof of Theorem~\ref{thm:exist_uniq_BUC} (in the $L^\infty$-case, cf.~Remark~\ref{rem:exist_uniq_Linf}), since
the mappings $B$ and $\Phi_\tau$, cf. \eqref{eq:exist_uniq_BUC:B}, \eqref{eq:exist_uniq_BUC:Phi_v}, map the set $\M$ into itself, we have that $\tilde{Q}_T$ has this property as well, cf.~Remark~\ref{rem:inclus}.
\end{proof}

Now we are going to prove the existence of the traveling wave solution to \eqref{eq:basic}. Denote, for any $\la>0$, $\xi\in\S $,
\begin{equation}\label{aplusexpla}
  \A_\xi(\la):=\int_\X a^+(x) e^{\la x\cdot \xi}\,dx\in[0,\infty].
\end{equation}

For a given $\xi\in\S $, consider the following assumption on $a^+$:
\begin{assum}\label{aplusexpint1}
  \text{there exists} \ \mu=\mu(\xi)>0 \ \text{such that} \ \A_{\xi}(\mu)<\infty.
\end{assum}

\begin{theorem}\label{thm:trwexists}
Let \eqref{as:chiplus_gr_m} and \eqref{as:aplus_gr_aminus} hold and $\xi\in\S $ be fixed. Suppose also that \eqref{aplusexpint1} holds.
Then there exists $c_*(\xi)\in\R$ such that
\begin{enumerate}[label={\arabic*})]
    \item for any $c\geq c_*(\xi)$, there exists a traveling wave solution, in the sense of Definition~\ref{def:trw}, with a profile $\psi\in\M$ and the speed $c$,
    \item for any $c<c_*(\xi)$, such a traveling wave does not exist.
\end{enumerate}
\end{theorem}
\begin{proof}
Let $\mu>0$ be such that \eqref{aplusexpint1} holds. Then, by \eqref{apm1dim},
\begin{align}\notag
\int_\R \check{a}^+(s) e^{\mu s}ds&=\int_\R \int_{\R^{d-1}} a^\pm(\tau_1\eta_1+\ldots+\tau_{d-1}\eta_{d-1}+s\xi)e^{\mu s}\,d\tau_1\ldots d\tau_{d-1} ds\\&=\A_\xi(\mu)<\infty.\label{expintla1}
\end{align}
Clearly, the integral equality in \eqref{expintla1} holds true for any $\la\in\R$ as well, with $\A_\xi(\la)\in[0,\infty]$.

Let $\mu>0$ be such that \eqref{aplusexpint1} holds. Define a function from $\M$ by
\begin{equation}\label{defvarphi}
\varphi(s):=\theta\min\{e^{-\mu s},1\}.
\end{equation}
Let us prove that there exists $c\in\R$ such that $\bar{\phi}(s,t):=\varphi(s-ct)$ is a super-solution to \eqref{eq:basic_one_dim}, i.e.
\begin{equation}\label{supersol}
\mathcal{F}\bar{\phi}(s,t)\geq0,\quad s\in\R, t\geq0,
\end{equation}
where $\mathcal{F}$ is given by \eqref{Foper} (in the case $d=1$).
We have
\begin{align*}
  (\mathcal{F}\bar{\phi})(s,t) & =-c\varphi'(s-ct)- \ka^+(\check{a}^+*\varphi)(s-ct)+m\varphi(s-ct)\\&\quad +\ka^-\varphi(s-ct) (\check{a}^-*\varphi)(s-ct),
\end{align*}
hence, to prove \eqref{supersol}, it is enough to show that, for all $s\in\R$,
\begin{equation}\label{suffcond}
  \mathcal{J}_c(s):=c\varphi'(s)+\ka^+(\check{a}^+*\varphi)(s)-m\varphi(s)-\ka^-\varphi(s)(\check{a}^-*\varphi)(s)\leq 0.
\end{equation}

By \eqref{defvarphi}, \eqref{acheckpos}, for $s<0$, we have
\[
\mathcal{J}_c(s)=\ka^+(\check{a}^+*\varphi)(s)-m\theta-\ka^-\theta(\check{a}^-*\varphi)(s)
\leq \bigl((\ka^+\check{a}-\ka^-\theta\check{a}^-)*\theta\bigr)(s)-m\theta=0.
\]
Next, by \eqref{defvarphi},
\[
(\check{a}^+*\varphi)(s)\leq\theta \int_\R \check{a}^+(\tau)e^{-\mu(s-\tau)}\,d\tau=\theta e^{-\mu s} \A_\xi(\mu),
\]
therefore, for $s\geq0$, we have
\[
    \mathcal{J}_c(s)\leq -\mu c\theta e^{-\mu s}
    +\ka^+\theta  e^{-\mu s} \A_\xi(\mu) -m\theta e^{-\mu s};
\]
and to get \eqref{suffcond} it is enough to demand that $\ka^+ \A_\xi(\mu)-m- \mu c\leq0$, in particular,
\begin{equation}\label{demand}
c=\frac{\ka^+ \A_\xi(\mu)-m}{\mu}.
\end{equation}
As a result, for $\bar\phi(s,t)=\varphi(s-ct)$ with $c$ given by \eqref{demand}, we have
\begin{equation}\label{supersol2}
\mathcal{F}\bar\phi\geq0=\mathcal{F}(\tilde{Q}_t\varphi),
\end{equation}
as $\tilde{Q}_t\varphi$ is a solution to \eqref{eq:basic_one_dim}. Then, by \eqref{as:aplus_gr_aminus} and the inequality $\bar\phi\leq\theta$, one can apply Proposition~\ref{compprabscont} and get that
\[
\tilde{Q}_t\varphi(s')\leq \bar\phi(t,s')=\varphi(s'-ct), \quad \text{a.a.}\ s'\in\R,
\]
where $c$ is given by \eqref{demand}; note that, by \eqref{defvarphi}, for any $s\in\R$, the function $\bar{\phi}(s,t)$ is absolutely continuous in $t$. In particular, for $t=1$, $s'=s+c$, we will have
\begin{equation}\label{mayYag}
\tilde{Q}_1\varphi(s+c)\leq \varphi(s), \quad \text{a.a.}\ s\in\R.
\end{equation}
And now one can apply \cite[Theorem 5]{Yag2009} which states that, if there exists a flow of abstract mappings $\tilde{Q}_t$, each of them maps $\M$ into itself and has properties \ref{eq:QBtheta_subset_Btheta}--\ref{prop:Q_cont} of Proposition~\ref{prop:Q_def}, and if, for some $t$ (e.g. $t=1$), for some $c\in\R$, and for some $\varphi\in\M$, the inequality \eqref{mayYag} holds, then there exists $\psi\in\M$ such that, for any $t\geq0$,
\begin{equation}\label{getbyYag}
(\tilde{Q}_t \psi)(s+ct)=\psi(s), \quad \text{a.a.}\ s\in\R,
\end{equation}
that yields the solution to \eqref{eq:basic_one_dim} in the form \eqref{tw1d}, and hence, by Remark~\ref{rem:multi-one}, we will get the existence of a solution to \eqref{eq:basic} in the form \eqref{trwv}. It is worth noting that, in \cite{Yag2009}, the results were obtained for increasing functions. By~Remark~\ref{incrinsteadofdecr}, the same results do hold for decreasing functions needed for our settings.

Next, by \cite[Theorem 6]{Yag2009}, there exists $c_*=c_*(\xi)\in(-\infty,\infty]$ such that, for any $c\geq c_*$, there exists $\psi=\psi_c\in\M$ such that \eqref{getbyYag} holds, and for any $c<c_*$ such a $\psi$ does not exist. Since for $c$ given by \eqref{demand} such a $\psi$ exists, we have that $c_*\leq c<\infty$, moreover, one can take any $\mu$ in \eqref{demand} for that \eqref{aplusexpint1} holds. Therefore,
\begin{equation}\label{cstarestimate}
c_*\leq \inf_{\la>0}\frac{\ka^+ \A_\xi(\la)-m}{\la}.
\end{equation}
The statement is proved.
\end{proof}

\begin{remark}
  It can be seen from the proof above that we didn't use the special form \eqref{defvarphi} of the function $\varphi$ after the inequality \eqref{supersol2}. Therefore, if a function $\varphi_1\in\M$ is such that the function $\bar\phi(s,t):=\varphi_1(s-ct)$, $s\in\R$, $t\geq0$, is a super-solution to \eqref{eq:basic_one_dim}, for some $c\in\R$, i.e. if \eqref{supersol} holds, then there exists a traveling wave solution to \eqref{eq:basic_one_dim}, and hence to \eqref{eq:basic}, with some profile $\psi\in\M$ and the same speed $c$.
\end{remark}

Next two statements describe the properties of a traveling wave solution.

\begin{proposition}\label{prop:reg_trw}
Let $\psi\in\M$ and $c\in\R$ be such that there exists a solution $u\in\tilde{\x}_\infty^1$ to the equation \eqref{eq:basic} such that \eqref{trwv} holds, for some $\xi\in\S $. Then $\psi\in C^{1}(\R\to[0,\theta])$, for $c\neq0$, and $\psi\in C(\R\to[0,\theta])$, otherwise.
\end{proposition}
\begin{proof}
The condition \eqref{trwv} implies \eqref{trwvincond} for the $\xi\in\S $. Then, by Proposition~\ref{prop:monot_sol}, there exists $\phi$ given by \eqref{repres} which solves \eqref{eq:basic_one_dim}; moreover, by Remark~\ref{rem:multi-one}, \eqref{tw1d} holds.

Let $c\neq0$. It is well-known that any monotone function is differentiable almost everywhere. Prove first that $\psi$ is differentiable everywhere on $\R$. Fix any $s_{0}\in\R$.
It follows directly from Proposition~\ref{prop:monot_sol}, that $\phi\in C^1((0,\infty)\to L^\infty(\R))$. Therefore, for any $t_0>0$ and for any $\varepsilon>0$, there exists $\delta=\delta(t_0,\varepsilon)>0$
such that, for all $t\in\R$ with $|ct|<\delta$ and $t_0+t>0$, the following inequalities hold, for a.a.~$s\in\R$,
\begin{gather}
\dfrac{\partial \phi}{\partial t}(s,t_{0}) -\eps< \dfrac{\phi(s,t_{0}+t)-\phi(s,t_{0})}{t}<\dfrac{\partial \phi}{\partial t}(s,t_{0})+\varepsilon,\label{firstfromeq}\\
\dfrac{\partial \phi}{\partial t}(s,t_{0})-\eps<\dfrac{\partial \phi}{\partial t}(s,t_{0}+t)<\dfrac{\partial \phi}{\partial t}(s,t_{0})+\varepsilon.\label{secondfromeq}
\end{gather}

Set, for the simplicity of notations, $x_0=s_0+ct_0$. Take any $0<h<1$ with $2h<\min\bigl\{\delta,|c|t_0, |c|\delta \bigr\}$.
Since $\psi$ is a decreasing function, one has, for almost all $s\in(x_0,x_0+h^{2})$,
\begin{align}
&\quad \dfrac{\psi(s_0+h)-\psi(s_0)}{h}\leq\dfrac{\psi(s-ct_0+h-h^{2})-\psi(s-ct_0)}{h} \nonumber \\
&=\dfrac{\phi(s,t_0+\frac{h^{2}-h}{c})-\phi(s,t_0)}{\frac{h^{2}-h}{c}}\dfrac{h^{2}-h}{ch}\leq\left(\dfrac{\partial \phi}{\partial t}(s,t_0)\mp\varepsilon\right)\dfrac{h-1}{c},\label{ff1}
\end{align}
by \eqref{firstfromeq} with $t=\frac{h^{2}-h}{c}$; note that then $|ct|=h-h^2<h<\delta$, and $t_0+t>0$ (the latter holds, for $c<0$, because of $t_0+t>t_0$ then; and, for $c>0$, it is equivalent to $ct_0>-ct=h-h^2$, that follows from $h<ct_0$).
Stress, that, in \eqref{ff1}, one needs to choose $-\eps$, for $c>0$, and $+\eps$, for $c<0$, according to the left and right inequalities in \eqref{firstfromeq}, correspondingly.

Similarly, for almost all $s\in(x_0-h^{2},x_0)$, one has
\begin{align}
&\quad \dfrac{\psi(s_0+h)-\psi(s_0)}{h}\geq\dfrac{\psi(s-ct_0+h+h^{2})-\psi(s-ct_0)}{h} \nonumber\\
&=\dfrac{\phi(s,t_0-\frac{h^{2}+h}{c})-\phi(s,t_0)}{-\frac{h^{2}+h}{c}}\dfrac{h^{2}+h}{-ch}\geq\left(\dfrac{\partial \phi}{\partial t}(s,t_0)\pm\varepsilon\right)\dfrac{h+1}{-c},\label{ff2}
\end{align}
where we take again the upper sign, for $c>0$, and the lower sign, for $c<0$; note also that $h+h^2<2h<\delta$.
Next, one needs to `shift' values of $s$ in \eqref{ff2} to get them the same as in \eqref{ff1}. To do this note that, by \eqref{tw1d},
\begin{equation}\label{ff3}
\phi\Bigl(s+h^2,t_0+\frac{h^2}{c}\Bigr)=\phi(s,t_0), \quad \text{a.a.}\ s\in\X.
\end{equation}
As a result,
\begin{equation}\label{ff4}
\begin{split}(\check{a}^\pm *\phi)\Bigl(s+h^2,t_0+\frac{h^2}{c}\Bigr)&=\int_{\R} \check{a}^\pm (s') \phi\Bigl(s-s'+h^2,t_0+\frac{h^2}{c}\Bigr)\,ds\\
&=(\check{a}^\pm *\phi)(s,t_0), \quad \text{a.a.}\ s\in\X.
\end{split}
\end{equation}
Then, by \eqref{eq:basic_one_dim}, \eqref{ff3}, \eqref{ff4}, one gets
\begin{equation}\label{ff5}
\frac{\partial}{\partial t}\phi\Bigl(s+h^2,t_0+\frac{h^2}{c}\Bigr)=\frac{\partial}{\partial t}\phi(s,t_0), \quad \text{a.a.}\ s\in\X.
\end{equation}
Therefore, by \eqref{ff5}, one gets from \eqref{ff2} that, for almost all $s\in(x_0,x_0+h^{2})$, cf. \eqref{ff1},
\begin{align}
 \dfrac{\psi(s_0+h)-\psi(s_0)}{h}&\geq\left(\dfrac{\partial \phi}{\partial t}\Bigl(s,t_0+\frac{h^2}{c}\Bigr)\pm\varepsilon\right)\dfrac{h+1}{-c},\notag\\
 \intertext{and, since $\bigl\lvert \frac{h^2}{c}\bigr\rvert<\delta$, one can apply the right and left inequalities in \eqref{secondfromeq}, for $c>0$ and $c<0$, correspondingly, to continue the estimate}
& \geq\left(\dfrac{\partial \phi}{\partial t}(s,t_0)\pm 2\varepsilon\right)\dfrac{h+1}{-c}.\label{ff6}
\end{align}
Combining \eqref{ff1} and \eqref{ff6}, we obtain
\begin{multline}
\left(\esssup_{s\in(x_0,x_0+h^2)}\dfrac{\partial \phi}{\partial t}(s,t_0)\pm 2\varepsilon\right)\dfrac{h+1}{-c} \leq \dfrac{\psi(s_0+h)-\psi(s_0)}{h}\\ \leq\left(\esssup_{s\in(x_0,x_0+h^2)}\dfrac{\partial \phi}{\partial t}(s,t_0)\mp\varepsilon\right)\dfrac{h-1}{c}. \label{ff7}
\end{multline}
For fixed $s_0\in\R$, $t_0>0$ and for $x_0=s_0+ct_0$, the function
\[
f(h):=\esssup\limits_{s\in(x_0,x_0+h^2)}\frac{\partial \phi}{\partial t}(s,t_0), \quad h\in(0,1)
\]
is bounded, as $|f(h)|\leq \bigl\lVert \frac{\partial \phi}{\partial t}(\cdot,t_0)\bigr\rVert_\infty<\infty$, and monotone; hence there exists $\bar f=\lim\limits_{h\to0+}f(h)$. As a result, for small enough $h$, \eqref{ff7} yields
\begin{equation*}
(\bar f\pm 2\varepsilon)\dfrac{1}{-c} -\eps \leq \dfrac{\psi(s_0+h)-\psi(s_0)}{h} \leq(\bar f\mp\varepsilon)\dfrac{-1}{c}+\eps,
\end{equation*}
and, therefore, there exists
$\dfrac{\partial\psi}{\partial s}(s_0+)=\dfrac{-\bar f}{c}$. In the same way, one can prove that there exists $\dfrac{\partial\psi}{\partial s}(s_0-)=\dfrac{-\bar f}{c}$, and, therefore, $\psi$ is differentiable at $s_0$.
As a result, $\psi$ is differentiable (and hence continuous) on the whole $\R$.

Next, for any $s_1,s_2,h\in\R$, we have
\begin{multline*}
  \biggl\lvert \frac{\psi(s_1+h)-  \psi(s_1)}{h}-\frac{\psi(s_2+h)-  \psi(s_2)}{h}\biggr\rvert\\
  =\frac{1}{|c|}\biggl\lvert \frac{\phi\bigl(s_1+ct_0,t_0-\frac{h}{c}\bigr)-  \phi(s_1+ct_0,t_0)}{-\frac{h}{c}}\qquad\qquad\qquad\\
  -\frac{\phi\bigl(s_1+ct_0,t_0+\frac{s_1-s_2}{c}-\frac{h}{c}\bigr)-  \phi\bigl(s_1+ct_0,t_0+\frac{s_1-s_2}{c}\bigr)}{-\frac{h}{c}}\biggr\rvert;
\end{multline*}
and if we pass $h$ to $0$, we get
\begin{align}\notag
  \lvert \psi'(s_1)-\psi'(s_2)\rvert&=\frac{1}{|c|}\biggl\lvert \frac{\partial}{\partial t}\phi(s_1+ct_0,t_0)
  -\frac{\partial}{\partial t}\phi\Bigl(s_1+ct_0,t_0+\frac{s_1-s_2}{c}\Bigr)\biggr\rvert
  \\& \leq \frac{1}{|c|}\biggl\lVert \frac{\partial}{\partial t}\phi(\cdot,t_0)
  -\frac{\partial}{\partial t}\phi\Bigl(\cdot,t_0+\frac{s_1-s_2}{c}\Bigr)\biggr\rVert.\label{eq333}
\end{align}
And now, by the continuity of $\frac{\partial}{\partial t}\phi(\cdot,t)$ in $t$ in the sense of the norm in $L^\infty(\R)$, we have that, by \eqref{secondfromeq}, the inequality $|s_1-s_2|\leq |c|\delta$ implies that, by \eqref{eq333},
\[
\lvert \psi'(s_1)-\psi'(s_2)\rvert\leq \frac{1}{|c|} \eps.
\]
As a result, $\psi'(s)$ is uniformly continuous on $\R$ and hence continuous.

Finally, consider the case $c=0$. Then \eqref{tw1d} implies that $\phi(s,t)$ must be constant in time, i.e. $\phi(s,t)=\psi(s)$, for a.a. $s\in\R$. Thus one can rewrite \eqref{eq:basic_one_dim} as follows
\begin{equation}\label{statwave}
  \ka^{+}(\check{a}^{+}*\psi)(s)-m\psi(s)-\ka^{-}\psi(s) (\check{a}^{-}*\psi)(s)=0,
\end{equation}
or, equivalently,
\begin{equation}\label{asquotient}
 \psi(s)=\frac{\ka^{+}(\check{a}^{+}*\psi)(s)}{m+\ka^{-}(\check{a}^{-}*\psi)(s)}.
\end{equation}
Since $\psi\in L^\infty(\R)$, then, by Lemma~\ref{le:simple}, the r.h.s. of \eqref{asquotient} is a continuous in $s$ function, and hence $\psi\in C(\R)$.
\end{proof}

Let $u\in\tilde{\x}_\infty^1$ be a traveling wave solution to \eqref{eq:basic}, in the sense of Definition~\ref{def:trw}, with a profile $\psi\in\M$ and a speed $c\in\R$. Then, by~Remark~\ref{rem:multi-one} and Proposition~\ref{prop:reg_trw}, for any $c\neq0$, one can differentiate $\psi(s-ct)$ in $t\geq0$. Thus (cf. also Lemma~\ref{le:simple}) we get
\begin{equation}
c\psi'(s)+\ka^{+}(\check{a}^{+}*\psi)(s)-m\psi(s)-\ka^{-}\psi(s) (\check{a}^{-}*\psi)(s)=0, \quad s\in\R.\label{eq:trw}
\end{equation}
For $c=0$, one has \eqref{statwave}, i.e. \eqref{eq:trw} holds in this case as well.

Let $k\in\N\cup\{\infty\}$ and $C_b^k(\R)$ denote the class of all functions on $\R$ which are $k$ times differentiable and whose derivatives (up to the order $k$) are continuous and bounded on $\R$.
\begin{corollary}\label{cor:infsmoothprofile}
In conditions and notations of Proposition~\ref{prop:reg_trw}, for any speed $c\neq0$, the profile $\psi\in C_b^\infty(\R)$.
\end{corollary}
\begin{proof}
By Lemma~\ref{le:simple}, $\check{a}^\pm*\psi\in C_b(\R)$. Then \eqref{eq:trw} yields $\psi'\in C_b(\R)$, i.e. $\psi\in C_b^1(\R)$. By e.g. \cite[Proposition~5.4.1]{Sta2005}, $\check{a}^\pm*\psi\in C_b^1(\R)$ and $(\check{a}^\pm*\psi)'=\check{a}^\pm*\psi'$, therefore, the equality \eqref{eq:trw} holds with $\psi'$ replaced by $\psi''$ and $\psi$ replaced by $\psi'$. Then, by the same arguments $\psi\in C_b^2(\R)$, and so on. The statement is proved.
\end{proof}

\begin{proposition}\label{prop:psidecaysstrictly}
  In conditions and notations of Proposition~\ref{prop:reg_trw}, $\psi$ is a strictly decaying function, for any speed $c$.
\end{proposition}
\begin{proof}
  Let $c\in\R$ be the speed of a traveling wave with a profile $\psi\in\M$ in a direction $\xi\in\S$. By~Proposition~\ref{prop:reg_trw}, $\psi\in C(\R)$. Suppose that $\psi$ is not strictly decaying, then there exists $\delta_0>0$ and $s_0\in\R$, such that $\psi(s)=\psi(s_0)$, for all $|s-s_0|\leq\delta_0$. Take any $\delta\in\bigl(0,\frac{\delta_0}{2}\bigr)$, and consider the function $\psi^\delta(s):=\psi(s+\delta)$. Clearly, $\psi^\delta(s)\leq\psi(s)$, $s\in\R$. By~Remark~\ref{shiftoftrw}, $\psi^\delta$ is a profile for a traveling wave with the same speed $c$. Therefore, one has two solutions to \eqref{eq:basic}: $u(x,t)=\psi(x\cdot\xi-ct)$ and $u^\delta(x,t)=\psi^\delta(x\cdot\xi-ct)$ and hence
  $u(x,t)\leq u^\delta(x,t)$, $x\in\X$, $t\geq0$. By the maximum principle, see Theorem~\ref{thm:strongmaxprinciple}, either $u\equiv u^\delta$, that contradicts $\delta>0$ or $u(x,t)< u^\delta(x,t)$, $x\in\X$, $t>0$. The latter, however, contradicts the equality $u(x,t)=u^\delta(x,t)$, which holds e.g.\! if $x\cdot\xi-ct=s_0$. Hence $\psi$ is a strictly decaying function.
\end{proof}

Under assumptions \eqref{as:chiplus_gr_m} and \eqref{as:aplus_gr_aminus}, define the following function, cf.~\eqref{Jq},
\begin{equation}\label{speckern}
  \check{J}_\upsilon(s):=\ka^+\check{a}^+(s)-\upsilon\ka^-\check{a}^-(s), \quad s\in\R, \upsilon\in(0,\theta].
\end{equation}
Then, by \eqref{acheckpos},
\begin{equation*}
\check{J}_\upsilon(s)\geq \check{J}_\theta(s)\geq0, \quad s\in\R, \upsilon\in(0,\theta].
\end{equation*}

\begin{proposition}\label{prop:trw_exp_est}
Let \eqref{as:chiplus_gr_m} and \eqref{as:aplus_gr_aminus} hold. Then, in the conditions and notations of Proposition \ref{prop:reg_trw}, there exists $\mu=\mu( c, a^+,\ka^-,\theta)>0$ such that
\begin{equation*}
  \int_\R\psi(s)e^{\mu s}\,ds<\infty.
\end{equation*}
\end{proposition}
\begin{proof}
At first, we prove that $\psi\in L^1(\R_+)$. Let $\upsilon\in(0,\theta)$ and $\check{J}_\upsilon(s)>0$, $s\in\R$ be given by \eqref{speckern}.
Since $\int_\R \check{J}_\upsilon(s)\,ds= \ka^+-\upsilon\ka^->m$, one can choose $R_0>0$, such that
\begin{equation}\label{Risproper}
\int_{-R_0}^{R_0} \check{J}_\upsilon(s)\,ds=m.
\end{equation}
We rewrite \eqref{eq:trw} as follows
\begin{equation}\label{eq:tr_w_ii}
c\psi '(s)+(\check{J}_\upsilon*\psi)(s)+\ka^-\bigl(\upsilon-\psi(s)\bigr)(\check{a}^-*\psi)(s)-m\psi(s)=0,\quad s\in\R.
\end{equation}
Fix arbitrary ${r_0}>0$, such that
\begin{equation}\label{rhocond}
  \psi({r_0})<\upsilon.
\end{equation}
Let $r>{r_0}+R_0$. Integrate \eqref{eq:tr_w_ii} over $[{r_0},r]$; one gets
\begin{equation}\label{eq:tr_w_ii_int}
c(\psi(r)-\psi({r_0}))+A+B=0,
\end{equation}
where
\begin{align*}
A&:=\int_{{r_0}}^{r}(\check{J}_\upsilon*\psi)(s)\, ds-m\int_{{r_0}}^{r}\psi(s)ds,\\
B&:=\ka^-\int_{{r_0}}^{r}(\upsilon-\psi(s))(\check{a}^-*\psi)(s)\,ds.
\end{align*}
By \eqref{speckern}, \eqref{Risproper}, one has
\begin{align}
A&\geq\int_{r_0}^{r}\int_{-R_0}^{R_0}\check{J}_\upsilon(\tau)\psi(s-\tau)d\tau ds-m\int_{r_0}^{r}\psi(s)\,ds\nonumber\\&= \int_{-R_0}^{R_0}\check{J}_\upsilon(\tau)\left( \int_{{r_0}-\tau}^{r-\tau}\psi(s)\,ds-\int_{{r_0}}^{r}\psi(s)\,ds \right)\,d\tau\nonumber \\
&=\int_{0}^{R_0}\check{J}_\upsilon(\tau)\left( \int_{{r_0}-\tau}^{{r_0}}\psi(s)\,ds-\int_{r-\tau}^{r}\psi(s)\,ds \right)\,d\tau\nonumber\\&\quad+\int_{-R_0}^{0}\check{J}_\upsilon(\tau)\left( \int_{r}^{r-\tau}\psi(s)\,ds-\int_{{r_0}}^{{r_0}-\tau}\psi(s)\,ds \right)\,d\tau; \label{eq:gen_est12}
\end{align}
and since $\psi$ is a decreasing function and $r-R_0>{r_0}$, we have from \eqref{eq:gen_est12}, that
\begin{align}
A &\geq (\psi({r_0})-\psi(r-R_0))\int_{0}^{R_0}\tau \check{J}_\upsilon(\tau)\,d\tau+(\psi(r+R_0)-\psi({r_0}))\int_{-R_0}^{0}(-\tau) J_\upsilon(\tau)\,d\tau \notag \\ &\geq -\theta \int_{-R_0}^{0}(-\tau) J_\upsilon(\tau)\,d\tau =:-\theta \bar{J}_{\upsilon,R_0}. \label{eq:gen_est}
\end{align}
Next, \eqref{rhocond} and monotonicity of $\psi$ imply
\begin{equation}\label{B-est}
  B\geq \ka^-(\upsilon-\psi({r_0}))\int_{{r_0}}^{r}(\check{a}^-*\psi)(s)\,ds.
\end{equation}
Then, by \eqref{eq:tr_w_ii_int}, \eqref{eq:gen_est}, \eqref{B-est}, \eqref{rhocond}, one gets
\begin{align*}
  0\leq\ka^-(\upsilon-\psi({r_0}))\int_{{r_0}}^{r}(\check{a}^-*\psi)(s)\,ds
  &\leq \theta \bar{J}_{\upsilon,R_0} + c(\psi({r_0})-\psi(r))\\
  &\to \theta \bar{J}_{\upsilon,R_0} + c\psi({r_0})<\infty, \quad r\to\infty,
\end{align*}
therefore, $\check{a}^-*\psi\in L^1(\R_+)$. Finally, \eqref{cleareq} implies that there exist a measurable bounded set $\Delta\subset\R$, with $m(\Delta):=\int_\Delta \,ds\in (0,\infty)$, and a constant $\mu>0$, such that $\check{a}^-(\tau)\geq\mu$, for a.a. $\tau\in\Delta$. Let $\delta=\inf \Delta\in\R$. Then, for any $s\in\R$, one has
\begin{equation*}
(\check{a}^-*\psi)(s)\geq \int_\Delta \check{a}^-(\tau) \psi(s-\tau)\,d\tau\geq \mu \psi(s-\delta) m(\Delta).
\end{equation*}
Therefore $\psi\in L^1(\R_+)$.

For any $N\in\N$, we define $\varphi_N(s):=\1_{(-\infty,N)}(s)+e^{-\la(s-N)}\1_{[N,\infty)}(s)$, where $\la>0$. By the proved above, $\psi,\check{a}^\pm*\psi\in L^1(\R_+)\cap L^\infty(\R)$ hence, by \eqref{eq:trw}, $c\psi'\in L^1(\R_+)\cap L^\infty(\R)$. Therefore, all terms of \eqref{eq:trw} being multiplied on $e^{\lambda s}\varphi_{N}(s)$ are~integrable over $\R$. After this integration, \eqref{eq:trw} will be read as follows
\begin{equation}
I_1+I_2+I_3=0,\label{eq:int_trw_zeta_exp}
\end{equation}
where (recall that $\ka^-\theta-\ka^+=-m$)
\begin{align*}
I_1&:=c\int_{\R}\psi' (s) e^{\lambda s}\varphi_{N}(s)\,ds,\\
I_2&:=\ka^{+}\int_\R\bigl((\check{a}^{+}*\psi)(s)-\psi(s)\bigr)e^{\lambda s}\varphi_{N}(s)\,ds,\\
I_3&:=\ka^{-}\int_{\R}\psi(s)\bigl(\theta-(\check{a}^{-}*\psi)(s)\bigr)
e^{\lambda s}\varphi_{N}(s)\,ds
\end{align*}
We estimate now $I_1,I_2,I_3$ from below.

We start with $I_2$. One can write
\begin{align}
&\quad\int_{\R}(\check{a}^{+}*\psi)(s)e^{\lambda s}\varphi_{N}(s)\,ds
=\int_{\R}\int_{\R}\check{a}^{+}(s-\tau)\psi(\tau)e^{\lambda s}\varphi_{N}(s)\,d\tau ds\notag\\
&=\int_{\R}\int_{\R}\check{a}^{+}(s)e^{\lambda s}\varphi_{N}(\tau+s)\,ds\, e^{\lambda \tau}\psi(\tau)\,d\tau\notag\\
&\ge\int_{\R}\biggl(\int_{-\infty}^{R}\check{a}^{+}(s)e^{\lambda s}\,ds\biggr)\varphi_{N}(\tau+R)e^{\lambda \tau}\psi(\tau)\,d\tau,\label{eq:111}
\end{align}
for any $R>0$, as $\varphi$ is nonincreasing. By \eqref{cleareq}, one can choose $R>0$ such that
\[
\int_{-\infty}^{R}\check{a}^{+}(\tau)\,d\tau>1-\dfrac{\ka^{-}\theta}{4}.
\]
By continuity arguments, there exists $\nu>0$ such that, for any $0<\la<\nu$,
\begin{equation}\label{eq:222}
  \int_{-\infty}^{R}\check{a}^{+}(\tau)e^{\lambda \tau}\,d\tau\geq\Bigl(1-\dfrac{\ka^{-}\theta}{4}\Bigr)e^{\lambda R}.
\end{equation}
Therefore, combining \eqref{eq:111} and \eqref{eq:222}, we get
\begin{align}
I_2&\geq\int _{\R}\Bigl(1-\dfrac{\ka^{-}\theta}{4}\Bigr)e^{\lambda R}\varphi_{N}(\tau+R)e^{\lambda \tau}\psi(\tau)\,d\tau-\int _{\R}\psi(s)e^{\lambda s}\varphi_{N}(s)\,ds\notag \\
&=\int_{\R}\Bigl(1-\dfrac{\ka^{-}\theta}{4}\Bigr)\varphi_{N}(\tau)e^{\lambda \tau}\psi(\tau-R)\,d\tau-\int _{\R}\psi(s)e^{\lambda s}\varphi_{N}(s)\,ds\notag \\
&\ge-\dfrac{\ka^{-}\theta}{4}\int_{\R}\psi(s)e^{\lambda s}\varphi_{N}(s)\,ds,\label{eq:trw_exp_est:i}
\end{align}
as $\psi(\tau-R)\geq\psi(\tau)$, $\tau\in\R$, $R>0$.

Now we estimate $I_3$. By \eqref{trwv}, it is easily seen that the function $(\check{a}^{-}*\psi)(s)$ decreases monotonically to $0$ as $s\to\infty$.
Suppose additionally that $R>0$ above is such that
\[
(\check{a}^{-}*\psi)(s)<\dfrac{\theta}{2}, \quad s>R.
\]
Then, one gets
\begin{align*}\notag
I_3&\geq\dfrac{\theta}{2}\int _{R}^{\infty}\psi (s)e^{\lambda s}\varphi_{N}(s)\,ds+\int _{-\infty}^{R}\psi(s)\bigl(\theta-(\check{a}^{-}*\psi)(s)\bigr)e^{\lambda s}\varphi_{N}(s)\,ds\notag\\
&\geq \dfrac{\theta}{2}\int _{R}^{\infty}\psi (s)e^{\lambda s}\varphi_{N}(s)\,ds,
\end{align*}
as $\psi\geq0$, $\varphi_N\geq0$, $(\check{a}^{-}*\psi)(s)\leq\theta$.

It remains to estimate $I_1$ (in the case $c\neq0$). Since
$\lim\limits_{s\to\pm\infty} \psi(s)e^{\la s}\varphi_N(s) =0$, we have from the inntegration by parts formula, that
\begin{equation*}
I_1=-c\int_{\R}\psi(s)(\lambda\varphi_{N}(s)+\varphi_{N}'(s))e^{\lambda s}\,ds.
\end{equation*}
For $c>0$, one can use that $\varphi_N'(s)\leq0$, $s\in\R$, and hence
\begin{equation*}
  I_1\geq -c \lambda\int _{\R}\psi(s)\varphi_{N}(s)e^{\lambda s}\,ds.
\end{equation*}
For $c<0$, we use that, by the definition of $\varphi_N$, $\lambda\varphi_{N}(s)+\varphi_{N}'(s)=0$, $s\geq N$; therefore,
\begin{equation}
I_1=-c\lambda\int_{-\infty}^N\psi(s)\,ds>0. \label{eq:trw_exp_est:iv}
\end{equation}

Therefore, combining \eqref{eq:trw_exp_est:i}--\eqref{eq:trw_exp_est:iv}, we get from \eqref{eq:int_trw_zeta_exp}, that
\begin{multline*}
0\geq
-\lambda \bar{c}\int _{\R}\psi(s)\varphi_{N}(s)e^{\lambda s}\,ds
-\dfrac{\ka^{-}\theta}{4}\int_{\R}\psi(s) e^{\lambda s}\varphi_{N}(s)\,ds
\\+\dfrac{\ka^{-}\theta}{2}\int_{R}^{\infty}\psi(s) e^{\lambda s}\varphi_{N}(s)\,ds,
\end{multline*}
where $\bar{c}=\max\{c,0\}$.

The latter inequality can be easily rewritten as
\begin{align}\notag
&\quad \Bigl(\dfrac{\ka^{-}\theta}{4}-\lambda \bar{c}\Bigr)\int_{R}^{\infty}\psi (s) e^{\lambda s}\varphi_{N}(s)\,ds\leq \Bigl(\dfrac{\ka^{-}\theta}{4}+\lambda \bar{c}\Bigr)\int_{-\infty}^{R}\psi(s)\varphi_{N}(s)e^{\lambda s}\,ds\\
&\leq \Bigl(\dfrac{\ka^{-}\theta}{4}+\lambda \bar{c}\Bigr)\theta \int_{-\infty}^{R}e^{\lambda s}\,ds=:I_{\la,R}<\infty,\label{eq:333}
\end{align}
for any $0<\la<\nu$.

Take now $\mu<\min\bigl\{\nu, \frac{\ka^-\theta}{4c}\bigr\}$, for $c>0$, and $\mu<\nu$, otherwise. Then, by \eqref{eq:333}, for any $N>R$, one get
\[
\infty>\Bigl(\dfrac{\ka^{-}\theta}{4}-\mu \bar{c}\Bigr)^{-1}I_{\mu,R}>\int_{R}^{\infty}\psi (s) e^{\mu s}\varphi_{N}(s)\,ds\geq
\int_{R}^{N}\psi (s) e^{\mu s}\,ds,
\]
thus,
\begin{align*}
\int_{\R}\psi (s) e^{\mu s}\,ds&=\int_{-\infty}^R \psi (s) e^{\mu s}\,ds
+\int_{R}^{\infty}\psi (s) e^{\mu s}\,ds
\\&\leq \theta \int_{-\infty}^R e^{\mu s}\,ds+I_{\mu,R}\Bigl(\dfrac{\ka^{-}\theta}{4}-\mu \bar{c}\Bigr)^{-1}<\infty,
\end{align*}
that gets the statement.
\end{proof}

\subsection{Speed and profile of a traveling wave}

Through this subsection we will suppose, additionally to \eqref{as:chiplus_gr_m} and \eqref{as:aplus_gr_aminus}, that
\begin{assum}\label{boundedkernels}
    a^+\in L^\infty(\X).
  \end{assum}
Clearly, \eqref{as:aplus_gr_aminus} and \eqref{boundedkernels} imply $a^-\in L^\infty(\X)$.
    \begin{remark}
  All further statements remain true if we change \eqref{boundedkernels} on the condition $\check{a}^+\in L^\infty(\R)$, where $\check{a}^+$ is given by \eqref{apm1dim}; evidently, the latter condition is, for $d\geq2$, weaker than \eqref{boundedkernels}.
\end{remark}

Let $\xi\in\S $ be fixed and \eqref{aplusexpint1} hold. Assume also that
  \begin{assum}\label{firstmoment}
   \int_\X \lvert x\cdot \xi \rvert \, a^+(x)\,dx<\infty.
  \end{assum}
Under assumption \eqref{firstmoment}, we define
\begin{equation}\label{firstdirmoment}
\m_\xi:=\int_\X x\cdot \xi \, \, a^+(x)\,dx.
\end{equation}

Suppose also, that the following modification of \eqref{as:a+nodeg} holds:  \begin{assum}\label{as:a+nodeg-mod}
    \begin{gathered}
    \text{there exist $r=r(\xi)\geq0$, $\rho=\rho(\xi)>0$, $\delta=\delta(\xi)>0$, such that}\\
    a^+(x)\geq\rho, \text{ for a.a.\! $x\in B_{\delta}(r\xi)$.}
    \end{gathered}
  \end{assum}

For an $f\in L^\infty(\R)$, let $\L f$ be a bilateral-type Laplace transform of $f$, cf.~\cite[Chapter~VI]{Wid1941}:
\begin{equation}\label{Laplace}
(\L f)(z)=\int_\R f(s)e^{z s}\,ds,\quad \Re z>0.
\end{equation}

We collect several results about $\L$ in the following lemma.
\begin{lemma}\label{lem:allaboutLapl}
Let $f\in L^\infty(\R)$.
\begin{enumerate}[label=\textnormal{(L\arabic*)}]
\item\label{L-converges} There exists $\la_0(f)\in [0,\infty]$ such that the integral \eqref{Laplace} converges in the strip $\{0<\Re z<\la_0(f)\}$ (provided that $\la_0(f)>0$) and diverges in the half plane $\{\Re z> \la_0(f)\}$ (provided that $\la_0(f)<\infty$).
\item\label{L-analytic} Let $\la_0(f)>0$. Then $(\L f) (z)$ is analytic in $\{0<\Re z<\la_0(f)\}$, and, for any $n\in\N$,
    \begin{equation*}
      \dfrac{d^n}{dz^n}(\L f)(z)=\int_\R e^{zs} s^n f(s)\,ds, \quad 0<\Re z<\la_0(f).
    \end{equation*}
\item\label{L-singular} Let $f\geq0$ a.e.\!\! and $0<\la_0(f)<\infty$. Then $(\L f)(z)$ has a singularity at $z=\la_0(f)$. In particular, $\L f$ has not an analytic extension to a strip  $0<\Re z<\nu$, with $\nu>\la_0(f)$.
\item\label{L-derivative} Let $f':=\frac{d}{ds}f\in L^\infty(\R)$, $f(\infty)=0$, and $\la_0(f')>0$. Then $\la_0(f)\geq\la_0(f')$ and, for any $0<\Re z<\la_0(f')$,
\begin{equation}\label{LaplaceofDer}
(\L f')(z)=-z (\L f)(z).
\end{equation}
\item\label{L-convolution} Let $g\in L^\infty(\R)\cap L^1(\R)$ and $\la_0(f)>0$, $\la_0(g)>0$. Then $\la_0(f*g)\geq\min\{\la_0(f),\la_0(g)\}$ and, for any $0<\Re z<\min\{\la_0(f),\la_0(g)\}$,
\begin{equation}\label{LaplaceofConv}
\bigl(\L(f*g)\bigr)(z)=(\L f)(z) (\L g)(z)
\end{equation}
\item\label{L-onesidelimit}
Let $0\leq f\in L^1(\R)\cap L^\infty(\R)$ and $\la_0(f)>0$. Then $\lim\limits_{\la\to 0+} (\L f)(\la)=\int_\R f(s)\,ds$.
\item\label{L-onesidelimit2}
Let $f\geq0$, $\la_0(f)\in(0,\infty)$ and $A:=\int_\R f(s) e^{\la_0(f)s}\,ds<\infty$. Then $\lim\limits_{\la\to \la_0(f)-} (\L f)(\la)=A$.
\item\label{L-decaying} Let $f\geq0$ be decreasing on $\R$, and let $\la_0(f)>0$. Then, for any $0<\la<\la_0(f)$,
\begin{equation}\label{expdecay}
f(s)\leq \frac{\la e^\la}{e^\la-1} (\L f)(\la) e^{-\la s}, \quad s\in\R.
\end{equation}
Moreover, for any $0\leq g\in L^\infty(\R)\cap L^1(\R)$, $\la_0(g)>0$,
\begin{equation}\label{absofsq}
  \la_0\bigl(f(g *f)\bigr)\geq\la_0(f)+ \min\bigl\{\la_0(g),\la_0(f)\bigr\}.
\end{equation}
\end{enumerate}
\end{lemma}
\begin{proof}
We can rewrite $\L=\L^++\L^-$, where
\begin{equation*}
(\L^\pm f)(z)=\int_{\R_\pm} f(s)e^{z s}\,ds, \quad \Re z>0,
\end{equation*}
$\R_+=[0,\infty)$, $\R_-=(\infty,0]$.  Let $\mathcal{L}$ denote the classical (unilateral) Laplace transform:
\[
(\mathcal{L}f)(z)=\int_{\R_+}f(s)e^{-z s}\,ds,
\]
and $\mathfrak{l}_0(f)$ be its abscissa of convergence (see details, e.g. in \cite[Chapter~II]{Wid1941}).
Then, clearly,
$(\L^+ f)(z)=(\mathcal{L}f)(-z)$, $(\L ^- f)(z)=(\mathcal{L}f^-)(z)$, where $f^-(s)=f(-s)$, $s\in\R$. As a result, $\la_0(f)=-\mathfrak{l}_0(f)$.

It is easily seen that, for $f\in L^\infty(\R)$, $\mathfrak{l}_0(f^-)\leq0$, in particular,the function $(\L^- f)(z)$ is analytic on $\Re z>0$.

Therefore, the properties \ref{L-converges}--\ref{L-singular} are direct consequences of \cite[Theorems II.1, II.5a, II.5b]{Wid1941}, respectively. The property \ref{L-derivative} may be easily derived from \cite[Theorem II.2.3a, II.2.3b]{Wid1941}, taking into account that $f(\infty)=0$. The property~\ref{L-convolution} one gets by a straightforward computation, cf. \cite[Theorem VI.16a]{Wid1941}; note that $f*g\in L^\infty(\R)$.

Next, $\la_0(f)>0$ implies $\mathfrak{l}_0(f)<0$, therefore, $\L^+f$ can be analytically continued to $0$. If $\mathfrak{l}(f^-)<0$, then $\L^- f$ can be analytically continued to $0$ as well, and \ref{L-onesidelimit} will be evident. Otherwise, if $\mathfrak{l}(f^-)=0$ then \ref{L-onesidelimit} follows from \cite[Theorem~V.1]{Wid1941}. Similar arguments prove \ref{L-onesidelimit2}.

To prove \ref{L-decaying} for decreasing nonnegative $f$, note that, for any $0<\la<\la_0(f)$,
\[
f(s) \int_{s-1}^{s}e^{\la \tau}\,d\tau\leq
\int_{s-1}^s f(\tau) e^{\la \tau}\,d\tau\leq (\L f)(\la), \quad s\in\R,
\]
that implies \eqref{expdecay}. Next, by \ref{L-convolution},
$\la_0(g*f)>0$, and conditions on $g$ yield that $g*f\geq0$ is decreasing as well. Therefore, by \eqref{expdecay}, for any $0<\la<\la_0(g *f)$,
\begin{align*}
  \bigl\lvert\bigl(\L(f(g *f))\bigr)(z)\bigr\rvert&\leq \int_\R f(s)(g *f)(s)
  e^{s\Re z}\,ds\\
  &\leq \frac{\la e^\la}{e^\la-1}\bigl(\L (g *f)\bigr)(\la)\int_\R f(s)e^{-s\la}
  e^{s\Re z}\,ds<\infty,
\end{align*}
provided that $\Re z<\la_0(f)+\la<\la_0(f)+\la_0(g*f)$. As a result,
$\la_0\bigl( f(g*f)\bigr)\geq\la_0(f)+\la_0(g*f)$ that, by \ref{L-convolution}, implies \eqref{absofsq}.
\end{proof}

\medskip

Fix any $\xi\in\S $. Then, by \eqref{expintla1}, one has that $\la_0(\check{a}^\pm)>0$.
Consider, cf. \eqref{demand}, \eqref{cstarestimate}, the following complex-valued function
\begin{equation}\label{dedGxi}
G_\xi(z):=\frac{\ka^+(\L\check{a}^+)(z)-m}{z}, \quad \Re z>0,
\end{equation}
which is well-defined on $0<\Re z<\la_0(\check{a}^+)$. Note that, by \eqref{expintla1},
\[
(\L\check{a}^+)(\la)= \A_\xi(\la), \qquad G_\xi(\la)=\frac{\ka^+\A_\xi(\la)-m}{\la}, \qquad 0<\la<\la_0(\check{a}^+),
\]
and hence, by \eqref{cstarestimate},
\begin{equation}\label{ineqwillbeeq}
c_*(\xi)\leq \inf\limits_{\la>0} G_\xi(\la),
\end{equation}
where $c_*(\xi)$ is the minimal speed of traveling waves, cf.~Theorem~\ref{thm:trwexists}. We will show below that in fact there exists equality in \eqref{ineqwillbeeq}, and hence in \eqref{cstarestimate}.

We start with the following notations to simplify the further statements.
\begin{definition}\label{def:Uxi}
  Let $m>0$, $\ka^\pm>0$, $0\leq a^-\in L^1(\X)$  be fixed, and \eqref{as:chiplus_gr_m} holds. For an arbitrary $\xi\in\S $, denote by $\Uxi$ the subset of functions $0\leq a^+\in L^1(\X)$ such that \eqref{as:aplus_gr_aminus} and \eqref{aplusexpint1}--\eqref{as:a+nodeg-mod} hold.

For $a^+\in\Uxi$, denote also the interval $I_\xi \subset (0,\infty)$ by
\[
I_\xi :=\begin{cases}
  (0,\infty), &\text{if}  \ \
  \la_0(\check{a}^+)=\infty,\\[2mm]
  \bigl(0,\la_0(\check{a}^+)\bigr), & \text{if} \ \la_0(\check{a}^+)<\infty \ \ \text{and} \ \ (\L \check{a}^+)\bigl(\la_0(\check{a}^+)\bigr)=\infty \\[2mm]
  \bigl(0,\la_0(\check{a}^+)\bigr], & \text{if} \  \la_0(\check{a}^+)<\infty \ \ \text{and} \ \   (\L \check{a}^+)\bigl(\la_0(\check{a}^+)\bigr)<\infty.
\end{cases}
\]
\end{definition}

\begin{proposition}\label{prop:infGisreached}
Let $\xi\in\S $ be fixed and $a^+\in\Uxi$.
Then there exists a unique $\la_*=\la_*(\xi)\in I_\xi $ such that
\begin{equation}\label{arginf}
  \inf\limits_{\la>0} G_\xi(\la)=\min_{\la\in I_\xi }G_\xi(\la)=G_\xi(\la_*)>\ka^+\m_\xi.
\end{equation}
Moreover, $G_\xi$ is strictly decreasing on $(0,\la_*]$ and $G_\xi$ is strictly increasing on $I_\xi\setminus(0,\la_*]$ (the latter interval may be empty).
\end{proposition}
\begin{proof}
First of all, by \eqref{apm1dim}, the condition \eqref{firstmoment} implies, cf.~\eqref{expintla1},
\begin{equation}\label{zeromoment1dim}
 \m_\xi=\int_\R s \check{a}^+(s)\,ds \in \R.
\end{equation}
Next, to simplify notations, we set $\la_0:=\la_0(\check{a}^+)\in (0,\infty]$. Denote also
  \begin{equation}\label{defF}
 F_\xi(\la):=\ka^+ \A_\xi(\la)-m=\la G_\xi(\la), \qquad \la\in I_\xi .
 \end{equation}

By \ref{L-analytic}, for any $\la\in(0,\la_0)$,
\begin{equation}\label{secderA}
  \A_\xi''(\la)=\int_\R s^2 \check{a}^+(s) e^{\la s}\,ds>0,
\end{equation}
therefore, $\A_\xi'(\la)$ is increasing on $(0,\la_0)$; in particular, by \eqref{zeromoment1dim}, we have, for any $\la\in(0,\la_0)$,
\begin{equation}\label{posderL}
\int_\R s \check{a}^+(s) e^{\la s}\,ds = \A_\xi'(\la)> \A_\xi'(0)=\int_\R  s \check{a}^+(s)\,ds=\m_\xi.
\end{equation}
Next, by \ref{L-onesidelimit}, $F_\xi(0+)=\ka^+-m>0$, hence,
\begin{equation}\label{valat0}
  G_\xi(0+)=\infty.
\end{equation}
Finally, for $\la\in (0,\la_0)$, we have
  \begin{align}\label{Gder1}
    G_\xi'(\la)&=\la^{-2}\bigl(\la F_\xi'(\la)-F_\xi(\la)\bigr)=\la^{-1}\bigl(F_\xi'(\la)-G_\xi(\la)\bigr),\\
    G_\xi''(\la)&=\la^{-1} (F_\xi''(\la)- 2G_\xi'(\la)).\label{Gder2}
  \end{align}

We will distinguish two cases.

\textit{Case 1.\/} There exists $\mu\in(0,\la_0)$ with $G_\xi'(\mu)=0$. Then, by \eqref{Gder2}, \eqref{secderA},
\begin{equation*}
   G_\xi''(\mu)=\mu^{-1} F_\xi''(\mu)=\mu^{-1}\ka^+ \A_\xi''(\mu)>0.
\end{equation*}
Hence any stationary point of $G_\xi$ is with necessity a point of local minimum, therefore, $G_\xi$ has at most one such a point, thus it will be a global minimum. Moreover, by \eqref{Gder1}, \eqref{posderL}, $G'(\mu)=0$ implies
\begin{equation}\label{gatmin}
    G_\xi(\mu)=F_\xi'(\mu) =\ka^+ \A_\xi'(\mu)> \ka^+\m_\xi.
\end{equation}
Therefore, in the Case 1, one can choose $\la_*=\mu$ (which is unique then) to fulfill the statement.

List the conditions under which the Case 1 is possible.
 \begin{enumerate}
 \item Let $\la_0=\infty$. Note that \eqref{as:a+nodeg-mod} implies that there exist  $\delta'>0$, $\rho'>0$, such that $\check{a}^+(s)\geq\rho'$, for a.a.\! $s\in[r-\delta',r+\delta']$. Indeed, fix, for the case $d\geq2$, a basis $\eta_1,\ldots,\eta_{d-1}$ of $H_\xi=\{\xi\}^\bot$, cf. definition of \eqref{apm1dim}, then
 \[
 B_{\delta}(r\xi)\supset\Bigl\{(r+\sigma)\xi+\tau_1\eta_1+\ldots+\tau_{d-1}\eta_{d-1}\Bigm\vert \lvert \sigma\rvert \leq \frac{\delta}{\sqrt{d}}, \lvert \tau_i \rvert \leq \frac{\delta}{\sqrt{d}}\Bigr\}.
 \]
 Therefore, by \eqref{apm1dim} and \eqref{as:a+nodeg-mod},
\begin{equation}\label{checkapluspos}
 \check{a}^+(s)\geq \rho \Bigl(\frac{2\delta}{\sqrt{d}}\Bigr)^{d-1}=:\rho', \quad s\in[r-\delta',r+\delta'], \quad \delta':=\frac{\delta}{\sqrt{d}}.
\end{equation}

Hence if $\la_0=\infty$, then
\begin{equation}\label{toinfasltoinf}
  \frac{1}{\la}\A_\xi(\la)\geq \frac{1}{\la}\int_r^{r+\delta'}\check{a}^+(s)e^{\la s}\,ds\geq\rho'\frac{1}{\la^2}\bigl(e^{\la(r+\delta')}-e^{\la r}\bigr)\to\infty,
\end{equation}
as $\la\to\infty$. Then, in such a case, $G_\xi(\infty)=\infty$. Therefore, by \eqref{valat0}, there exists a zero of $G_\xi'$.

\item Let $\la_0<\infty$ and $\A_\xi(\la_0)=\infty$. Then, again, \eqref{valat0} implies the existence of a zero of $G_\xi'$ on $(0,\la_0)$.

\item Let $\la_0<\infty$ and $\A_\xi(\la_0)<\infty$. By \eqref{defF}, \eqref{Gder1},
\begin{equation*}
\lim_{\la\to0+}\la^2 G_\xi'(\la)=-F_\xi(0+)=-(\ka^+-m)<0.
\end{equation*}
Therefore, the function $G_\xi'$ has a zero on $(0,\la_0)$ if and only if takes a positive value at some point from $(0,\la_0)$.
\end{enumerate}

Now, one can formulate and consider the opposite to the Case 1.

\textit{Case 2.\/} Let $\la_0<\infty$, $\A_\xi(\la_0)<\infty$, and
\begin{equation}\label{Gneg}
  G_\xi'(\la)<0,\quad \la\in(0,\la_0).
\end{equation}
Therefore,
\begin{equation}\label{infattheend}
\inf\limits_{\la>0} G_\xi(\la)=\inf\limits_{\la\in(0,\la_0]} G_\xi(\la)=\lim_{\la\to\la_0-} G_\xi(\la)=G_\xi(\la_0),
\end{equation}
by \ref{L-onesidelimit2}. Hence we have the first equality in \eqref{arginf}, by setting $\la_*:=\la_0$. To~prove the second inequality in \eqref{arginf}, note that, by \eqref{Gder1}, the inequality \eqref{Gneg} is equivalent to $F_\xi'(\la)< G_\xi(\la)$, $\la\in(0,\la_0)$. Therefore, by \eqref{infattheend}, \eqref{defF}, \eqref{posderL},
  \[
  G_\xi(\la_0)=\inf\limits_{\la\in\bigl(\frac{\la_0}{2},\la_0\bigr)} G_\xi(\la)\geq \inf\limits_{\la\in\bigl(\frac{\la_0}{2},\la_0\bigr)} F_\xi'(\la)\geq \ka^+ \A_\xi'\Bigl(\frac{\la_0}{2}\Bigr)> \ka^+\m_\xi,
  \]
where we used again that, by \eqref{secderA}, $\A_\xi'$ and hence $F_\xi'$ are increasing on $(0,\la_0)$.
The statement is fully proved now.
\end{proof}

The second case in the proof of Proposition~\ref{prop:infGisreached} requires additional analysis. Let $\xi\in\S $ be fixed and $a^+\in\Uxi$, $\la_0:=\la_0(\check{a}^+)$. By \ref{L-analytic}, one can define the following function
\begin{equation}\label{specfunc}
    \T_\xi(\la):=\ka^+\int_\R (1-\la s)\check{a}^+(s)e^{\la s}\,ds\in\R, \qquad \la\in[0,\la_0).
\end{equation}
Note that
\begin{equation}\label{triv}
    \int_{\R_-}|s|\check{a}^+(s)e^{\la_0 s}\,ds<\infty,
\end{equation}
and
$\int_{\R_+}s\check{a}^+(s)e^{\la_0 s}\,ds\in(0,\infty]$ is well-defined. Then, in the case $\la_0<\infty$ and $\A_\xi(\la_0)<\infty$, one can continue $\T_\xi$ at $\la_0$, namely,
  \begin{equation}\label{newquant}
    \T_\xi(\la_0):=\ka^+\int_\R (1-\la_0 s)\check{a}^+(s)e^{\la_0 s}\,ds\in[-\infty,\ka^+).
  \end{equation}
To prove the latter inclusion, i.e. that $\T_\xi(\la_0)<\ka^+$, consider the function
$f_0(s):=(1-\la_0 s)e^{\la_0 s}$, $s\in\R$. Then, $f'_0(s)=-\la_0^2 se^{\la_0 s}$, and thus $f_0(s)<f_0(0)=1$, $s\neq 0$. Moreover, the function $g_0(s)=f_0(-s)-f_0(s)$, $s\geq0$ is such that $g_0'(s)=\la_0^2 s(e^{\la_0 s}-e^{-\la_0 s})>0$, $s>0$. As a result, for any $\delta>0$, $f_0(-\delta)>f_0(\delta)$, and
\[
\int_\R f_0(s)\check{a}^+(s)\,ds\leq f_0(-\delta)\int_{\R\setminus[-\delta,\delta]}\check{a}^+(s)\,ds
+\int_{[-\delta,\delta]}\check{a}^+(s)\,ds<\int_\R \check{a}^+(s)\,ds=1.
\]

\begin{proposition}\label{prop:Gnegholds}
  Let $\xi\in\S $ be fixed and $a^+\in\Uxi$. Suppose also that $\la_0:=\la_0(\check{a}^+)<\infty$ and $\A_\xi(\la_0)<\infty$. Then \eqref{Gneg} holds iff
   \begin{gather}\label{posquant}
    \T_\xi(\la_0)\in(0,\ka^+),\\
    \label{smallm}
     m\leq\T_\xi(\la_0).
   \end{gather}
\end{proposition}
\begin{proof}
  Define the function, cf. \eqref{defF},
\begin{equation}\label{defH}
  H_\xi(\la):=\la F'_\xi(\la)-F_\xi(\la), \quad \la\in(0,\la_0).
\end{equation}
By \eqref{Gder1}, the condition \eqref{Gneg} holds iff $H_\xi$ is negative on $(0,\la_0)$.
By \eqref{defH}, \eqref{secderA}, one has $H'_\xi(\la)=\la F''_\xi(\la)>0$, $\la\in(0,\la_0)$ and, therefore, $H_\xi$ is (strictly) increasing on $(0,\la_0)$.
By~Proposition~\ref{prop:infGisreached}, $G_\xi'$, and hence $H_\xi$, are negative on a right-neighborhood of $0$. As a result, $H_\xi(\la)<0$ on $(0,\la_0)$ iff
\begin{equation}\label{leftlimHisneg}
\lim\limits_{\la\to\la_0-}H_\xi(\la)\leq0.
\end{equation}
On the other hand, by \eqref{defF}, \eqref{specfunc}, one can rewrite $H_\xi(\la)$ as follows:
\begin{equation}\label{rewrH}
  H_\xi(\la)=-\T_\xi(\la)+m, \quad \la\in(0,\la_0).
\end{equation}
By the monotone convergence theorem,
\[
\lim\limits_{\la\to\la_0-}\int_{\R_+} (\la s-1)\check{a}^+(s)e^{\la s}\,ds=\int_{\R_+} (\la_0 s-1)\check{a}^+(s)e^{\la_0 s}\,ds\in(-1,\infty].
\]
Therefore, by \eqref{triv}, \eqref{rewrH}, $\T_\xi(\la_0)\in\R$ iff $H_\xi(\la_0)=\lim\limits_{\la\to\la_0-}H_\xi(\la)\in\R$. Next, clearly, $H_\xi(\la_0)\in(m-\ka^+,0]$ holds true iff both \eqref{smallm} and \eqref{posquant} hold.

As a result, \eqref{Gneg} is equivalent to \eqref{leftlimHisneg} and the latter, by \eqref{triv}, implies that $\T_\xi(\la_0)\in\R$ and hence $H_\xi(\la_0)\in(m-\ka^+,0]$.
Vice versa, \eqref{posquant} yields $\T_\xi(\la_0)\in\R$ that together with \eqref{smallm} give that $H_\xi(\la_0)\leq0$, i.e. that \eqref{Gneg} holds.
\end{proof}

According to the above, it is natural to consider two subclasses of functions from $\Uxi$, cf.~Definition~\ref{def:Uxi}.

\begin{definition}\label{def:VWxi}
  Let $\xi\in\S $ be fixed. We denote by $\Vxi$ the class of all kernels $a^+\in\Uxi$ such that one of the following assumptions does hold:
  \begin{enumerate}
    \item $\la_0:=\la_0(\check{a}^+)=\infty$;
    \item $\la_0<\infty$ and $\A_\xi(\la_0)=\infty$;
    \item $\la_0<\infty$, $\A_\xi(\la_0)<\infty$ and
  $\T_\xi(\la_0)\in[-\infty,m)$, where $\T_\xi(\la_0)$ is given by \eqref{newquant}.
  \end{enumerate}
\noindent Correspondingly, we denote by $\Wxi$ the class of all kernels $a^+\in\Uxi$ such that
  $\la_0<\infty$, $\A_\xi(\la_0)<\infty$, and $\T_\xi(\la_0)\in[m,\ka^+)$.
  Clearly, $\Uxi=\Vxi\cup\Wxi$.
\end{definition}

As a result, combining the proofs and statements of Propositions~\ref{prop:infGisreached} and~\ref{prop:Gnegholds}, one immediately gets the following corollary.
\begin{corollary}\label{cor_alltogether}
   Let $\xi\in\S $ be fixed, $a^+\in\Uxi$, and $\la_*$ be the same as in Proposition~\ref{prop:infGisreached}. Then $\la_*<\la_0:=\la_0(\check{a}^+)$ iff $a^+\in\Vxi$; moreover, then $G'(\la_*)=0$. Correspondingly, $\la_*=\la_0$ iff $a^+\in\Wxi$; in this case,
   \begin{equation}\label{leftlimG}
     \lim\limits_{\la\to\la_0-}G_\xi'(\la)=\frac{m-\T_\xi(\la_0)}{\la_0^2}\leq 0.
   \end{equation}
\end{corollary}

\begin{example}\label{ex:spefunc}
To demonstrate the cases of Definition~\ref{def:VWxi} on an example, consider the following family of functions
\begin{equation}\label{exofa}
\check{a}^+(s):=\frac{\alpha e^{-\mu|s|^p}}{1+|s|^q}, \quad s\in\R, p\geq0, q\geq0, \mu>0,
\end{equation}
where $\alpha>0$ is a normalising constant to get \eqref{cleareq}. Clearly, the case $p\in[0,1)$ implies $\la_0(\check{a}^+)=0$, that is impossible under assumption \eqref{aplusexpint1}. Next, $p>1$ leads to $\la_0(\check{a}^+)=\infty$, in particular, the corresponding $a^+\in\Vxi$. Let now $p=1$, then $\la_0(\check{a}^+)=\mu$. The case $q\in[0,1]$ gives $\A_\xi(\la_0)=\infty$, i.e. $a^+\in\Vxi$ as well. In~the~case $q\in (1,2]$, we will have that $\A_\xi(\la_0)<\infty$, however,$\int_\R s \check{a}^+(s)e^{\mu s}\,ds=\infty$, i.e. $\T_\xi(\mu)=-\infty$, and again $a^+\in\Vxi$. Let $q>2$; then, by \eqref{specfunc},
\begin{align*}
\T_\xi(\mu)&=\ka^+\alpha \int_{\R_-} \frac{1-\mu s}{1+|s|^q}e^{2\mu s}\,ds
+\ka^+\alpha \int_{\R_+} \frac{1-\mu s}{1+s^q}\,ds\\&\geq \ka^+\alpha \int_{\R_+} \frac{1-\mu s}{1+s^q}\,ds=\frac{\pi\ka^+\alpha }{q}\biggl(\frac{1}{\sin\frac{\pi}{q}}-\frac{\mu}{\sin\frac{2\pi}{q}}\biggr)\geq m,
\end{align*}
if only $\mu\leq 2\cos\frac{\pi}{q}- \frac{m q}{\ka^+\alpha\pi}\sin\frac{2\pi}{q}$ (note that $q>2$ implies $\sin\frac{2\pi}{q}>0$); then we have $a^+\in\Wxi$. On the other hand, using the inequality $te^{-t}\leq e^{-1}$, $t\geq0$, one gets
\begin{align}\label{muexpr}
  \T_\xi(\mu)&=\ka^+\alpha \int_{\R_+} \frac{(1+\mu s)e^{-2\mu s}+1-\mu s}{1+s^q}\,ds\\
  &\leq \ka^+\alpha \int_{\R_+} \frac{1+\frac{1}{2e}+1-\mu s}{1+s^q}\,ds=
  \frac{\pi\ka^+\alpha }{q}\biggl(\frac{1+4e}{2e\sin\frac{\pi}{q}}-\frac{\mu}{\sin\frac{2\pi}{q}}\biggr)<m,  \notag
\end{align}
if only $\mu>\frac{1+4e}{e}\cos\frac{\pi}{q}- \frac{m q}{\ka^+\alpha\pi}\sin\frac{2\pi}{q}$; then we have $a^+\in\Vxi$.
 Since
 \[
 \frac{d}{d\mu}\bigl((1+\mu s)e^{-2\mu s}+1-\mu s\bigr)=-se^{-2\mu s}(1+2s\mu)-s<0, \quad s>0, \mu>0,
 \]
we have from \eqref{muexpr}, that $\T_\xi(\mu)$ is strictly decreasing and continuous in $\mu$, therefore, there exist a critical value
\[
\mu_*\in \Bigl(2\cos\frac{\pi}{q}- \frac{m q}{\ka^+\alpha\pi}\sin\frac{2\pi}{q},(4+e^{-1})\cos\frac{\pi}{q}- \frac{m q}{\ka^+\alpha\pi}\sin\frac{2\pi}{q}\Bigr),
\]
such that, for all $\mu>\mu_*$, $a^+\in\Vxi$, whereas, for $\mu\in(0,\mu_*]$, $a^+\in\Wxi$.
\end{example}

Now we are ready to prove the main statement of this subsection.

\begin{theorem}\label{thm:speedandprofile}
 Let $\xi\in\S $ be fixed and $a^+\in\Uxi$. Let $c_*(\xi)$ be the minimal traveling wave speed according to Theorem~\ref{thm:trwexists}, and let, for any $c\geq c_*(\xi)$, the function $\psi=\psi_c\in\M$ be a traveling wave profile corresponding to the speed~$c$. Let $\la_*\in I_\xi$ be the same as in Proposition~\ref{prop:infGisreached}. Denote, as usual, $\la_0:=\la_0(\check{a}^+)$.
\begin{enumerate}
\item The following relations hold
\begin{gather}
c_*(\xi)=\min_{\la>0}\frac{\ka^+\A_\xi(\la)-m}{\la}=
\frac{\ka^+\A_\xi(\la_*)-m}{\la_*}>\ka^+\m_\xi,\label{minspeed}
\\
\label{finitespeed}
\la_0(\psi)\in(0,\la_*],\\
\label{Laplinfatabs}
(\L \psi)\bigl(\la_0(\psi)\bigr)=\infty;
\end{gather}
and the mapping $(0,\la_*]\ni\la_0(\psi)\mapsto c\in[c_*(\xi),\infty)$ is a (strictly) monotonically decreasing bijection, given by
\begin{equation}\label{speedviaabs}
c=\frac{\ka^+\A_\xi\bigl(\la_0(\psi)\bigr)-m}{\la_0(\psi)}.
\end{equation}
In particular, $\la_0(\psi_{c_*(\xi)})=\la_*$.
\item For $a^+\in\Vxi$, one has $\la_*<\la_0$ and there exists another representation for the minimal speed than \eqref{speedviaabs}, namely,
\begin{equation}\label{minspeed-spec}
    \begin{split}c_*(\xi)&= \ka^+\int_{\X} x\cdot \xi \, a^+(x) e^{\la_* x\cdot \xi} \,dx\\&=\ka^+\int_\R s\check{a}^+(s)e^{\la_* s}\,ds>\ka^+\m_\xi.
    \end{split}
\end{equation}
Moreover, for all $\la\in(0,\la_*]$,
\begin{equation}\label{estonm}
\T_\xi(\la)\geq m,
\end{equation}
and the equality holds for $\la=\la_*$ only.
\item For $a^+\in\Wxi$, one has $\la_*=\la_0$.
Moreover, the inequality \eqref{estonm} also holds as well as, for all $\la\in(0,\la_*]$,
\begin{equation}
    c\geq \ka^+\int_\R s\check{a}^+(s)e^{\la s}\,ds,\label{estonc}
\end{equation}
whereas the equalities in \eqref{estonm} and \eqref{estonc} hold true now for $m=\T_\xi(\la_0)$, $\la=\la_*$, $c=c_*(\xi)$ only.
\end{enumerate}
\end{theorem}
\begin{proof}
By Theorem~\ref{thm:trwexists}, for any $c\geq c_*(\xi)$,
there exists a profile $\psi\in\M$, cf.~Remark~\ref{shiftoftrw}, which define a traveling wave solution \eqref{trwv} to \eqref{eq:basic} in the direction $\xi$. Then, by \eqref{eq:trw}, we get
\begin{equation}
-c\psi'(s)=\ka^{+}(\check{a}^{+}*\psi)(s)-m\psi(s)-\ka^{-}\psi(s) (\check{a}^{-}*\psi)(s), \quad s\in\R.\label{eq:trwder}
\end{equation}

\textit{Step 1}.
By Proposition~\ref{prop:trw_exp_est}, we have that $\la_0(\psi)>0$.
Note also that the condition \eqref{as:aplus_gr_aminus} implies \eqref{acheckpos}, therefore, $\la_0(\check{a}^-)\geq \la_0(\check{a}^+)>0$. Take any $z\in\mathbb{C}$ with
\begin{equation}\label{unexpineq}
0<\Re z<\min\bigl\{\la_0(\check{a}^+),\la_0(\psi)\bigr\}\leq \la_0(\psi)<\la_0\bigl(\psi(\check{a}^-*\psi)\bigr),
\end{equation}
where the later inequality holds by \eqref{absofsq}. As a result, by \ref{L-convolution}, \ref{L-decaying}, being multiplied on $e^{z s}$ the l.h.s.\,of \eqref{eq:trwder} will be  integrable (in $s$) over $\R$. Hence,~for~any~$z$ which satisfies \eqref{unexpineq}, $(\L \psi')(z)$ converges. By~\ref{L-derivative}, it yields $\la_0(\psi)\geq\la_0(\psi')\geq\min\bigl\{\la_0(\check{a}^+),\la_0(\psi)\bigr\}$.

Therefore, by \eqref{LaplaceofDer}, \eqref{LaplaceofConv}, we get from \eqref{eq:trwder}
\begin{multline}
c z  (\L \psi)(z)=\ka^{+}(\L \check{a}^{+})(z) (\L\psi)(z)\\-m(\L\psi)(z)-
\ka^{-}\bigl(\L(\psi(\check{a}^{-}*\psi))\bigr)(z),\label{eqforLtr}
\end{multline}
if only
\begin{equation}\label{condz}
0<\Re z<\min\bigl\{\la_0(\check{a}^+),\la_0(\psi)\bigr\}.
\end{equation}

Since $\psi\not\equiv 0$, we have that $(\L \psi)(z)\neq 0$, therefore, one can rewrite \eqref{eqforLtr} as follows
\begin{equation}\label{neweq}
G_\xi(z)-c=\frac{\ka^{-}\bigl(\L(\psi(\check{a}^{-}*\psi))\bigr)(z)}{z(\L \psi)(z)},
\end{equation}
if \eqref{condz} holds.
By~\eqref{unexpineq}, both nominator and denominator in the r.h.s. of \eqref{neweq} are analytic on $0<\Re z<\la_0(\psi)$, therefore.
Suppose that $\la_0(\psi)>\la_0(\check{a}^+)$, then \eqref{neweq} holds on $0<\Re z<\la_0(\check{a}^+)$, however, the r.h.s. of \eqref{neweq} would be analytic at $z=\la_0(\check{a}^+)$, whereas, by \ref{L-singular}, the l.h.s. of \eqref{neweq} has a singularity at this point. As a result,
\begin{equation}\label{absofkernelisbigger}
\la_0(\check{a}^+)\geq\la_0(\psi),
\end{equation}
for any traveling wave profile $\psi\in\M$.
Thus one gets that \eqref{neweq} holds true on
$0<\Re z<\la_0(\psi)$.

Prove that
\begin{equation}\label{finiteabswavenew}
\la_0(\psi)<\infty.
\end{equation}
Since $0\leq\psi\leq\theta$ yields $0\leq a^-*\psi\leq\theta$, one gets from \eqref{neweq} that, for any $0<\la<\la_0(\psi)$,
\begin{equation}
c\geq G_\xi(\la)-\ka^{-}\dfrac{\theta}{\la}
=\frac{\ka^+(\L\check{a}^+)(\la)-\ka^+}{\la}.\label{lowestforspeed}
\end{equation}
If $\la_0(\check{a}^+)<\infty$ then \eqref{finiteabswavenew} holds by \eqref{absofkernelisbigger}. Suppose that $\la_0(\check{a}^+)=\infty$. By \eqref{toinfasltoinf}, the r.h.s. of \eqref{lowestforspeed} tends to $\infty$ as $\la\to\infty$, thus the latter inequality cannot hold for all $\la>0$; and, as a result, \eqref{finiteabswavenew} does hold.

\textit{Step 2}. Recall that \eqref{ineqwillbeeq} holds. Suppose that $c\geq c_*(\xi)$ is such that, cf.~\eqref{arginf},
\begin{equation}\label{cbiginf}
c\geq G_\xi(\la_*)=\inf_{\la_0\in (0,\la_*] } G_\xi(\la)=\inf_{\la_0\in I_\xi } G_\xi(\la).
\end{equation}
Then, by Proposition~\ref{prop:infGisreached}, the equation $G_\xi(\la)=c$, $\la\in I_\xi $, has one or two solutions. Let $\la_c$ be the unique solution in the first case or the smaller of the solutions in the second one. Since $G_\xi$ is decreasing on $(0,\la_*]$, we have $\la_c\leq \la_*$. Since the nominator in the r.h.s. of \eqref{neweq} is positive, we immediately get from \eqref{neweq} that
\begin{equation}\label{eqdopinfty}
(\L\psi)(\la_c)=\infty,
\end{equation}
therefore, $\la_c\geq\la_0(\psi)$.
On the other hand, one can rewrite \eqref{neweq} as follows
\begin{equation}\label{rewriting}
(\L\psi)(z)=\frac{\ka^{-}\bigl(\L(\psi(\check{a}^{-}*\psi))\bigr)(z)}{z(G_\xi(z)-c)}.
\end{equation}
By \eqref{neweq}, $G_\xi(z)\neq c$, for all $0<\Re z<\la_0(\psi)\leq\la_c\leq\la_*\leq\la_0(\check{a}^+)$. As a result, by \eqref{unexpineq}, \ref{L-converges}, and \ref{L-singular}, $\la_c=\la_0(\psi)$, that together with \eqref{eqdopinfty} proves \eqref{finitespeed} and \eqref{Laplinfatabs}, for waves whose speeds satisfy \eqref{cbiginf}. By \eqref{expintla1}, we immediately get, for such speeds, \eqref{speedviaabs} as well. Moreover, \eqref{speedviaabs} defines a strictly monotone function $(0,\la_*]\ni\la_0(\psi)\mapsto c\in[G_\xi(\la_*),\infty)$.

Next, by \eqref{specfunc}, \ref{L-analytic}, \eqref{defF}, \eqref{Gder1}, we have that, for any $\la\in I_\xi$,
\begin{equation}
  \T_\xi(\la)=\ka^+\A_\xi(\la)-\ka^+\la\A_\xi'(\la)=
m+F_\xi(\la)-\la F_\xi'(\la)=m-\la^2 G_\xi'(\la).\label{Txirepr}
\end{equation}
Recall that, by Proposition~\ref{prop:infGisreached}, the function $G_\xi$ is strictly decreasing on $(0,\la_*)$. Then \eqref{Txirepr} implies that $\T_\xi(\la)>m$, $\la\in(0,\la_*)$. On the other hand, by the second equality in \eqref{Gder1}, the inequality $G_\xi'(\la)<0$, $\la\in(0,\la_*)$, yields $G_\xi(\la)>F_\xi'(\la)$, for such a $\la$.
Let $c>G_\xi(\la_*)$. By \eqref{speedviaabs}, \eqref{defF}, we have then
$c>\ka^+\A_\xi'(\la)$, for all $\la\in[\la_0(\psi),\la_*)$. By~\eqref{secderA}, $F_\xi'$ is increasing, hence, by~\ref{L-analytic}, the strict inequality in \eqref{estonc} does hold, for $\la\in(0,\la_*)$.

 Let again $c\geq G_\xi(\la_*)$, and let $a^+\in\Vxi$. Then, by~Corollary~\ref{cor_alltogether},
 $\la_*<\la_0(\check{a}^+)$ and $G'(\la_*)=0$. By~\eqref{Gder1}, the latter equality and \eqref{Txirepr} give $\T_\xi(\la_*)=m$, that fulfills the proof of \eqref{estonm}, for such $a^+$ and $m$. Moreover, by \eqref{gatmin},
 \begin{equation}\label{dopequal}
   G_\xi(\la_*)=\ka^+\A_\xi'(\la_*)=\ka^+\int_\R s\check{a}^+(s)e^{\la_* s}\,ds.
 \end{equation}
 Let $a^+\in\Wxi$, then $\la_*=\la_0(\check{a}^+)$.
 It means that $\T_\xi(\la_*)=m$ if $m=\T_\xi(\la_0)$ only, otherwise, $\T_\xi(\la_*)>m$.
 Next, we get from \eqref{cbiginf}, \eqref{Gder1}
 \eqref{leftlimG},
 \begin{equation}\label{dsgdsggdetwet}
  c\geq G_\xi(\la_*)\geq \lim\limits_{\la\to\la_*-}F_\xi'(\la)
  =\ka^+\int_\R s\check{a}^+(s)e^{\la_* s}\,ds,
 \end{equation}
 where the latter equality may be easily verified if we rewrite, for $\la\in(0,\la_*)$,
 \begin{equation}\label{tew463634436}
   F_\xi'(\la)=\ka^+\int_{\R_-}s\check{a}^+(s)e^{\la s}\,ds
 +\ka^+\int_{\R_+}s\check{a}^+(s)e^{\la s}\,ds,
 \end{equation}
 and apply the dominated convergence theorem to the first integral and the monotone convergence theorem for the second one. On the other hand, \eqref{leftlimG} implies that the second inequality in \eqref{dsgdsggdetwet} will be strict iff $m<\T_\xi(\la_0)$, whereas, for $c=G_\xi(\la_*)=\inf\limits_{\la>0}G_\xi(\la)$ and $m=\T_\xi(\la_0)$, we will get all equalities in \eqref{dsgdsggdetwet}.

\textit{Step 3}. Let now $c\geq c_*(\xi)$ and suppose that $\la_0(\check{a}^+)>\la_0(\psi)$.
Prove that \eqref{cbiginf} does hold.
On the contrary, suppose that the $c$ is such that
\begin{equation}\label{hypoteticspeed}
c_*(\xi)\leq c<\inf_{\la\in (0,\la_*] } G_\xi(\la)=\inf_{\la>0 } G_\xi(\la).
\end{equation}
Again, by \eqref{neweq}, $G_\xi(z)\neq c$, for all $0<\Re z<\la_0(\psi)$, and \eqref{rewriting} holds, for such a~$z$.
Since we supposed that $\la_0(\check{a}^+)>\la_0(\psi)$, one gets from
\eqref{unexpineq}, that both
nominator and denominator of the r.h.s. of \eqref{rewriting} are analytic on
\[
\{0<\Re z<\nu\} \supsetneq \{0<\Re z<\la_0(\psi)\},
\]
where $\nu=\min\bigl\{ \la_0(\check{a}^+),\la_0\bigl(\psi(\check{a}^-*\psi)\bigr) \bigr\}$.
On the other hand, \ref{L-singular} implies that $\L\psi$ has a singularity at $z=\la_0(\psi)$. Since $\L(\psi(\check{a}^{-}*\psi))\bigr)(\la_0(\psi))>0$, the equality \eqref{rewriting} would be possible if only $G_\xi(\la_0(\psi))=c$, that contradicts \eqref{hypoteticspeed}.

\textit{Step 4}. By \eqref{absofkernelisbigger}, it remains to prove that, for $c\geq c_*(\xi)$, \eqref{cbiginf} does holds, provided that we have $\la_0(\check{a}^+)=\la_0(\psi)$.
Again on the contrary, suppose that \eqref{hypoteticspeed} holds.
For $0<\Re z<\la_0(\psi)$, we can rewrite \eqref{eqforLtr} as follows
\begin{equation}\label{rewriting2}
z(\L\psi)(z)(G_\xi(z)-c)=\ka^{-}\bigl(\L(\psi(\check{a}^{-}*\psi))\bigr)(z).
\end{equation}
In the notations of the proof of Lemma~\ref{lem:allaboutLapl}, the functions
$\L^-\psi$ and $\L^-\check{a}^+$ are analytic on $\Re z>0$. Moreover,
$(\L^+\psi)(\la)$ and $(\L^+\check{a}^+)(\la)$ are increasing on $0<\la<\la_0(\check{a}^+)=\la_0(\psi)$.
Then, cf.~\eqref{tew463634436}, by the monotone convergence theorem, we will get from \eqref{rewriting2} and \eqref{unexpineq}, that
\begin{equation}\label{finiteleftlims}
  \int_\R \psi(s) e^{\la_0(\psi)s}\,ds<\infty, \qquad
\int_\R \check{a}^+(s) e^{\la_0(\check{a}^+)s}\,ds<\infty.
\end{equation}

We are going to apply now Proposition~\ref{lowestsuppCub}, in the case $d=1$, to the equation~\eqref{eq:basic_one_dim}, where the initial condition $\psi$ is a wave profile with the speed $c$ which satisfies \eqref{hypoteticspeed}. Namely, we set $\Delta_R:=(-\infty,R)\nearrow\R$, $R\to\infty$, and let $\check{a}_R^\pm$, $\check{A}_R^\pm$ be given by \eqref{trkern}, \eqref{ARdef} respectively with $d=1$ and $a^\pm$ replaced by $\check{a}^\pm$. Consider a strictly monotonic sequence $\{R_n\mid n\in\N\}$, such that $0<R_n\to\infty$, $n\to\infty$ and
\begin{equation}\label{iiaslater}
  \check{A}_{R_n}^+>\frac{m}{\ka^+},
\end{equation}
cf.~\eqref{bigR}. Let $\theta_n:=\theta_{R_n}$ be given by \eqref{defofthetaR} with $A^\pm_R$ replaced by $\check{A}^\pm_{R_n}$. Then, by \eqref{thetaRlesstheta}, $\theta_n\leq \theta$, $n\in\N$. Fix an arbitrary $n\in\N$. Consider the `truncated' equation \eqref{eq:basic_R} with $d=1$, $a^\pm_R$ replaced by $a^\pm_{R_n}$, and the initial condition $w_0(s):=\min\{\psi(s),\theta_n\}\in C_{ub}(\R)$. By~Proposition~\ref{lowestsuppCub}, there exists the unique solution $w^{(n)}(s,t)$ of the latter equation. Moreover, if we denote the corresponding nonlinear mapping, cf.~Definition~\ref{def:def:Q_t} and Proposition~\ref{prop:Qtilde}, by~$\tilde{Q}^{(n)}_t$, we will have from \eqref{wlessthetaR} and \eqref{ineqtrunc} that
\begin{equation}\label{saewtwetwtgvxc}
  (\tilde{Q}^{(n)}_t w_0)(s)\leq\theta_n, \quad s\in\R, t\geq0,
\end{equation}
 and
\begin{equation}\label{s4366643}
  (\tilde{Q}^{(n)}_t w_0)(s)\leq \phi(s,t),
\end{equation}
where $\phi$ solves \eqref{eq:basic_one_dim}. By~\eqref{tw1d}, we get from \eqref{s4366643} that $(\tilde{Q}^{(n)}_1 w_0)(s+c)\leq \psi(s)$, $s\in\R$. The latter inequality together with \eqref{saewtwetwtgvxc} imply
\begin{equation}\label{wavedoesexists}
  (\tilde{Q}^{(n)}_1 w_0)(s+c)\leq w_0(s).
\end{equation}
Then, by the same arguments as in the proof of Theorem~\ref{thm:trwexists}, cf.~\eqref{mayYag}, we obtain from \cite[Theorem 5]{Yag2009} that there exists a traveling wave $\psi_n$ for the equation \eqref{eq:basic_R} (with $d=1$ and $a^\pm_R$ replaced by $a^\pm_{R_n}$), whose speed will be exactly $c$ (and $c$ satisfies \eqref{hypoteticspeed}).

Now we are going to get a contradiction, by proving that
\begin{equation}\label{infsareequal}
  \inf_{\la>0}G_\xi(\la)=\lim_{n\to\infty}\inf_{\la>0}G^{(n)}_\xi(\la),
\end{equation}
where $G^{(n)}_\xi$ is given by \eqref{dedGxi} with $\check{a}^\pm$ replaced by
$\check{a}_{n}^\pm:=\check{a}_{R_n}^\pm$.
The sequence of functions $G_\xi^{(n)}$ is point-wise monotone in $n$ and it converges to $G_\xi$ point-wise, for $0<\la\leq \la_0(\check{a}^+)$; note we may include $\la_0(\check{a}^+)$ here, according to \eqref{finiteleftlims}. Moreover, $G_\xi^{(n)}(\la)\leq G_\xi(\la)$, $0<\la\leq \la_0(\check{a}^+)$. As a result, for any $n\in\N$,
\begin{equation}\label{asft346eyeryd}
G_\xi^{(n)}(\la_*^{(n)})=\inf_{\la>0}G_\xi^{(n)}(\la)
\leq \inf_{\la>0}G_\xi(\la)=G_\xi(\la_*).
\end{equation}
Hence if we suppose that \eqref{infsareequal} does not hold, then
\[
\inf\limits_{\la>0}G_\xi(\la)-\lim\limits_{n\to\infty}\inf\limits_{\la>0}G^{(n)}_\xi(\la)>0.
\]
Therefore, there exist $\delta>0$ and $N\in\N$, such that
\begin{equation}\label{wqwqwqwqwe}
  G_\xi^{(n)}(\la_*^{(n)})=\inf\limits_{\la>0}G^{(n)}_\xi(\la)\leq \inf\limits_{\la>0}G_\xi(\la)-\delta=G_\xi(\la_*)-\delta, \quad n\geq N.
\end{equation}

Clearly, \eqref{trkern} with $\Delta_{R_n}=(-\infty,R_n)$ implies that $\la_0(\check{a}^+_n)=\infty$, hence $G^{(n)}_\xi$ is analytic on~\mbox{$\Re z>0$}. One can repeat all considerations of the first three steps of this proof for the equation \eqref{eq:basic_R}. Let $c^{(n)}_*(\xi)$ be the corresponding minimal traveling wave speed, according to Theorem~\ref{thm:trwexists}. Then the corresponding inequality \eqref{finiteabswavenew} will show that the abscissa of an arbitrary traveling wave to \eqref{eq:basic_R} is less than $\la_0(\check{a}_{n}^+)=\infty$. As a result, the inequality
$c^{(n)}_*(\xi)<\inf\limits_{\la>0 } G^{(n)}_\xi(\la)$, cf.~\eqref{hypoteticspeed}, is impossible, and hence, by the Step~3,
\begin{equation}\label{minspeedfortrunc}
  c\geq c^{(n)}_*(\xi)=\inf_{\la>0 } G^{(n)}_\xi(\la)=G^{(n)}_\xi(\la_*^{(n)}),
\end{equation}
where $\la_*^{(n)}$ is the unique zero of the function $\frac{d}{d\la}G^{(n)}_\xi (\la)$. Let $\T_\xi^{(n)}$ be given on $(0,\infty)$ by \eqref{specfunc} with $\check{a}^+$ replaced by
$\check{a}_n^+$. Then
\begin{equation}\label{derofT}
  \frac{d}{d\la}\T_\xi^{(n)}(\la)=-\la\ka^+ \int_{-\infty}^{R_n} \check{a}^+(s) s^2 e^{\la s}\,ds<0, \quad \la>0.
\end{equation}
By \eqref{estonm}, the unique point of intersection of the strictly decreasing function
$y=\T_\xi^{(n)}(\la)$ and the horizontal line $y=m$ is exactly the point  $(\la_*^{(n)},0)$.

Prove that there exist $\la_1>0$, such that $\la_*^{(n)}>\la_1$, $n\geq N$, and there exists $N_1\geq N$, such that $\T_\xi^{(n)}(\la)\leq \T_\xi^{(m)}(\la)$, $n> m\geq N_1$, $\la\geq\la_1$.
Recall that \eqref{iiaslater} holds; we have
\begin{align*}
  \la G_\xi^{(n)}(\la) & = \ka^+\int_\R \check{a}^+_n(s) (e^{\la s}-1)\,ds +\ka^+ \check{A}^+_{R_n}-m\\
  & \geq \ka^+\int_{-\infty}^0 \check{a}^+_n(s) (e^{\la s}-1)\,ds +\ka^+ \check{A}^+_{R_1}-m,
\end{align*}
and the inequality $1-e^{-s}\leq s$, $s\geq 0$ implies that
\[
\biggl\lvert \int_{-\infty}^0 \check{a}^+_n(s) (e^{\la s}-1)\,ds\biggr\rvert
\leq \la \int_{-\infty}^0 \check{a}^+_n(s) |s|\,ds\leq \la \int_\R \check{a}^+(s) |s|\,ds<\infty,
\]
by \eqref{firstmoment}. As a result, if we set
\[
\la_1:=(\ka^+ \check{A}^+_{R_1}-m)\biggl(\ka^+\int_\R \check{a}^+(s) |s|\,ds+|G_\xi(\la_*)|\biggr)^{-1}>0,
\]
then, for any $\la\in (0,\la_1)$, we have
\[
  \la G_\xi^{(n)}(\la)\geq \ka^+ \check{A}^+_{R_1}-m-\la_1\ka^+\int_\R \check{a}^+_n(s) |s|\,ds
  = \la_1 |G_\xi(\la_*)|\geq \la G_\xi(\la_*),
\]
i.e. $G_\xi^{(n)}(\la)\geq G_\xi(\la_*)=\inf\limits_{\la>0}G_\xi(\la)$. Then, for any $n\geq N$, \eqref{wqwqwqwqwe} implies that $\la_*^{(n)}$, being the minimum point for $G_\xi^{(n)}$, does not belong to the interval $(0,\la_1)$. Next, let $N_1\geq N$ be such that $R_n\geq \frac{1}{\la_1}$, for all $n\geq N_1$. Then, for any $\la\geq\la_1$, and for any $n>m\geq N_1$, we have $R_n>R_m$ and
\begin{align*}
\T_\xi^{(n)}(\la)-\T_\xi^{(m)}(\la)&=\ka^+\int_{R_m}^{R_n}(1-\la s)\check{a}^+(s)e^{\la s}\,ds\\ &\leq
\ka^+\int_{R_m}^{R_n}(1-\la_1 s)\check{a}^+(s)e^{\la s}\,ds\leq 0.
\end{align*}

As a result, the sequence $\{\la_*^{(n)}\mid n\geq N_1\}\subset[\la_1,\infty)$ is monotonically decreasing (cf.~\eqref{derofT}). We set
\begin{equation}\label{barla}
  \vartheta:=\lim_{n\to\infty}\la_*^{(n)}\geq\la_1.
\end{equation}
Next, for any $n,m\in\N$, $n>m\geq N_1$,
\begin{equation}\label{doubleineqcoolmy}
  G^{(n)}_\xi(\la_*^{(n)})\geq G^{(m)}_\xi(\la_*^{(n)})\geq
G^{(m)}_\xi(\la_*^{(m)}),
\end{equation}
where we used that $G^{(n)}_\xi$ is increasing in $n$ and $\la_*^{(m)}$ is the minimum point of $G^{(m)}_\xi$. Therefore, the sequence $\{G^{(n)}_\xi(\la_*^{(n)})\}$ is increasing and, by \eqref{wqwqwqwqwe}, is bounded. Then, there exists
\begin{equation}\label{limitgmy}
  \lim\limits_{n\to\infty}G^{(n)}_\xi(\la_*^{(n)})=:g \leq G_\xi(\la_*)-\delta.
\end{equation}
Fix $m\geq N_1$ in \eqref{doubleineqcoolmy} and pass $n$ to infinity; then, by the continuity of $G_\xi^{(n)}$,
\begin{equation}\label{wt3udh}
  g\geq \lim\limits_{\la\to\vartheta+}G_\xi^{(m)}(\la)=G_\xi^{(m)}(\vartheta)\geq
  G_\xi^{(m)}(\la^{(m)}),
\end{equation}
in particular, $\vartheta>0$, as $G_\xi^{(m)}(0+)=\infty$. Next, if we pass $m$ to $\infty$ in \eqref{wt3udh}, we will get from \eqref{limitgmy}
\begin{equation}\label{contradictionmy}
  \lim_{m\to\infty}G_\xi^{(m)}(\vartheta)=g\leq G_\xi(\la_*)-\delta<G_\xi(\la_*).
\end{equation}
If $0<\vartheta\leq\la_0(\check{a}^+)$ then
\[
\lim_{m\to\infty}G_\xi^{(m)}(\vartheta)=G_\xi(\vartheta)\geq G_\xi(\la_*),
\]
that contradicts \eqref{contradictionmy}. If $\vartheta>\la_0(\check{a}^+)$, then
$\lim\limits_{m\to\infty}G_\xi^{(m)}(\vartheta)=\infty$ (recall again that $\L^-(\check{a}^+)(\la)$ is analytic and $\L^-(\check{a}^+)(\la)$ is monotone in $\la$), that contradicts \eqref{contradictionmy} as well.

The contradiction we obtained shows that \eqref{infsareequal} does hold. Then, for the chosen $c\geq c_*(\xi)$ which satisfies \eqref{hypoteticspeed}, one can find $n$ big enough to ensure that, cf.~\eqref{minspeedfortrunc},
\[
c<\inf\limits_{\la>0}G^{(n)}_\xi(\la)=c_*^{(n)}(\xi).
\]
However, as it was shown above, for this $n$ there exists a profile $\psi_n$ of a traveling wave to the `truncated' equation
\eqref{eq:basic_R} (with, recall, $d=1$ and $a^\pm_R$ replaced by $a^\pm_{R_n}$). The latter contradicts the statement of Theorem~\ref{thm:trwexists} applied to this equation, as $c_*^{(n)}(\xi)$ has to be a minimal possible speed for such waves.

Therefore, the strict inequality in \eqref{hypoteticspeed} is impossible, hence, we have equality in \eqref{ineqwillbeeq}. As a result, \eqref{expintla1} implies \eqref{minspeed}, and \eqref{dopequal} may be read as \eqref{minspeed-spec}. The rest of the statement is evident now.
\end{proof}

\begin{remark}\label{rem:evenimplypos}
  Clearly, the assumption $a^+(-x)=a^+(x)$, $x\in\X$, implies $\m_\xi=0$, for any $\xi\in\S $. As a result, all speeds of traveling waves in any directions are positive, by \eqref{minspeed}.
\end{remark}

\subsection{Uniqueness of traveling waves}
In this subsection we will prove the uniqueness (up to shifts) of a profile $\psi$ for a traveling wave with given speed $c\geq c_*(\xi)$, $c\neq0$.
We will use the almost traditional now approach, namely, we find an {\em a priori} asymptotic for $\psi(t)$, $t\to\infty$, cf. e.g. \cite{CC2004,AGT2012} and the references therein.

We start with the so-called characteristic function of the equation \eqref{eq:basic}. Namely, for a given $\xi\in\S $ and for any $c\in [c_*(\xi),\infty)$, we set
\begin{equation}\label{charfunc}
\h_{\xi,c}(z):= \ka^+(\L \check{a}^+)(z)-m-z c=zG_\xi(z)-zc, \qquad \Re z\in I_\xi .
\end{equation}
\begin{proposition}\label{prop:charfuncpropert}
Let $\xi\in\S $ be fixed, $a^+\in\Uxi$, $\la_0:=\la_0(\check{a}^+)$, $c_*(\xi)$ be the minimal traveling wave speed in the direction $\xi$. Let, for any $c\geq c_*(\xi)$, the function $\psi\in\M$ be a traveling wave profile corresponding to the speed~$c$. For the case $a^+\in\Wxi$ with $m=\T_\xi(\la_0)$, we will assume, additionally, that
\begin{equation}\label{secmomadd}
  \int_\R s^2 \check{a}^+(s)e^{\la_0 s}\,ds<\infty.
\end{equation}
Then the function $\h_{\xi,c}$ is analytic on $\{0<\Re z<\la_0(\psi)\}$. Moreover, for any $\beta\in(0,\la_0(\psi))$, the function $\h_{\xi,c}$ is continuous and does not equal to $0$ on the closed strip $\{\beta\leq\Re z\leq\la_0(\psi)\}$, except the root at $z=\la_0(\psi)$, whose multiplicity $j$ may be $1$ or $2$ only.
\end{proposition}
\begin{proof}
By \eqref{neweq} and the arguments around, $\h_{\xi,c}(z)=z (G_\xi(z) -c)$ is analytic on $\{0<\Re z<\la_0(\psi)\}\subset I_\xi$ and does not equal to $0$ there. Then, by \eqref{speedviaabs} and Proposition~\ref{prop:infGisreached}, the smallest positive root of the function $\h_{\xi,c}(\la)$ on $\R$ is exactly $\la_0(\psi)$. Prove that if $z_0:=\la_0(\psi)+i \beta$ is a root of $\h_{\xi,c}$, then $\beta=0$. Indeed,  $\h_{\xi,c}(z_0)=0$ yields
\[
\ka^+\int_\R \check{a}^+(s) e^{\la_0(\psi) s}\cos\beta s \,ds=m+c\la_0(\psi),
\]
that together with \eqref{speedviaabs} leads to
\[
\ka^+\int_\R \check{a}^+(s) e^{\la_0(\psi) s}(\cos\beta s -1)\,ds=0,
\]
and thus $\beta=0$.

Regarding multiplicity of the root $z=\la_0(\psi)$, we note that, by Proposition~\ref{prop:infGisreached} and Corollary~\ref{cor_alltogether}, there exist two possibilities. If $a^+\in\Vxi$, then $\la_0(\psi)\leq\la_*<\la_0(\check{a}^+)$ and, therefore, $G_\xi$ is analytic at $z=\la_0(\psi)$. By the second equality in \eqref{charfunc}, the multiplicity $j$ of this root for $\h_{\xi,c}$ is the same as for the function $G_\xi(z)-c$. By Proposition~\ref{prop:infGisreached}, $G_\xi$ is strictly decreasing on $(0,\la_*)$ and, therefore, $j=1$ for $c>c_*(\xi)$. By Corollary~\ref{cor_alltogether}, for $c=c_*(\xi)$, we have $G_\xi'(\la_0(\psi))=G_\xi'(\la_*)=0$ and, since $\h_{\xi,c}''(\la_0)>0$, one gets $j=2$.

Let now $a^+\in\Wxi$. Then, we recall, $\la_*=\la_0:=\la_0(\check{a}^+)<\infty$, $G_\xi(\la_0)<\infty$ and \eqref{leftlimG} hold.
For $c>c_*(\xi)$, the arguments are the same as before, and they yield $j=1$. Let $c=c_*(\xi)$. Then $\h_{\xi,c}(\la_0)=0$, and,  for all $z\in\mathbb{C}$, $\Re z\in(0,{\la_0})$, one has
\begin{align}
 \h_{\xi,c}({\la_0}-z)&=\h_{\xi,c}({\la_0}-z)-\h_{\xi,c}({\la_0})=\ka^+\int_\R \check{a}^+(\tau )(e^{({\la_0}-z)\tau }-e^{{\la_0} \tau })d\tau +cz \notag\\
    &=z\biggl(-\ka^+\int_\R \check{a}^+(\tau )e^{{\la_0} \tau } \int_0^\tau e^{-zs}\,ds d\tau +c
    \biggr).\label{eq:i}
\end{align}
Let $z=\alpha+\beta i$, $\alpha\in(0,\la_0)$. Then $\bigl\lvert e^{\la_0 \tau }  e^{-zs}\bigr\rvert=e^{\la_0\tau-\alpha s}$. Next, for $\tau\geq0$, $s\in[0,\tau]$, we have
$e^{\la_0\tau-\alpha s}\leq e^{\la_0\tau}$; whereas, for $\tau<0$, $s\in[\tau,0]$, one has $e^{\la_0\tau-\alpha s}=e^{\la_0(\tau-s)}e^{(\la_0-\alpha) s}\leq 1$.
As a result, $\bigl\lvert e^{\la_0 \tau }  e^{-zs}\bigr\rvert\leq e^{\la_0\max\{\tau,0\}}$.
Then, using that $a^+\in\Wxi$ implies $\int_\R\check{a}^+(\tau)e^{\la_0\max\{\tau,0\}}\,ds<\infty$, one can apply the dominated convergence theorem to the double integral in \eqref{eq:i}; we get then
\begin{equation}\label{impeq}
  \lim_{\substack{\Re z\to0+\\ \Im z\to0}}\frac{\h_{\xi,c}({\la_0}-z)}{z}
  =-\ka^+\int_\R \check{a}^+(\tau )e^{{\la_0} \tau } \tau d\tau +c.
\end{equation}
According to the statement 3 of Theorem~\ref{thm:speedandprofile}, for $m<\T_\xi(\la_0)$, the r.h.s. of \eqref{impeq} is positive, i.e. $j=1$ in such a case. Let now $m=\T_\xi(\la_0)$, then the r.h.s. of \eqref{impeq} is equal to $0$. It is easily seen that one can rewrite then \eqref{eq:i} as follows
\begin{align}\notag
  \frac{\h_{\xi,c}({\la_0}-z)}{z}&=\ka^+\int_\R \check{a}^+(\tau )e^{{\la_0} \tau } \int_0^\tau (1-e^{-zs})\,ds d\tau\\\label{eq:ii}
  &=z\ka^+\int_\R \check{a}^+(\tau )e^{{\la_0} \tau } \int_0^\tau \int_0^s e^{-zt}\,dt\,ds\, d\tau.
\end{align}
Similarly to the above, for $\Re z\in (0,\la_0)$, one has that
$\lvert e^{{\la_0} \tau -zt}\rvert \leq e^{{\la_0} \max\{\tau,0\} }$.
Then, by \eqref{secmomadd} and the dominated convergence theorem, we get from \eqref{eq:ii} that
\begin{equation*}
  \lim_{\substack{\Re z\to0+\\ \Im z\to0}}\frac{\h_{\xi,c}({\la_0}-z)}{z^2}
  =\frac{\ka^+}{2}\int_\R \check{a}^+(\tau )e^{{\la_0} \tau } \tau^2 d\tau\in(0,\infty).
\end{equation*}
Thus $j=2$ in such a case. The statement is fully proved now.
\end{proof}

\begin{remark}
Combining results of Theorem~\ref{thm:speedandprofile} and Proposition~\ref{prop:charfuncpropert}, we immediately get that, for the case $j=2$, the minimal traveling wave speed $c_*(\xi)$ always satisfies \eqref{minspeed-spec}.
\end{remark}

\begin{remark}\label{very-critical-case}
  If $\check{a}^+$ is given by \eqref{exofa}, then, cf.~Example~\ref{ex:spefunc},
  the case $a^+\in\Wxi$, $m=\T_\xi(\la_0)$ together with \eqref{secmomadd} requires $p=1$, $\mu<\mu_*$, $q>3$.
\end{remark}

We consider now the following Ikehara--Delange type Tauberian theorem, cf.~\cite{Ten1995,Del1954,Kab2008}.

For any $\mu>\beta>0$, $T>0$, we set
\[
K_{\beta,\mu,T}:=\bigl\{z\in\mathbb{C} \bigm| \beta\leq \Re z \leq \mu,\ \lvert \Im z \rvert \leq T\bigr\}.
\]
Let, for any $D\subset\mathbb{C}$, $\An(D)$ be the class of all analytic functions on $D$.
\begin{proposition}\label{prop:tauber}
Let $\mu>\beta>0$ be fixed. Let $\varphi\in C^1(\R_+\to  \R_+)$ be a non-increasing function such that, for some $a>0$, the function $\varphi(t)e^{(\mu+a) t}$ is non-decreasing, and the integral
\begin{equation}\label{eq:psi_onesided_lap}
\int\limits_0^{\infty}e^{z t}\varphi'(t)dt, \quad 0<\Re z<\mu
\end{equation}
converges. Suppose also that there exist a constant $j\in\{1,2\}$ and complex-valued  functions $H,F:\{0<\Re z\leq \mu\}\to\mathbb{C}$, such that $H\in\An(0<\Re z\leq \mu)$, $F\in\An (0<\Re z<\mu)\cap C(0<\Re z\leq \mu)$, and, for any $T>0$,
\begin{equation}\label{loglim}
  \lim_{\sigma\to0+}(\log\sigma)^{2-j}\sup_{\tau\in[-T,T]}\bigl\lvert F(\mu-2\sigma-i\tau)-F(\mu-\sigma-i\tau)\bigr\rvert=0,
\end{equation}
and also that the following representation holds
\begin{equation}\label{eq:sing_repres}
\int\limits_0^{\infty}e^{z t}\varphi(t)dt=\dfrac{F(z)}{(z-\mu)^j}+H(z),\quad 0<\Re z<\mu.
\end{equation}
Then $\varphi$ has the following asymptotic
\begin{equation}\label{eq:psi_asympt}
\varphi(t)\sim F(\mu) e^{-\mu t}t^{j-1},\quad t\to +\infty.
\end{equation}
\end{proposition}

The proof of Proposition~\ref{prop:tauber} is based on the following Tenenbaum's result.
\begin{lemma}[{``Effective'' Ikehara--Ingham Theorem, cf.~{\cite[Theorem 7.5.11]{Ten1995}}}]\label{lem:EII}
Let $\alpha(t)$ be a non-decreasing function such that, for some fixed $a>0$, the following integral converges:
\begin{equation}\label{eq:onesided_laplace}
\int\limits_0^{\infty}e^{-z t}d\alpha(t), \quad \Re z>a.
\end{equation}
Let also there exist constants $D\geq0$ and $j>0$, such that for the functions
\begin{align}\label{eq:G}
G(z)&:=\dfrac{1}{a+z}\int\limits_0^\infty e^{-(a+z)t}d\alpha(t)-\dfrac{D}{z^{j}},\quad \Re z>0,\\
\eta(\sigma,T)&:=\sigma^{j-1}\int\limits_{-T}^{T}\bigl\lvert G(2\sigma+i\tau)-G(\sigma+i\tau)\bigr\rvert d\tau, \quad T>0,\notag
\end{align}
one has that
\begin{equation}\label{eq:eta}
\lim_{\sigma\to0+}\eta(\sigma,T)=0, \quad T>0.
\end{equation}
Then
\begin{equation}\label{eq:alpha}
\alpha(t)=\left\{ \dfrac{D}{\Gamma(j)}+O\bigl(\rho(t)\bigr)\right\}  e^{at}t^{j-1},\quad t\geq1,
\end{equation}
where
\begin{equation}\label{eq:rho}
\rho(t):=\inf\limits_{T\geq32(a+1)} \Bigl\{ T^{-1}+\eta\bigl(t^{-1},T\bigr)+(Tt)^{-j} \Bigr\}.
\end{equation}
\end{lemma}
\begin{proof}[Proof of Proposition~\ref{prop:tauber}]
We first express $\int_{0}^{\infty}e^{\la t}\varphi(t)dt$ in the form \eqref{eq:onesided_laplace}. By the assumption, the function $\alpha(t):= e^{(\mu+a)t}\varphi(t)$ is non-decreasing. For any $0<\Re z<\mu$, one has
\begin{equation}\label{eq:psi_iprime}
\int\limits_0^\infty e^{-(a+z)t}d\alpha(t)=(\mu+a)\int\limits_0^{\infty}e^{(\mu-z)t}\varphi(t)dt +\int\limits_0^\infty e^{(\mu-z)t}\varphi'(t)dt,
\end{equation}
and the r.h.s. of \eqref{eq:psi_iprime} converges, by \eqref{eq:psi_onesided_lap} and \ref{L-derivative}. Then, by \cite[Corollary~II.1.1a]{Wid1941}, the l.h.s. of \eqref{eq:psi_iprime} converges, for $\Re z>0$, and hence, by \cite[Theorem II.2.3a]{Wid1941}, one gets
\begin{equation}\label{eq:psi_i}
\int\limits_0^\infty e^{-(a+z)t}d\alpha(t)=-\varphi(0)+(a+z)\int\limits_0^\infty e^{(\mu-z)t}\varphi(t)dt.
\end{equation}
Therefore, by \eqref{eq:sing_repres} and \eqref{eq:psi_i}, we have
\begin{equation*}
\dfrac{1}{a+z}\int\limits_0^\infty e^{-(a+z)t}d\alpha(t)=\dfrac{F(\mu-z)}{z^j}+K(z),\quad 0<\Re z<\mu,
\end{equation*}
where $K(z):=H(\mu-z)-\dfrac{\varphi(0)}{a+z}$, $0< \Re z\leq\mu$.

Let now $G$ be given by \eqref{eq:G} with $\alpha(t)$ as above and $D:=F(\mu)$.
Check the condition \eqref{eq:eta}; one can assume, clearly, that $0<\sigma<2\sigma<\mu$. Since $K\in\An(0<\Re z\leq \mu)$, one easily gets that
\begin{align}
&\quad \lim\limits_{\sigma\to  0+} \sigma^{j-1}\int\limits_{-T}^{T} \bigl\lvert G(2\sigma+i\tau)-G(\sigma+i\tau)\bigr\rvert d\tau \notag \\
 &\leq \lim\limits_{\sigma\to  0+}\sigma^{j-1}\int\limits_{-T}^{T}\Bigl\lvert \dfrac{F(\mu-2\sigma-i\tau)-F(\mu)}{(2\sigma+i\tau)^j}-\dfrac{F(\mu-\sigma-i\tau)-F(\mu)}{(\sigma+i\tau)^j}\Bigr\rvert d\tau\notag\\
 &\leq \lim\limits_{\sigma\to  0+}\sigma^{j-1}\int\limits_{-T}^{T}\Bigl\lvert \dfrac{F(\mu-2\sigma-i\tau)-F(\mu-\sigma-i\tau)}{(\sigma+i\tau)^j}\Bigr\rvert d\tau\notag\\
 &\quad+\lim\limits_{\sigma\to  0+}\sigma^{j-1}\int\limits_{-T}^{T}\bigl\lvert F(\mu-2\sigma-i\tau)-F(\mu)\bigr\rvert \Bigl\lvert \dfrac{1}{(2\sigma+i\tau)^j}-\dfrac{1}{(\sigma+i\tau)^j}\Bigr\rvert d\tau\notag,\\
 &=: \lim\limits_{\sigma\to  0+} A_j(\sigma)+\lim\limits_{\sigma\to  0+} B_j(\sigma).\label{ABexpans}
 \end{align}
One has
 \begin{equation*}
    A_j(\sigma) \leq \sup_{\tau\in[-T,T]}\bigl\lvert F(\mu-2\sigma-i\tau)-F(\mu-\sigma-i\tau)\bigr\rvert \sigma^{j-1}\int_{-T}^T\frac{1}{|\sigma+i\tau|^j}d\tau,
 \end{equation*}
and since
\[
 \sigma^{j-1}\int_{-T}^T\frac{1}{|\sigma+i\tau|^j}d\tau=
 \sigma^{j-1}\int_{-T}^T\frac{1}{(\sigma^2+\tau^2)^\frac{j}{2}}d\tau=\begin{cases}
     2\log \dfrac{\sqrt{T^2+\sigma^2}+T}{\sigma}, &j=1,\\[2mm]
     2\arctan\dfrac{T}{\sigma}, &j=2,
   \end{cases}
\]
we get, by \eqref{loglim}, that $\lim\limits_{\sigma\to  0+} A_j(\sigma) =0$.

 Next, since $F\in C(K_{\beta,\mu,T})$, there exists $C_1>0$ such that $|F(z)|\leq C_1$, $z\in K_{\beta,\mu,T}$. Therefore,
 \begin{align}
   B_j(\sigma) & \leq \sigma^{j-1}\sup_{|\tau|\leq\sqrt{\sigma}}\bigl\lvert F(\mu-2\sigma-i\tau)-F(\mu)\bigr\rvert  \int\limits_{|\tau|\leq\sqrt{\sigma}} \Bigl\lvert\dfrac{1}{(2\sigma+i\tau)^j}-\dfrac{1}{(\sigma+i\tau)^j}\Bigr\rvert d\tau \notag\\
    & \quad+ 2C_1 \sigma^{j-1}\int\limits_{\sqrt{\sigma}\leq|\tau|\leq T}\Bigl\lvert \dfrac{1}{(2\sigma+i\tau)^j}-\dfrac{1}{(\sigma+i\tau)^j}\Bigr\rvert d\tau.\label{Bsigmaest}
 \end{align}
Note that, for any $a<b$,
\begin{align*}
\int\limits_a^b\left\vert \dfrac{1}{2\sigma+i\tau}-\dfrac{1}{\sigma+i\tau} \right\vert d\tau &=\int\limits_a^b \dfrac{\sigma}{\sqrt{(2\sigma^2-\tau^2)^2+9\sigma^2\tau^2}}d\tau=
\int\limits_{\frac{a}{\sigma}}^{\frac{b}{\sigma}}h_1(x)dx;\\
\int\limits_a^b\left\vert \dfrac{1}{(2\sigma+i\tau)^2}-\dfrac{1}{(\sigma+i\tau)^2} \right\vert d\tau&=\sigma\int\limits_a^b \dfrac{\sqrt{9\sigma^2+4\tau^2}}{(2\sigma^2-\tau^2)^2+9\sigma^2\tau^2}d\tau=
\dfrac{1}{\sigma}\int\limits_{\frac{a}{\sigma}}^{\frac{b}{\sigma}}
h_2(x)dx,
\end{align*}
where
\[
h_1(x):=\dfrac{1}{\sqrt{4+5x^2+x^4}}, \qquad h_2(x):=\dfrac{\sqrt{9+4x^2}}{4+5x^2+x^4}.
\]
Now, one can estimate terms in \eqref{Bsigmaest} separately. We have
\[
\sigma^{j-1} \int\limits_{|\tau|\leq\sqrt{\sigma}} \Bigl\lvert\dfrac{1}{(2\sigma+i\tau)^j}-\dfrac{1}{(\sigma+i\tau)^j}\Bigr\rvert d\tau
=\int\limits_{|x|\leq\frac{\sqrt{\sigma}}{\sigma}} h_j(x)dx<\int_{\R} h_j(x)dx<\infty.
\]
Next, since $F$ is uniformly continuous on $K_{\beta,\mu,T}$, we have that, for any $\eps>0$ there exists $\delta>0$ such that
$f(\mu,\sigma,\tau):=\bigl\lvert F(\mu-2\sigma-i\tau)-F(\mu)\bigr\rvert<\eps$,
if only $4\sigma^2+\tau^2<\delta$. Therefore, if $\sigma>0$ is such that $4\sigma^2+\sigma<\delta$ then $\sup_{|\tau|\leq\sqrt{\sigma}}f(\mu,\sigma,\tau)<\eps$ hence
\[
\sup_{|\tau|\leq\sqrt{\sigma}}\bigl\lvert F(\mu-2\sigma-i\tau)-F(\mu)\bigr\rvert\to 0, \quad \sigma\to0+.
\]
Finally,
\[
  \sigma^{j-1}\int\limits_{\sqrt{\sigma}\leq|\tau|\leq T}\Bigl\lvert \dfrac{1}{(2\sigma+i\tau)^j}-\dfrac{1}{(\sigma+i\tau)^j}\Bigr\rvert  =2\int\limits_{\frac{\sqrt{\sigma}}{\sigma}}^{\frac{T}{\sigma}}h_j(x)dx\to 0, \quad \sigma\to0+,
\]
as $\int_\R h_j(x)dx<\infty$. As a result, \eqref{Bsigmaest} gives $ B_j(\sigma)\to 0$, as $\sigma\to0+$. Combining this with $A_j(\sigma)\to 0$, one gets \eqref{eq:eta} from \eqref{ABexpans}; and we can apply Lemma~\ref{lem:EII}. Namely, by \eqref{eq:alpha},
there exist $C>0$ and $t_0\geq1$, such that
\[
D e^{at}t^{j-1}\leq \varphi(t)e^{(\mu+a)t} \leq \left\{ D+C\rho(t) \right\}e^{at}t^{j-1},\quad t\geq t_0.
\]
as $\Gamma(j)=1$, for $j\in\{1,2\}$.
By \eqref{eq:eta}, \eqref{eq:rho} $\rho(t)\to  0$ as $t\to \infty$. Therefore,
\[
\varphi(t)e^{(\mu+a)t}\sim D e^{at}t^{j-1},\quad t\to  \infty,
\]
that is equivalent to \eqref{eq:psi_asympt} and finishes the proof.
\end{proof}

To apply Proposition~\ref{prop:tauber} to our settings, we will need the following statement, which is an adaptation of \cite[Lemma 3.2, Proposition 3.7]{ZLW2012}.

\begin{proposition}\label{beincreas}
Let $\xi\in\S $ be fixed, $a^+\in\Uxi$, $c_*(\xi)$ be the minimal traveling wave speed in the direction $\xi$. Let a traveling wave profile $\psi\in\M$ correspond to a speed $c\geq c_*(\xi)$, $c\neq0$. Then there exists $\nu>0$, such that $\psi(t)e^{\nu t}$ is a monotonically increasing function.
\end{proposition}
\begin{proof}
We start from the case $c>0$. Since $\psi(t)>0,\ t\in\R$, it is sufficient to prove that
  \begin{equation}\label{weneed}
    \frac{\psi'(t)}{\psi(t)}> -\nu,\quad t\in\R.
  \end{equation}
  Fix any $\mu\geq\dfrac{\ka^+}{c}>0$. Then, clearly,
  \[
  \ka^-(\check{a}^-*\psi)(t)+m\leq \ka^-\theta+m=\ka^+\leq c\mu,
  \]
  and we will get from \eqref{eq:trwder}, that
  \begin{equation}
0\geq c\psi'(s)+\ka^{+}(\check{a}^{+}*\psi)(s)-c\mu\psi(s), \quad s\in\R.
\label{ineq1}
\end{equation}
Multiply both parts of \eqref{ineq1} on $e^{-\mu s}>0$ and set
\[
w(s):=\psi(s)e^{-\mu s}>0, \quad s\in\R.
\]
Then $w'(s)=\psi'(s)e^{-\mu s}-\mu w(s)$ and one can rewrite \eqref{ineq1} as follows
  \begin{align}
0&\geq c w'(s)+\ka^{+}(\check{a}^{+}*\psi)(s)e^{-\mu s}\notag\\&=c w'(s)+\ka^+\int_\R\check{a}^+(\tau)w(s-\tau)e^{-\mu \tau}d\tau, \quad s\in\R.
\label{eq:w_est}
\end{align}

As it was shown in the proof of Proposition~\ref{prop:infGisreached}, \eqref{as:a+nodeg-mod} implies that there exists $\varrho>0$, such that
\begin{equation}\label{intfrom2rhoispos}
\int_{2\varrho}^{\infty}\check{a}^+(s)e^{-\mu s}ds>0;
\end{equation}
indeed, it is enough to set $2\varrho:=r+\frac{\delta'}{2}$ in \eqref{checkapluspos}.

Integrate \eqref{eq:w_est} over $s\in[t,t+\varrho]$; one gets
  \begin{equation}\label{asd3}
  0 \geq c(w(t+\varrho)-w(t))+\ka^+\int_{t}^{t+\varrho}\int_\R\check{a}^+(\tau)w(s-\tau)e^{-\mu \tau}d\tau ds.
\end{equation}
Since $w(t)$ is a monotonically decreasing function, we have
  \begin{align}
     \int_{t}^{t+\varrho}\int_\R\check{a}^+(\tau)w(s-\tau)e^{-\mu \tau}d\tau ds     &\geq \varrho\int_\R\check{a}^+(\tau)w(t+\varrho-\tau)e^{-\mu \tau}d\tau \nonumber \\
      \geq \varrho \int_{2\varrho}^{\infty}\check{a}^+(\tau)w(t+\varrho-\tau)e^{-\mu \tau}d\tau &\geq \varrho w(t-\varrho)\int_{2\varrho}^{\infty}\check{a}^+(\tau)e^{-\mu \tau}d\tau.\label{asd4}
  \end{align}
We set, cf. \eqref{intfrom2rhoispos},
\[
C(\mu,\rho):=\dfrac{\ka^+}{c}\int_{2\varrho}^{\infty}\check{a}^+(s)e^{-\mu s}ds>0.
\]
Then \eqref{asd3} and \eqref{asd4} yield
  \begin{equation}\label{eq:ln_fy_est_i}
    w(t) -\varrho C(\mu,\rho)w(t-\varrho)\geq w(t+\varrho)>0,\quad t\in\R.
  \end{equation}

  Now we integrate \eqref{eq:w_est} over $s\in[t-\varrho,t]$. Similarly to above, one gets
  \begin{align}
    0 &\geq c(w(t)-w(t-\varrho))+\ka^+\int_{t-\varrho}^{t}\int_\R\check{a}^+(\tau)w(s-\tau)e^{-\mu \tau}d\tau ds \nonumber \\
      &\geq c(w(t)-w(t-\varrho))+\varrho\ka^+\int_\R\check{a}^+(\tau)w(t-\tau)e^{-\mu \tau}d\tau.\label{asda12}
  \end{align}
  By \eqref{eq:ln_fy_est_i} and \eqref{asda12}, we have
  \begin{equation}\label{eq:ln_fy_est_ii}
  \frac{1}{\varrho C(\mu,\rho)}\geq \dfrac{ w(t-\varrho)}{w(t)}\geq 1 +\dfrac{\varrho\ka^+}{c}\int_\R\check{a}^+(\tau)\dfrac{w(t-\tau)}{w(t)}e^{-\mu \tau}d\tau.
  \end{equation}
  On the other hand, \eqref{eq:trwder} implies that
  \begin{equation}
-\frac{\psi'(t)}{\psi(t)}\leq \frac{\ka^{+}}{c}\frac{(\check{a}^{+}*\psi)(t)}{\psi(t)}=\dfrac{\ka^+}{c}\int_\R\check{a}^+(\tau)\dfrac{w(t-\tau)}{w(t)}e^{-\mu \tau}d\tau, \quad t\in\R.\label{eq:ln_fy_est_ii11}
\end{equation}
Finally, \eqref{eq:ln_fy_est_ii} and \eqref{eq:ln_fy_est_ii11} yield \eqref{weneed} with $\nu=\dfrac{1}{\rho^2 C(\mu,\rho)}>0$.

Let now $c<0$. For any $\nu\in\R$, one has
\begin{equation*}
\psi'(s)=e^{-\nu s}(\psi(s)e^{\nu s})'-\nu \psi(s),\quad s\in\R.
\end{equation*}
Hence, by \eqref{eq:trw}, \eqref{as:aplus_gr_aminus},
\begin{align*}
0&= ce^{-\nu s}(\psi(s)e^{\nu s})'-c\nu\psi(s)+\chi^+(\check{a}^+*\psi)(s)-\chi^-\psi(s)(\check{a}^-*\psi)(s)-m\psi(s)\nonumber \\
&\geq ce^{-\nu s}(\psi(s)e^{\nu s})'-c\nu\psi(s)+\chi^+(\check{a}^+*\psi)(s)-\chi^-\theta(\check{a}^-*\psi)(s)-m\psi(s)\nonumber \\
&\geq ce^{-\nu s}(\psi(s)e^{\nu s})'-c\nu\psi(s)-m\psi(s),\quad s\in\R.
\end{align*}
As a result, choosing $\nu>\dfrac{m}{-c}$, one gets
\begin{equation*}
-c e^{-\nu s}(\psi(s)e^{\nu s})'\geq(-c\nu-m)\psi(s)>0,\quad s\in\R,
\end{equation*}
i.e. $\psi(s)e^{\nu s}$ is an increasing function.
\end{proof}

Now, we can apply Proposition~\ref{prop:tauber} to find the asymptotic of the profile of a traveling wave.
\begin{proposition}\label{asymptexact}
In conditions and notations of Proposition~\ref{prop:charfuncpropert},
for $c\neq0$, there exists $D=D_j>0$, such that
\begin{equation}\label{eq:trw_asympt}
\psi(t)\sim D e^{-\la_0(\psi) t}t^{j-1},\quad t\to  \infty.
\end{equation}
\end{proposition}
\begin{proof}
We set $\mu:=\la_0(\psi)$ and
\begin{equation}\label{deffunc}
\begin{aligned}
f(z)&:=\ka^-\bigl(\L(\psi (\check{a}^-*\psi))\bigr)(z), & g_j(z)&:=\dfrac{\h_{\xi,c}(z)}{(z-\mu)^j},\\
H(z)&:=-\int\limits_{-\infty}^0\psi(t)e^{z t}dt, & F(z)&:=\dfrac{f(z)}{g_j(z)}.
\end{aligned}
\end{equation}
By \eqref{unexpineq} and Lemma~\ref{lem:allaboutLapl},  we have that $f, H\in\An(0<\Re z\leq \mu)$; in particular, for any $T>0$, $\beta>0$,
\begin{equation}\label{fbdd}
\bar f:=\sup_{z\in K_{\beta,\mu,T}}|f(z)|<\infty.
\end{equation}
By Proposition~\ref{prop:charfuncpropert}, the function $g_j$ is continuous and does not equal to $0$ on the strip $\{0<\Re z\leq\mu\}$, in particular, for any $T>0$, $\beta>0$,
\begin{equation}\label{gpos}
\bar g_j:=\inf_{z\in K_{\beta,\mu,T}}|g(z)|>0.
\end{equation}
Therefore, $F\in\An (0<\Re z<\mu)\cap C(0< \Re z \leq \mu)$. As a result, one can rewrite \eqref{rewriting} in the form \eqref{eq:sing_repres}, with $\varphi=\psi$ and with $F$, $H$ as in \eqref{deffunc}.

Taking into account Proposition~\ref{beincreas}, to apply Proposition~\ref{prop:tauber} it is enough to prove that \eqref{loglim} holds. Assume that $0<2\sigma<\mu$.

Let $j=2$. Clearly, $F\in C(0< \Re z \leq \mu)$ implies that $F$ is uniformly continuous
on $ K_{\beta,\mu,T}$. Then, for any $\eps>0$ there exists $\delta>0$ such that, for any $\tau\in[-T,T]$, the inequality
\[
|\sigma|=|(\mu-2\sigma-i\tau)-(\mu-\sigma-i\tau)|<\delta
\]
implies
\[
|F(\mu-2\sigma-i\tau)-F(\mu-\sigma-i\tau)|<\eps,
\]
and hence \eqref{loglim} holds (with $j=2$).

Let now $j=1$. If $F\in \An(K_{\beta,\mu,T})$, we have, evidently, that $F'$ is bounded on $K_{\beta,\mu,T}$, and one can apply a mean-value-type theorem for complex-valued functions, see e.g. \cite{EJ1992}, to get that $F$ is a Lipschitz function on $K_{\beta,\mu,T}$. Therefore, for some $K>0$,
\[
|F(\mu-2\sigma-i\tau)-F(\mu-\sigma-i\tau)|<K |\sigma|,
\]
for all $\tau\in[-T,T]$, that yields \eqref{loglim} (with $j=1$). By~Proposition~\ref{prop:infGisreached} and Corollary~\ref{cor_alltogether}, the inclusion $F\in \An(K_{\beta,\mu,T})$ always holds for $c>c_*$; whereas, for $c=c_*$ it does hold iff $a^+\in\Vxi$. Moreover, the case $a^+\in\Wxi$ with $m=\T_\xi(\la_0)$ and $c=c_*$ implies, by Proposition~\ref{prop:charfuncpropert}, $j=2$ and hence it was considered above.

Therefore, it remains to prove \eqref{loglim} for the case $a^+\in\Wxi$ with $m<\T_\xi(\la_0)$, $c=c_*$ (then $j=1$).
Denote, for simplicity,
\begin{equation}\label{zi}
z_1:=\mu-\sigma-i\tau, \qquad z_2:=\mu-2\sigma-i\tau.
\end{equation}
Then, by \eqref{deffunc}, \eqref{fbdd}, \eqref{gpos}, one has
\begin{align}
  \bigl\lvert F(z_2)-F(z_1)\bigr\rvert &\leq
  \Bigl\lvert \frac{f(z_2)}{g_1(z_2)}-
  \frac{f(z_1)}{g_1(z_2)}\Bigr\rvert+\Bigl\lvert
  \frac{f(z_1)}{g_1(z_2)}
  -\frac{f(z_1)}{g_1(z_1)}\Bigr\rvert\notag\\
  &\leq \frac{1}{\bar g_1}\bigl\lvert f(z_2)-
  f(z_1)\bigr\rvert+ \frac{\bar f}{\bar g_1^2 }|g_1(z_1)-g_1(z_2)|.\label{fgest}
\end{align}
Note that, if $0<\phi\in L^\infty(\R)\cap L^1(\R)$ be such that $\la_0(\phi)>\mu$ then
\begin{align}
&  \bigl\lvert (\L\phi)(z_2)-
  (\L\phi)(z_1)\bigr\rvert  \leq \int_\R \phi(s)e^{\mu s}|e^{-2\sigma s}-e^{-\sigma s}|ds\notag\\
   &\quad \leq\sigma \int_0^\infty \phi(s)e^{(\mu-\sigma) s}sds+
   \sigma \int_{-\infty}^0 \phi(s)e^{(\mu -2\sigma)s} |s|\,ds=O(\sigma), \label{genconv}
\end{align}
as $\sigma\to0+$, where we used that $\sup_{s<0}e^{(\mu -2\sigma)s} |s|<\infty$,  $0<2\sigma<\mu$, and that \ref{L-analytic} holds.
Applying \eqref{genconv} to $\phi=\psi (\check{a}^-*\psi)\leq \theta^2 \check{a}^-\in L^1(\R)\cap L^\infty(\R)$, one gets
\[
  \sup_{\tau\in[-T,T]}\bigl\lvert f(z_2)-
  f(z_1)\bigr\rvert  = O(\sigma), \quad \sigma\to0+.
\]
Therefore, by \eqref{fgest}, it remains to show that
\begin{equation}\label{enough}
  \lim_{\sigma\to0+} \log \sigma \sup_{\tau\in[-T,T]} |g_1(z_1)-g_1(z_2)|=0.
\end{equation}
Recall that, in the considered case $c=c_*$, one has $\h_{\xi,c}(\mu)=0$. Therefore, by \eqref{charfunc}, \eqref{deffunc}, \eqref{zi}, we have
\begin{align}
&\quad|g_1(z_1)-g_1(z_2)|=\biggl\lvert \frac{\h_{\xi,c}(z_1)-\h_{\xi,c}(\mu)}{z_1-\mu}
-\frac{\h_{\xi,c}(z_2)-\h_{\xi,c}(\mu)}{z_2-\mu}\biggr\rvert\notag\\&=
\biggl\lvert \frac{\ka^+(\L\check{a}^+)(z_1)-\ka^+(\L\check{a}^+)(\mu)}{z_1-\mu}
-\frac{\ka^+(\L\check{a}^+)(z_2)-\ka^+(\L\check{a}^+)(\mu)}{z_2-\mu}\biggr\rvert\notag\\
&\leq \ka^+\int_\R \check{a}^+(s) e^{\mu s}
\biggl\lvert \frac{1-e^{(-\sigma-i\tau)s}}{\sigma+i\tau}
-\frac{1-e^{(-2\sigma-i\tau)s}}{2\sigma+i\tau}\biggr\rvert\,ds\notag\\
&= \ka^+\int_\R \check{a}^+(s) e^{\mu s}\biggl\lvert\int_0^s \bigl( e^{(-\sigma-i\tau)t}-
e^{(-2\sigma-i\tau)t}\bigr)\,dt\biggr\rvert\,ds\notag\\&\leq
\ka^+\int_0^\infty \check{a}^+(s) e^{\mu s}\int_0^s \bigl| e^{-\sigma t}-
e^{-2\sigma t}\bigr|\,dt\,ds\notag
\\&\quad+\ka^+\int_{-\infty}^0 \check{a}^+(s) e^{\mu s}\int_s^0 \bigl| e^{-\sigma t}-
e^{-2\sigma t}\bigr|\,dt\,ds\label{forcont}\\
\intertext{and since, for $t\geq0$, $\bigl| e^{-\sigma t}-
e^{-2\sigma t}\bigr|\leq \sigma t$; and, for $s\leq t\leq0$,
\[
\bigl| e^{-\sigma t}- e^{-2\sigma t}\bigr|=e^{-2\sigma t}\bigl| e^{\sigma t}-
1\bigr|\leq e^{-2\sigma s} \sigma |t|,
\]
one can continue \eqref{forcont}}
&\leq \frac{1}{2}\sigma\ka^+\int_0^\infty \check{a}^+(s) e^{\mu s}s^2\,ds
+\frac{1}{2}\sigma\ka^+\int_{-\infty}^0 \check{a}^+(s) e^{(\mu-2\sigma) s}s^2\,ds.\notag
\end{align}
Since $\mu>2\sigma$, one has $\sup\limits_{s\leq0}e^{(\mu-2\sigma) s}s^2<\infty$, therefore, by \eqref{secmomadd}, one gets
\[
\sup_{\tau\in[-T,T]}|g_1(z_1)-g_1(z_2)|\leq\mathrm{const}\cdot\sigma,
\]
that proves \eqref{enough}. The statement is fully proved now.
\end{proof}

\begin{remark}\label{rem:shifting}
  By \eqref{eq:psi_asympt} and \eqref{deffunc}, one has that the constant $D=D_j$ in \eqref{eq:trw_asympt} is given by
  \[
  D=D(\psi)=\ka^-\bigl(\L(\psi (\check{a}^-*\psi))\bigr)(\mu) \lim_{z\to\mu}\dfrac{(z-\mu)^j}{\h_{\xi,c}(z)},
  \]
  where $\mu=\la_0(\psi)$. Note that, by Proposition~\ref{prop:charfuncpropert}, the limit above is finite and does not depend on $\psi$. Next, by Remark~\ref{shiftoftrw}, for any $q\in\R$, $\psi_q(s):=\psi(s+q)$, $s\in\R$ is a traveling wave with the same speed, and hence, by Theorem~\ref{thm:speedandprofile}, $\la_0(\psi_q)=\la_0(\psi)$. Moreover,
  \begin{align*}
  \bigl(\L(\psi_q (\check{a}^-*\psi_q))\bigr)(\mu)&=\int_\R \psi(s+q)\int_\R \check{a}^-(t)\psi(s-t+q)\,dt\, e^{\mu s}\,ds\\&=e^{-\mu q}\bigl(\L(\psi (\check{a}^-*\psi))\bigr)(\mu).
  \end{align*}
  Thus, for a traveling wave profile $\psi$ one can always choose a $q\in\R$ such that, for the shifted profile $\psi_q$, the corresponding $D=D(\psi_q)$ will be equal to $1$.
\end{remark}

Finally, we are ready to prove the uniqueness result.

\begin{theorem}\label{thm:tr_w_uniq}
Let $\xi\in\S $ be fixed and $a^+\in\Uxi$. Suppose, additionally, that \eqref{as:aplus-aminus-is-pos} holds. Let $c_*(\xi)$ be the minimal traveling wave speed according to Theorem~\ref{thm:trwexists}.
For the case $a^+\in\Wxi$ with $m=\T_\xi(\la_0)$, we will assume, additionally, that
\eqref{secmomadd} holds.
Then, for any $c\geq c_*$, such that $c\neq0$, there exists a unique, up~to~a~shift, traveling wave profile $\psi$ for \eqref{eq:basic}.
\end{theorem}
\begin{proof}
  We will follow the sliding technique from \cite{CDM2008}. Let $\psi_1,\psi_2\in C^1(\R)\cap\M$ are traveling wave profiles with a speed $c\geq c_*$, $c\neq0$, cf.~Proposition~\ref{prop:reg_trw}. By~Proposition~\ref{asymptexact} and Remark~\ref{rem:shifting}, we may assume, without lost of generality, that \eqref{eq:trw_asympt} holds for both $\psi_1$ and $\psi_2$ with $D=1$. By~the~proof of Proposition~\ref{prop:charfuncpropert}, the corresponding $j\in\{1,2\}$ depends on $a^\pm$, $\ka^\pm$, $m$ only, and does not depend on the choice of $\psi_1$, $\psi_2$. By~Theorem~\ref{thm:speedandprofile}, $\la_0(\psi_1)=\la_0(\psi_2)=:\la_c\in (0,\infty)$.

\smallskip

  \textit{Step 1.} Prove that, for any $\tau>0$, there exists $T=T(\tau)>0$, such that
  \begin{equation}\label{ineqoutoftau}
    \psi^\tau_1(s):=\psi_1(s-\tau)>\psi_2(s), \quad s\geq T.
  \end{equation}

  Indeed, take an arbitrary $\tau>0$. Then \eqref{eq:trw_asympt} with $D=1$ yields
  \[
    \lim_{s\to\infty}\frac{\psi_1^\tau(s)}{(s-\tau)^{j-1}e^{-\la_c(s-\tau)}}=1=
    \lim_{s\to\infty}\frac{\psi_2(s)}{s^{j-1}e^{-\la_c s}}.
  \]
  Then, for any $\eps>0$, there exists $T_1=T_1(\eps)>\tau$, such that, for any $s>T_1$,
  \[
  \frac{\psi_1^\tau(s)}{(s-\tau)^{j-1}e^{-\la_c(s-\tau)}}-1>-\eps, \qquad
  \frac{\psi_2(s)}{s^{j-1}e^{-\la_c s}}-1<\eps.
  \]
  As a result, for $s>T_1>\tau$,
  \begin{align}
  &\quad \psi_1^\tau(s)-\psi_2(s)>(1-\eps)(s-\tau)^{j-1}e^{-\la_c(s-\tau)} -(1+\eps)s^{j-1}e^{-\la_c s}\notag\\
  &=s^{j-1}e^{-\la_c s}\biggl( \Bigl(1-\frac{\tau}{s}\Bigr)^{j-1}e^{\la_c \tau}-1 -\eps\Bigl(\Bigl(1-\frac{\tau}{s}\Bigr)^{j-1}e^{\la_c \tau} +1\Bigr)\biggr)\notag\\
  &\geq s^{j-1}e^{-\la_c s}\biggl( \Bigl(1-\frac{\tau}{T_1}\Bigr)^{j-1}e^{\la_c \tau}-1 -\eps\bigl(e^{\la_c \tau} +1\bigr)\biggr)>0  \label{poswillbe}
  \end{align}
  if only
  \begin{equation}\label{smallepsmy}
      0<\eps<\frac{\Bigl(1-\dfrac{\tau}{T_1}\Bigr)^{j-1}e^{\la_c \tau}-1 }{e^{\la_c \tau} +1}=:g(\tau,T_1).
  \end{equation}

  For $j=1$, the nominator in the r.h.s. of \eqref{smalleps} is positive. For $j=2$, consider $f(t):= \bigl(1-\frac{t}{T_1}\bigr)e^{\la_c t}-1$, $t\geq0$. Then
 $f'(t)=\frac{1}{T_1}e^{\la_c t}(\la_c T_1-\la_c t-1)>0$, if~only $T_1>t+\frac{1}{\la_c}$, that implies $f(t)>f(0)=0$, $t\in\bigl(0,T_1-\frac{1}{\la_c}\bigr)$.

 As a result, choose $\eps=\eps(\tau)>0$ with $\eps< g\bigl(\tau,\tau+\frac{1}{\la_c}\bigr)$, then, without loss of generality, suppose that $T_1=T_1(\eps)=T_1(\tau)>\tau+\frac{1}{\la_c}>\tau$. Therefore, $0<\eps<g\bigl(\tau,\tau+\frac{1}{\la_c}\bigr)\leq g(\tau,T_1)$, that fulfills \eqref{smallepsmy}, and hence
\eqref{poswillbe} yields \eqref{ineqoutoftau}, with any $T>T_1$.

\smallskip

\textit{Step 2.} Prove that there exists $\nu >0$, such that, cf.~\eqref{ineqoutoftau},
  \begin{equation}\label{ineqeverywhere}
    \psi^\nu _1(s) \geq \psi_2(s), \quad s\in\R.
  \end{equation}

  Let $\tau>0$ be arbitrary and $T=T(\tau)$ be as above.
  Choose any $\delta\in\bigl(0,\frac{\theta}{4}\bigr)$.
  By \eqref{cleareq}, \eqref{trwv}, and the dominated convergence theorem,
\begin{equation}\label{convtotthetaconv}
    \lim\limits_{s\to-\infty} (\check{a}^-*\psi_2)(s)=\lim\limits_{s\to-\infty} \int_\R \check{a}^-(\tau)\psi_2(s-\tau)\,d\tau=\theta>\delta.
\end{equation}
  Then, one can choose $T_2=T_2(\delta)>T$, such that, for all $s<-T_2$,
  \begin{gather}\label{bigT2}
    \psi_1^\tau(s)>\theta-\delta,\\
    (\check{a}^-*\psi_2)(s)>\delta.\label{bigT2conv}
  \end{gather}
  Note also that \eqref{ineqoutoftau} holds, for all $s\geq T_2>T$, as well.
  Clearly, for any $\nu\geq\tau$,
  \[
  \psi_1^\nu (s)=\psi_1(s-\nu )\geq\psi_1(s-\tau)>\psi_2(s), \quad s>T_2.
  \]
  Next, $\lim\limits_{\nu \to\infty}\psi_1^\nu (T_2)=\theta>\psi_2(-T_2)$ implies that there exists $\nu _1=\nu _1(T_2)=\nu _1(\delta)>\tau$, such that, for all $\nu >\nu _1$,
  \[
    \psi_1^\nu (s)\geq \psi_1^\nu (T_2)>\psi_2(-T_2)\geq \psi_2(s), \quad s\in[-T_2,T_2].
  \]
  Let such a $\nu >\nu _1$ be chosen and fixed.
  As a result,
  \begin{align}
  \psi_1^\nu (s)\geq \psi_2(s), \quad s\geq -T_2, \label{ineqright}\\
  \intertext{and, by \eqref{bigT2},}
  \psi_1^\nu (s)+\delta>\theta>\psi_2(s), \quad s<-T_2. \label{ineqleft}
  \end{align}
  For the $\nu >\nu _1$ chosen above, define
  \begin{equation}\label{phinu}
   \varphi_\nu (s):=\psi_1^\nu (s)-\psi_2(s), \quad s\in\R.
  \end{equation}
  To prove \eqref{ineqeverywhere}, it is enough to show that $\varphi_\nu (s)\geq0$, $s\in\R$.

  On the contrary, suppose that $\varphi_\nu $ takes negative values. By~\eqref{ineqright}, \eqref{ineqleft},
  \begin{equation}\label{varnubelow}
      \varphi_\nu (s)\geq-\delta, \quad s<-T_2; \qquad \varphi_\nu (s)\geq 0,\quad s\geq-T_2.
  \end{equation}
  Since $\lim\limits_{s\to-\infty} \varphi_\nu (s)=0$ and $\varphi_\nu \in C^1(\R)$, our assumption implies that there exists $s_0<-T_2$, such that
  \begin{equation}\label{minpointmy}
      \varphi_\nu (s_0)=\min\limits_{s\in\R}\varphi_\nu (s)\in [-\delta,0).
  \end{equation}
  We set also
  \begin{equation}\label{deltamy}
    \delta_*:=-\varphi_\nu (s_0)=\psi_2(s_0)-\psi_1^\nu(s_0)\in (0,\delta].
  \end{equation}

Next, both $\psi_1^\nu $ and $\psi_2$ solve \eqref{eq:trw}.
Let $\check{J}_\theta$ be given by \eqref{speckern}. Then, recall, $\int_\R \check{J}_\theta(s)\,ds=m$. Denote, cf.~\eqref{jump}, $\check{L}_\theta \varphi:=\check{J}_\theta *\varphi-m\varphi$. Then one can rewrite \eqref{eq:trw}, cf.~\eqref{eq:tr_w_ii},
\[
  c\psi'(s)+(\check{L}_\theta\psi)(s)+\ka^-(\theta-\psi(s))(\check{a}^-*\psi)(s)=0.
\]
Writing the latter equation for $\psi_1^\nu$ and $\psi_2$ and subtracting the results, one gets
\begin{equation}\label{neweqnpos}
  \begin{gathered}
  c\varphi_\nu'(s)+(\check{L}_\theta\varphi_\nu)(s)+\ka^-A(s)=0,\\
  A(s):=(\theta-\psi_1^\nu(s))(\check{a}^-*\psi_1^\nu)(s)-(\theta-\psi_2(s))(\check{a}^-*\psi_2)(s).
  \end{gathered}
\end{equation}
Consider \eqref{neweqnpos} at the point $s_0$. By~\eqref{minpointmy},
\begin{equation}\label{maxprinciple}
\varphi_\nu'(s_0)=0, \qquad (\check{L}_\theta\varphi_\nu)(s_0)\geq0.
\end{equation}
 Next, \eqref{deltamy} yields
\begin{align}
  A(s_0) & =
  (\theta-\psi_1^\nu(s_0))(\check{a}^-*\psi_1^\nu)(s_0)+(\delta_*-(\theta-\psi_1^\nu(s_0))(\check{a}^-*\psi_2)(s_0)\notag\\
  &=(\theta-\psi_1^\nu(s_0))(\check{a}^-*\varphi_\nu)(s_0)+\delta_*(\check{a}^-*\psi_2)(s_0)\notag\\
  &=(\theta-\psi_1^\nu(s_0))(\check{a}^-*(\varphi_\nu+\delta_*))(s_0)
  +\delta_*\bigl((\check{a}^-*\psi_2)(s_0)-(\theta-\psi_1^\nu(s_0))\bigr)\notag\\
  &>0,\label{ufffffffff}
\end{align}
because of \eqref{minpointmy}, \eqref{bigT2}, and \eqref{bigT2conv}. The strict inequality in \eqref{ufffffffff} together with \eqref{maxprinciple} contradict to~\eqref{neweqnpos}. Therefore, \eqref{ineqeverywhere} holds, for any $\nu>\nu_1$.

\smallskip

\textit{Step 3.} Prove that, cf.~\eqref{ineqeverywhere},
  \begin{equation}\label{infiszero}
    \vartheta_*:=\inf\{\vartheta>0\mid \psi^\vartheta _1(s) \geq \psi_2(s), s\in\R\}=0.
  \end{equation}

On the contrary, suppose that $\vartheta_*>0$. Let $\varphi_*:=\varphi_{\vartheta_*}$ be given by \eqref{phinu}. By the continuity of the profiles, $\varphi_*\geq0$.

First, assume that $\varphi_*(s_0)=0$, for some $s_0\in\R$, i.e. $\varphi_*$ attains its minimum at $s_0$. Then \eqref{maxprinciple} holds with $\vartheta$ replaced by $\vartheta_*$, and, moreover, cf.~\eqref{neweqnpos},
\[
A(s_0)=(\theta-\psi_1^\vartheta(s_0))(\check{a}^-*\varphi_*)(s_0)\geq0.
\]
Therefore, \eqref{neweqnpos} implies
\begin{equation}\label{maxprLtheta}
  (\check{L}_\theta\varphi_*)(s_0)=0.
\end{equation}
By the same arguments as in the proof of Proposition~\ref{prop:infGisreached}, one can show that \eqref{as:aplus-aminus-is-pos} implies that the function $\check{J}_\theta$ also satisfies \eqref{as:aplus-aminus-is-pos}, for $d=1$, with some another constants.
Then, arguing in the same way as in the proof of Proposition~\ref{prop:u_gr_0} (with $d=1$ and $a^+$ replaced by $\check{J}_\theta$), one gets that \eqref{maxprLtheta} implies
that $\varphi_*$ is a constant, and thus $\varphi_*\equiv0$, i.e. $\psi_1^{\vartheta_*}\equiv\psi_2$. The latter contradicts \eqref{ineqoutoftau}.

Therefore, $\varphi_*(s)>0$, i.e. $\psi_1^{\vartheta_*}(s)>\psi_2(s)$, $s\in\R$. By~\eqref{ineqoutoftau} and \eqref{convtotthetaconv},  there exists $T_3=T_3(\vartheta_*)>0$, such that
$\psi_1^{\frac{\vartheta_*}{2}}(s)>\psi_2(s)$, $s>T_3$, and also, for any $s<-T_3$,
\eqref{bigT2conv} holds and \eqref{ineqleft} holds with $\vartheta$ replaced by $\frac{\vartheta_*}{2}$ (for some fixed $\delta\in\bigl(0,\frac{\theta}{4}\bigr)$).
For any $\eps\in\bigl(0,\frac{\vartheta_*}{2}\bigr)$, $\psi_1^{\vartheta_*-\eps}\geq\psi_1^{\frac{\vartheta_*}{2}}$, therefore,
\[
\psi_1^{\vartheta_*-\eps}(s)>\psi_2(s), \quad s>T_3,
\]
and also \eqref{ineqleft} holds with $\vartheta$ replaced by $\vartheta_*-\eps$, for $s<-T_3$.
We set
\[
\alpha:=\inf\limits_{t\in[-T_3,T_3]}(\psi_1^{\vartheta_*}(s)-\psi_2(s))>0.
\]
Since the family $\bigl\{\psi_1^{\vartheta_*-\eps}\mid \eps\in\bigl(0,\frac{\vartheta_*}{2}\bigr)\bigr\}$ is monotone in $\eps$,
and $\lim\limits_{\eps\to0}\psi_1^{\vartheta_*-\eps}(t)=
\psi_1^{\vartheta_*}(t)$, $t\in\R$, we have, by Dini's theorem, that the latter convergence is uniform on $[-T_3,T_3]$. As~a~result, there exists $\eps=\eps(\alpha)\in \bigl(0,\frac{\vartheta_*}{2}\bigr)$, such that
\[
\psi_1^{\vartheta_*}(s)\geq\psi_1^{\vartheta_*-\eps}(s)\geq\psi_2(s), \quad s\in[-T_3,T_3].
\]
Then, the same arguments as in the Step~2 prove that $\psi_1^{\vartheta_*-\eps}(s)\geq\psi_2(s)$, for all $s\in\R$, that contradicts the definition \eqref{infiszero} of $\vartheta_*$.

As a result, $\vartheta_*=0$, and by the continuity of profiles, $\psi_1\geq\psi_2$. By the same arguments, $\psi_2\geq\psi_1$, that fulfills the statement.
\end{proof}

\section{Long-time behavior of solutions}\label{sec:long-time}

We will study here the behavior of $u(t x,t)$, where $u$ solves \eqref{eq:basic}, for big $t\geq0$. The results of Section~\ref{sec:tr-waves} together with the comparison principle imply that if an initial condition $u_0(x)$ to \eqref{eq:basic} has a minorant/majorant which has a form $\psi(x\cdot\xi)$, $\xi\in\S$, where $\psi\in\M$ is a traveling wave profile in the direction $\xi$ with a speed $c\geq c_*(\xi)$, then for the corresponding solution $u$ to \eqref{eq:basic}, the function $u(t x,t)$ will have the minorant/majorant $\psi(t(x\cdot\xi-c))$, correspondingly. In particular, if the initial condition is ``below'' of any traveling wave in a given direction, then one can estimate the corresponding value of $u(tx,t)$ (Theorem~\ref{thm:decayoutsidedirectional}). Considering such a behavior in different directions, one can obtain a (bounded, cf.~Proposition~\ref{prop:front_is_non-empty}) set, out of which the solution exponentially decays to $0$ (Theorem~\ref{thm:outoffront}). Inside of this set the solution will uniformly converge to $\theta$ (Theorem~\ref{thm:convtotheta}).
We will study stationary solutions (Proposition~\ref{uniqstationarysolutions}) and consider the case of slow decaying kernels $a^\pm$ (Subsection~\ref{subsec:fastpropag}) as well.

\subsection{Long-time behavior along a direction}
We will follow the abstract scheme proposed in \cite{Wei1982}. Note that all statements there were considered in the space $C_{b}(\X)$, however, it can be checked straightforward that they remain true in the space $\Buc$.
We will assume that  \eqref{as:chiplus_gr_m} and \eqref{as:aplus_gr_aminus} hold.
Recall that $\theta$, $\Utheta$, $\Ltheta$ are given by \eqref{theta_def}, \eqref{eq:B_theta}, and \eqref{eq:L_theta}, respectively.

Consider the set $N_\theta$ of all nonincreasing functions $\varphi\in C(\R)$, such that
$\varphi(s)=0$, $s\geq0$, and
\begin{equation*}
\varphi(-\infty):=\lim _{s\to-\infty}\varphi(s)\in(0,\theta).\label{def:fy_Weinberger}
\end{equation*}
It is easily seen that $N_\theta\subset U_\theta$.

For arbitrary $s\in\R$, $c\in\R$, $\xi\in\S $,
define the following  mapping $V_{s,c,\xi}:L^\infty(\R)\to L^\infty(\X)$
\begin{equation}\label{defofV}
  (V_{s,c,\xi}f)(x)=f(x\cdot \xi+s+c), \quad x\in\X.
\end{equation}
Fix an arbitrary $\varphi\in N_\theta$.
 For $T>0$, $c\in\R$, $\xi\in\S $, consider the mapping $R_{T,c,\xi}:\ L^{\infty}(\R)\to L^{\infty}(\R)$, given by
\begin{equation}
(R_{T,c,\xi}f)(s)=\max\bigl\{ \varphi(s),(Q_T (V_{s,c,\xi}f))(0)\bigr\},\quad s\in\R,\label{eq:iterop_by_Weinberger}
\end{equation}
where $Q_T$ is given by \eqref{def:Q_T}, cf. Proposition~\ref{prop:Q_def}.
Consider now the following sequence of functions
\begin{equation}\label{fiteration}
  f_{n+1}(s)=(R_{T,c,\xi}f_n)(s),\quad f_0(s)=\varphi(s),\qquad s\in\R, n\in\N\cup\{0\}.
\end{equation}
By Proposition~\ref{prop:Q_def} and \cite[Lemma~5.1]{Wei1982}, $\varphi\in U_\theta$ implies $f_n\in U_\theta$ and $f_{n+1}(s)\geq f_{n}(s)$, $s\in\R$, $n\in\N$; hence one can define the following limit
\begin{equation}
f_{T,c,\xi}(s):= \lim_{n\to\infty}f_n(s), \quad s\in\R.\label{eq:limit_func_Weinberger}
\end{equation}
Also, by \cite[Lemma~5.1]{Wei1982}, for fixed $\xi\in\S $, $T>0$, $n\in\N$, the functions $f_n(s)$ and $f_{T,c,\xi}(s)$ are nonincreasing in $s$ and in $c$; moreover, $f_{T,c,\xi}(s)$ is a lower semicontinuous function of $s,c,\xi$, as a result, this function is continuous from the right in $s$ and in $c$. Note also, that $0\leq f_{T,c,\xi}\leq\theta$.
Then, for any $c,\xi$, one can define the limiting value
\[
f_{T,c,\xi}(\infty):=\lim_{s\to\infty}f_{T,c,\xi}(s).
\]
Next, for any $T>0$, $\xi\in\S $, we define
\begin{equation*}
c_T^{*}(\xi)=\sup\{ c\mid f_{T,c,\xi}(\infty)=\theta\} \in\R\cup\{-\infty,\infty\},
\end{equation*}
where, as usual, $\sup\emptyset:=-\infty$. By \cite[Propositions~5.1, 5.2]{Wei1982}, one has
\begin{equation}\label{jumpfunc}
  f_{T,c,\xi}(\infty)=\begin{cases}
    \theta, & c<c_T^*(\xi),\\
    0, & c\geq c_T^*(\xi),
  \end{cases}
\end{equation}
cf. also \cite[Lemma~5.5]{Wei1982}; moreover, $c_T^*(\xi)$ is a lower semicontinuous function of~$\xi$. It is crucial that, by \cite[Lemma 5.4]{Wei1982}, neither $f_{T,c,\xi}(\infty)$ nor $c_T^{*}(\xi)$ depends on the choice of $\varphi\in N_\theta$. Note that the monotonicity of $f_{T,c,\xi}(s)$ in $s$ and \eqref{jumpfunc} imply that, for $c<c_T^*(\xi)$, $f_{T,c,\xi}(s)=\theta$, $s\in\R$.

\begin{proposition}\label{cstarsareequal}
 Let $\xi\in\S $ and suppose that \eqref{as:chiplus_gr_m}, \eqref{as:aplus_gr_aminus}, and \eqref{aplusexpint1} hold. Let $c_*(\xi)$ be as in Theorem~\ref{thm:trwexists}. Then
 \begin{equation}\label{ct=tc}
 c^*_T(\xi)=Tc_*(\xi), \quad T>0.
 \end{equation}
\end{proposition}
\begin{proof}
Take any $c\in\R$ with $cT\geq c^*_T(\xi)$. Then, by \eqref{jumpfunc},
$f_{T,cT,\xi}\not\equiv\theta$.
By \eqref{eq:iterop_by_Weinberger}, \eqref{fiteration}, one has
\begin{equation}\label{qns1}
  f_{n+1}(s)\geq (Q_T (V_{s,cT,\xi}f_n))(0), \quad s\in\R.
\end{equation}
Since $f_n(s)$ is nonincreasing in $s$, one gets, by \eqref{defofV}, that, for a fixed $x\in\X$, the function $(V_{s,cT,\xi}f_n)(x)$ is also nonincreasing in $s$.
Next, by \eqref{defofV}, \eqref{eq:limit_func_Weinberger} and Propositions \ref{prop:stab_Linf_ae},
\begin{equation}\label{convspec}
  (Q_T (V_{s,cT,\xi}f_n))(x)\to (Q_T (V_{s,cT,\xi}f_{T,cT,\xi}))(x), \text{  a.a. } x\in\X.
\end{equation}
Note that, by \eqref{defofV} and Proposition~\ref{prop:monot_sol},
\begin{equation}\label{eqqqphi}
  (Q_T (V_{s,cT,\xi}f_{T,cT,\xi}))(x)=\phi (x\cdot\xi,T),
\end{equation}
where $\phi(\tau,t)$, $\tau\in\R$, $t\in\R_+$ solves \eqref{eq:basic_one_dim} with $\psi(\tau)=f_{T,cT,\xi}(\tau+s+cT)$ (note that $s$ is a parameter now, cf.~\eqref{eq:basic_one_dim}). On the other hand, the evident equality
$(V_{s,cT,\xi}f_{T,cT,\xi})(x+\tau\xi)=f_{T,cT,\xi}(x\cdot\xi+\tau+s+cT)$, $\tau\in\R$
shows that the function $V_{s,cT,\xi}f_{T,cT,\xi}$ is a decreasing function on $\X$ along the $\xi\in\S $, cf.~Definition~\ref{def:monotoneindirection}, as $f_{T,cT,\xi}$ is a decreasing function on $\R$. Then, by Proposition~\ref{prop:monot_along_vector_sol} and \eqref{eqqqphi}, the function $\X\ni x\mapsto\phi (x\cdot\xi,T)\in[0,\theta]$ is decreasing
along the $\xi$ as well, i.e. $\phi (x\cdot\xi+\tau,T)=\phi ((x+\tau\xi)\cdot\xi,T)\leq
\phi(x\cdot\xi,T)$, $\tau\geq0$. As a result, the function $\phi (s,T)$ is monotone (almost everywhere) in $s$. Since $f_{T,cT,\xi}(s)$ was continuous from the right in~$s$, one gets from \eqref{qns1}, \eqref{convspec}, that
\[
f_{T,cT,\xi}(s)\geq (\tilde{Q}_T f_{T,cT,\xi})(s+cT),
\]
where $\tilde{Q}_T$ is given as in Proposition~\ref{prop:Qtilde}. Since $f_{T,cT,\xi}\not\equiv\theta$, one has that, by  \cite[Theorem 5]{Yag2009} (cf. the proof of Theorem~\ref{thm:trwexists}), there exists a traveling wave profile with speed $c$. By Theorem~\ref{thm:trwexists}, we have that $c\geq c_*(\xi)$, and hence $c^*_T(\xi)\geq Tc_*(\xi)$.

Take now any $c\geq c_*(\xi)$ and consider, by Theorem~\ref{thm:trwexists}, a traveling wave in a direction $\xi\in\S $, with a profile $\psi\in\M$ and a speed $c$. Then, by \eqref{defofV} and \eqref{trwv},
\[
(Q_T(V_{s,cT,\xi}\psi))(x)=\psi((x\cdot\xi-cT)+s+cT)=\psi(x\cdot\xi+s).
\]
Choose $\varphi\in N_\theta$ such that $\varphi(s)\leq \psi(s)$, $s\in\R$ (recall that all constructions are independent on the choice of $\varphi$). Then, one gets from \eqref{eq:iterop_by_Weinberger} and \ref{prop:Q_preserves_order} of Proposition~\ref{prop:Q_def},
that
\[
(R_{T,cT,\xi}\varphi)(s)\leq(R_{T,cT,\xi}\psi)(s)=\psi(s), \quad s\in\R.
\]
Then, by \eqref{fiteration} and \eqref{eq:limit_func_Weinberger}, $f_{T,cT,\xi}(s)\leq\psi(s)$, $s\in\R$, and thus \eqref{jumpfunc} implies $cT\geq c_T^*(\xi)$; as a result, $Tc_*(\xi)\geq c^*_T(\xi)$, that fulfills the statement.
\end{proof}

We describe now how the solution to \eqref{eq:basic} behaves, for big times, along a direction $\xi\in\S $. We start with a result about an exponential decaying along such a direction. It is worth noting that we do not need to assume either \eqref{as:chiplus_gr_m} or \eqref{as:aplus_gr_aminus} to prove Proposition~\ref{prop:expdecay} below.

For any $\xi\in\S $ and $\la>0$, consider the following set of bounded functions on $\X$:
\begin{equation}\label{sefElambda}
E_{\la,\xi}(\X):=\bigl\{ f\in L^\infty(\X) \bigm| \|f\|_{\la,\xi}:=\esssup_{x\in\X} \lvert f(x)\rvert e^{\la  x\cdot\xi }<\infty \bigr\}.
\end{equation}
Evidently, for $f\in L^\infty(\X)$,
\[
\esssup\limits_{x\in\X} |f(x)| e^{\la  x\cdot\xi }<\infty \quad \text{if and only if}\quad
\esssup\limits_{x\cdot\xi\geq0} |f(x)| e^{\la  x\cdot\xi }<\infty,
\]
therefore,
\begin{equation*}
  E_{\la,\xi}(\X)\subset E_{\la',\xi}(\X), \quad \la>\la'>0, \xi\in\S .
\end{equation*}
\begin{proposition}\label{prop:expdecay}
Let $\xi\in\S $ and $\la>0$ be fixed and suppose that \eqref{aplusexpint1} holds with  $\mu=\la$. Let $0\leq u_0\in E_{\la,\xi}(\X)$ and let $u=u(x,t)$ be a solution to \eqref{eq:basic}. Then
\begin{equation}\label{expdecaying}
  \| u(\cdot,t)\|_{\la,\xi}\leq \|u_0\|_{\la,\xi}e^{ p  t}, \quad t\geq0,
\end{equation}
where
\begin{equation}\label{bignu}
 p = p (\xi,\la)=\ka^+ \int_\X a^+(x) e^{\la x\cdot\xi }\,dx-m\in\R.
\end{equation}
\end{proposition}
\begin{proof}
First, we note that, for any $a\in L^1(\X)$,
\begin{align}\notag
\bigl\lvert(a*f)(x)e^{\la  x \cdot\xi }\bigr\rvert &\leq
\int_\X|a(x-y)|e^{\la  (x-y) \cdot\xi }|f(y)|e^{\la  y \cdot\xi }\,dy\\&\leq \|f\|_{\la,\xi} \int_\X |a(y)|e^{\la  y \cdot\xi }\,dy.\label{elaest}
\end{align}

We will follow the notations from the proof of Theorem~\ref{thm:exist_uniq_BUC}, cf.~Remark~\ref{rem:exist_uniq_Linf}. Let $p$ is given by \eqref{bignu} and suppose that, for some $\tau\in[0,T)$, $\|u_\tau\|_{\la,\xi}\leq \|u_0\|_{\la,\xi}\,e^{ p  \tau}$. Take any $v\in \x_{\tau,\Tau}^+(r)$ with $\Tau$, $r$ given by \eqref{settings}, \eqref{need111}, such that
\begin{equation}\label{assumptonv}
\|v(\cdot,t)\|_{\la,\xi} \leq \|u_0\|_{\la,\xi}\, e^{ p  t}, \quad t\in[\tau,\Tau].
\end{equation}
Then, by \eqref{eq:exist_uniq_BUC:Phi_v}, \eqref{eq:exist_uniq_BUC:B}, one gets, for any $t\in[\tau,\Tau]$,
\begin{align*}\notag
  0&\leq (\Phi_\tau v)(x,t)e^{\la  x \cdot\xi }\\&\leq e^{-(t-\tau)m}u_\tau(x)e^{\la  x \cdot\xi }
  +\int_\tau^t e^{-m(t-s)}\ka^+(a^+*v)(x,s) e^{\la  x \cdot\xi } \,ds\\
  &\leq \|u_0\|_{\la,\xi}\,e^{-m(t-\tau)}e^{ p  \tau}
  +\|u_0\|_{\la,\xi}\,\ka^+\int_\X a^+(y) e^{\la  y \cdot\xi }\,dy \int_\tau^t e^{-m(t-s)}e^{ p  s} \,ds,\\&=\|u_0\|_{\la,\xi}\,e^{-mt}e^{( p +m)\tau}+\|u_0\|_{\la,\xi}\,( p +m)e^{-mt}\int_\tau^t e^{( p +m)s}\,ds\\&=\|u_0\|_{\la,\xi}\,e^{ p  t},
\end{align*}
where we used \eqref{elaest} and \eqref{assumptonv}. Therefore,
$\|(\Phi_\tau v)(\cdot,t)\|_{\la,\xi} \leq \|u_0\|_{\la,\xi}\,e^{ p  t}$, $t\in[\tau,\Tau]$.
As a result,
\[
\|(\Phi_\tau^n v)(\cdot,t)\|_{\la,\xi}\leq \|u_0\|_{\la,\xi}\,e^{ p  t}, \quad n\in\N, \ t\in[\tau,\Tau].
\]
Then $\|u(\cdot,t)\|_{\la,\xi}$ satisfies the same inequality on $[\tau,\Tau]$; and, by the proof of Theorem~\ref{thm:exist_uniq_BUC}, we have the statement.
\end{proof}

\begin{remark}
  It follows from \ref{L-converges} of Lemma~\ref{lem:allaboutLapl} and the considerations thereafter, that the statement of Proposition~\ref{prop:expdecay} remains true if \eqref{aplusexpint1} holds for some $\mu>\la$, provided that we assume, additionally, \eqref{boundedkernels}.
\end{remark}

Define now the following set
\begin{equation}
    \Tau_{T,\xi}=\bigl\{ x\in\X \mid x\cdot\xi\leq c_T^{*}(\xi)\bigr\}, \quad \xi\in\S , T>0.\label{eq:TauTxi}
\end{equation}
Clearly, the set $\Tau_{T,\xi}$ is convex and closed. Moreover, by \eqref{ct=tc},
\begin{equation}\label{taut=ttau1-dir}
  \Tau_{T,\xi}=T\Tau_{1,\xi}
\end{equation}
Here and below, for any measurable $A\subset\X$, we define $tA:=\{tx\mid x\in A\}\subset\X$.
We are going to explain now how a solution $u(x,t)$ to \eqref{eq:basic} behaves outside of the set $t\Tau_{1,\xi}=\Tau_{t,\xi}$, $t>0$.

\begin{theorem}\label{thm:decayoutsidedirectional}
  Let $\xi\in\S $ and $a^+\in\Uxi$; i.e. all conditions of Definition~\ref{def:Uxi} hold. Let $\la_*=\la_*(\xi)\in I_\xi$ be the same as in Proposition~\ref{prop:infGisreached}. Suppose that $u_0\in E_{\la_*,\xi}(\X)\cap \Ltheta$ and let $u\in\tXinf$ be the corresponding solution to \eqref{eq:basic}. Let $\Tauout_\xi\subset\X$ be an open set, such that $\Tau_{1,\xi}\subset\Tauout_\xi$ and $\delta:=\dist (\Tau_{1,\xi},\X\setminus\Tauout_\xi )>0$. Then the following estimate holds
\begin{equation}\label{supconvto0xi}
  \esssup_{x\notin t\Tauout_\xi } u(x,t)\leq \lVert u_0\rVert_{\la_*,\xi} e^{-\la_*\delta t}, \quad t>0.
\end{equation}
\end{theorem}
\begin{proof}
Let $p_*:=p(\xi,\la_*)$ be given by \eqref{bignu}. By \eqref{expdecaying}, \eqref{sefElambda}, one has
\begin{equation}\label{mainassasf}
  0\leq u(x,t)\leq \lVert u_0\rVert_{\la_*,\xi} \exp\bigl\{ p_* t - \la_* x\cdot\xi\bigr\}, \quad \text{a.a. } x\in\X.
\end{equation}
Next, by \eqref{eq:TauTxi} and Proposition~\ref{cstarsareequal},  for any $t>0$ and for all $x\in\X\setminus t\Tauout_\xi$, one has $x\cdot\xi\geq t c^*_1(\xi)+t\delta=t c_*(\xi)+t\delta$. Then, by~\eqref{speedviaabs},
\begin{multline*}
\inf_{x\notin t\Tauout_\xi } (\la_* x\cdot\xi)\geq t\la_* c_*(\xi)+t\la_* \delta\\=t\Bigl(\ka^+ \int_\X a^+(x) e^{\la_* x\cdot\xi }\,dx-m\Bigr)+t\la_* \delta=tp_*+t\la_* \delta.
\end{multline*}
Therefore, \eqref{mainassasf} implies the statement.
\end{proof}

\begin{remark}\label{rem:ads1}
 The assumption $u_0\in E_{\la_*,\xi}(\X)$ is close, in some sense, to the weakest possible assumption on an initial condition $u_0\in\Ltheta$ for the equation \eqref{eq:basic} to have
 \begin{equation}\label{convtp0general}
    \lim_{t\to\infty}\esssup_{x\notin t\Tauout_{\xi}} u(x,t)=0,
 \end{equation}
 for an arbitrary open set $\Tauout_{\xi}\supset \Tau_{1,\xi}$, where $\Tau_{1,\xi}$ is defined by \eqref{eq:TauTxi}.
 Indeed, take any $\la_1,\la$ with $0<\la_1<\la<\la_*=\la_*(\xi)$. By~Theorem~\ref{thm:speedandprofile}, there exists a traveling wave solution to \eqref{eq:basic} with a profile $\psi_1\in\M$ such that $\la_0(\psi_1)=\la_1$. By Proposition~\ref{asymptexact} (with $j=1$ as $\la_1<\la_*$) we have that $\psi_1(t)\sim De^{-\la_1 t}$, $t\to\infty$. It is easily seen that one can choose a function $\varphi\in\M\cap C(\R)$ such that there exist $p>0$, $T>0$, such that $\varphi(t)\geq\psi_1(t)$, $t\in\R$ and $\varphi(t)=pe^{-\la t}$, $t>T$. Take now $u_0(x)=\varphi(x\cdot\xi)$, $x\in\X$. We have $u_0\in E_{\la,\xi}(\X)\setminus E_{\la_*,\xi}(\X)$. Then, by Proposition~\ref{prop:monot_sol}, the corresponding solution has the form $u(x,t)=\phi(x\cdot\xi,t)$. By Proposition~\ref{prop:comp_pr_BUC} applied to the equation \eqref{eq:basic_one_dim}, $\phi(s,t)\geq\psi_1(s-c_1t)$, $s\in\R$, $t\geq0$, where $c_1=G_\xi(\la_1)>c_*(\xi)$, cf.~\eqref{dedGxi} and \eqref{speedviaabs}. Take $c\in (c_*(\xi),c_1)$ and consider an open set $\Tauout_\xi: =\{x\in\X\mid x\cdot\xi<c\}$, then $\Tau_{1,\xi}\subset \Tauout_\xi \subset \{x\in\X\mid x\cdot\xi\leq c_1\}=:A_1$. One has
 \begin{align*}
 \sup_{x\notin t\Tauout_{\xi}}u(x,t)&\geq \sup_{x\in tA_1\setminus t\Tauout_{\xi}}\phi(x\cdot\xi,t)\\
 &\geq \sup_{ct<s\leq c_1t}\psi_1(s-c_1t)=\psi_1(ct-c_1t)>\psi_1(0),
 \end{align*}
 as $c<c_1$ and $\psi_1$ is decreasing. As a result, \eqref{convtp0general} does not hold.

 On the other hand, if $\psi_*\in\M$ is a profile with the minimal speed $c_*(\xi)\neq0$ and if $j=2$, cf.~Proposition~\ref{prop:charfuncpropert}, then $u_0(x):=\psi_*(x\cdot\xi)$ does not belong to the space $E_{\la_*,\xi}(\X)$, and the arguments above do not contradict
\eqref{convtp0general} anymore. In the next remark, we consider this case in more details.
\end{remark}

\begin{remark}\label{rem:ads2}
  In connection with the previous remark, it is worth noting also that one can easily generalize Theorem~\ref{thm:decayoutsidedirectional} in the following way. Let $u_0\in E_{\la,\xi}(\X)\cap \Ltheta$, for some $\la\in (0,\la_*]$, and let $u\in\tXinf$ be the corresponding solution to \eqref{eq:basic}. Consider the set $A_{c,\xi}:=\bigl\{ x\in\X \mid x\cdot\xi\leq c\bigr\}$, where  $c=\la^{-1}(\ka^+\A_\xi(\la)-m)$
   cf.~\eqref{speedviaabs}. Then, for any open set $B_{c,\xi}\supset A_{c,\xi}$ with $\delta_c:=\dist(A_{c,\xi},\X\setminus B_{c,\xi})>0$, one gets
\begin{equation}\label{equivasdas}
  \esssup_{x\notin t B_{c,\xi}} u(x,t)\leq \lVert u_0\rVert_{\la,\xi} e^{-\la\delta_c t}.
\end{equation}
Therefore, if $u_0(x)=\psi_*(x\cdot \xi)$, where $\psi_*$ is as in Remark~\ref{rem:ads1} above, then, evidently, $u_0\in E_{\la,\xi}(\X)$, for any $\la\in(0,\la_*)$. Then, for any open $\Tauout_\xi \supset\Tau_{1,\xi}$ with $\delta:=\dist (\Tau_{1,\xi},\X\setminus\Tauout_\xi )>0$  one can choose, for any $\eps\in(0,1)$, $c_1=c_*(\xi)+\delta\eps$. By Theorem~\ref{thm:speedandprofile}, there exists a unique $\la_1=\la_1(\eps)\in(0,\la_*)$ such that $c_1=\la_1^{-1}(\ka^+\A_\xi(\la_1)-m)$. Then $u_0\in E_{\la_1,\xi}(\X)$ and
$A_{c_1,\xi}\subset \Tauout_\xi $, i.e. $\Tauout_\xi $ may be considered as a set $B_{c_1,\xi}$, cf.~above. As a result, \eqref{equivasdas} gives \eqref{supconvto0xi}, with the constant $\lVert u_0\rVert_{\la_1,\xi}<\lVert u_0\rVert_{\la_*,\xi}$, and with $\la_*\delta$ replaced by $\la_1 \delta(1-\eps)$. Note that, clearly, $\lVert u_0\rVert_{\la_1,\xi}\nearrow\lVert u_0\rVert_{\la_*,\xi}$, $\la_1\nearrow\la_*$, $\eps\to0$.
\end{remark}

\subsection{Global long-time behavior}
We are going to consider now the global long-time behavior along all possible directions $\xi\in\S $ simultaneously. Define, cf. \eqref{eq:TauTxi},
\begin{equation}
\Tau_{T}=\left\{ x\in\X |x\cdot\xi\leq c_T^{*}(\xi),\ \xi\in\S \right\}, \quad T>0.\label{eq:TauT}
\end{equation}
By \eqref{eq:TauTxi}, \eqref{ct=tc}, \eqref{taut=ttau1-dir},
\begin{equation}\label{TauT=TTau1}
  \Tau_{T}=\bigcap_{\xi\in\S }\Tau_{T,\xi}=\bigcap_{\xi\in\S }T\Tau_{1,\xi}=T\Tau_1, \quad T>0.
\end{equation}
Clearly, the set $\Tau_T$, $T>0$ is convex and closed. To have an analog of Theorem~\ref{thm:decayoutsidedirectional} for the set $\Tau_{T}$, one needs to have $a^+\in\Uxi$, for all $\xi\in\S $, cf.~Definition~\ref{def:Uxi}.

 Since $\int_{x\cdot \xi\leq0}a^+(x)e^{\la x\cdot\xi}\,dx\in[0,1]$, $\xi\in\S $, $\la>0$, we have the following observation. If, for some $\xi\in\S $, there exist $\mu^\pm>0$, such that, cf.~\eqref{aplusexpla}, $\A_{\pm\xi}(\mu^\pm)<\infty$, i.e. if \eqref{aplusexpint1} holds for both $\xi$ and $-\xi$, then, for $\mu=\min\{\mu^+,\mu^-\}$,
 \begin{align}
&\quad \int_\X a^+(x) e^{\mu |x\cdot\xi|}\,dx=
 \int_{x\cdot\xi\geq0} a^+(x) e^{\mu x\cdot\xi}\,dx
 +\int_{x\cdot\xi<0} a^+(x) e^{-\mu x\cdot\xi}\,dx\notag
 \\&\leq \int_{x\cdot\xi\geq0} a^+(x) e^{\mu^+ x\cdot\xi}\,dx
 +\int_{x\cdot(-\xi)> 0} a^+(x) e^{\mu^- x\cdot(-\xi)}\,dx<\infty.
 \label{expmoddir}
\end{align}

Let now $\{e_i\mid 1\leq i\leq d\}$ be an orthonormal basis in $\X$. Let \eqref{aplusexpint1} holds for $2d$ directions $\{\pm e_i\mid 1\leq i\leq d\}\subset\S $ and let $\mu_i=\min\{\mu(e_i),\mu(-e_i)\}$, $1\leq i\leq d$, cf.~\eqref{expmoddir}. Set $\mu=\frac{1}{d}\min\{\mu_i\mid 1\leq i\leq d\}$. Then, by the triangle and Jensen's inequalities and \eqref{expmoddir}, one has
\begin{align*}
      \int_\X a^+(x) e^{\mu|x|}\,dx
&\leq \int_\X a^+(x) \exp\biggl(\sum_{i=1}^d \frac{1}{d}\mu_i|x\cdot e_i|\biggr)\,dx\\
&\leq \sum_{i=1}^d \frac{1}{d} \int_\X a^+(x) e^{\mu_i|x\cdot e_i|}\,dx<\infty.
\end{align*}
As a result, the assumption that \eqref{aplusexpint1} holds, for all $\xi\in\S $, is equivalent to the following one
\begin{assum}\label{expradialmoment}
  \text{there exists $\mu_d>0$, such that} \ \int_\X a^+(x) e^{\mu_d|x|}\,dx<\infty.
\end{assum}
Clearly, \eqref{expradialmoment} implies
\begin{equation}\label{firstglobalmoment}
  \int_\X |x| a^+(x)\,dx<\infty,
\end{equation}
and thus \eqref{firstmoment} holds, for any $\xi\in\S $. Then, one can define the (global) first moment vector of $a^+$, cf.~\eqref{firstdirmoment},
\begin{equation}\label{firstfullmoment}
  \m:=\int_\X x a^+(x)\,dx\in\X.
\end{equation}
The most `anisotropic' assumption is \eqref{as:a+nodeg-mod}. We will assume, for simplicity, that \eqref{as:a+nodeg} holds; then \eqref{as:a+nodeg-mod} holds with $r(\xi)=0$, for all $\xi\in\S $.

\begin{proposition}\label{prop:front_is_non-empty}
Let \eqref{as:chiplus_gr_m}, \eqref{as:aplus_gr_aminus}, \eqref{as:a+nodeg}, \eqref{boundedkernels}, \eqref{expradialmoment} hold. Then, for any $T>0$, $T\ka^+\m\in\X$ is an interior point of $\Tau_T$, and $\Tau_T$ is a bounded set.
\end{proposition}
\begin{proof}
By \eqref{TauT=TTau1}, it is enough to prove the statement, for $T=1$.
By~\eqref{firstdirmoment}, for any orthonormal basis $\{e_i\mid 1\leq i \leq d\}\subset\S $, $\m=\sum\limits_{i=1}^d \m_{e_i}$. As it was shown above, the assumptions of the statement imply that Theorem~\ref{thm:speedandprofile} holds, for any $\xi\in\S $. Therefore, by
\eqref{minspeed} and Proposition~\ref{cstarsareequal},
\begin{equation}\label{dsasasa}
  (\ka^+ \m)\cdot\xi =\ka^+ \int_\X x\cdot \xi a^+(x)\,dx=\ka^+\m_\xi<c_*(\xi)=c^*_1(\xi),
\end{equation}
for all $\xi\in\S $; thus $\ka^+ \m\in \Tau_1$. Since the inequality in \eqref{dsasasa} is strict, the point $\ka^+\m$ is an interior point of $\Tau_1$.

Next, by Proposition~\ref{cstarsareequal}, $x\in\Tau_1$ implies that, for any fixed $\xi\in\S $, $x\cdot \xi\leq c_*(\xi)$ and $x\cdot (-\xi)\leq c_*(-\xi)$, i.e.
\begin{equation}\label{dirbundforfront}
  -c_*(-\xi)\leq x\cdot\xi\leq c_*(\xi), \quad x\in\Tau_1,\xi\in\S .
\end{equation}
Then \eqref{dirbundforfront} implies
\[
\lvert x\cdot\xi\rvert\leq \max\bigl\{\lvert c_*(\xi)\rvert,\lvert c_*(-\xi)\rvert\bigr\}, \quad
x\in\Tau_1,\xi\in\S ;
\]
in particular, for an orthonormal basis $\{e_i\mid 1\leq i\leq d\}$ of $\X$, one gets
\[
|x|\leq \sum_{i=1}^d \lvert x\cdot e_i\rvert\leq \sum_{i=1}^d
\max\bigl\{\lvert c_*(e_i)\rvert,\lvert c_*(-e_i)\rvert\bigr\}=:R<\infty,\quad x\in\Tau_1,
\]
that fulfills the statement.
\end{proof}

\begin{remark}\label{rem:minspeedopospos}
  It is worth noting that, by \eqref{minspeed}, \eqref{firstdirmoment}, the following inequality holds, cf.~\eqref{dirbundforfront},
\begin{equation*}
  c_*(\xi)+c_*(-\xi)>\ka^+(\m_\xi+\m_{-\xi})=0.
\end{equation*}
\end{remark}

For any $T>0$, consider the set $\Mn(T)$ of all subsets from $\X$ of the following form:
\begin{equation}
M_{T}=M_{T,\eps,K,\xi_1,\ldots,\xi_K}=\bigl\{x\in\X \mid
x\cdot\xi_i\leq c^*_T(\xi_i)+\eps,\ i=1,\ldots,K\bigr\},\label{eq:eps_neighb_ofGamma}
\end{equation}
for some $\eps>0$, $K\in\N$, $\xi_1,\ldots,\xi_K\in\S $.

We are ready now to prove a result about the long-time behavior at infinity in space.
\begin{theorem}\label{thm:outoffront}
Let the conditions \eqref{as:chiplus_gr_m}, \eqref{as:aplus_gr_aminus}, \eqref{as:a+nodeg}, \eqref{boundedkernels}, \eqref{expradialmoment} hold. Let $u_0\in \Ltheta$ be such that
\begin{equation}\label{initcondexp}
  \lvert\lvert\lvert u_0\rvert\rvert\rvert :=\sup_{\la>0}\esssup_{x\in\X} u_0(x)e^{\la|x|}<\infty
\end{equation}
and let $u\in\tXinf$ be the corresponding solution to \eqref{eq:basic}. Then, for any open set $\Tauout\supset\Tau_1$, there exists $\nu=\nu(\Tauout)>0$, such that
\begin{equation*}
  \esssup_{x\notin t\Tauout} u(x,t)\leq \lvert\lvert\lvert u_0\rvert\rvert\rvert  e^{-\nu t}, \quad t>0.
\end{equation*}
\end{theorem}
\begin{proof}
By Proposition~\ref{prop:front_is_non-empty}, the set $\Tau_1$ is bounded and nonempty. Then, by \cite[Lemma~7.2]{Wei1982}, there exists $\eps>0$, $K\in\N$, $\xi_1,\ldots,\xi_K\in\S $ and a set $M\in\Mn(1)$ of the form \eqref{eq:eps_neighb_ofGamma}, with $T=1$, such that
 \begin{equation}\label{maininclW}
   \Tau_1\subset M\subset\Tauout.
\end{equation}
Choose now
\[
\Tauout_{\xi_i}=\Bigl\{x\in\X\Bigm\vert x\cdot\xi_i <c_1^*(\xi_i)+\frac{\eps}{2}\Bigr\}\supset \Tau_{1,\xi_i}, \quad 1\leq i\leq K.
\]
Then, by \eqref{maininclW},
\[
\Tau_1=\bigcap_{\xi\in\S }\Tau_{1,\xi}\subset\bigcap_{i=1}^K \Tau_{1,\xi_i}\subset
\bigcap_{i=1}^K \Tauout_{\xi_i}\subset M\subset\Tauout,
\]
and, therefore,
\begin{equation}\label{complincl}
    \X\setminus \Tauout \subset \bigcup_{i=1}^K (\X\setminus\Tauout_{\xi_i}).
\end{equation}
By \eqref{sefElambda}, the assumption \eqref{initcondexp} implies,
\begin{align*}
  \|u_0\|_{\la,\xi}&\leq \max\Bigl\{\esssup_{x\cdot\xi\geq0} \lvert u_0(x)\rvert e^{\la  x\cdot\xi}, \esssup_{x\cdot\xi<0} \lvert u_0(x)\rvert\Bigr\}\\
  & \leq \max\Bigl\{\esssup_{x\cdot\xi\geq0} \lvert u_0(x)\rvert e^{\la  |x|}, \esssup_{x\cdot\xi<0} \lvert u_0(x)\rvert\Bigr\}\leq
  \esssup_{x\in\X} \lvert u_0(x)\rvert e^{\la  |x|}\leq \lvert\lvert\lvert u_0\rvert\rvert\rvert,
\end{align*}
for any $\la>0$, $\xi\in\S $.
Denote
 \[
  \nu_{i} :=\la_*(\xi_i) \dist (\Tau_{1,\xi_i},\X\setminus\Tauout_{\xi_i})=\la_*(\xi_i)\frac{\eps}{2}, \quad 1\leq i\leq K.
\]
Then, by Theorem~\ref{thm:decayoutsidedirectional} and \eqref{complincl}, one gets, for any $t>0$,
\[
  \esssup_{x\notin t\Tauout} u(x,t) \leq \max_{1\leq i\leq K}\esssup_{x\notin t\Tauout_{\xi_i}} u(x,t)\leq \lVert u_0\rVert_{\la_*(\xi_i),\xi_i} e^{-\nu_{i} t}\leq \lvert\lvert\lvert u_0\rvert\rvert\rvert e^{-\nu t},
\]
with $\nu:=\min\{\nu_{i}\mid 1\leq i\leq K\}$.
\end{proof}

Our second main result about the long-time behavior states that the solution $u\in\Xinf$ uniformly converges to $\theta$ inside the set $t\Tau_1=\Tau_t$. The proof of this result is quite technical. For the convenience of the reader, we present here the statement of Theorem~\ref{thm:convtotheta} only, and explain the proof in the next subsection.

For a closed set $A\subset\X$, we denote by $\inter(A)$ the interior of $A$.

\begin{theorem}\label{thm:convtotheta}
  Let the conditions \eqref{as:chiplus_gr_m}, \eqref{as:aplus_gr_aminus}, \eqref{as:aplus-aminus-is-pos}, \eqref{boundedkernels}, \eqref{expradialmoment} hold. Let $u_0\in \Utheta$, $u_0\not\equiv0$, and  $u\in\Xinf$ be the corresponding solution to \eqref{eq:basic}. Then, for any compact set $\Tauin\subset\inter(\Tau_1)$,
  \begin{equation}\label{convtotheta}
    \lim_{t\to\infty} \min_{x\in t\Tauin} u(x,t)=\theta.
  \end{equation}
\end{theorem}
\begin{corollary}\label{cor_essinf}
  Let the conditions \eqref{as:chiplus_gr_m}, \eqref{as:aplus_gr_aminus}, \eqref{as:aplus-aminus-is-pos}, \eqref{boundedkernels}, \eqref{expradialmoment} hold. Let $u_0\in \Ltheta$ be such that there exist $x_0\in\X$, $\eta>0$, $r>0$, with $u_0\geq\eta$, for a.a. $x\in B_r(x_0)$. Let $u\in\tXinf$ be the corresponding solution to \eqref{eq:basic}. Then, for any compact set $\Tauin\subset\inter(\Tau_1)$,
  \begin{equation*}
    \lim_{t\to\infty} \essinf_{x\in t\Tauin} u(x,t)=\theta.
  \end{equation*}
\end{corollary}
\begin{proof}
  The assumption on $u_0$ implies that there exists a function $v_0\in\Utheta\subset\Ltheta$, $v_0\not\equiv0$, such that $u_0(x)\geq v_0(x)$, for a.a. $x\in\X$. Then, by
    Remark~\ref{rem:comp_pr_Linf}, $u(x,t)\geq v(x,t)$, for a.a. $x\in\X$, and for all $t\geq0$, where $v\in\tXinf$ is the corresponding to $v_0$ solution to \eqref{eq:basic}. By~Proposition~\ref{prop:comp_pr_BUC}, $v\in\Xinf$, and one has \eqref{convtotheta} for $v$, with the same $\Tau_1$, cf.~\ref{eq:QBtheta_subset_Btheta} of Proposition~\ref{prop:Q_def}. The statement follows then from the evident inequality
  \[
  \min\limits_{x\in t\Tauin}v(x,t)=\essinf\limits_{x\in t\Tauin}v(x,t)\leq
  \essinf\limits_{x\in t\Tauin}u(x,t)\leq\theta. \qedhere
  \]
\end{proof}

As an important application of Theorem~\ref{thm:convtotheta},
we will prove that there are not stationary solutions $u\geq0$  to \eqref{eq:basic} (i.e. solutions with $\frac{\partial}{\partial t}u=0$), except $u\equiv0$ and $u\equiv\theta$, provided that the origin belongs to $\inter(\Tau_1)$.

\begin{proposition}\label{uniqstationarysolutions}
Let the conditions \eqref{as:chiplus_gr_m}, \eqref{as:aplus_gr_aminus}, \eqref{as:aplus-aminus-is-pos}, \eqref{boundedkernels}, \eqref{expradialmoment}, and \eqref{infpos} hold.
Let the origin belongs to $\inter(\Tau_1)$.
Then there exist only two non-negative stationary solutions to \eqref{eq:basic}
in $L^\infty(\X)$, namely, $u=0$ and $u=\theta$.
\end{proposition}
\begin{proof}
Since $\frac{\partial}{\partial t}u=0$, one gets from \eqref{eq:basic} that
\begin{equation}\label{stateq1}
  u(x)=\frac{\ka^+ (a^+*u)(x)}{m+\ka^- (a^-*u)(x)}, \quad x\in\X.
\end{equation}
Then, by Lemma~\ref{le:simple}, one easily gets that $u\in\Buc$.

Denote $M:=\|u\|=\sup\limits_{x\in\X}u(x)$. We are going to prove now that $M\leq\theta$. On the contrary, suppose that $M>\theta$.
One can rewrite \eqref{stateq1} as follows:
\begin{equation}
mu(x)+\ka^-(a^-*u)(x)(u(x)-\theta)=(J_\theta*u)(x)
\leq Mm, \label{stateq2}
\end{equation}
where $J_\theta\geq0$ is given by \eqref{Jq} and hence $\int_\X J_\theta(x)\,dx=m$.

Choose a sequence $x_n\in\X$, $n\in\N$, such that $u(x_n)\to M$, $n\to\infty$. Substitute $x_n$ to the inequality \eqref{stateq2} and pass $n\to\infty$. Since $M>\theta$ and $u\geq0$, one gets then that $(a^-*u)(x_n)\to0$, $n\to\infty$. Passing to a subsequence of $\{x_n\}$ and keeping the same notation, for simplicity, one gets that
\[
(a^-*u)(x_n)\leq\dfrac{1}{n},\ n\geq 1.
\]

For all $n\geq r_0^{-2d}$, set $r_n:=n^{-\frac{1}{2d}}\leq r_0$; then the inequality \eqref{infpos} holds, for any $x\in B_{r_n}(0)$, and hence
\begin{equation}\label{eq:st_sol_inf_est}
\dfrac{1}{n}\geq(a^-*u)(x_n)\geq\alpha(\1_{B_{r_n}(0)}*u)(x_n)\geq \alpha V_d(r_n)\min\limits_{x\in B_{r_n}(x_n)}u(x),
\end{equation}
where $V_d(R)$ is a volume of a sphere with the radius $R>0$ in $\X$. Since $V(r_n)=r_n^d V_d(1)=n^{-\frac{1}{2}}V_d(1)$, we have from \eqref{eq:st_sol_inf_est}, that, for any $n\geq r_0^{-2d}$, there exists $y_n\in B_{r_n}(x_n)$, such that
\[
u(y_n)\leq\dfrac{1}{\alpha\sqrt{n}V_d(1)}.
\]
Thus $u(y_n)\to0$, $n\to\infty$. Recall that $u(x_n)\to M>0$, $n\to\infty$, however, $|x_n-y_n|\leq r_n =n^{-\frac{1}{2d}}$, that may be arbitrary small. This contradicts the fact that $u\in\Buc$.

As a result, $0\leq u(x)\leq \theta=M$, $x\in\X$. Let $u\not\equiv0$. By Theorem~\ref{thm:convtotheta}, for  any compact set $\Tauin\subset\inter(\Tau_1)$,
  $\min\limits_{x\in t\Tauin} u(x)\to\theta$, $t\to\infty$, as $u(x,t)=u(x)$ now. Since $0\in\inter(\Tau_1)$, the latter convergence is obviously possible for $u\equiv\theta$ only.
\end{proof}

\begin{remark}\label{rem:zero_is_in_front}
It is worth noting that, by \eqref{eq:TauTxi}, \eqref{taut=ttau1-dir}, and \eqref{ct=tc}, the assumption $0\in\inter(\Tau_1)$ implies that $c_*(\xi)\geq0$, for all $\xi\in\S$. It means that all traveling waves in all directions have nonnegative speeds only.
\end{remark}

\subsection{Proof of Theorem~\ref{thm:convtotheta}}

We will do as follows. At first, in Proposition~\ref{prop:conv_to_theta_cont_time}, we apply results of \cite{Wei1982} for discrete time, to prove \eqref{convtotheta} for continuous time, provided that $u_0$ is separated from $0$ on a big enough set. Next, in Proposition~\ref{prop:subsolution}, we show that there exists a proper subsolution to \eqref{eq:basic}, which will reach (as we explain thereafter) any needed level after a finite time. Finally, we properly use in Proposition~\ref{useBrandle} the results of \cite{BCF2011}, to prove that the solution to \eqref{eq:basic} will dominate the subsolution after a finite time.

We start with the following Weinberger's result (rephrased in our settings).
Note that \eqref{as:aplus-aminus-is-pos} implies \eqref{as:a+nodeg}, hence, under conditions of Theorem~\ref{thm:convtotheta}, we have by Proposition~\ref{prop:front_is_non-empty}, that $\Tau_T\neq\emptyset$, $T>0$.

\begin{lemma}[{cf. \cite[Theorem~6.2]{Wei1982}}] \label{lem:conv_to_theta-1}
Let \eqref{as:chiplus_gr_m}, \eqref{as:aplus_gr_aminus}, \eqref{as:aplus-aminus-is-pos}, \eqref{boundedkernels}, \eqref{expradialmoment} hold.
Let $u_0\in \Utheta$ and $T>0$ be arbitrary, and $Q_T$ be given by \eqref{def:Q_T}. Define
\begin{equation}\label{uniterarion}
u_{n+1}(x):=(Q_Tu_n)(x), \quad n\geq0.
\end{equation}
Then, for any compact set $\Tauin_T\subset\inter(\Tau_T)$ and for any $\sigma\in(0,\theta)$, one can choose a radius $r_\sigma=r_\sigma(Q_T,\Tauin_T)$, such that
\begin{equation}\label{initcondissepfrom0}
  u_0(x)\geq\sigma, \quad x\in B_{r_\sigma}(0),
\end{equation}
implies
\begin{equation}\label{liminfconv}
  \lim_{n\rightarrow\infty}\min\limits _{x\in n\Tauin_T}u_{n}(x)=\theta.
\end{equation}
\end{lemma}
\begin{remark}\label{justAradius}
By the proof of \cite[Theorem~6.2]{Wei1982}, the radius $r_\sigma(Q_T,\Tauin_T)$ is not defined uniquely. In the sequel, $r_\sigma(Q_T,\Tauin_T)$ means just a~radius which fulfills the assertion of Lemma~\ref{lem:conv_to_theta-1} for the chosen $Q_T$ and $\Tauin_T$, rather than a~function of~$Q_T$~and $\Tauin_T$.
\end{remark}
\begin{remark}
It is worth noting, that, by \eqref{def:Q_T} and the uniqueness of the solution to \eqref{eq:basic}, the iteration \eqref{uniterarion} is just given by
\begin{equation}\label{realiteration}
  u_{n}(x)=u(x,nT), \quad x\in\X, n\in\N\cup\{0\}.
\end{equation}
Therefore, \eqref{liminfconv} with $T=1$ yields \eqref{convtotheta}, for $\N\ni t\to\infty$, namely,
\begin{equation}\label{convtothetanatural}
  \lim_{n\to\infty}\min_{x\in n\Tauin}u(x,n)=\theta,
\end{equation}
provided that \eqref{initcondissepfrom0} holds with $r_\sigma=r_\sigma(Q_1,\Tauin)$, $\Tauin\subset \inter(\Tau_1)$.
\end{remark}

\begin{lemma}
Let the conditions of Theorem~\ref{thm:convtotheta} hold. Fix a $\sigma\in(0,\theta)$ and a compact set $\Tauin\subset \inter(\Tau_1)$. Let $u_0\in\Utheta$ be such that $u_0(x)\geq\sigma$, $x\in B_{r_{\sigma}(Q_{1},\Tauin)}(0)$. Then, for any $k\in\N$,
\begin{equation}
\lim_{n\to\infty}\min_{x\in \frac{n}{k}\Tauin}u\Bigl(x,\dfrac{n}{k}\Bigr)=\theta.
\label{eq:rsigma1_geq_rsigma05:i}
\end{equation}
\end{lemma}
\begin{proof}
Since $\Tauin\subset\inter(\Tau_1)$, one can choose a compact
set $\tilde{\Tauin}\subset\inter(\Tau_1)$ such that
\begin{equation}\label{eq:rsigma1_geq_rsigma05:ii}
\Tauin\subset\inter(\tilde{\Tauin}).
\end{equation}
By \eqref{realiteration} and Lemma \ref{lem:conv_to_theta-1} (with $T=1$), the assumption
$u_0(x)\geq\sigma$, $x\in B_{r_{\sigma}(Q_{1},\Tauin)}(0)$ implies \eqref{convtothetanatural}. Fix $k\in\N$, take $p=\frac{1}{k}$; then choose and fix the radius $r_{\sigma}\bigl(Q_{p},p\tilde{\Tauin}\bigr)$. By~\eqref{convtothetanatural}, there exists an~$N=N(k)\in\N$, such that
\begin{gather*}
u(x,N)\geq\sigma,\quad x\in N\Tauin, \\
 B_{r_{\sigma}(Q_{p},p\tilde{\Tauin})}(0)\subset N\Tauin.
\end{gather*}
Apply now Lemma~\ref{lem:conv_to_theta-1}, with $u_0(x)=u(x,N)$, $x\in\X$, $T=p$, and
\[
\Tauin_T=\Tauin_p:=p\tilde{\Tauin}\subset p\,\inter(\Tau_1)=\inter(\Tau_p),
\]
as, by \eqref{TauT=TTau1}, $p\Tau_1=\Tau_p$. We will get then
\begin{equation}\label{eq:rsigma1_geq_rsigma05:iii}
\lim_{n\to\infty}\min_{x\in np\tilde{\Tauin}}u(x,N+np)=\theta.
\end{equation}
By (\ref{eq:rsigma1_geq_rsigma05:ii}), there exists $M\in\N$ such
that one has
\begin{equation}\label{inclwithn}
  \Bigl(\frac{N}{n}+p\Bigr)\Tauin\subset p\tilde{\Tauin},\quad n\geq M.
\end{equation}
Therefore, by \eqref{inclwithn}, one gets, for $n\geq M$,
\begin{align}\notag
\min_{x\in np\tilde{\Tauin}}u(x,N+np)&\leq
\min_{x\in n(\frac{N}{n}+p)\Tauin}u(x,N+np)\\&=
\min_{x\in(Nk+n)\frac{1}{k}\Tauin}u\Bigl(x,(Nk+n)\frac{1}{k}\Bigr)\leq\theta.\label{eq:rsigma1_geq_rsigma05:iv}
\end{align}
By \eqref{eq:rsigma1_geq_rsigma05:iii} and \eqref{eq:rsigma1_geq_rsigma05:iv}, one gets the statement.
\end{proof}

Now, one can prove Theorem~\ref{thm:convtotheta}, under assumption on the initial condition.
\begin{proposition}\label{prop:conv_to_theta_cont_time}
Let the conditions of Theorem~\ref{thm:convtotheta} hold. Fix a $\sigma\in(0,\theta)$ and a compact set $\Tauin\subset \inter(\Tau_1)$. Let $u_0\in\Utheta$ be such that $u_0(x)\geq\sigma$, $x\in B_{r_{\sigma}(Q_{1},\Tauin)}(0)$, and $u\in\Xinf$ be the corresponding solution to \eqref{eq:basic}. Then
\begin{equation}\label{convtothetaunderassumtion}
  \lim_{t\to\infty}\min_{x\in t\Tauin}u(x,t)=\theta.
\end{equation}
\end{proposition}
\begin{proof}
Suppose \eqref{convtothetaunderassumtion} were false.
Then, there exist $\eps>0$ and a sequence $t_N\to\infty$, such that
$\min\limits_{x\in t_N\Tauin}u(x,t_N)<\theta-\eps$, $n\in\N$.
Since $t_N\Tauin$ is a compact set and $u(\cdot,t)\in\Utheta$, $t\geq0$, there exists $x_N\in t_N\Tauin$, such that
\begin{equation}\label{eq:conv_to_theta:i}
u(x_N,t_N)<\theta-\eps, \quad n\in\N.
\end{equation}
Next, by~Proposition~\ref{prop:u_in_BUC}, there exists a $\delta=\delta(\eps)>0$
such that, for all $x',x''\in\X$ and for all $t', t''>0$, with $|x'-x''|+|t'-t''|<\delta$,
one has
\begin{equation}\label{eq:conv_to_theta:ii}
|u(x',t')-u(x'',t'')|<\dfrac{\eps}{2}.
\end{equation}
Since $\Tauin$ is a compact, $p(\Tauin):=\sup\limits_{x\in\Tauin}\lVert x\rVert<\infty$. Choose $k\in\N$, such that $\frac{1}{k}<\frac{\delta}{1+p(\Tauin)}$.
By~\eqref{eq:rsigma1_geq_rsigma05:i}, there exists $M(k)\in\N$, such that, for all $n\geq M(k)$, \begin{equation}\label{conseqofle}
  u\Bigl(x,\frac{n}{k}\Bigr)>\theta-\frac{\eps}{2}, \quad x\in \frac{n}{k}\Tauin.
\end{equation}
Choose $N>N_0$ big enough to ensure $t_N>\frac{M(k)}{k}$.
Then, there exists $n\geq M(k)$, such that $t_N\in\bigl[\frac{n}{k},\frac{n+1}{k}\bigr)$. Hence
\begin{equation}\label{sadas}
  \Bigl\lvert t_N-\frac{n}{k}\Bigr\rvert <\frac{1}{k}<\frac{\delta}{1+p(\Tauin)}.
\end{equation}
Next, for the chosen $N$, there exists $y_N\in\Tauin$, such that $x_N=t_N y_N$.
Set $t'=t_N$, $t''=\frac{n}{k}$, $x'=x_N=t_N y_N$, and $x''=\frac{n}{k}y_N$.
Then, by \eqref{sadas},
\[
|t'-t''|+|x'-x''|=\Bigl\lvert t_N-\frac{n}{k}\Bigr\rvert\bigl(1+|y_N|\bigr)<\delta.
\]
Therefore, one can apply \eqref{eq:conv_to_theta:ii}. Combining this with \eqref{eq:conv_to_theta:i}, one gets
\[
u\Bigl(\frac{n}{k}y_N,\frac{n}{k}\Bigr)=
u\Bigl(\frac{n}{k}y_N,\frac{n}{k}\Bigr)-u(t_N y_N,t_N)+
u(x_N,t_N)<\frac{\eps}{2}+\theta-\eps=\theta-\frac{\eps}{2},
\]
that contradicts \eqref{conseqofle}, as $\frac{n}{k} y_N\in\frac{n}{k}\Tauin$.
Hence the statement is proved.
\end{proof}

Next two statements will allow us to get rid the restriction on $u_0$ in Proposition~\ref{prop:conv_to_theta_cont_time}.

\begin{proposition}\label{prop:subsolution}
  Let \eqref{as:chiplus_gr_m}, \eqref{as:aplus_gr_aminus}, and \eqref{as:aplus-aminus-is-pos} hold; assume also that \eqref{firstglobalmoment} holds, and $\m$ is given by \eqref{firstfullmoment}.
  Then there exists $\alpha_0>0$, such that, for all $\alpha\in(0,\alpha_0)$, there exists
  $q_0=q_0(\alpha)\in(0,\theta)$, such that there exists $T=T(\alpha,q_0)>0$, such that, for all $q\in(0,q_0)$, the function
   \begin{equation}\label{greatsubsol}
     w(x,t)=q\exp\biggl(-\frac{|x- t \m   |^2}{\alpha t}\biggr), \quad x\in\X, t>T,
\end{equation}
is a subsolution to \eqref{eq:basic} on $t>T$; i.e. $\mathcal{F} w(x,t)\leq 0$, $x\in\X$, $t>T$, where $\mathcal{F}$ is given by \eqref{Foper}.
\end{proposition}
\begin{proof}
Let $J_q$, $q\in(0,\theta)$ be given by \eqref{Jq}, and consider the function \eqref{greatsubsol}. Since $w(x,t)\leq q$, we have from \eqref{Foper}, that
\begin{align}\notag
    (\mathcal{F}w)(x,t)&=
    w(x,t)\biggl(\frac{|x|^2}{\alpha t^2}-\frac{|\m   |^2}{\alpha}\biggr) - \ka^+(a^+*w)(x,t)
    \\&\quad+\ka^-w(x,t) (a^-*w)(x,t)+mw(x,t) \notag
    \\&\leq w(x,t)\biggl(\frac{|x|^2}{\alpha t^2}-\frac{|\m   |^2}{\alpha}\biggr) - (J_{q}*w)(x,t)+mw(x,t).\label{suffineq}
  \end{align}
Since, for any $q_0\in(0,\theta)$ and for any $q\in(0,q_0)$, $J_q(x)\geq J_{q_0}(x)$, $x\in\X$, one gets from \eqref{suffineq}, that, to have $\mathcal{F}w\leq 0$, it is enough to claim that, for all $x\in\X$,
\[
m+\frac{|x|^2}{\alpha t^2}-\frac{|\m   |^2}{\alpha}\leq
\exp\biggl(\frac{|x-t \m   |^2}{\alpha t}\biggr)
\int_\X J_{q_0}(y)\exp\biggl(-\frac{|x-y-t \m   |^2}{\alpha t}\biggr)\,dy.
\]
By changing $x$ onto $x+t \m   $ and a simplification, one gets an equivalent inequality
\begin{equation}\label{ineqtoprove}
 m+ \frac{|x|^2}{\alpha t^2}+\frac{2\, x\cdot\m   }{\alpha t}\leq
\int_\X J_{q_0}(y)\exp\biggl(\frac{2x\cdot y}{\alpha t}\biggr)\exp\biggl(-\frac{|y|^2}{\alpha t}\biggr)\,dy=:I(t).
\end{equation}
One can rewrite $I(t)=I_0(t)+I^+(t)+I^-(t)$, where
\begin{gather*}
I_0(t):=\int_\X J_{q_0}(y)e^{-\frac{|y|^2}{\alpha t}}dy;\qquad\quad
I^+(t):=\int_{x\cdot y\geq0} J_{q_0}(y)e^{-\frac{|y|^2}{\alpha t}}\Bigl(e^{\frac{2x\cdot y}{\alpha t}}-1\Bigr)dy;\\
I^-(t):=\int_{x\cdot y<0} J_{q_0}(y)e^{-\frac{|y|^2}{\alpha t}}\Bigl(e^{\frac{2x\cdot y}{\alpha t}}-1\Bigr)dy.
\end{gather*}
Using that $e^s-1\geq s$, for all $s\in\R$, and $e^s-1\geq s+\frac{s^2}{2}$, for all $s\geq0$, one gets the following estimates
\[
I^+(t)\geq \frac{2}{\alpha t}\int_{x\cdot y\geq0} J_{q_0}(y)e^{-\frac{|y|^2}{\alpha t}}(x\cdot y)dy+\frac{2}{\alpha^2 t^2}\int_{x\cdot y\geq0} J_{q_0}(y)e^{-\frac{|y|^2}{\alpha t}}(x\cdot y)^2dy
\]
and
\[
I^-(t)\geq\frac{2}{\alpha t}\int_{x\cdot y<0} J_{q_0}(y)e^{-\frac{|y|^2}{\alpha t}}(x\cdot y)dy.
\]
Therefore,
\begin{multline}
I(t)\geq I_0(t) + \frac{2}{\alpha t}\Bigl(x\cdot \int_{\X} J_{q_0}(y)e^{-\frac{|y|^2}{\alpha t}} ydy\Bigr)\\+\frac{2}{\alpha^2 t^2}\int_{x\cdot y\geq0} J_{q_0}(y)e^{-\frac{|y|^2}{\alpha t}}(x\cdot y)^2dy.\label{estIfrombelow}
\end{multline}
By the dominated convergence theorem,
\begin{equation}\label{con1eret}
  I_0(t)\nearrow \int_\X J_{q_0}(x)\,dx=\ka^+-q_0\ka^->m, \quad t\to\infty,
\end{equation}
for any $q_0\in(0,\theta)$. Set also
\[
I_1(t):=\int_{\X} J_{q_0}(y)e^{-\frac{|y|^2}{\alpha t}} y \,dy.
\]
By \eqref{firstglobalmoment} and \eqref{as:aplus_gr_aminus}, one has $\int_\X a^-(x)|x|\,dx<\infty$ and hence $\int_\X J_{q_0}(x)|x|\,dx<\infty$.
Then, by the dominated convergence theorem,
\begin{equation}\label{con2eret}
I_1(t)\to\int_{\X} J_{q_0}(y) ydy=:\mu(q_0) \in\X, \quad t\to\infty.
\end{equation}
Since $0\leq J_{q_0}(x)\leq \ka^+ a^+(x)$, $x\in\X$, we have, by \eqref{firstfullmoment} and
the dominated convergence theorem, that $m(q_0)\to\m$, $q_0\to0$.

For any $\eps>0$ with $m+2\eps<\ka^+$, one can choose $q_0=q_0(\eps)\in(0,\theta)$, such that
\begin{equation}\label{qwrrwsfasa}
  \ka^+>\ka^+-q_0\ka^->m+2\eps, \qquad |\m-\mu(q_0)|<\frac{\eps}{2}.
\end{equation}
By \eqref{con1eret}, \eqref{con2eret}, there exists $T_1=T_1(\eps,q_0)>0$, such that, for all $\alpha>0$ and $t>0$ with $\alpha t>T_1$, one has, cf.~\eqref{qwrrwsfasa},
\begin{equation}\label{estlim}
\ka^+\geq I_0(t)>m+\eps, \qquad \lvert I_1(t)-\mu(q_0)   \rvert<\frac{\eps}{2}.
\end{equation}

Let $T>\frac{T_1}{\alpha}$ be chosen later. The function
\[
I_2(t):=\int_{x\cdot y\geq0} J_{q_0}(y)e^{-\frac{|y|^2}{\alpha t}}(x\cdot y)^2dy
\]
is also increasing in $t>0$.
Clearly, from \eqref{qwrrwsfasa} and \eqref{estlim}, one has $|I_1(t)-\m|<\eps$.
Therefore, by \eqref{estIfrombelow} and \eqref{estlim}, one gets, for $t>T>\frac{T_1}{\alpha}$,
\begin{align}
I(t)&>m+\eps +\frac{2}{\alpha t}x\cdot(I_1(t)-\m   )+\frac{2}{\alpha t}x\cdot\m
+\frac{2}{\alpha^2 t^2}I_2(t) \notag\\
& \geq m+\eps -\frac{2\eps}{\alpha t}|x|+\frac{2}{\alpha t}x\cdot\m
+\frac{2}{\alpha^2 t^2}I_2(T).\label{ineqenoughtoprove2}
\end{align}

Next, by \eqref{as:aplus_gr_aminus}, \eqref{as:aplus-aminus-is-pos}, and \eqref{Jq},
$J_{q_0}(y)\geq\rho$, for a.a. $y\in B_\delta(0)$. For an arbitrary $x\in\X$, consider the set
\[
B_x=\Bigl\{y\in\X \Bigm\vert |y|\leq \delta, \frac{1}{2}\leq\frac{x\cdot y}{|x||y|}\leq 1\Bigr\}.
\]
Then
\begin{equation}\label{weneedit}
I_2(T)\geq \frac{\rho}{4}|x|^2 \int_{B_x} |y|^2 e^{-\frac{|y|^2}{\alpha T}} dy.
\end{equation}
The set $B_x$ is a cone inside the ball $B_\delta(0)$, with the apex at the origin, the height which lies along $x$, and the apex angle $2\pi/3$. Since the function inside the integral in the r.h.s. of \eqref{weneedit} is radially symmetric, the integral does not depend on $x$. Fix an arbitrary $\bar{x}\in\X$ and denote
\begin{equation}\label{Alimit}
  A(\tau)=A(\tau,\delta)=\int_{B_{\bar{x}}} |y|^2 e^{-\frac{|y|^2}{\tau}} dy\nearrow\int_{B_{\bar{x}}} |y|^2 dy=:\bar B_\delta, \quad \tau\to\infty.
\end{equation}
Then, by \eqref{ineqenoughtoprove2} and \eqref{weneedit}, one has, for $t>T$,
\begin{equation}
I(t)> m+\eps -\frac{2\eps}{\alpha t}|x|+\frac{2}{\alpha t}x\cdot\m
+\frac{\rho A(\alpha T)}{2\alpha^2 t^2}|x|^2 .\label{ineqenoughtoprove3}
\end{equation}
By \eqref{ineqenoughtoprove3}, to prove \eqref{ineqtoprove}, it is enough to show that
\begin{equation*}
  \eps -\frac{2\eps}{\alpha t}|x|
+\frac{\rho A(\alpha T)}{2\alpha^2 t^2}|x|^2\geq  \frac{|x|^2}{\alpha t^2}, \qquad t>T, \ x\in\X,
\end{equation*}
or, equivalently, for $2\alpha<\rho  A(\alpha T)$,
\begin{multline}\label{ineqenoughtoprove}
\biggl(\sqrt{\frac{\rho A(\alpha T)-2\alpha}{2 }}\frac{|x|}{\alpha t}-
 \eps\sqrt{\frac{2}{\rho A(\alpha T)-2\alpha}}\biggr)^2\\+\eps-\eps^2\frac{2}{\rho A(\alpha T)-2\alpha}\geq0.
\end{multline}
To get \eqref{ineqenoughtoprove}, we proceed as follows. For a given $\rho>0$, $\delta>0$ which provide \eqref{as:aplus-aminus-is-pos}, we set $\alpha_0:=\frac{1}{2}\rho\bar{B}_\delta$, cf.~\eqref{Alimit}.
Then, for any $\alpha\in(0,\alpha_0)$, there exists $T_2=T_2(\alpha)>0$, such that
\[
2\alpha<\rho A(\alpha T_2)<\rho\bar{B}_\delta.
\]
Choose now $\eps=\eps(\alpha)>0$, such that $m+2\eps<\ka^+$ and
\begin{equation}\label{smalleps}
  \eps<\frac{1}{2}(\rho A(\alpha T_2)-2\alpha)<
\frac{1}{2}(\rho A(\alpha T)-2\alpha), \quad T>T_2.
\end{equation}
For the chosen $\eps$, find $q_0=q_0(\alpha)\in(0,\theta)$ which ensures
 \eqref{qwrrwsfasa}. Then, find $T_1=T_1(\alpha,q_0)>0$ which gives \eqref{estlim}; and, finally, take $T=T(\alpha,q)>T_2$ such that $\alpha T> T_1$. As a result, for $t>T$, one has $\alpha t>\alpha T>T_1$, thus
\eqref{estlim} holds, whereas \eqref{smalleps} yields \eqref{ineqenoughtoprove}. The latter inequality gives \eqref{ineqtoprove}, and hence, for all $q\in(0,q_0)$, $\mathcal{F}w\leq 0$, for $w$ given by \eqref{greatsubsol}. The statement is proved.
\end{proof}

\begin{proposition}\label{useBrandle}
Let \eqref{as:chiplus_gr_m}, \eqref{as:aplus_gr_aminus}, and \eqref{as:aplus-aminus-is-pos} hold. Then, there exists $t_1>0$, such that, for any $t>t_1$ and for any $\tau>0$, there exists $q_1=q_1(t,\tau)>0$, such that the following holds. If $u_0\in\Ltheta$ is such that there exist $\eta>0$, $r >0$, $x_0\in\X$ with
$u_0(x)\geq\eta$, $x\in B_r (x_0)$ and $u\in\tXinf$ is the corresponding solution to \eqref{eq:basic}, then
\begin{equation}
u(x,t)\geq q_1e^{-\frac{|x-x_0|^2}{\tau}},\quad x\in\X.\label{onegetsfromBr}
\end{equation}
\end{proposition}
\begin{proof}
At first, we note that \eqref{onegetsfromBr} may be rewritten as follows:
\[
q_1e^{-\frac{|x|^2}{\tau}}\leq u(x+x_0,t_0)=T_{-x_0}Q_{t_0}u_0(x)=Q_{t_0}T_{-x_0}u_0(x),
\]
cf.~\eqref{def:Q_T}, \eqref{shiftoper}, \eqref{eq:QTy=TyQ}, and one has
\[
T_{-x_0}u_0(x)=u_0(x+x_0)\geq \eta, \quad \lvert (x+x_0)-x_0\rvert=\lvert x\rvert\leq r.
\]
Therefore, it is enough to prove the statement for $x_0=0$.

Consider now arbitrary functions $b,v_0\in C^\infty(\X)$, such that
 \begin{align*}
 \supp b =B_\delta(0), &\qquad 0 < b(x)=b(|x|) \leq \rho, && x\in \inter(B_\delta(0));\\
 \supp v_0=B_r(0), &\qquad 0 < v_0(x) \leq \eta, && x\in \inter(B_r(0));\\
 \exists \, 0<p<\min\{r,1\},\, 0<\nu<\eta, &\qquad \text{such that} \ v_0(x)\geq \nu, && x\in B_{p }(0),
 \end{align*}
where $\rho$ and $\delta$ are the same as in \eqref{as:aplus-aminus-is-pos}.
Set $\langle b\rangle:=\int_\X b(x)\,dx>0$. Define two bounded operators in the space $L^\infty(\X)$, cf.~\eqref{jump}: $Bu=b*u$, $L_bu=Bu-\langle b\rangle u$.
One can rewrite \eqref{eq:basic} as follows
\begin{align*}
\frac{\partial}{\partial t} u(x,t)&=(J_\theta*u)(x,t)-mu(x,t)+\ka^-(\theta-u(x,t))(a^-*u)(x,t)
\\&=(b*u)(x,t)-mu(x,t)+f(x,t),
\end{align*}
where, for any $x\in\X$, $t\geq0$,
\[
f(x,t):=((J_\theta-b)*u)(x,t)+\ka^-(\theta-u(x,t))(a^-*u)(x,t).
\]
By \eqref{as:aplus-aminus-is-pos} and the choice of $b$, $J_\theta(x)\geq b(x)$, $x\in\X$. In particular, $m=\int_\X J_\theta(x)\,dx\geq\langle b\rangle$, and $f(x,t)\geq0$, $x\in\X$, $t\geq0$. Next, for any $t\geq0$,
$\lVert f(\cdot,t)\rVert_\infty\leq\theta (m-\langle b\rangle)+\ka^-\theta^2<\infty$. Since $b\geq0$ and $Bu=b*u$ defines a bounded operator on $L^\infty(\X)$, one has that $e^{t B}f(x,s)\geq0$, for all $t,s\geq0$, $x\in\X$. By the same argument, $u_0(x)\geq \eta \1_{B_r(0)}(x)\geq v_0(x)\geq0$ implies
$(e^{tB}u_0)(x)\geq (e^{tB}v_0)(x)$.
Therefore,
\begin{align}\notag
u(x,t)&=e^{-tm}(e^{tB}u_0)(x)+\int_0^t e^{-(t-s)m}(e^{(t-s)B}f)(x,s)ds\\
&\geq e^{-tm}(e^{tB}u_0)(x)
\geq e^{-(m-\langle b\rangle)t}(e^{tL_b}v_0)(x),\quad x\in\X.\label{estbelow1}
\end{align}

We are going to apply now the results of \cite{BCF2011}. To do this, set $\beta:=\langle b\rangle ^{-1}$. Then
\begin{equation}\label{asfsaffsa}
(e^{tL_b}v_0)(x)=(e^{\langle b\rangle t(\beta L_b)}v_0)(x)=v(x,\langle b\rangle t),
\end{equation}
where $v$ solves the differential equation $\frac{d}{dt}v=\beta L_b$.
Since $\int_\X \beta b(x)\,dx=1$, then, by \cite[Theorem~2.1, Lemma~2.2]{CCR2006},
\begin{equation}\label{soltolinear}
  v(x,t)=e^{-t}v_0(x)+(w*v_0)(x,t),
\end{equation}
where $w(x,t)$ is a smooth function. Moreover, by \cite[Proposition 5.1]{BCF2011}, for any $\omega\in (0,\delta)$ there exist $c_1=c_1(\omega)>0$ and $c_2=c_2(\omega)\in\R$, such that
\begin{equation}\label{lowbound}
\begin{split}
w(x,t)&\geq h(x,t), \quad x\in\X, t\geq0,\\
h(x,t)&:=c_1t\exp\Bigl(-t-\frac{1}{\omega}|x|\log|x|+(\log t-c_2)\Bigl[\frac{|x|}{\omega}\Bigr]\Bigr).
\end{split}
\end{equation}
Here $[\alpha]$ means the entire part of an $\alpha\in\R$, and $0\log 0:=1$, $\log 0:=-\infty$.

Set $t_1=e^{c_2}>0$. Since $[\alpha]>\alpha-1$, $\alpha\in\R$, one has, for $t>t_1$,
\[
h(x,t)\geq c_1e^{c_2}\exp\Bigl(-t-\frac{1}{\omega}|x|\log|x|+(\log t-c_2)\frac{|x|}{\omega}\Bigr)\geq c_3g(x,t),
\]
where $c_3=c_1e^{c_2}>0$ and
\[
g(x,t):=\exp\Bigl(-t-\frac{1}{\omega}|x|\log|x|\Bigr), \quad x\in\X, t>t_1.
\]
Since $v_0\geq\nu\1_{B_{p }(0)}$, one gets from \eqref{soltolinear} and \eqref{lowbound}, that
\begin{equation}\label{aswgddsye}
   v(x,t)\geq \nu e^{-t}\1_{B_{p }(0)}(x)+\nu c_3\int_{B_{p }(x)}g(y,t)\,dy
\end{equation}
Set $V_p :=\int_{B_p(0)}\,dx$. For any fixed $t>t_1$, since $g(\cdot,t)\in C(B_p(x))$, there exists $y_0,y_1\in B_p(x)$, such that $g(y,t)$ attains its minimal and maximal values on $B_p(x)$ at these points, respectively. Since $B_p(x)$ is a convex set, one gets that, for any $\gamma\in(0,1)$, $y_\gamma:=\gamma y_1+(1-\gamma)y_0\in B_p(x)$. Then
\[
V_p g(y_0,t)\leq \int_{B_{p }(x)}g(y_\gamma,t)\,dy\leq V_p g(y_1,t).
\]
Therefore, by the intermediate value theorem there exists, $\tilde{y}_t=\tilde{y}(x,t)\in B_p (x)$, $t>t_1$, $x\in\X$, such that $\int_{B_p(x)}g(y,t)\,dy=V_p  g(\tilde{y}_t,t)$. Hence one gets from \eqref{estbelow1}, \eqref{asfsaffsa}, \eqref{aswgddsye}, that
\begin{equation}
  u(x,t)\geq c_4 e^{-(m-\langle b\rangle)t}g\bigl(\tilde{y}_t,\langle b\rangle t\bigr)
  =c_4
  \exp\Bigl(-mt-\frac{1}{\omega}|\tilde{y}_t|\log|\tilde{y}_t|\Bigr),\label{estbelow2}
\end{equation}
for $\tilde{y}_t=\tilde{y}(x,t)\in B_p (x)$, $t>t_1$; here $c_4=c_3 \nu V_p>0$.

As a result, to get the statement, it is enough to show that, for any $t>t_1$ and for any $\tau>0$, there exists $q_1=q_1(t,\tau)>0$, such that the r.h.s. of \eqref{estbelow2} is estimated from below by $q_1 e^{-\frac{|x|^2}{\tau}}$, i.e. that
\begin{equation}\label{needtoproveineq2}
mt+\frac{1}{\omega}|\tilde{y}_t|\log|\tilde{y}_t|-\log c_4\leq \frac{|x|^2}{\tau}-\log q_1, \quad x\in\X,
\end{equation}
Note that $\tilde{y}_t\in B_p (x)$ implies $|\tilde{y}_t|\leq p+|x|$, $x\in\X$.

Let $p+|x|\leq 1$. Then $\log |\tilde{y}_t|\leq 0$, and the l.h.s. of \eqref{needtoproveineq2} is majorized by $mt-\log c_4$. Therefore, to get \eqref{needtoproveineq2}, it is enough to have $q_1< c_4e^{-mt}$, regardless~of~$\tau$.

Let now $|x|+p >1$. Recall that we chose $p<1$. The function $s\log s$ is increasing on $s>1$. Hence to get \eqref{needtoproveineq2}, we claim
\begin{equation}\label{sasfaasftewtew}
  (|x|+1)\log(|x|+1)\leq \frac{\omega}{\tau}|x|^2-\omega mt+\omega\log c_4-\omega\log q_1.
\end{equation}
Consider now the function $f(s)=as^2-(s+1)\log(s+1)$, $s\geq0$, $a=\frac{\omega}{\tau}>0$. Then $f(0)=0$, $f'(s)=2as-\log(s+1)-1$, $f'(0)=-1$, $f''(s)=2a-\frac{1}{s+1}$. Since $f''(s)\nearrow 2a>0$, $s\to\infty$, there exists $s_0>0$, such that $f''(s)>0$, for all $s>s_0$, i.e. $f'(s)$ increases on $s>s_0$. Since $f'(s)\to\infty$, $s\to\infty$, there exists $s_1>s_0$, such that $f'(s)>0$, for all $s>s_1$, i.e. $f$ is increasing on $s>s_1$. Finally, for any $t>t_1$, one can choose $q_1=q_1(t,\tau)>0$ small enough, to get
\[
\min\limits_{s\in[0,s_1]}f(s)-\omega mt+\omega\log c_4-\omega\log q_1>0
\]
and to fulfill \eqref{sasfaasftewtew}, for all $x\in\X$.
The statement is proved.
\end{proof}

Now, we are ready to prove the main Theorem~\ref{thm:convtotheta}.

\begin{proof}[Proof of Theorem~\ref{thm:convtotheta}]
For $u_0\equiv\theta$, the statement is trivial. Hence let $u_0\not\equiv\theta$, $u_0\not\equiv0$. Next, recall that, \eqref{as:aplus-aminus-is-pos} implies
  \eqref{as:a+nodeg} and \eqref{expradialmoment} implies \eqref{firstglobalmoment}. Therefore, one may use the statements of Propositions~\ref{prop:front_is_non-empty},
  \ref{prop:subsolution}, \ref{useBrandle}.

  According to Proposition~\ref{prop:subsolution}, choose any $\alpha\in(0,\alpha_0)$ and take the corresponding $q_0=q_0(\alpha)\in(0,\theta)$ and $T=T(\alpha,q_0)>0$. Choose then arbitrary $t_2>T$. Let $\m$ be given by \eqref{firstfullmoment}. Set $x_0=t_2\m\in\X$. By~Proposition~\ref{prop:u_gr_0}, there exist $\eta=\eta(t_2)>0$ and $r=r(t_2)>0$, such that $u(x,t_2)\geq\eta$, $|x-x_0|=|x-t_2\m|\leq r$. Apply now Proposition~\ref{useBrandle}, with $u_0(x)=u(x,t_2)$; let $t_1$ be the moment of time stated there. Take, for the $\alpha$ chosen above, $\tau=\alpha t_2>0$. Take any $t_3>\max\{t_1,t_2\}$ and the corresponding $q_1=q_1(t_3,\tau)>0$. We will get then, by~\eqref{onegetsfromBr}, that
\begin{equation}
u(x,t_3+t_2)\geq q_1\exp\Bigl(-\frac{|x-t_2\m|^2}{\alpha t_2}\Bigr),\quad x\in\X.\label{onegetsfromBrforproof}
\end{equation}
Of course, one can assume that $q_1<q_0$ (otherwise, we just pass to a weaker inequality in \eqref{onegetsfromBrforproof}).

We are going to apply now Theorem~\ref{thm:compar_pr}, with $c=\theta$ and, for $t\geq0$,
\[
   \begin{split}
    u_1(x,t)&=q_1\exp\Bigl(-\frac{|x-(t+t_2)\m|^2}{\alpha (t+t_2)}\Bigr)\geq0,  \\
    u_2(x,t)&=u(x,t+t_3+t_2)\in[0,\theta].
    \end{split}
\]
By \eqref{onegetsfromBrforproof}, $u_1(x,0)\leq u_2(x,0)$, $x\in\X$. Since $u$ solve \eqref{eq:basic}, $\mathcal{F}u_2\equiv0$. Next, by Proposition~\ref{prop:subsolution},
if we set $q=q_1$, we will have $\mathcal{F}u_1\leq 0$, as $t+t_2\geq t_2>T$. Therefore, by Theorem~\ref{thm:compar_pr},
\begin{equation*}
  u(x,t+t_3+t_2)\geq q_1\exp\Bigl(-\frac{|x-(t+t_2)\m|^2}{\alpha (t+t_2)}\Bigr), \quad t\geq0, x\in\X,
\end{equation*}
or, equivalently,
\begin{equation*}
  u(x+(t+t_2)\m,t+t_3+t_2)\geq q_1\exp\Bigl(-\frac{|x|^2}{\alpha (t+t_2)}\Bigr), \quad t\geq0,\ x\in\X,
\end{equation*}

Let now $\K\subset\inter(\Tau_1)$ be a compact set. Choose any $\sigma\in(0,q_1)$ and consider a radius $r_\sigma=r_\sigma(Q_1,\K)$ which fulfills Proposition~\ref{prop:conv_to_theta_cont_time}, cf.~Remark~\ref{justAradius}. Then $|x|\leq r_\sigma$ implies that there exists $t_4=t_4(\sigma,\K)>0$, such that, for all $t\geq t_4$,
\[
q_1\exp\Bigl(-\frac{|x|^2}{\alpha (t+t_2)}\Bigr)\geq
q_1\exp\Bigl(-\frac{r_\sigma^2}{\alpha (t+t_2)}\Bigr)>\sigma.
\]
Then, one can apply Proposition~\ref{prop:conv_to_theta_cont_time} with
$u_0(x)=u(x+(t_4+t_2)\m,t_4+t_3+t_2)$, $x\in\X$; by~\eqref{convtothetaunderassumtion}, we have
\begin{equation}\label{convtothetaunderassumtion231}
  \lim_{t\to\infty}\min_{x\in t\K}u(x+(t_2+t_4)\m,t+t_2+t_3+t_4)=\theta.
\end{equation}

Let, finally, $\Tauin\subset\inter(\Tau_1)$ be an arbitrary compact set from the statement of Theorem~\ref{thm:convtotheta}. It is well-known, that the distance between disjoint compact and closed sets is positive; in particular, one can consider the compact $\Tauin$ and the closure of $\X\setminus\Tau_1$. Therefore, there exists a compact set $\K\subset\inter(\Tau_1)$, such that $\Tauin\subset\inter(\K)$. Let $\delta_0>0$ be the distance between $\Tauin$ and the closure of $\X\setminus\K$.
One has then that \eqref{convtothetaunderassumtion231} does hold with $t_4=t_4(\sigma,\K)>0$.

By \eqref{convtothetaunderassumtion231}, for any $\eps>0$, there exists $t_5>0$ such that, for all $t>t_2+t_3+t_4+t_5=:t_6>0$ and for all $y\in\K$,
\begin{equation}\label{sasasfsfwr}
 u\bigl((t-t_2-t_3-t_4)y+(t_2+t_4)\m,t\bigr)>\theta-\eps
\end{equation}
Without loss of generality we can assume that $t_5$ is big enough to ensure
\begin{equation}\label{bigt5}
  (t_2+t_3+t_4)\max\limits_{x\in\Tauin}|x|+(t_2+t_4)|\m|<\delta_0 t_5.
\end{equation}
Then, for any $x\in\Tauin$ and for any $t>t_6$, the vector
\[
y(x,t):=\frac{tx-(t_2+t_4)\m}{t-t_2-t_3-t_4}
\]
is such that
\[
\lvert y(x,t)-x\rvert
=\frac{\bigl\lvert(t_2+t_3+t_4)x-(t_2+t_4)\m\bigr\rvert}{t-t_2-t_3-t_4}
<\delta_0,
\]
where we used \eqref{bigt5}.
Therefore, $y(x,t)\in\K$, for all $x\in\Tauin$ and $t>t_6$, and hence \eqref{sasasfsfwr}, being applied for any such $y(x,t)$, yields
$u(tx,t)>\theta-\eps$, $x\in\Tauin$, $t>t_6$, that fulfils the proof.
\end{proof}

\subsection{Fast propagation for slow decaying dispersal kernels}
\label{subsec:fastpropag}

All result above about traveling waves and long-time behavior of the solutions were obtained under exponential integrability assumptions, cf. \eqref{aplusexpint1} or \eqref{expradialmoment}. In \cite{Gar2011}, it was proved, for the equation \eqref{nl+l} on $\R$ with local nonlinear term, that the case with $a^+$ which does not satisfy such conditions leads to `accelerating' solutions, i.e. in this case the equality like \eqref{convtotheta} holds for arbitrary big compact $\Tauin\subset\X$. The aim of this Subsection is to show an analogous result for the equation \eqref{eq:basic}. The detailed analysis of the propagation for the slow decaying $a^+$ will be done in a forthcoming paper.

We will prove an analog of the first statement in \cite[Theorem~1]{Gar2011}.

\begin{theorem}\label{thm:infinitespeed}
Let the conditions \eqref{as:chiplus_gr_m}, \eqref{as:aplus_gr_aminus}, \eqref{as:aplus-aminus-is-pos}, \eqref{boundedkernels}, and \eqref{firstglobalmoment} hold.
Suppose also there exists a function $0\leq b\in L^1(\R_+)\cap L^\infty(\R_+)$, such that $a^+(x)\geq b(|x|)$, for a.a. $x\in\X$, and that, cf.~\eqref{expradialmoment}, for any $\la>0$ and for any $\xi\in\S $,
\begin{equation}\label{slowdecaying}
  \int_{\X}b(|x|)e^{\la \, x\cdot\xi}dx=\infty.
\end{equation}
Let $u_0\in \Ltheta$ be such that there exist $x_0\in\X$, $\eta>0$, $r>0$, with $u_0\geq\eta$, for a.a. $x\in B_r(x_0)$. Let $u\in\tXinf$ be the corresponding solution to \eqref{eq:basic}.
Then, for any compact set $\K\subset\X$,
  \begin{equation}\label{convtothetaaccel}
    \lim_{t\to\infty} \essinf_{x\in t\K} u(x,t)=\theta.
  \end{equation}
\end{theorem}
\begin{proof}
By the same arguments as in the proof of Corollary~\ref{cor_essinf}, there exists $v_0\in\Utheta$, $v_0\not\equiv0$, such that $u_0(x)\geq v_0(x)$, for a.a. $x\in\X$, and $u(x,t)\geq v(x,t)$, for a.a. $x\in\X$ and for all $t\geq0$, where $v\in\tXinf$ is the corresponding to $v_0$ solution to \eqref{eq:basic}, moreover, $v\in\Xinf$.

Let $\bar{\theta}\in(0,\theta)$ be chosen and fixed. We are going to apply now Proposition~\ref{lowestsuppCub} to \eqref{trkern}--\eqref{ARdef} with $\Delta_R:=B_R(0)\nearrow\X$, $R\to\infty$.
Consider an increasing sequence $\{R_n\mid n\in\N\}$, such that

(i)~$\delta<R_n\to\infty$, $n\to\infty$, where $\delta$ is the same as in~\eqref{as:aplus-aminus-is-pos};

(ii)~$A^+_{R_n}>\frac{m}{\ka^+}$, $n\in\N$, cf. \eqref{bigR};

(iii)~$\bar{\theta}<\theta_{R_n}\leq\theta$, cf.~\eqref{defofthetaR}, \eqref{thetaRlesstheta}.

Let $w_0\in\Buc$, $w_0\not\equiv0$ be such that $0\leq w_0(x)\leq v_0(x)$, $x\in\X$ and $\lVert w_0\rVert\leq \bar{\theta}$.
Let, for any $n\in\N$, $w^{(n)}\in\Xinf$ be the corresponding solution to the equation \eqref{eq:basic_R} with $R$ replaced by $R_n$. Then, by \eqref{ineqtrunc}, $w^{(n)}(x,t)\leq v(x,t)$, for all $x\in\X$, $t\geq0$, $n\in\N$. As a result,
\begin{equation}\label{ineqtruncseq}
w^{(n)}(x,t)\leq v(x,t)\leq \theta, \quad \text{a.a. } x\in\X, \ t\geq0, \ n\in\N.
\end{equation}

For an arbitrary $\xi\in\S$, consider the corresponding $\check{a}_{R_n}^+$, cf.~\eqref{apm1dim}. Clearly, $\la_0(\check{a}_{R_n}^+)=\infty$, i.e. $a_{R_n}^+\in \Vxi$, $n\in\N$, cf.~Definition~\ref{def:VWxi}. Let $\A_\xi^{(n)}(\la)$, $n\in\N$, $\la>0$ be defined by \eqref{aplusexpla}, with $a^+$ replaced by $a^+_{R_n}$. Finally, let $c^{(n)}_*(\xi)$ be the corresponding minimal traveling wave's speed for the equation \eqref{eq:basic_R} (with $R$ replaced by $R_n$).
Prove that
\begin{equation}\label{mainaimacc}
  \lim_{n\to\infty}\inf\limits_{\xi\in\S}c^{(n)}_*(\xi)=\infty.
\end{equation}

By \eqref{minspeed}, it is enough to show that, cf. \eqref{firstglobalmoment}, for any
\begin{equation}\label{Cbelow}
  C>\ka^+\int_\X a^+(x)|x|dx,
\end{equation}
there exists $N=N(C)\in\N$, such that, for all $\la>0$,
\begin{equation}\label{needtoproveC}
  \frac{1}{\la}\bigl(\ka^+\A_\xi^{(n)}(\la)-m\bigr)\geq C, \qquad  \xi\in\S, \ n\geq N.
\end{equation}

Denote $\Xi^\pm_\xi:=\{x\in\X\mid \pm x\cdot\xi\geq0\}$; i.e. $\Xi^+_\xi\cup\Xi^-_\xi=\X$.
Then, by (ii) above,
\begin{align*}
  \ka^+\A_\xi^{(n)}(\la)-m &=\ka^+ \int_\X a^+_{R_n}(x) (e^{\la x\cdot\xi}-1)dx+\ka^+ A^+_{R_n}-m\\
  & \geq \ka^+ \int_{\Xi^-_\xi} a^+_{R_n}(x) (e^{\la x\cdot\xi}-1)dx+\ka^+ A^+_{R_1}-m,
\end{align*}
as $\int_{\Xi^+_\xi} a^+_{R_n}(x) (e^{\la x\cdot\xi}-1)dx\geq0$. By the inequality $1-e^{-s}\leq s$, $s\geq 0$, one has that
\[
\biggl\lvert \int_{\Xi^-_\xi} a^+_{R_n}(x) (e^{\la x\cdot\xi}-1)dx\biggr\rvert\leq
\la \int_{\Xi^-_\xi} a^+_{R_n}(x) |x\cdot\xi|dx\leq \la \int_\X a^+(x)|x|dx.
\]
Hence, cf. (ii), \eqref{firstglobalmoment}, and \eqref{Cbelow}, if we set
\[
\la_1:=\frac{\ka^+ A^+_{R_1}-m}{2C}>0,
\]
then, for any $\la\in (0,\la_1)$ and for any $\xi\in\S$,
\[
  \ka^+\A_\xi^{(n)}(\la)-m\geq \ka^+ A^+_{R_1}-m-\la_1\ka^+\int_\X a^+(x)|x|dx\geq\frac{\ka^+ A^+_{R_1}-m}{2}>C\la,
\]
i.e. \eqref{needtoproveC} holds.

On the other hand, \eqref{as:aplus-aminus-is-pos} and the condition (i) imply that, for any $n\in\N$, the assumption \eqref{as:a+nodeg-mod} holds with $a^+$ replaced by $a^+_{R_n}$, where $r=0$ and $\rho,\delta$ are the same as in \eqref{as:aplus-aminus-is-pos}, and thus are independent on $n$. Hence, by \eqref{toinfasltoinf},
\[
  \frac{1}{\la}\A^{(n)}_\xi(\la)\geq \rho'\frac{1}{\la^2}(e^{\la \delta'}-1)\to\infty,
\]
for all $n\in\N$, and here $\rho',\delta'$ are independent on $n$ and on $\xi$. Therefore, there exists $\la_2>0$, such that, for all $\la>\la_2$, $\xi\in\S$, $n\in\N$, \eqref{needtoproveC} holds.

Let, finally, $\la\in[\la_1,\la_2]$. Since $a^+_{R_n}$ are compactly supported, one has
\begin{equation}\label{sparing}
  \frac{d}{d\la}\A_\xi^{(n)}(\la)=
  \int_{\Xi_\xi^+} a^+_{R_n}(x) (x\cdot\xi) e^{\la x\cdot\xi}dx+
  \int_{\Xi_\xi^-} a^+_{R_n}(x) (x\cdot\xi) e^{\la x\cdot\xi}dx.
\end{equation}
The inequality $se^{-s}\leq \frac{1}{e}$, $s\geq0$ implies
\begin{equation}\label{sparing2}
  \biggl\lvert \int_{\Xi_\xi^-} a^+_{R_n}(x) (x\cdot\xi) e^{\la x\cdot\xi}dx
\biggr\rvert\leq \frac{1}{e}\int_{\Xi_\xi^-} a^+_{R_n}(x)dx\leq \frac{1}{e}.
\end{equation}
Since
\[
\int_{x\cdot\xi\leq 1}b(|x|)e^{\la \, x\cdot\xi}dx\leq e^{\la}<\infty, \quad \la>0,
\]
one has, by \eqref{slowdecaying}, that
\begin{equation}\label{sparing3}
    \int_{x\cdot\xi\geq1}b(|x|)e^{\la\, x\cdot\xi}dx=\infty, \quad \la>0.
\end{equation}
Then, by \eqref{sparing}, \eqref{sparing2}, \eqref{sparing3}, for all $\la\geq \la_1$,
\begin{align*}
&\frac{d}{d\la} \int_\X a^+_{R_n}(x) e^{\la x\cdot\xi}dx
\geq \int_{\Xi_\xi^+} a^+_{R_n}(x) (x\cdot\xi) e^{\la x\cdot\xi}dx-\frac{1}{e}\\
&\qquad\geq \int_{x\cdot\xi\geq1}a_{R_n}^+(x)e^{\la_1\, x\cdot\xi}dx-\frac{1}{e}\\
&\qquad\geq\int_{x\cdot\xi\geq1}b(|x|)\1_{B_{R_n}(0)}(x)e^{\la_1\, x\cdot\xi}dx-\frac{1}{e}
\to\infty, \quad n\to\infty,
\end{align*}
and the latter integral, evidently, does not depend on $\xi\in\S$.
Therefore, there exists $N_1=N_1(\la_1)\in\N$, such that, for all $n\geq N_1$ and for all $\xi\in\S$, the function $\A_\xi^{(n)}(\la)$ is increasing on $[\la_1,\la_2]$. As a result, for $\la\in[\la_1,\la_2]$, $n\geq N_1$, $\xi\in\S$,
\begin{multline*}
  \frac{1}{\la}\bigl(\ka^+\A_\xi^{(n)}(\la)-m\bigr)\geq
  \frac{1}{\la_2}\bigl(\ka^+\A_\xi^{(n)}(\la_1)-m\bigr)\\
  \geq \frac{1}{\la_2}\biggl(\ka^+\int_\X b(|x|)\1_{B_{R_n}(0)}(x)e^{\la_1\, x\cdot\xi}dx-m\biggr)
  \to\infty, \quad n\to\infty,
\end{multline*}
and, again, the latter expression does not depend on $\xi\in\S$, thus the convergence is uniform in $\xi$. Therefore, one gets \eqref{needtoproveC}, for a big enough $N>N_1$ and all $\la\in[\la_1,\la_2]$, $\xi\in\S$.

As a result, we have \eqref{mainaimacc}.
Take an arbitrary compact $\K\in\X$. Choose $n\in\N$ big enough to ensure that
\[
\max\limits_{x\in\K,\xi\in\S} x\cdot \xi < \min\limits_{\xi\in\S}c_*^{(n)}(\xi).
\]
As a result, $\K\in\inter(\Tau_1^{(n)})$, where $\Tau_1^{(n)}$ is defined according to \eqref{eq:TauT}, but for the kernels $a^\pm_{R_n}$. Then \eqref{convtotheta}, with $\Tauin=\K$, yields $\min\limits_{x\in t\K} w^{(n)}(x,t)=\theta$, $t\to\infty$. Hence the inequality \eqref{ineqtruncseq} fulfills the statement.
\end{proof}

\begin{corollary}\label{cor:nowaves}
  Let conditions of Theorem~\ref{thm:infinitespeed} hold. Then there does not exist  a traveling wave solution, in the sense of Definition~\ref{def:trw}, to the equation~\eqref{eq:basic}.
\end{corollary}
\begin{proof}
  Suppose that, for some $\xi\in\S$, $c\in\R$, and $\psi\in\M$, \eqref{trwv} holds. Then $u_0(x)=\psi(x\cdot\xi)$ satisfies the assumptions of Theorem~\ref{thm:infinitespeed}. Take a compact set $\K\subset\X$, such that $c_1:=\max\limits_{y\in\K}y\cdot\xi >c$. Then \eqref{convtothetaaccel} implies
  \begin{align*}
    \theta&=\lim_{t\to\infty} \essinf_{x\in t\K} \psi(x\cdot\xi -ct)=\lim_{t\to\infty} \essinf_{y\in \K} \psi\bigl(t(y\cdot\xi -c)\bigr)\\
    &=
    \lim_{t\to\infty} \psi\bigl(t(c_1 -c)\bigr)=0,
  \end{align*}
  where we used that $\psi$ is decreasing. One gets a contradiction which proves the statement.
\end{proof}

\enlargethispage{-2\baselineskip}

\section{Concluding remarks}\label{sec:concl}
\subsection{Historical comments}\label{subsec:history}
The solution $u=u(x,t)$ to \eqref{log} describes approximately a density (at the moment
of time $t$\ and at the position $x$ of the space $\X$) for a particle system
evolving in the continuum. In course of the evolution, particles might reproduce
themselves, die, and compete (say, for resources). Namely, a particle located
at a point $y\in\X$ may produce a `child' at a point $x\in\X$ with the intensity $\ka^+$ and according to the dispersion kernel $a^+(x-y)$. Next, any particle
may die with the constant intensity $m$. And additionally, a particle located
at $x$ may die according to the competition with the rest of particles; the
intensity of the death because of a competitive particle located at $y$ is equal to $\ka^-$ and the distribution of the competition  is described by $a^-(x-y)$.

This model was originally proposed in mathematical ecology, see \cite{BP1997},
and subsequent papers \cite{BP1999,DL2000,LMD2003,OC2006}; for further biological
references see e.g. \cite{OFKCBK2014} and the recent review \cite{PL2014}.
Rigorous mathematical constructions were done in \cite{FKK2009}, see also
\cite{FKK2011a, FKK2014}. The mathematical approach
was realized using the theory of Markov statistical dynamics on the so-called
configuration spaces expressed in terms of evolution of time-dependent correlation functions of the system, see e.g. \cite{FKO2009, KKM2008, KK2002}.

The particle density of such a system can not be described by a single evolution
equation in a closed form (except the case $\ka^-=0,$ which is out of our
considerations, see for it \cite{KKP2008}). Namely, the evolutional equation for
the density includes time-dependent correlations between pairs of particles,
whereas the evolutional equation  for that pair correlations includes correlations
between triples of particles, and so on. This situation is quite common
in statistical physics, see e.g. \cite{FKO2009} and the references therein;
in particular, cf. BBGKY-hierarchy for the Hamiltonian dynamics \cite{DSS1989}.

On the other hand, for purposes of applications, one needs to find though
approximate numeric values of the density. To do this, the so-called moment
closure procedure, which goes back at least to \cite{Whi1957}, was realized in \cite{BP1997, BP1999, DL2000, LMD2003}. The idea was to rewrite higher
order correlations as (nonlinear) combinations of low-order ones that yields
a closed system of (nonlinear) evolutional equations. Unfortunately, this
procedure had not any rigorous mathematical background and may be considered
informally only. Another problem was that different `closings' gave different
answers (even numerically). On the contrary, in \cite{OC2006}, a~mesoscopic-type
scaling, cf. \cite{Pre2009}, for the considered model was proposed. It~was
realized there (heuristically) for the case of homogeneous in space initial
density, that gave the homogeneous version \eqref{Ver} of \eqref{log}, for $u=u(t)$.

The nonhomogeneous equation \eqref{log} was rigorously derived in \cite{FKK2011a,
FKK2010c} from the dynamics of infinite particle systems described above, using the
so-called Vlasov scaling technique for continuous particle systems developed in \cite{FKK2010a}.
The infiniteness of the systems reflected in the fact that solutions to \eqref{log}
should be bounded but, in general, non integrable on the whole $\X$. The derivation was realized for all times, however, under conditions $m\gg \ka^+$ and
$Ca^-(x)\geq a^+(x)$, $x\in\X$,
with some $C>0$; and under the assumption that
the functions $a^\pm$ are symmetric: $a^\pm(-x)=a^\pm(x)$.
In the recent paper \cite{FKKozK2014}, the condition on $m$ was dropped, however, the equation \eqref{log} was derived
on a finite time-interval only. A rigorous derivation of \eqref{log} from
an infinite particle system under
the opposite assumption
$a^+(x)\geq C a^-(x)$, $x\in\X$,
which is crucial for the present paper (see, in particular, the discussion in Subsection~\ref{subsec:comparison_pr}), is still an open problem. Note also that the same question for finite systems, that leads to an integrable in space function $u$, should not require any comparison between kernels $a^+$ and $a^-$, cf.~\cite{FM2004}.

The relations between the particle density of the considered system and the solutions
to \eqref{Ver} and \eqref{log} were studied in \cite{OC2006} and \cite{OFKCBK2014},
correspondingly.

\subsection{Remarks and open problems}\label{subsec:remarks}

\begin{enumerate}[leftmargin=*]

  \item The assumption \eqref{infpos} in Theorem~\ref{thm:unibdd} does not need to be postulated at the origin only. It is enough to have the kernel $a^-$ separated from zero elsewhere. Clearly, if $a^-$ is (piecewise) continuous, then the condition \eqref{normed} yields such a property. Note also that the similar assumptions \eqref{as:a+nodeg} and \eqref{as:aplus-aminus-is-pos} cannot be weakened by choosing a neighborhood of an arbitrary point.

  \item We did not try to find an optimal upper bound for $u$ in \eqref{globalbound}. By~\eqref{Mchoice}, one can see that the upper bound we found is inversely proportional to the level $\alpha$ of the separation of $a^-$ from zero (and directly proportional to the initial value upper bound $\Vert u_0\rVert$). According to the previous remark, the highest level which $a^-$ achieves on $\X$ (uniformly on a ball) will give a better upper bound for $u$. Note that the kernel $a^-$ described a competition in a particle system mentioned in Subsection~\ref{subsec:history}; hence it is quite natural that a higher level of competition leads to a lower level of the system's density.

  \item The comparison principle shows, in particular, that any subsolution $\underline{u}$ to \eqref{eq:basic} (i.e. $\mathcal{F} \underline{u}\leq0$, where $\mathcal{F}$ is given by \eqref{Foper}) does not exceed a supersolution $\overline{u}$ (i.e.  $\mathcal{F} \overline{u}\geq0$). The sufficient condition for this, is the inequality \eqref{compare}, where we have a parameter $c>0$. For the case $c\geq\theta$, \eqref{compare} is the necessary condition to have a comparison (Remark~\ref{rem:necineq}). Namely, we have proved that a solution to \eqref{eq:basic} with an initial condition, which did not exceed $\theta$, becomes bigger than $\theta$ (that is a stationary solution to\eqref{eq:basic}), provided that \eqref{compare} fails somewhere. It is an open and interesting problem to study whether the solution will `return' (say, asymptotically) to $\theta$ in such a case (if, for example, the initial condition was compactly supported).

Stress that one has a majorant for a solution to \eqref{eq:basic} regardless of the condition \eqref{compare}. Indeed, neither Theorem~\ref{thm:unibdd} nor Proposition~\ref{prop:expdecay} required \eqref{compare} (note that Proposition~\ref{prop:expdecay} required though \eqref{aplusexpint1}). Therefore, the solution will be bounded uniformly in time by a constant inside a ball and exponentially (in time) decaying at infinity (in space).

\item The most part of the present paper requires that \eqref{compare} does hold with $c=\theta$. We have covered in Subsection~\ref{subsec:comparison_pr} the case \eqref{compare} with $c>\theta$ and the initial condition $u_0$ with $\lVert u_0\rVert\in (\theta,c]$. In particular, if, additionally, $\inf_\X u_0 >0$, then the solution to \eqref{eq:basic} will converge in $E$ to $\theta$ exponentially fast. The same result may be obtained if \eqref{compare} holds with $c=\theta$, however, $u_0$ may exceed $\theta$ (even everywhere) and separated from zero. The are going to present this in a forthcoming paper.

\item The case of \eqref{compare} with $c<\theta$ is less clear. As it was mentioned in a previous remark, one can get then a majorant for the solution. It is worth noting that one can reformulate Theorem~\ref{thm:compar_pr} by replacing \eqref{eq:init_for_compar} on $0\leq u_1(x,t)\leq c$, $0\leq u_2(x,t)\leq d$, $u_1(x,0)\leq u_2(x,0)$, $0<c\leq d$. Then it can be shown that $u_1(x,t)\leq u_2(x,t)$, in particular, one can expect to have a `useful' subsolution $u_1$ for a solution $u=u_2$.

\item Proposition~\ref{prop:u_gr_0} and Corollary~\ref{cor:lesstheta} constitute the maximum principle, whereas Theorem~\ref{thm:strongmaxprinciple} is usually addressed to the strong maximum principle, see e.g. \cite{CDM2008,CDM2008a}. There is an open problem whether a strong comparison principle does hold for the equation \eqref{eq:basic}, i.e. whether the strict inequality in \eqref{eq:max_pr_BUC:ineq} implies $u_1(x,t)<u_2(x,t)$; cf. e.g. \cite[Section 6 of Chapter 2]{Fri1964}.

\item It should be emphasized that though the assumption \eqref{as:a+nodeg-mod} is weaker than \eqref{as:a+nodeg}, it requires $r\geq0$, i.e. one excludes the situation in which $a^+(x)=0$, for all $x\in\X$ with $x\cdot\xi\geq 0$. Note that we needed \eqref{as:a+nodeg-mod} to get \eqref{toinfasltoinf} in the case of quickly decaying kernel $a^+$, i.e. when $\la_0(\check{a}^+)=\infty$. Next, for $a^+\in\Vxi$, \eqref{minspeed-spec} implies that, to have a traveling wave moving to the left, i.e. whose speed $c<0$, one can be looking for a kernel $a^+$ such that $\check{a}^+$ is `concentrated' on $\R_-$. However, \eqref{as:a+nodeg} with $r\geq0$ still has to hold.
Let, for example, $a^+(x)=(1+p)^{-1}\1_{x\cdot\xi\in[-p,1]}(x)$, where $p>0$. Then, clearly, $\la_0(\check{a}^+)=\infty$ hence $a^+\in\Vxi$. By \eqref{minspeed}, to show that $c_*(\xi)<0$ it is enough to prove that $\ka^+\A_\xi(\la)<m$, for some $\la>0$. One has
\[
\ka^+\A_\xi(\la) = \frac{\ka^+e^\la}{\la(1+p)}(1-e^{-(1+p)\la})<m,
\]
if we fix $\la>0$ and choose $p>0$ big enough, since the function $f(x)=x^{-1}(1-e^{-x})$ decreases monotonically to zero as $x\to\infty$.

\item The original Ikehara's Theorem dealt with increasing to infinity functions which allow a representation like \eqref{eq:sing_repres}, where $j=1$, $F$ is a constant and $H\equiv 0$, see e.g. \cite[Theorem~V.17]{Wid1941} (and the Laplace transform in \eqref{eq:sing_repres} was with $z$ replaced by $-z$, as $\varphi$ was increasing to infinity). A generalization for the case $j=2$ was obtained by Delange \cite{Del1954}; the general conditions were slightly similar to \eqref{eq:eta}, and more simple sufficient conditions were formulated. The latter required that the function $F$ in \eqref{eq:sing_repres} has to be analytic on the right closed strip $0<\Re z\leq\mu$. It was not appropriate for us, taking into account \ref{L-singular} and the `critical' case with $\la_*=\la_0(\check{a}^+)$, i.e. $a^+\in\Wxi$, see e.g. Example~\ref{ex:spefunc}. However, the most important point is that we needed to have an analog for this theorem for decaying to zero functions. The corresponding result was postulated in \cite[Proposition~2.3]{CC2004} (see also \cite{GW2008} and the proof of \cite[Theorem~1.6]{CDM2008}). However, the proof in \cite{CC2004} was proposed to be realized by a modification of \cite[Theorem~2.12]{EE1985}; the latter was just a formulation (without a proof) of Delange's results. Of course, such a modification might exist, however, it does not seem to be straightforward. In~particular, a decaying to zero function $\varphi$ may change the order of decay: for example, on some decreasing in length but placed arbitrary far intervals $\Delta_k$ it may decay as $e^{-a_k t^2}$, whereas on the complement to their union it decays as $e^{-\mu t}$. This might be enough to preserve the same singularity as in \eqref{eq:sing_repres}, however, the asymptotic \eqref{eq:trw_asympt} will not hold. To exclude such a case, it is enough to assume that $e^{\nu t}\varphi(t)$ is increasing, for a big enough $\nu$ (see the exact formulation in Proposition~\ref{prop:tauber}). Surely, for a traveling wave profile, it may be obtain from the equation \eqref{eq:trw} only (see Proposition~\ref{beincreas}).

\item The results of \cite[Theorems 1--3]{YY2013} partially cover the statements of Theorem~\ref{thm:trwexists}, Proposition~\ref{asymptexact}, and Theorem~\ref{thm:tr_w_uniq}, in the particular case, where $a^+=a^-$ and \eqref{aplusexpint1} holds for all $\mu>0$, i.e. $\la_0(\check{a}^+)=\infty$. Unfortunately, the proofs in \cite{YY2013} contain a lot of misprints and several gaps; in particular, the proof of the crucial for the uniqueness \cite[Lemma~4.2]{YY2013} is insufficient, cf. the Step~3 in the proof of Theorem~\ref{thm:tr_w_uniq} of the present paper.

\item It is worth noting that the usage of Proposition~\ref{prop:expdecay} (which, we recall, does not require any comparison like \eqref{as:aplus_gr_aminus} between kernels) allowed us to describe the decaying of a solution $u$ to \eqref{eq:basic} `outside' of the set $\Tau_{1,\xi}$ (Theorem~\ref{thm:decayoutsidedirectional}). In particular, the initial condition $u_0$ did not need to be compactly supported. This yielded the corresponding result for the set $\Tau_1$ (Theorem~\ref{thm:outoffront}). On Figure~\ref{fig:f1}, we sketched a relation between
$\Tau_{1,\xi}$ and $\Tau_1$. The arrows describe the motion of these sets to `become' $\Tau_{t,\xi}$ and $\Tau_t$. Recall that, by Proposition~\ref{prop:front_is_non-empty}, $\ka^+\m\in\Tau_1$, however, the origin may be out of $\Tau_1$ or even $\Tau_{1,\xi}$, for some $\xi\in\S$. For the latter, recall, however, that, by Remark~\ref{rem:minspeedopospos}, the origin must belong at least to one of the sets $\Tau_{1,\xi}$, $\Tau_{1,-\xi}$, for any $\xi\in\S$. Recall also that $0\in\inter{\Tau_1}$ (which does hold, if e.g. $a^+(-x)=a^+(x)$, $x\in\X$, then $\m=0$) imply that all traveling waves move to their `right directions', cf. Remarks~\ref{rem:evenimplypos}, \ref{rem:zero_is_in_front}.

\begin{figure}[hb!]
  \centering
   \begin{tikzpicture}[xscale=0.33,yscale=0.33,line width=1pt,>=stealth]
\draw[pattern=north west lines, pattern color=blue!40!white] plot [smooth cycle,tension=0.4] coordinates {(2,2) (1,4) (2,7) (3,8) (5,9) (9,9) (12,7) (13,6) (14,4) (13,2) (12,1) (11,0.7) (10,0.6) (9,0.5) (8,0.6) (7,0.7) (6,0.8) (5,0.9) (4,1) (3,1.2)};
\node (T) at (9,6) {$\Tau_1$};
\fill[black] (5,3) circle (1.5ex) node[right] {$\ka^+\m$};
\draw (19,0) -- (5,16);
\fill[pattern=north east lines, pattern color=green!20!white] (19,0) -- (5,16) -- (0,16) -- (0,-2) -- (19,-2) -- cycle;
\node (T1) at (7,11) {$\Tau_{1,\xi}$};
\draw[->] (12.5,-2) node[below] {$O$} -- (17.1,2.2) node[right] {$c_*(\xi)\xi$};
\fill[black] (12.5,-2)  circle (1.5ex);
\draw[->] (12.5,-2) -- (14.5,{4.2*2/4.6-2}) node[below] {$\xi$};
\foreach \x in {6,...,16}
    \draw[->] (\x,{-8*\x/7+8*19/7}) -- ({\x+1}, {7*(\x+1)/8+(-8*\x/7+8*19/7)-7*\x/8});
\node[right,align=center] (F) at ({12+1.5}, {7*(12+1)/8+(-8*12/7+8*19/7)-7*12/8}) {front  propagation\\ in a direction $\xi$};
\draw[->] (9,9)--(9,10);
\draw[->] (5,9)--(5,10);
\draw[->] (11,0.7)--(11,-0.3);
\draw[->] (8,0.6)--(8,-0.4);
\draw[->] (5,0.9)--(5,-0.1);
\draw[->] (1,4)--(0,4);
\draw[->] (13,2)--(14,1.2);
\draw[->] (2,2)--(1,1);
\draw[->] (2,7)--(0.8,7.6);
\node[below,align=center] (H) at (7,-0.1) {front  propagation};
\end{tikzpicture}
  \caption{Relationship between the sets $\Tau_{1,\xi}$ and $\Tau_1$}
  \label{fig:f1}
\end{figure}
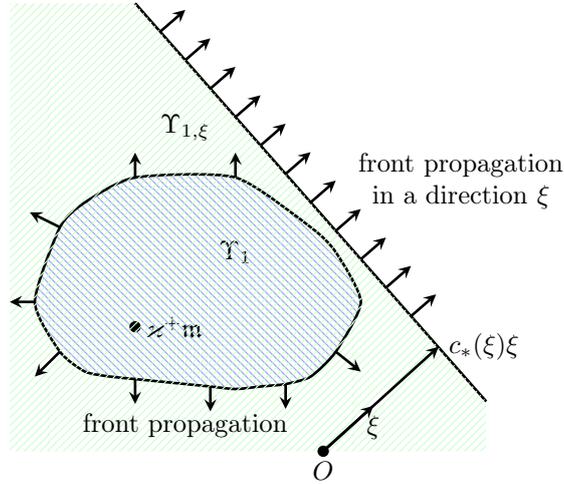

\item The notion of `front' has several slightly different definitions, see e.g. \cite{PS2005,Xin2009}. Informally, front for \eqref{eq:basic} has to be a set which separates a part $\Tauin$ of $\X$, where $u(tx,t)\to\theta$, $x\in\Tauin$, $t\to\infty$ and a part $\Tauout$ of $\X$, where $u(tx,t)\to0$, $x\in\Tauout$, $t\to\infty$. Thus one is interested to have the front as `thin' as possible (one expects to have a surface). Roughly speaking, the front itself is a set, where we do not know the corresponding long-time behavior. The results of Theorems~\ref{thm:outoffront} and \ref{thm:convtotheta} show that any $\eps$-neighborhood of the boundary of $\Tau_1$ can be considered as a front set in the meaning above.
Figures~\ref{fig:f2}, \ref{fig:f3} describe two `projections' of the three-dimensional graph for $u=u(x,t)$.

\begin{figure}[ht!]
  \centering
  \begin{tikzpicture}[xscale=0.5,yscale=0.5,line width=1pt,every node/.style={font=\small},declare function={u(\x)=(8-0.4)*exp(-0.025*\x*\x);},>=stealth]
  \pgfmathsetmacro{\bb}{10};
	\pgfmathsetmacro{\th}{8};
  \draw[dashed, line width=0.7pt,->] ({-(\bb+0.5)},0)--({\bb+1},0) node[below] {$\xi$};
  \draw[dashed, line width=0.7pt,->] (-1,-1)--(-1,{\th+1.5}) node[left] {$u$};
  \draw[domain=-10:10,smooth,variable=\x,blue] plot (\x,{u(\x)});
  \draw({-(\bb+0.5)},\th)--({\bb+0.5},\th);
  \node[above left] (T) at (-1,\th) {$\theta$};
  \pgfmathsetmacro{\cxi}{\bb-3};
  \pgfmathsetmacro{\bbm}{-\bb+1};
  \pgfmathsetmacro{\cxim}{-\bb+4};
  \pgfmathsetmacro{\cximmm}{-\bb+4};
  \foreach \x in {\cxi,...,\bb}{
    \draw[->,line width=0.5pt] (\x,{u(\x)})--(\x,0.1);
    \draw[->,line width=0.5pt] ({\x-0.5},{u(\x-0.5)})--({\x-0.5},0.1);}
  \foreach \x in {\bbm,...,\cximmm}{
    \draw[->,line width=0.5pt] ({\x},{u(\x)})--({\x},0.1);
    \draw[->,line width=0.5pt] ({\x-0.5},{u(\x-0.5)})--({\x-0.5},0.1);}
    \pgfmathsetmacro{\cximm}{\cxim+2.5};
    \pgfmathsetmacro{\cxin}{\cxi-3}
  \foreach \x in {\cximm,...,\cxin}{
    \draw[->,line width=0.5pt] (\x,{u(\x)})--(\x,{\th-0.1});
    \draw[->,line width=0.5pt] ({\x+0.5},{u(\x+0.5)})--({\x+0.5},{\th-0.1});
  }
  \fill[black] ({\cxi-1.5},0) circle(0.7ex) node[below] {$t c_*(\xi)\xi$};
  \fill[black] ({\cxim+1.3},0) circle(0.7ex) node[below] {$-t c_*(-\xi)\xi$};
  \draw[dashed,line width=0.7pt] ({\cxi-1.5},0)--({\cxi-1.5},\th);
  \draw[dashed,line width=0.7pt] ({\cxim+1.3},0)--({\cxim+1.3},\th);
  \draw[dashed,line width=0.7pt,->] ({\cxim+1.3},{u(\cxim+1.3)}) --({\cxim+2.1},{u(\cxim+1.3)}) node[right] {$\eps t\xi$};
  \draw[dashed,line width=0.7pt,->] ({\cxim+1.3},{u(\cxim+1.3)}) --({\cxim+0.6},{u(\cxim+1.3)}) node[left] {$-\eps t\xi$};
  \draw[dashed,line width=0.7pt,->] ({\cxi-1.5},{u(\cxi-1.5)}) --({\cxi-0.8},{u(\cxi-1.5)}) node[right] {$\eps t\xi$};
  \draw[dashed,line width=0.7pt,->] ({\cxi-1.5},{u(\cxi-1.5)}) --({\cxi-2.2},{u(\cxi-1.5)}) node[left] {$-\eps t\xi$};
  \draw[dashed,line width=0.7pt] ({\cxim+2.1},0)--({\cxim+2.1},\th);
  \draw[dashed,line width=0.7pt] ({\cxim+0.6},0)--({\cxim+0.6},\th);
  \draw[dashed,line width=0.7pt] ({\cxi-0.8},0)--({\cxi-0.8},\th);
  \draw[dashed,line width=0.7pt] ({\cxi-2.2},0)--({\cxi-2.2},\th);
      \end{tikzpicture}
 \caption{Space-value diagram}
  \label{fig:f2}
\end{figure}
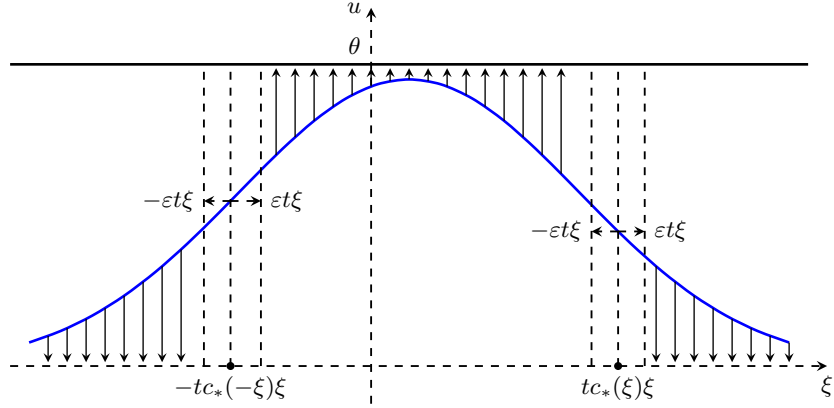

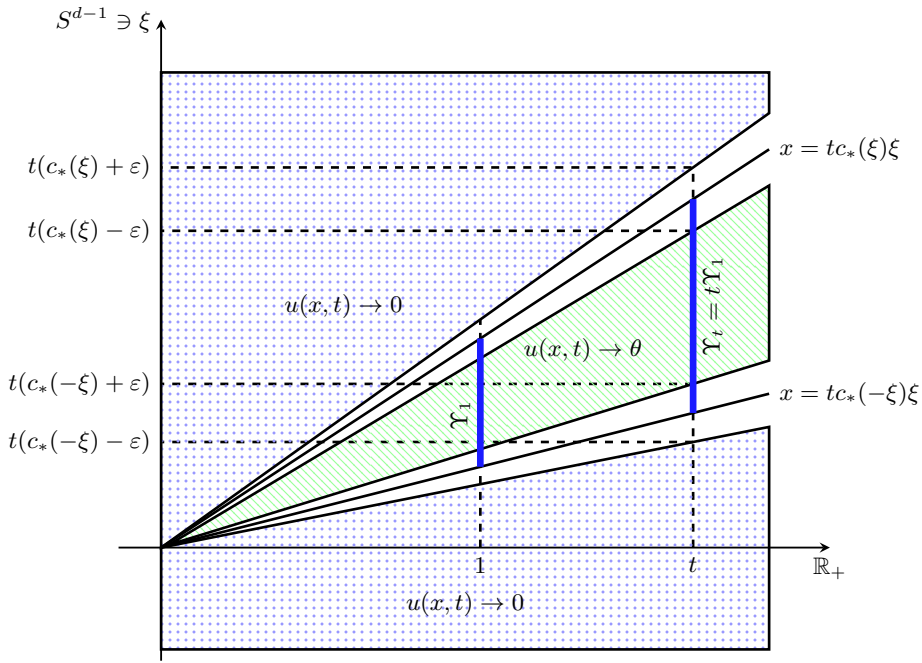
\begin{figure}[ht!]
  \centering
  \begin{tikzpicture}[xscale=1,yscale=1,line width=1pt,every node/.style={font=\small},>=stealth]
	\pgfmathsetmacro{\r}{8};
	\pgfmathsetmacro{\p}{7};
	\pgfmathsetmacro{\pp}{-1.5};
	\pgfmathsetmacro{\ra}{0.2};
	\pgfmathsetmacro{\rb}{0.31};
	\pgfmathsetmacro{\rc}{0.6};
	\pgfmathsetmacro{\rd}{0.72};
	\coordinate (O) at (0,0);
	\coordinate (T) at ({\r*1.1},0);
	\coordinate (X) at (0,\p);
	\draw[->,line width=0.7pt] (0,\pp) -- (X);
	\draw[->,line width=0.7pt] ({-\r*0.07},0) -- (T);
      \coordinate (A) at (\r,{\r*\ra});
      \coordinate (B) at (\r,{\r*\rb});
      \coordinate (C) at (\r,{\r*\rc});
      \coordinate (D) at (\r,{\r*\rd});
      \coordinate (E1) at (\r,{0.9*\p});
      \coordinate (E2) at (0,{0.9*\p});
      \coordinate (F1) at (\r,{0.9*\pp});
      \coordinate (F2) at (0,{0.9*\pp});
      \pgfmathsetmacro{\q}{7};
      \coordinate (G) at (\q,0);
      \coordinate (GA) at (\q,{\q*\ra});
      \coordinate (GB) at (\q,{\q*\rb});
      \coordinate (GC) at (\q,{\q*\rc});
      \coordinate (GD) at (\q,{\q*\rd});
      \coordinate (GAO) at (0,{\q*\ra});
      \coordinate (GBO) at (0,{\q*\rb});
      \coordinate (GCO) at (0,{\q*\rc});
      \coordinate (GDO) at (0,{\q*\rd});
      \draw [dashed] (G) -- (GD) -- (GDO);
      \draw [dashed] (GA) -- (GAO);
      \draw [dashed] (GB) -- (GBO);
      \draw [dashed] (GC) -- (GCO);
      \draw[pattern=north west lines, pattern color=green!40!white] (0.1,{0.04*(\rb+\rc)})--(B)--(C)--cycle;
      \draw[pattern=dots, pattern color=blue!40!white] (D)--(E1)--(E2)--(O)--cycle;
      \draw[pattern=dots, pattern color=blue!40!white] (A)--(F1)--(F2)--(O)--cycle;
      \node [left] at (GAO) {$t(c_*(-\xi)-\varepsilon)$};
      \node [left] at (GBO) {$t(c_*(-\xi)+\varepsilon)$};
      \node [left] at (GCO) {$t(c_*(\xi)-\varepsilon)$};
      \node [left] at (GDO) {$t(c_*(\xi)+\varepsilon)$};
      \node [left] at (X) {$S^{d-1}\ni\xi$};
       \node [below] at (T) {$\mathbb{R}_+$};
        \node [below] at (G) {$t$};
        \draw (O) -- (\r,{\r*(\ra + \rb)/2}) node [right] {$x=t c_*(-\xi)\xi$};
        \draw (O) -- (\r,{\r*(\rc + \rd)/2}) node [right] {$x=t c_*(\xi)\xi$};
      \coordinate (K1) at ({0.3*\r},{0.48*\r*(\rb+\rc)});
      \node [below] at (K1) {$u(x,t)\to 0$};
      \coordinate (K2) at ({0.5*\r},{0.5*\pp});
      \node at (K2) {$u(x,t)\to 0$};
      \coordinate (K3) at ({0.58*\r},{0.36*\r*(\rb+\rc)});
      \node[right] at (K3) {$u(x,t)\to \theta$};
      \coordinate (t1) at (\q,{\q*(\ra + \rb)/2});
      \coordinate (t2) at (\q,{\q*(\rc + \rd)/2});
    \coordinate (t3) at (\q,{\q*(\ra + \rb+\rc + \rd)/4});
      \draw[line width=2.5pt,color=blue!90!white] (t1) -- (t2);
    \node[below,rotate=90] at (t3) {$\Tau_t=t\Tau_1$};
      \pgfmathsetmacro{\u}{{0.6*\q}};
     \node[below] (U) at (\u,0) {$1$};
    \draw[dashed] (U) -- (\u,{\u*\rd});
    \draw[line width=2.5pt,color=blue!90!white] (\u,{\u*(\ra+\rb)/2}) -- (\u,{\u*(\rc+\rd)/2});
    \node[above,rotate=90] at (\u,{0.46*\u*(\rb+\rc)}) {$\Tau_1$};
      \end{tikzpicture}
 \caption{Space-time diagram}
  \label{fig:f3}
\end{figure}

\item Theorem~\ref{thm:infinitespeed} states that, for a heavy-tailed kernel $a^+$ (when the Mollison condition \eqref{aplusexpint1} fails, for all $\mu>0$), the front becomes unbounded. In other words, the speed of propagation is not constant anymore. We are going to realise detail analysis of this topic in a forthcoming paper.

\item It was shown  recently in \cite{HR2010}, that a fast propagation in the classical Fisher--KPP equation \eqref{kpp} may be obtained provided that the initial condition decays slowly. Taking into account Remark~\ref{rem:ads1}, one can postulate the same question, for the equation \eqref{eq:basic}, as an open problem.

\item In the proof of Proposition~\ref{uniqstationarysolutions}, we did not use the assumption $0\in\inter(\Tau_1)$ to show that a nonnegative stationary solution to \eqref{eq:basic} has to be bounded by $\theta$. However, if $0\notin\inter(\Tau_1)$, then, for some $\xi\in\S$, $c_*(\xi)<0$ and, therefore, there exists a traveling wave in the direction $\xi$ with the zero speed; it evidently generates a stationary solution. It is unknown whether another stationary solutions do exist in such a case.

\item Finally, we would like to mention that an exact behavior of $u(tx,t)$ in a neighborhood of the boundary of $\Tau_1$ seems to be very delicate open problem. It has to be related to the more easy (but still open) questions about the stability of traveling waves and the convergence of solutions to the equation \eqref{eq:basic} to its traveling waves, cf. \cite{Bra1983}; the latter convergence was originally, in \cite{KPP1937}, the reason to introduce the notion of traveling waves at all.
\end{enumerate}

\section*{Acknowledgements}
Financial support of DFG through CRC 701, Research Group ``Stochastic
Dynamics: Mathematical Theory and Applications'' is gratefully acknowledged.

\end{document}